%% file: ThesisIN.tex
\pdfoutput=1
\documentclass[11pt,a4paper]{report}      %% LaTeX2e document.
\usepackage{warwickthesis,setspace,graphicx}     %!  setspace is used to control linepacing
\usepackage[numbers]{natbib}
\usepackage{enumerate}  %! used for the library form, but you might find it useful too.
\usepackage{amsmath}
\usepackage[dvipsnames]{xcolor}
\usepackage{amssymb}
\usepackage{bbm}
\usepackage{enumerate}
\usepackage{caption}
\usepackage{amsthm}
\usepackage{relsize}
\usepackage{hyperref}
\usepackage{xspace}
\usepackage{csquotes}
\usepackage{tikz}
\newtheorem{theorem}{Theorem}[section]
\newtheorem{lemma}[theorem]{Lemma}
\newtheorem{definition}[theorem]{Definition}
\newtheorem{proposition}[theorem]{Proposition}
\newtheorem{corollary}[theorem]{Corollary}
\newtheorem{remark}[theorem]{Remark}
\newtheorem{example}[theorem]{Example}
\newtheorem{conjecture}[theorem]{Conjecture}
 \leftchapter                       %% Uncomment one of these if you want
\renewcommand{\thesisdepartmentname}{Department of Statistics}    %! The name of
\renewcommand{\thesissubmission}{Submitted to the University of Warwick\\
                        for the degree of}
\renewcommand{\thesisauthor} {Ioanna Nteka}    %% Your official name.
\renewcommand{\thesisauthorno}{1159210}  %! your university number, used on the library copyright page.

\renewcommand{\thesismonth}{September}     %% Your month of graduation.

\renewcommand{\thesisyear}{2016}      %% Your year of graduation.

\renewcommand{\thesistitle}{Positive temperature dynamics on Gelfand-Tsetlin patterns restricted by wall}     %% The title of your thesis; use
\renewcommand{\thesisauthorpreviousdegrees}{....}

\DeclareMathAlphabet{\mathpzc}{OT1}{pzc}{m}{it}

\begin{document}

%\thesiscopyrightpage                 %! Generate the copyright page for the library.
%%%%%%%%%%% \thesiscopyrightpagehardcopyonly This only applies for a masters thesis that will not go online.

%%* Uncomment a ttitle page.
 \thesistitlepage                     %% Generate the title page.
%\thesistitlecolourpage           %! Generates a COLOUR title page.

%%* Start roman page numbering here for contents, etc
\pagenumbering{roman} %! Begins roman numerals start from page i.

\tableofcontents                     %% Generate table of contents.
% \listoftables                      %% Uncomment this to generate list
                                     %% of tables.
% \listoffigures                     %% Uncomment this to generate list
                                     %% of figures.

\begin{thesisacknowledgments}        %% Use this to write your
%  \input ack.tex                    %% acknowledgments; it can be anything
                                     %% allowed in LaTeX2e par-mode.

                                     %! This following is not needed, but you may like to add it.
%This \lowercase\expandafter{\thesistype} was typeset with
%\LaTeXe\footnote{\LaTeXe{} is an extension of \LaTeX. \LaTeX{} is
%a collection of macros for \TeX. \TeX{} is a trademark of the
%American Mathematical Society. The style package {\em warwickthesis} was
%used.} by \thesistypist.
I would like to thank my supervisors Dr Jonathan Warren and Dr Nikolaos Zygouras for the invaluable guidance. I am also grateful to Dr Bruce Westbury for mentioning the Berele correspondence on which my work on Chapter 4 was based. Finally, I would like to thank the reviewers, Dr Sigurd Assing and Dr Sergey Oblezin, for carefully reading my manuscript and for giving detailed comments and suggestions.
\end{thesisacknowledgments}

\begin{thesisdeclaration}        %! Use this to declare the extent of
                 %! the original work,
                 %! collaboration, other published
                                 %! material etc.it can be anything
                                 %% allowed in LaTeX2e par-mode.
%Replace this text with a declaration of the extent of the original work,
%collaboration, other published material etc. You can use any \LaTeX\
%constructs.
I hereby declare that this Ph.D. thesis entitled \enquote{\textbf{Positive temperature dynamics on Gelfand-Tsetlin patterns restricted by wall}} was carried out by me for the degree of Doctor of
Philosophy in Statistics under the guidance and supervision of Dr Jonathan Warren and Dr Nikolaos Zygouras. None of this work have been previously published or submitted for a degree or any other qualification at this University or to any other institute. Work of others I have made use of is acknowledged at the respective place in the text.
\end{thesisdeclaration}

\begin{thesisabstract} 
The thesis focuses on processes on symplectic Gelfand-Tsetlin patterns. In chapter 4, a process with dynamics inspired by the Berele correspondence \cite{Berele_1986} is presented. It is proved that the shape of the pattern is a Doob $h$-transform of independent random walks with $h$ given by the symplectic Schur function. This is followed by an extension to a $q$-weighted version. This randomised version has itself a branching structure and is related to a $q$-deformation of the $\mathfrak{so}_{2n+1}$-Whittaker functions. In chapter 5, we present a fully randomised process. This process $q$-deforms a process proposed in \cite{Warren_Windridge_2009}. In chapter 7 we prove the convergence of the $q$-deformation of the $\mathfrak{so}_{2n+1}$-Whittaker functions to the classical $\mathfrak{so}_{2n+1}$-Whittaker functions when $q\to 1$. Finally, in chapter 8 we turn our interest to the continuous setting and construct a process on patterns which contains a positive temperature analogue of the Dyson's Brownian motion of type $B/C$. The processes obtained are $h$-transforms of Brownian motions killed at a continuous rate that depends on their distance from the boundary of the Weyl chamber of type $B/C$, with $h$ related with the $\mathfrak{so}_{2n+1}$-Whittaker functions.

              %% Use this to write your thesis
                                     %% abstract; it can be anything
                                     %% allowed in LaTeX2e par-mode.
%!  \begin{singlespace}       %! uncomment this if you need single spacing
%   \input abstract.tex       %!           don't forget the end spacing!
                                     %! It must fit on one page.
                                     %! single spacing and smaller
                                     %! font size
                                     %!  is allowed here.
%!   \end{singlespace}
\end{thesisabstract}

%\begin{thesisabbreviations}       %! Use this to give a list of
                                   %! abbreviateons
                                   %! It can be anything
%\end{thesisabbreviations}         %! allowed in LaTeX2e par-mode.
                                   %!The following may be useful':
                     %!\begin{itemize}
                     %!     \item[symbol]descriptive text..
                     %!\end{itemize}

%\end{thesisabbreviations}
%!!!!!!!!!!!!!!!                     %% Begin your thesis text here; follow
                                     %% the report style and group your text
                                     %% in chapters, sections, etc. eg:
%%* don't need this with one-sided printing
%\newpage{\pagestyle{empty}\cleardoublepage} %! ensure that Chapter 1 starts on an odd
                                           %! page when using double sided printing.
%%* Start arabic numbering of main text here
\pagenumbering{arabic} %! Begins arabic numerals start from page 1.

\chapter{Introduction}
\input{tex/introduction}

\chapter{The Macdonald-Koornwinder polynomials}
\label{Koornwinder}
\input{tex/Koornwinder}

\chapter{Probabilistic prerequisites}
\label{prerequisites}
\input{tex/intertwining}

\chapter{The Berele insertion algorithm}
\label{Berele}
\input{tex/berele.tex}

\chapter{A $q$-deformed Markov process on symplectic Gelfand Tsetlin patterns}
\label{discrete}
\input{tex/discrete}

\chapter{Calculating the law of processes on patterns}
\label{q-whittaker-process}
\input{tex/q-whittaker_processes}

\chapter{On a classical limit of the $q$-deformed $\mathfrak{so}_{2n+1}$-Whittaker functions}
\input{tex/whittaker_processes}

\chapter{A positive temperature analogue of Dyson's Brownian motion with boundary}
\input{tex/continuous}

%\begin{verbatim}\citet[chap. 2]{ballentine82}|
%\end{verbatim}
%for a textual one, as \citet[chap. 2]{ballentine82}.\\
% \\
%\begin{verbatim}\citep{abraham_etal}
% \end{verbatim}
% for a parenthetical citation \citep{abraham_etal},\\
%
% \begin{verbatim}\citep*{MTW}
% \end{verbatim}
% for a full list of authors use a * parenthetical citation \citep*{MTW},\\
% \\

%!!!!!!!!!!!!!!!

%  \appendix                            %% this will do the appendices
%  \chapter{Proof of Fred's theorem}
%  \input{app1.tex}
%  \chapter{listing of Fred's program}
%  \input{app2.tex}

%\bibliographystyle{plainnat}
\bibliographystyle{alpha} 
%\bibliography{sample}            %% Start your bibliography here;
                                 %! with sample.bib as your bibliography file. You can
                               %% also use:
                %! \begin{thebibliography}
                %!    \bibitem{etc....
                %! \end{thebibliography}
                               %% to generate your bibliography.

\bibliography{refs}
%\begin{thesisauthorvita}             %% Write your vita here; it can be
%                                     %% anything in LaTeX2e par-mode.
%\end{thesisauthorvita}               %%

\end{document}

%% file: tex/introduction.tex
In recent years there has been intense activity around the interplay of probabilistic models and algebraic/integrable structures. Integrable probability is concerned with the study of the structure of probabilistic models combining ideas and techniques from algebraic combinatorics, representation theory, theory of symmetric functions, integrable systems etc.

The onset of this activity was the work of Baik-Deift-Johansson \cite{Baik_Deift_Johansson_1999}, where the authors studied the distribution of the length of the longest increasing subsequence in a random permutation and proved its relation with the \emph{Tracy-Widom GUE} distribution, namely the distribution of the largest eigenvalue of a random matrix from the \emph{Gaussian Unitary Ensemble}. The Gaussian Unitary Ensemble GUE$(N)$ is described by a probability measure on the set of Hermitian matrices $H=\{H(i,j)\}_{1\leq i,j\leq N}$ with density given by
\[\dfrac{1}{Z_{GUE(N)}}e^{-\frac{N}{2}trH^2}\]
where $Z_{GUE(N)}$ is the normalising constant and $tr$ denotes the \emph{trace} of a matrix.

The Robinson-Schensted correspondence, is a combinatorial bijection between permutations of integers and a pair of Young tableaux $(P,Q)$ \cite{Robinson_1938, Scensted_1961, Stanley_2001}. The longest subsequence of a permutation is then given by the length of the first row of the two tableaux providing a connection between random matrices and Young tableaux.

Further connections of probabilistic models to algebraic structures were due to Okounkov \cite{Okounkov_2001} and Okounkov-Reshetikhin \cite{Okounkov_Reshetikhin_2003}, where the \emph{Schur functions}, namely the characters of the symmetric group, play an important role. Okounkov and Reshetikhin defined the Schur process to be a probability measure on sequences of Young diagrams given in terms of Schur functions.

A parallel approach in the spirit of Okounkov-Reshetikhin was developed in the Macdonald processes setting. A complete review of processes on Young diagrams is due to Borodin-Corwin \cite{Borodin_Corwin_2011} where the authors study probability measures on Young tableaux. The main object for the study of these models is the Macdonald polynomials, which are generalisation of the Schur functions introduced by Macdonald \cite{Macdonald_2003}. An example of such processes comes from the $q$-deformation of the Robinson-Schensted correspondence proposed by O'Connell-Pei \cite{O'Connell_Pei_2013}. We refer to the work of Borodin-Petrov \cite{Borodin_Petrov_2014} for a full classification of dynamics that preserve the class of Macdonald processes. 

In the continuous setting, the main object of interest for us is the O' Connell-Yor semi-discrete random polymer \cite{O'Connell_Yor_2001} and a path transformation of Brownian paths proposed by O' Connell \cite{O'Connell_2012} that generalises Matsumoto-Yor's process \cite{Matsumoto_Yor_1999,Matsumoto_Yor_2000} in the type $A$ setting. In the core of these models we find the $\mathfrak{gl}_{N}$-Whittaker functions. 

These works correspond to processes associated with root systems of type $A$. We extend this setting in studying dynamics of particles that correspond and are linked to root systems of type $B/C$. Adopting the Gelfand-Tsetlin framework this corresponds to looking at dynamics on particles restricted by a wall. This brought us over to objects like \emph{Macdonald-Koornwinder polynomials}, and more specifically \emph{$q$-Whittaker functions of type B} (see e.g. \cite{Koornwinder_1992}, \cite{vanDiejen_1999}, \cite{vanDiejen_Emsiz_2015a}, \cite{vanDiejen_Emsiz_2015b}) in the discrete setting and \emph{Whittaker of type B} (see e.g. \cite{Gerasimov_et_al_2012}) in the continuous setting. Many of these objects had already appeared in the literature but we arrived to these following a probabilistic paths.

In the rest of the Introduction we will present some models mainly in the type $A$ setting in order to motivate the study of the models that appear in the main body of this thesis.
% discrete
\section{Examples of processes on Gelfand-Tsetlin patterns in the zero-temperature setting}
Let us start with the simplest example of a discrete-space process, namely an $N$-dimensional random walk. Suppose we are instead interested in an $N$-dimensional random walk killed upon exiting the set
\[\mathcal{W}^N_{\mathbb{Z}} = \{z=(z_{1},...,z_{n})\in \mathbb{Z}^N: z_{1}\geq ...\geq z_{N}\}\]
which is conditioned not to be killed. The formalism for the \enquote{conditioning} is via Doob's $h$-transform and we will explain this concept more in Chapter \ref{prerequisites}.

In order to construct such a process we need a function which is harmonic for the generator of the killed process. We fix a parameter $a=(a_{1},...,a_{N})\in \mathbb{R}^N_{>0}$. Let us consider the process $Z=(Z_{1},...,Z_{N})$ which evolves as follows. The particle $Z_{i}$ waits, independently of the other particles, an exponential waiting time with parameter $a_{i}$ and then jumps to the right. 

The Schur function $S_{x}^{(N)}$, which is the character of the unitary group, is a symmetric function parametrized by $x \in \mathcal{W}^N_{\mathbb{Z}}$ defined as follows
\[S_{x}^{(N)}(a)=\dfrac{\det (a_{j}^{x_{i}+N-i})_{1\leq i,j \leq N}}{\det (a_{j}^{N-i})_{1\leq i,j\leq N}}.\]
These functions are known to satisfy the Pieri formula (see for example \cite{Stanley_2001}, cor. 7.15.3)
\[\sum_{i=1}^N \mathbbm{1}_{\{z+e_{i}\in \mathcal{W}^N_{\mathbb{Z}}\}}S^{(N)}_{z+e_{i}}(a)=\sum_{i=1}^Na_{i}S^{(N)}_{z}(a)\]
where $e_{1},...,e_{N}$ denotes the standard basis of $\mathbb{Z}^N$. The Pieri formula implies that for the function $h_{N}(z)=a_{1}^{-z_{1}}...a_{N}^{-z_{N}}S_{z}^{(N)}(a)$, the following identity holds
\[\sum_{i=1}^N \mathbbm{1}_{\{z+e_{i}\in \mathcal{W}^N_{\mathbb{Z}}\}}a_{i}h_{N}(z+e_{i})=\sum_{i=1}^Na_{i} h_{N}(z)\]
and hence the function $h_{N}$ is harmonic for the generator of the $N$-dimensional random walk killed when it exits the set $\mathcal{W}^N_{\mathbb{Z}}$.

Then the $h$-transform of $Z$ is a Markov process with state space $\mathcal{W}^N_{\mathbb{Z}}$ whose transition rates are given by
\begin{equation}
Q(z,z+e_{i})=a_{i}\dfrac{h_{N}(z+e_{i})}{h_{N}(z)}\mathbbm{1}_{\{z+e_{i}\in \mathcal{W}^N_{\mathbb{Z}}\}}=\dfrac{S^{(N)}_{z+e_{i}}(a)}{S^{(N)}_{z}(a)}\mathbbm{1}_{\{z+e_{i}\in \mathcal{W}^N_{\mathbb{Z}}\}}, \text{ for }1\leq i \leq N
\label{eq:Schur_process_rates(intro)}
\end{equation}
and $Q(z,z)=-\sum_{i=1}^N a_{i}$.

Moving to higher dimensional objects we may define the \emph{integer-valued Gelfand-Tsetlin pattern} to be a collection of points $\mathbf{z}=(z^1,...,z^N)$ where $z^k \in \mathbb{Z}^k$, satisfying the interlacing condition
\[z^{k+1}_{i+1}\leq z^k_{i}\leq z^{k+1}_{i}, \text{ for }1\leq i \leq k < N.\]
An integer-valued Gelfand-Tsetlin pattern is in one-to-one correspondence with the Young tableau (\cite{Berenshtein_Zeleviskii_1988}). Let us denote by $\mathbb{K}^N_{\mathbb{Z}}$ the set of integer-valued Gelfand-Tsetlin patterns with $N$ levels. We observe that $\mathbb{K}^N_{\mathbb{Z}}$ essentially consists of nested elements of $\{\mathcal{W}^k_{\mathbb{Z}},1\leq k \leq N\}$. There are several probabilistic models studied in the literature with state space $\mathbb{K}^N_{\mathbb{Z}}$. We present here only a few of those.

The first example was proposed by Warren-Windridge in \cite{Warren_Windridge_2009}. The particles perform independent simple random walks, but a jump of a particle $Z^k_{i}$ is suppressed if otherwise it would land outside of the interval 
\begin{equation*}
\mathcal{I}^k_{i}:=\left\{ \begin{array}{ll}
\lbrack Z^{k-1}_{1}, \infty) & \text{ if }i=1\\
\lbrack Z^{k-1}_{i}, Z^{k-1}_{i-1} \rbrack & \text{ if }1<i<k\\
( -\infty,Z^{k-1}_{k-1}\rbrack & \text{ if } i=k
\end{array} \right. 
\end{equation*}
Moreover, there is a pushing mechanism which ensures the interlacing condition with the levels below.

We formally define the process $\mathbf{Z}$ as follows.
\begin{definition}\label{Warren-Windridge-process}
Let $a = (a_{1},...,a_{N})\in \mathbb{R}_{>0}^N$ be fixed parameter. The continuous-time Markov process $\mathbf{Z}=(\mathbf{Z}(t),t\geq 0)$ has state space $\mathbb{K}^N_{\mathbb{Z}}$ and evolves as follows. The particle $Z^k_{i}$ waits, independently of the other particles, an exponential time with parameter $a_{k}$ and then attempts a rightward jump. If, for the jump time $T$, $Z^k_{i}(T-)=Z^{k-1}_{i-1}(T-)$, then the jump is suppressed since such a transition would lead to violation of the interlacing condition. Otherwise, the jump is performed. If at a jump time $T$, $Z^k_{i}$ performs a right jump and  $Z^k_{i}(T-)=Z^{k+1}_{i}(T-)$ then $Z^{k+1}_{i}$ simultaneously performs itself a right jump, in order to maintain the ordering. 
\end{definition}

We already gave a definition for the Schur functions. Let us now give a second combinatorial definition. At each $\mathbf{z}\in \mathbb{K}^N_{\mathbb{Z}}$ we attach a weight
\[w^N_a(\mathbf{z})=\prod_{k=1}^Na_{k}^{\sum_{i=1}^k z^k_{i}-\sum_{i=1}^{k-1}z^{k-1}_{i}}.\]
The Schur function parametrized by $z\in \mathcal{W}^N_{\mathbb{Z}}=\{(z_{1},...,z_{N})\in \mathbb{Z}^N: z_{1}\geq ... \geq z_{N}\}$ is given by
\begin{equation}
S_{z}^{(N)}(a_{1},...,a_{N})=\sum_{z \in \mathbb{K}^N_{\mathbb{Z}}[z]}w^N_{a}(\mathbf{z})
\label{eq:Schur_comb}
\end{equation}
where $\mathbb{K}^N_{\mathbb{Z}}[z]$ denotes the set of Gelfand-Tsetlin patterns with shape $z^N=z$. Then the function
\[M_{a}^N(\mathbf{z};z)=\dfrac{w^N_{a}(\mathbf{z})}{S^{(N)}_{z}(a)}\]
is a probability measure on the set of Gelfand-Tsetlin patterns with shape $z$.
\begin{theorem}[\cite{Warren_Windridge_2009}]\label{intro_theorem_schur}Suppose the process $\mathbf{Z}$, defined in \ref{Warren-Windridge-process}, has initial distribution $M^N_{a}(\cdot;z)$, for some $z \in \mathcal{W}^N_{\mathbb{Z}}$. Then, $(Z^N(t),t\geq 0)$ is a Markov process, started at $z$, with state space $\mathcal{W}^N_{\mathbb{Z}}$ and transition rate matrix $Q=\{Q(z,z'),z,z' \in \mathcal{W}^N_{\mathbb{Z}}\}$, where $Q(z,z')$ are as in \eqref{eq:Schur_process_rates(intro)} . 
\end{theorem}
A second example of a process on $\mathbb{K}^N_{\mathbb{Z}}$, with the same behaviour regarding the bottom level is inspired from the Robinson-Schensted correspondence and corresponds to a process where only the leftmost particles $Z^k_{k}$ can jump independently to the right at rate $a_{k}$. Whenever a particle $Z^k_{j}$, $1\leq j \leq k$ jumps to the right, it triggers the right jump of exactly one of its lower nearest neighbours, the right neighbour $Z^{k+1}_{j}$, if $Z^{k+1}_{j}=Z^k_{j}$, or the left, $Z^{k+1}_{j+1}$, otherwise. This process is a Poissonization of a process described by O' Connell \cite{O'Connell_2003}. \\

Moving to the continuous setting, one of the most well-known examples of an interacting particles system is the so called Dyson's Brownian Motion. Let us consider a Hermitian random matrix $H=(H(t))_{t\geq 0}$ of size $N$ of the form
\begin{equation*}
\small
H(t)=\left [ 
\begin{array}{llll}
B_{11}(t) & \frac{1}{\sqrt{2}}(B^{(R)}_{12}(t)+iB^{(I)}_{12}(t)) & ... & \frac{1}{\sqrt{2}}(B^{(R)}_{1N}(t)+iB^{(I)}_{1N}(t))\\
\frac{1}{\sqrt{2}}(B^{(R)}_{21}(t)-iB^{(I)}_{21}(t)) & B_{22}(t) & ... & \frac{1}{\sqrt{2}}(B^{(R)}_{2N}(t)+iB^{(I)}_{2N}(t))\\
\vdots & \vdots & \ddots & \vdots \\
\frac{1}{\sqrt{2}}(B^{(R)}_{N1}(t)+iB^{(I)}_{N1}(t))& \frac{1}{\sqrt{2}}(B^{(R)}_{N2}(t)+iB^{(I)}_{N2}(t))& ... & B_{NN}(t)
\end{array}
\right ] 
\end{equation*} 
where $B_{jk}^{(R)}=B_{kj}^{(R)}$, $B_{jk}^{(I)}=B_{kj}^{(I)}$ and $B_{jj}$, for $1\leq j \leq N$, $B_{jk}^{(R)},B_{jk}^{(I)}$ for $1\leq j <k \leq N$ are standard independent Brownian motions.

Since the matrix $H$ is Hermitian all its eigenvalues are real. Dyson proved in \cite{Dyson_1962} that the eigenvalues $(X^N_{1}(t)\geq ... \geq X^N_{N}(t))_{t\geq 0}$ form a diffusion process with evolution given by 
\[dX_{i} = dB_{i}+\sum_{1\leq j\neq i \leq N}\dfrac{dt}{X_{i}-X_{j}}, \qquad 1\leq i \leq N\]
where $B_{i}$, for $1\leq i \leq N$, are independent, standard Brownian motions.

Let us consider the function
\[h_{N}(x)=\prod_{1\leq i<j \leq N}(x_{i}-x_{j}).\]
The function $h_{N}$ has the following properties
\begin{itemize}
\item it is harmonic for the generator of the $N$-dimensional standard Brownian motions;
\item it vanishes at the boundary of the Weyl chamber $\mathcal{W}^N$, i.e.
\[h_{N}(x)\big|_{x_{i}=x_{i+1}}=0.\]
\end{itemize}
Therefore, the function $h_{N}$ can be used to re-write the generator of the Dyson's Brownian motion as follows
\begin{equation}
\mathcal{L}=\dfrac{1}{2}\sum_{i=1}^N \dfrac{\partial^2}{\partial x_{i}^2}+\sum_{i=1}^N \dfrac{\partial}{\partial x_{i}}\log h_{N} \dfrac{\partial}{\partial x_{i}}=\dfrac{1}{h_{N}}\Big(\, \dfrac{1}{2}\sum_{i=1}^N \dfrac{\partial^2}{\partial x_{i}^2}\, \Big)h_{N}.
\label{eq:DysonA_generator_intro}
\end{equation}
Therefore, $X$ can be identified as an $N$-dimensional standard Brownian motion killed whenever two particles collide, conditioned not to be killed. 

The top-left $(N-1) \times (N-1)$ minor of $H$ is itself a Hermitian matrix $H_{N-1}$ with the same structure as $H$, therefore its eigenvalues are also real and are distributed as an $(N-1)$-dimensional Dyson's Brownian motion. Moreover, they satisfy the interlacing condition 
\[x^N_{N}\leq x^{N-1}_{N-1} \leq ... \leq x^{N-1}_{1}\leq x^N_{1}.\]
In \cite{Warren_2007}, Warren extended the observation that two Dyson's Brownian motions are intertwined and constructed a process on the real-valued Gelfand-Tsetlin cone, denoted by $\mathbb{K}^N$, which is defined as the set of points $(x^{1},...,x^N)$ with $x^k \in \mathbb{R}^k$ satisfying the interlacing condition
\[x^{k+1}_{i+1}\leq x^{k}_{i}\leq x^{k+1}_{i},\text{ for }1\leq i \leq k \leq N-1. \]

\begin{definition}\label{Warren-process}
Let $(B^k_{i},1\leq i \leq k \leq N)$ be a collection of independent standard Brownian motions. We define the Markov process $\mathbf{X}=(\mathbf{X}(t),t\geq 0)$ with state space the Gelfand-Tsetlin cone $\mathbb{K}^N$ and evolution described as follows. For $t>0$
\begin{equation}
\begin{split}
X^1_{1}(t)&=B^1_{1}(t)\\
X^k_{i}(t)&=B^k_{i}(t)+L^{k,-}_{i}(t)-L^{k,+}_{i}(t)
\end{split}
\label{eq:Warren_process}
\end{equation}
where 
\[L^{k,-}_{i}(t)=\int_{0}^t \mathbbm{1}_{X^k_{i}(s)=X^{k-1}_{i}(s)}dL_{i}^{k,-}(s),\quad  L^{k,+}_{i}(t)=\int_{0}^t \mathbbm{1}_{X^k_{i}(s)=X^{k-1}_{i-1}(s)}dL_{i}^{k,+}(s)\]
are continuous, non-decreasing processes that increase only when $X^k_{i}(s)=X^{k-1}_{i}(s)$ and $X^k_{i}(s)=X^{k-1}_{i-1}(s)$ respectively. By convention, we set $L^{k,-}_{k}=L^{k,+}_{1}\equiv 0$.
\end{definition}
Then the following result holds.
\begin{theorem}[\cite{Warren_2007},prop. 6] Let $(\mathbf{X}(t);t\geq 0)$ be the process defined in \ref{Warren-process} started at the origin. Then for $k=1,2,...,N$, $(X^k_{1}(t),...,X^k_{k}(t);t\geq 0)$ is distributed as a $k$-dimensional Dyson's Brownian motion.
\end{theorem}

In the process $\mathbf{X}$ defined in \ref{Warren-process}, a particle $X^{k+1}_{i}$ is reflected, in the Skorokhod sense, down by $X^{k}_{i-1}$ and up by $X^k_{i}$.  This is not the only process on the Gelfand-Tsetlin cone where the shape (i.e. the bottom level) evolves as a Dyson's Brownian motion. O' Connell, in \cite{O'Connell_2003}, proposed a transformation of continuous paths $\eta:\mathbb{R}_{+}\to \mathbb{R}^N$ with $\eta(0)=0$ inspired by the column insertion version of the Robinson-Schensted algorithm as follows.

For $k=1,...,N-1$, define
\begin{equation*}
(\mathcal{P}_{k}\eta)(t)=\eta(t)+\inf_{0\leq s \leq t}\big(\eta_{k}(s)-\eta_{k+1}(s)\big)(e_{k}-e_{k+1})
\end{equation*}
where $e_{1},...,e_{N}$ denotes the standard basis in $\mathbb{R}^N$. Let $\mathcal{G}_{1}$ denote the identity operator and for $2\leq k \leq N$, let
\begin{equation}
\mathcal{G}_{k}=(\mathcal{P}_{1}\circ ... \circ \mathcal{P}_{k-1})\circ \mathcal{G}_{k-1}.
\label{eq:OConnell_transform_zero_intro}
\end{equation}
The path transformation $\mathcal{G}_{k}$ is a generalisation of Pitman's $2M-X$ transformation \cite{Pitman_1975}. More specifically for $N=2$ and $\eta = (\eta_{1},\eta_{2})$ a standard 2-dimensional Brownian motion we have for $t\geq 0$
\[\dfrac{1}{\sqrt{2}}\Big((\mathcal{G}_{2}\eta)_{2}(t)-(\mathcal{G}_{2}\eta)_{1}(t) \Big)=2\sup_{0\leq s \leq t}\Big(\dfrac{1}{\sqrt{2}}\big(\eta_{1}(s)-\eta_{2}(s)\big)\Big)-\dfrac{1}{\sqrt{2}}\big(\eta_{1}(t)-\eta_{2}(t)\big).\]
If we set for $t\geq 0$, $B(t)=\dfrac{1}{\sqrt{2}}\big(\eta_{1}(t)-\eta_{2}(t)\big)$, then $(B(t);t\geq 0)$ is a standard Brownian motion and the Pitman transform is obtained, i.e.
\[\dfrac{1}{\sqrt{2}}\Big((\mathcal{G}_{2}\eta)_{2}(t)-(\mathcal{G}_{2}\eta)_{1}(t) \Big)=2\sup_{0\leq s \leq t}B(s)-B(t).\]
Pitman, in \cite{Pitman_1975} proved that the process $Y=(Y(t);t\geq 0)$ with $Y(t)=2\sup_{0\leq s \leq t}B(s)-B(t)$ is a 3-dimensional Bessel process, i.e. a diffusion in $\mathbb{R}_{\geq 0}$ with infinitesimal generator
\[\dfrac{1}{2}\dfrac{d^2}{dy^2}+\dfrac{1}{y}\dfrac{d}{dy}.\]
O' Connell proved an analogous result for general $N>1$. We first remark that for every continuous path $\eta$, $\mathcal{G}_{N}\eta$ remains in the Weyl chamber $\mathcal{W}^N=\{x \in \mathbb{R}^N: x_{1}>...>x_{N}\}$. Moreover, if $\eta=(\eta_{1},...,\eta_{N})$ is a standard $N$-dimensional Brownian motion, then $((\mathcal{G}_{k}\eta)_{i}(t),1\leq i \leq k \leq N)$ has state space the Gelfand-Tsetlin cone $\mathbb{K}^N$ and the following result holds for its shape.
\begin{theorem}[\cite{O'Connell_2003}, th. 8.1] $\mathcal{G}_{N}\eta = ((\mathcal{G}_{N}\eta)_{i},1\leq i \leq N)$ is distributed as a Dyson Brownian motion.
\end{theorem}
We remark that, even though the bottom level of the process defined in \ref{Warren-process} has the same evolution as the bottom level of $(\mathcal{G}_{k}\eta, 1\leq k \leq N)$, they are completely different. In the process defined via the path transformation \ref{eq:OConnell_transform_zero_intro}, the only source of randomness is contained in the $N$-dimensional Brownian motion driving the evolution of the left edge of the pattern $((\mathcal{G}_{k}\eta)_{k}(t),1\leq k \leq N)_{t\geq 0}$ whereas the process defined in \ref{Warren-process} is fully randomised, in the sense that all the coordinates are driven by their own independent Brownian motions.

\section{Examples of processes on Gelfand-Tsetlin patterns in the positive temperature setting}
In the previous section we presented a transformation of Brownian paths proposed by O' Connell \cite{O'Connell_2012} which when $N=2$ reduces to the Pitman's transform. Let us start this section by presenting the positive temperature analogue of this process and then we will move to higher dimensional processes.

Matsumoto and Yor in \cite{Matsumoto_Yor_1999,Matsumoto_Yor_2000}, studied the following transformation of a Brownian motion. Fix $\mu \in \mathbb{R}$ and consider the standard Brownian motion with drift $\mu$, $B^{\mu}_{t}=B_{t}+\mu t$. For $t\geq 0$, we define the random variable
\[Y(t)=\log \int_{0}^t e^{2B^{\mu}_{s}}ds-B^{\mu}_{t}.\]
The process $Y=(Y(t);t\geq 0)$ is a diffusion in $\mathbb{R}$ started from $-\infty$ with infinitesimal generator
\[\mathcal{L}^{MY}_{\mu} = \dfrac{1}{2}\dfrac{d^2}{dy^2}+\dfrac{d}{dy}\log K_{\mu}(e^{-y})\dfrac{d}{dy}\]
where $K_{\mu}(z)$ denotes the modified Bessel function of second kind (also called the Macdonald function) which has the following integral representation
\[K_{\mu}(z)=\frac{1}{2}\int_{0}^{\infty}t^{\mu - 1}\exp \Big(-\frac{z}{2}\big(t+\frac{1}{t}\big)\Big)dt.\]
The Macdonald function is a solution to the following differential equation
\begin{equation}
\frac{d^2}{dz^2}K_{\mu}(z)+\dfrac{1}{z}\dfrac{d}{dz}K_{\mu}(z)-\big(1+\dfrac{\mu^2}{z^2}\big)K_{\mu}(z)=0
\label{eq:eigen_macdonald_intro}
\end{equation}
therefore the generator $\mathcal{L}^{MY}_{\mu}$ can be rewritten as follows
\[\mathcal{L}^{MY}_{\mu} = \dfrac{1}{\phi_{\mu}}\dfrac{1}{2}\Big(\dfrac{d^2}{dy^2}-e^{-2y}-\mu^2\Big)\phi_{\mu}\]
where $\phi_{\mu}(y)=K_{\mu}(e^{-y})$.

The last expression shows that $\mathcal{L}^{MY}$ is essentially the generator of a process killed at rate $\frac{1}{2}e^{-2y}$ which is conditioned, in the Doob's sense, not to be killed. We observe that the killing term becomes very large when $x\leq 0$ and is negligible for $x>0$.

Let us now fix a positive parameter $\beta$, which in the terminology of random polymers is called the \emph{inverse temperature parameter}, and let us consider, for $t\geq 0$, the random variable
\[Y^{\beta}(t)=\dfrac{1}{\beta}\log \int_{0}^t e^{2\beta B^{\mu}_{s}}ds-B^{\mu}_{t}.\]
As $\beta \to \infty$, in which case the temperature tends to zero, by Laplace method the following asymptotic equivalence holds
\[\dfrac{1}{\beta}\log \int_{0}^t e^{2\beta B^\mu_{s}}ds\approx \sup_{0\leq s\leq t}(2B^{\mu}_{s})\]
therefore as $\beta$ tends to infinity we recover the Pitman transformation. In this sense Matsumoto-Yor's process can be thought as a positive temperature analogue of the 3-dimensional Bessel process. Proceeding to higher dimensional processes, O' Connell in \cite{O'Connell_2012} proposed a transformation of Brownian paths that can be though as a positive temperature analogue of the transformation in \ref{eq:OConnell_transform_zero_intro}. For $k=1,...,N-1$ and $\eta:(0,\infty)\mapsto \mathbb{R}^N$ a continuous map, define
\begin{equation}
(\mathcal{T}_{k}\eta)(t)=\eta(t)+\big( \log \int_{0}^te^{\eta_{k+1}(s)-\eta_{k}(s)}ds\big)(e_{k}-e_{k+1}).
\label{eq:OConnell_transform_positive_intro}
\end{equation}
Let $\Pi_{1}$ denote the identity operator and for $2\leq k \leq N$, let
\[\Pi_{k}=(\mathcal{T}_{1}\circ ... \circ \mathcal{T}_{k-1})\circ \Pi_{k-1}.\]
Let $\nu=(\nu_{1},...,\nu_{N})$ be a real-valued vector and assume that $\eta=(\eta_{1},...,\eta_{N})$ is an $N$-dimensional standard Brownian motion where $\eta_{k}$ has drift $\nu_{k}$. Then $((\Pi_{k}\eta)_{i},1\leq i \leq k \leq N)$ is a diffusion in $\mathbb{R}^{N(N+1)/2}$. Similarly to the zero-temperature case, it is proved in \cite{O'Connell_2012} that $\Pi_{N}\eta=((\Pi_{N}\eta)(t),t\geq 0)$ is a diffusion in $\mathbb{R}^N$. In order to formulate the result we need some functions that will play the role of the functions $h_{N}(x)=\prod_{i<j}(x_{i}-x_{j})$ in the sense that they will allow us to condition a diffusion with killing not to be killed. 

The Toda operator associated with the $\mathfrak{gl}_{N}$-Lie algebra is a differential operator given by
\[\mathcal{H}^{\mathfrak{gl}_{N}}=\frac{1}{2}\Delta-\sum_{i=1}^{N-1}e^{x_{i+1}-x_{i}}\]
where $\Delta = \sum_{i=1}^N \dfrac{\partial^2}{\partial x_{i}^2}$ denotes the Laplacian.\\
The Toda operator can be though as the generator of a diffusion killed at rate
\[k(x)=\sum_{i=1}^{N-1}e^{x_{i+1}-x_{i}}>0.\]
Observe that the killing term is negligible when $x$ is away from the boundary of the Weyl chamber $\mathcal{W}^N$, so that the process then essentially evolves as an $N$-dimensional Brownian motion. The $\mathfrak{gl}_{N}$-Whittaker functions are eigenfunctions of the operator $\mathcal{H}^{\mathfrak{gl}_{N}}$, parametrized by $\nu = (\nu_{1},...,\nu_{N})\in \mathbb{C}^N$, that satisfy the eigenrelation
\[\mathcal{H}^{\mathfrak{gl}_{N}}\Psi^{\mathfrak{gl}_{N}}_{\nu}(x)=\dfrac{1}{2}\sum_{i=1}^N \nu_{i}^2\Psi^{\mathfrak{gl}_{N}}_{\nu}(x).\]
Givental in \cite{Givental_1996} proved that the $\mathfrak{gl}_{N}$-Whittaker function has an integral representation as follows
\begin{equation}
\Psi_{\nu}^{\mathfrak{gl}_{N}}(x)=\int_{\Gamma[x]}\prod_{k=1}^{N-1}\prod_{i=1}^k dx^k_{i}e^{\mathcal{F}_{\nu}(\mathbf{x})}
\end{equation}
where $\Gamma[x]$ is a deformation of the subspace $\{\mathbf{x}=(x^1,...,x^N): x^{k}\in \mathbb{R}^k \text{ such that }x^N=x\}$ such that the integral converges and 
\[\mathcal{F}_{\nu}(\mathbf{x}) = \exp\Big(\sum_{k=1}^{N-1}\nu_{k}(|x^k|-|x^{k-1}|)-\sum_{k=1}^N \sum_{i=1}^{k-1}(e^{x^k_{i+1}-x^{k-1}_{i}}+e^{x^{k-1}_{i}-x^k_{i}}) \Big).\]

For $x , \nu \in \mathbb{R}^N$ let us consider the probability measure on defined on $\mathbb{R}^{N(N+1)/2}$ by
\[\int fd\sigma^x_{\nu}=\Psi_{\nu}^{\mathfrak{gl}_{N}}(x)^{-1}\int_{\Gamma[x]}\prod_{i=1}^k dx^k_{i}f(\mathbf{x})e^{\mathcal{F}_{\nu}(\mathbf{x})} .\]
Then the following holds.

\begin{theorem}[\cite{O'Connell_2012}, th. 3.1]
Let $(\eta(t),t\geq 0)$ be a standard Brownian motion in $\mathbb{R}^N$ with drift $\nu$, then $((\Pi_{N}\eta)(t),t\geq 0)$ is a diffusion in $\mathbb{R}^N$ with infinitesimal generator given by
\begin{equation}
\mathcal{L}^{\mathfrak{gl}_{N}}_{\nu} = \dfrac{1}{\Psi^{\mathfrak{gl}_{N}}_{\nu}}\Big(\frac{1}{2}\Delta-\sum_{i=1}^{N-1}e^{x_{i+1}-x_{i}}-\frac{1}{2}\sum_{i=1}^N\nu_{i}^2\Big)\Psi_{\nu}^{\mathfrak{gl}_{N}}= \frac{1}{2}\Delta + \nabla \log \Psi^{\mathfrak{gl}_{N}}_{\nu}\cdot \nabla
\label{eq:gln_whittaker_process_generator_intro}
\end{equation}
where $\nabla$ denotes the nabla operator, i.e. $\nabla = \big(\frac{\partial}{\partial x_{1}},...,\frac{\partial}{\partial x_{N}}\big)$. \\
Moreover, for each $t\geq 0$, the conditional law of $\{(\Pi_{k}\eta)_{i}(t),1\leq i \leq k \leq N\}$ given $\{(\Pi_{N}\eta)(s),s\leq t, (\Pi_{N}\eta)(t)=x\}$, for some $x \in \mathbb{R}^N$, is given by $\sigma^x_{\nu}$.
\end{theorem}
We remark that $(\Pi_{N}\eta)_{1}$ is the logarithm of the partition function of the O'Connell-Yor semi-discrete polymer. As a corrolary one can calculate the Laplace transform of $(\Pi_{N}\eta)_{1}$ in terms of the $\mathfrak{gl}_{N}$-Whittaker function. The $\mathfrak{gl}_{N}$-Whittaker functions also appear in the study of the point-to-point log-gamma directed polymers studied in \cite{O'Connell_Seppalainen_Zygouras_2014, Corwin_O'Connell_Seppalainen_Zygouras_2014}.

As we mentioned before, the positive temperature for discrete-state Markov processes is explained in the setting of Macdonald processes. Let us fix a parameter $q \in (0,1)$ which plays the role of a tuning parameter. We can then $q$-deform the process on Gelfand-Tsetlin patterns we described in Definition \ref{Warren-Windridge-process}. Let $\mathbf{Z}$ be a continuous-time Markov process evolving on $\mathbb{K}^N_{\mathbb{Z}}$ as follows. Each particle $Z^k_{j}$, $1\leq j \leq k \leq N$ performs a right jump after waiting an exponential time with parameter
\begin{equation}
\label{eq:symmetric_jump_rates_intro}
R^k_{j} = a_{k}(1-q^{Z^{k-1}_{j-1}-Z^k_{j}})\dfrac{1-q^{Z^{k}_{j}-Z^k_{j+1}+1}}{1-q^{Z^{k}_{j}-Z^{k-1}_{j}+1}}.
\end{equation}
Moreover there exists a pushing mechanism in order to maintain the ordering. More specifically, if $Z^k_{j}$ performs a right jump and $Z^k_{j}=Z^{k+1}_{j}$, then $Z^{k+1}_{j}$ is simultaneously pushed to the right.

We observe that when $q\to 0$ the jump rate becomes $a_{k}\mathbbm{1}_{\{Z^{k-1}_{j-1}=Z^k_{j}\}}$, therefore we obtain the process of Warren-Windridge defined in \ref{Warren-Windridge-process}.

A $q$-deformation of the Robinson-Schensted correspondence was proposed in \cite{O'Connell_Pei_2013}. We describe the algorithm again in terms of dynamics on Gelfand-Tsetlin patterns. Only the leftmost particles $Z^k_{k}$, $1\leq k \leq N$ may jump of their own volition to the right at rate $a_{k}$. Whenever a particle $Z^k_{j}$ performs a right jump, it triggers the right move of exactly one of its nearest neighbours, the right $Z^{k+1}_{j}$ with probability $r^k_{j}$ and the left, $Z^{k+1}_{j+1}$, with probability $1-r^k_{j}$, where
\begin{equation}
\label{eq:q_push_probability_intro}
r^k_{j}=q^{Z^{k+1}_{j}-Z^k_{j}}\dfrac{1-q^{Z^k_{j-1}-Z^k_{j}}}{1-q^{Z^{k}_{j-1}-Z^k_{j}}}.
\end{equation}

The evolution of the bottom level of the pattern in both cases is then associated with a special case of the Macdonald polynomial called the $q$-Whittaker function. A combinatorial formula for the $q$-Whittaker functions was proved in \cite{Gerasimov_et_al_2010}. Let us denote by $w^N_{a,q}(\mathbf{z})$, the kernel of the $q$-Whittaker function for $\mathbf{z}\in \mathbb{K}^N_{\mathbb{Z}}[z]$, i.e. if $P_{z}^N(a;q)$ denotes the $q$-Whittaker function parametrized by $z \in \mathcal{W}^N_{\mathbb{Z}}$, then
\[P^{N}_{z}(a;q)=\sum_{\mathbf{z}\in \mathbb{K}^N_{\mathbb{Z}}[z]}w^N_{a;q}(\mathbf{z}).\]
The last formula allows us to define a probability measure on the set of Gelfand-Tsetlin patterns with shape $z \in \mathcal{W}^N_{\mathbb{Z}} $ as follows. For $\mathbf{z}\in \mathbb{K}^N_{\mathbb{Z}}[z]$ we set
\[M_{a,q}^N(\mathbf{z};z)=\dfrac{w^N_{a,q}(\mathbf{z})}{P^{(N)}_{z}(a;q)}.\]
If the pattern starts according to $M^N_{a,q}(\cdot;z)$, for some $z \in \mathcal{W}^N_{\mathbb{Z}}$, then $Z^N$ is a continuous time Markov process with transition rates given in terms of the $q$-Whittaker functions. For a detailed discussion on these processes we refer the reader to \cite{Borodin_Corwin_2011}.

\section{An overview of processes on Gelfand-Tsetlin pattern with wall}

Similarly to the type-$A$ Dyson's Brownian motion, one can obtain a process evolving in the interior of the type-$B$ Weyl chamber $\mathcal{W}^n_{0}=\{x\in \mathbb{R}^n:x_{1}>x_{2}>...>x_{n}>0\}$ by considering the Doob $h$-transform, with $h(x)=h^B_{n}(x)=\prod_{i=1}^n x_{i}\prod_{1\leq i\leq j\leq n}(x_{i}^2-x_{j}^2)$, of an $n$-dimensional standard Brownian motion killed when it exits $\mathcal{W}^n_{0}$. The Dyson's Brownian motion of type $B$ is a $n$-dimensional diffusion, started from the origin, satisfying the system of stochastic differential equations
\[dX_{i}^{(B)}=dB_{i}+\dfrac{1}{X_{i}^{(B)}}dt+\sum_{\substack{1\leq j \leq n \\ j\neq i}}\Big( \dfrac{1}{X^{(B)}_{i}-X^{(B)}_{j}}+\dfrac{1}{X_{i}^{(B)}+X_{j}^{(B)}}\Big)dt,\, 1\leq i \leq n.\]

The process $X^{(B)}$ also descibes the evolution of the eigenvalues of some special ensemble. We refer the reader to \cite{Katori_Tanemura_2004} for more details on the subject. 

We may consider another process where the wall is replaced by a reflected wall. The Dyson's Brownian motion of type D is a process with state space $\{x_{1}>...>x_{n}\geq 0\}$ whose evolution is given by
\[dX_{i}^{(D)}=dB_{i}+\frac{1}{2}\mathbbm{1}_{\{i=n\}}dL(t)+\sum_{\substack{1\leq j \leq n \\ j\neq i}}\Big( \dfrac{1}{X^{(D)}_{i}-X^{(D)}_{j}}+\dfrac{1}{X_{i}^{(D)}+X_{j}^{(D)}}\Big)dt, 1\leq i \leq n,\]
where $L(t)$ denotes the local time of $X^{(D)}_{n}$ at the origin.

In \cite{Borodin_et_al_2009} and \cite{Sasamoto_2011}, a multilevel extension of the type-$B$ Dyson's Brownian motion is proposed. The natural state space for such a process is the real-valued symplectic Gelfand-Tsetlin cone $\mathbb{K}^N_{0}$, i.e. the set of points $(x^1,...,x^N)$ with $x^k \in \mathbb{R}^{[\frac{k+1}{2}]}_{\geq 0}$, where $[\frac{k+1}{2}]$ denotes the integer part of $\frac{k+1}{2}$, satisfying the interlacing conditions
\[x^{2l-1}_{l}\leq x^{2l}_{l}\leq x^{2l-1}_{l-1}\leq ... \leq x^{2l-1}_{1}\leq x^{2l}_{1} \text{ for }1\leq l \leq \big[\dfrac{N+1}{2}\big]\]
and
\[x^{2l+1}_{l+1}\leq x^{2l}_{l}\leq x^{2l+1}_{l}\leq ... \leq x^{2l}_{1} \leq x^{2l+1}_{1}\text{ for }1\leq l \leq \big[\dfrac{N-1}{2}\big].\]
Let $\mathbf{X}=(\mathbf{X}(t),t\geq 0)$ be a $\mathbb{K}^N_{0}$-valued process evolving according to the rules: $X^1_{1}$ is a Brownian motion reflected by the wall at the origin and for $1\leq i \leq [\frac{k+1}{2}], 1\leq k \leq N$, $X^k_{i}$ evolves as a Brownian motion reflected, in the Skorokhod sense, down by $X^{k-1}_{i-1}$, if $i>1$, and up by $X^{k-1}_{i}$ or by the wall if $i=\frac{k+1}{2}$ and $k$ is odd.
\begin{theorem}[\cite{Borodin_et_al_2009}, prop. 2] Let $\mathbf{X}$ starts at the origin, then if $N=2n$
\[(X^N(t),t\geq 0)\overset{d}{=}(X^{(B)}(t),t\geq 0)\]
and if $N=2n-1$
\[(X^N(t),t\geq 0)\overset{d}{=}(X^{(D)}(t),t\geq 0)\]
where $X^{(B)}$ and $X^{(D)}$ are the $n$-dimensional Dyson's Brownian motions of type $B$ and type $D$ respectively.
\end{theorem}

The natural question that arises is whether there exists a Pitman-type transformation that gives the $B/D$-type Dyson's Brownian motion. The original Pitman's transformation accommodates the $N=2$ case. In \cite{Biane_et_al_2005} the authors proposed the following transformation. For $\eta :(0,\infty)\mapsto \mathbb{R}^n$ a continuous path, define
\[(\mathcal{P}_{k}\eta)(t)=\eta(t) - \inf_{0\leq s \leq t}\big(\eta_{k}(s)-\eta_{k+1}(s)\big)(e_{k}-e_{k+1}), \text{ for }1\leq k \leq n-1 \]
and 
\[(\mathcal{P}_{n}\eta)(t)=\eta(t) - \inf_{0\leq s \leq t}(2\eta_{n}(s))e_{n}.\]
We also define the following transform: set $\mathcal{G}_{1}=\mathcal{P}_{n}$ and for $2\leq k \leq n$
\[\mathcal{G}_{k} = (\mathcal{P}_{n-k+1}\circ ...\circ \mathcal{P}_{n-1} \circ \mathcal{P}_{n}\circ\mathcal{P}_{n-1}\circ... \circ \mathcal{P}_{n-k+1} )\circ \mathcal{G}_{k-1}.\]

If $\eta=(\eta(t),t\geq 0)$ is a standard $n$-dimensional Brownian motion, then $((\mathcal{G}_{k}\eta)_{i},1\leq i \leq k \leq n)$ is a Markov process of dimension $\frac{n(n+1)}{2}$. Observe that this transformation only describes the coordinates corresponding to the even levels of the symplectic Gelfand-Tsetlin pattern. We are not aware of any Pitman-type transformation that would describe the whole symplectic Gelfand-Tsetlin pattern. 
\begin{theorem}[\cite{Biane_et_al_2005}, th. 5.6] 
For $\eta=(\eta(t),t\geq 0)$ a standard $n$-dimensional Brownian motion, $\mathcal{G}_{n}\eta=((\mathcal{G}_{n}\eta)(t),t\geq 0)$ is distributed as a type-$B$ Dyson's Brownian motion.
\end{theorem}

As expected a positive-temperature analogue of the transformations $\mathcal{P}$ and $\mathcal{G}$ also exists. Biane, Bougerol and O' Connell in \cite{Biane_et_al_2005} defined and Chhaibi in \cite{Chhaibi_2013} further studied the following transformation of Brownian paths.

For $t>0$ define the transformation
\[(\mathcal{T}_{k}\eta)(t)=\eta(t) + \big(\log \int_{0}^t e^{\eta_{k+1}(s)-\eta_{k}(s)}ds\big)(e_{k}-e_{k+1}), \text{ for }1\leq k \leq n-1\]
\[(\mathcal{T}_{n}\eta)(t)=\eta(t) + \big(\log \int_{0}^t e^{-2\eta_{n}(s)}ds\big)e_{n}\]
and set $\Pi_{1}=\mathcal{T}_{n}$ and for $1<k \leq n$
\[\Pi_{k} = (\mathcal{T}_{n-k+1}\circ ...\circ \mathcal{T}_{n-1} \circ \mathcal{T}_{n}\circ\mathcal{T}_{n-1}\circ... \circ \mathcal{T}_{n-k+1} )\circ \Pi_{k-1}.\]

Chhaibi proved that $\Pi_{n}\eta$ is a diffusion in $\mathbb{R}^n$ with drift given in terms of the Whittaker function associated with the $\mathfrak{so}_{2n+1}$-Lie algebra. Moreover he was able to obtain a spectral decomposition theorem for the Whittaker functions and hence compute the entrance law for the process $\Pi_{n}\eta$.

As we can see, similarly to the type $A$ setting, $(\Pi_{n}\eta)_{1}$ can be identified as the logarithm of the partition function of a semi-discrete directed polymer. Bisi and Zygouras in \cite{Bisi_Zygouras_2017} studied the partition function of the point-to-line and point-to-half-line log-gamma directed polymer and obtained its Laplace transform which is also expressed in terms of the $\mathfrak{so}_{2n+1}$-Whittaker functions.

Here we need to mention that the Whittaker functions that appear in Chhaibi's work are obtained by the Whittaker functions introduced by Jacquet in \cite{Jacquet_1967}.  In the main part of the thesis we will present an integral representation for the $\mathfrak{so}_{2n+1}$-Whittaker functions proved in \cite{Gerasimov_et_al_2012} but we were unable to identify how the two forms of the Whittaker functions that appear in the literature relate. 

Finally, we close this section by presenting a discrete-space process with wall, constructed by Warren-Windridge in \cite{Warren_Windridge_2009}. The  state space of the process is $\mathbb{K}^N_{\mathbb{Z}_{\geq 0}}$, the set of symplectic Gelfand-Tsetlin patterns, i.e. the set of points $\mathbf{z}=(z^1,...,z^N)$ where $z^k \in \mathbb{Z}^{[\frac{k+1}{2}]}_{\geq 0}$ satisfying the intertwining conditions
\[z^{2l-1}_{l}\leq z^{2l}_{l}\leq z^{2l-1}_{l-1}\leq ... \leq z^{2l-1}_{1}\leq z^{2l}_{1} \text{ for }1\leq l \leq \big[\dfrac{N+1}{2}\big]\]
and
\[z^{2l+1}_{l+1}\leq z^{2l}_{l}\leq z^{2l+1}_{l}\leq ... \leq z^{2l}_{1} \leq z^{2l+1}_{1}\text{ for }1\leq l \leq \big[\dfrac{N-1}{2}\big].\]

More specifically, we construct a process $\mathbf{Z}=(\mathbf{Z}(t);t\geq 0)$ with state space $\mathbb{K}^N_{\mathbb{Z}_{\geq 0}}$ and dynamics as follows. We fix $n \in \mathbb{N}$ and assume that $N=2n$ or $N=2n-1$. Moreover we fix parameters $a=(a_{1},...,a_{n})\in \mathbb{R}^n_{>0}$.

The particle $Z^{2l-1}_{i}$, for $1\leq i \leq l$, attempts a jump to the right after waiting an exponential time with parameter $a_{l}$ and to the left after waiting an exponential time with parameter $a_{l}^{-1}$. The jump is then performed unless it would lead to violation of the intertwining conditions. The transitions for a particle of even level $Z^{2l}_{i}$ for $1\leq i \leq l$ are similar but now the exponential time responsible for the right jumps has parameter $a_{l}^{-1}$ and the one for the left has parameter $a_{l}$. Again the transition is suppressed whenever a jump breaks the intertwining of the particles. We remark here that all the exponential clocks we described above are independent.

Let us now describe the pushing mechanism. Suppose that, for $1\leq i \leq [\frac{k+1}{2}]$, $Z^{k}_{i}=Z^{k+1}_{i}$ and $Z^{k}_{i}$ performs a right jump. Then the particle $Z^{k+1}_{i}$ is simultaneously pushed one step to the right. Similarly, if $Z^{k}_{i}=Z^{k+1}_{i+1}$ and $Z^k_{i}$ performs a left jump then the particle $Z^{k+1}_{i+1}$ is simultaneously pushed one step to the left.

The main result in \cite{Warren_Windridge_2009} is that under appropriate initial condition the bottom level of the process $Z$, $Z^N=(Z^{N}(t),t\geq 0)$ evolves as an autonomous Markov process with state space $\mathcal{W}^n_{0}=\{(z_{1},...,z_{n})\in \mathbb{Z}^n_{\geq 0}:z_{1}\geq z_{2}\geq ... \geq z_{n}\}$. The transition rate matrix is then given in terms of the \emph{symplectic Schur functions}, namely the character of the symplectic group.

\section{Thesis outline} The remaining of the thesis is organised as follows.
\begin{enumerate}
\setcounter{enumi}{1}

\item \textbf{The Macdonald-Koornwinder polynomials.} The Macdonald-Koornwinder polynomials are Laurent polynomials that are invariant under permutation and inversion of its variables. We present some results as they were proved in \cite{vanDiejen_Emsiz_2015a} . Also in \cite{vanDiejen_Emsiz_2015b} and \cite{vanDiejen_Emsiz_2016} the authors provide a branching formula for the Macdonald-Koornwinder polynomials. We conjecture a combinatorial formula for a special case of these polynomials and prove the validity of the conjecture for some special cases.

\item \textbf{Probabilistic prerequisites.} We present some known result from Theory of Markov functions that provide requirements for the function of a Markov process to be Markovian. The main result is due to Rogers-Pitman \cite{Rogers_Pitman_1981}. We also present the formalism of Doob's $h$-processes, namely a transformation of a Markov process with killing.

\item \textbf{The Berele insertion algorithm.} The Berele insertion correspondence is a combinatorial bijection, introduced by A. Berele in \cite{Berele_1986}, between words from an alphabet $\{1<\bar{1}<...<n<\bar{n}\}$ and a symplectic Young tableau along with a sequence of up/down diagrams. A symplectic tableau can also be represented as a symplectic Gelfand-Tsetlin pattern whose bottom level is given by the shape of the Young tableau. We will prove that if the GT pattern evolves according to dynamics dictated by the Berele algorithm, then the bottom level evolves as a Markov process. Moreover, we propose a randomization of the algorithm by introducing a parameter $q\in(0,1)$. When $q\to 0$, we recover the original version. Finally, we prove that under these new dynamics the shape of the tableau also evolves as a Markov process with transition rates a $q$-deformation of the symplectic Schur functions.

\item \textbf{A $q$-deformed Markov process on symplectic Gelfand Tsetlin patterns.} We propose a $q$-deformation of a model in the integer-valued Gelfand-Tsetlin cone by introducing long--range interactions among the particles with the effect of a boundary and present their behaviour. We prove that under certain initial conditions the bottom row evolves as a Markov chain. This model is a q-deformation of the model in \cite{Warren_Windridge_2009} and can be thought as a \enquote{fully randomized} version of the model we studied in Chapter 4.

\item \textbf{Calculating the law of processes on patterns.} Using properties of the Macdonald-Koornwinder polynomials presented in Chapter 2 and assuming that these polynomials have also
the combinatorial representation we conjecture in Chapter 2, we construct a law for the processes on Gelfand-Tsetlin pattern with dynamics as in Chapters 4 and 5.

\item \textbf{On a classical limit of the $q$-deformed $\mathfrak{so}_{2n+1}$-Whittaker functions.} In this chapter we present the $\mathfrak{so}_{2n+1}$-Whittaker functions, namely the eigenfunctions of the operator of the quantum Toda lattice associated with the $\mathfrak{so}_{2n+1}$ Lie algebra. Our main references for a branching formula are \cite{Gerasimov_et_al_2008} and \cite{Gerasimov_et_al_2012}. We then prove that these functions appear as limit of the polynomials we conjectured in Chapter 2 under an appropriate scaling.

\item \textbf{A positive temperature analogue of Dyson's Brownian motion with boundary.} We propose a modification of a multi-layered model of continuous state space studied by N. O'Connell in \cite{O'Connell_2012} where we introduce the effect of a \enquote{soft} wall at zero. We remark that this model can be thought as a positive temperature version of a model pre-existed in the literature, like in \cite{Borodin_et_al_2009}, that was related to the Dyson non-colliding Brownian Motion with wall. We prove that each level evolves as a continuous time Markov process with infinitesimal generator involving a transformation of the $\mathfrak{so}_{2n+1}$-Whittaker functions. Finally, we informally comment on the relation of the particle of the highest order of the bottom level with the corresponding particle in O'Connell's model.

\end{enumerate}

%% file: tex/Koornwinder.tex
\section{Partitions and Young diagrams}
\label{partitions}
A \emph{partition} is a sequence $\lambda = (\lambda_{1},\lambda_{2}, ...)$ of integers satisfying $\lambda_{1}\geq \lambda_{2}\geq ... \geq 0$ with finitely many non-zero terms. We call each $\lambda_{i}$ a part of $\lambda$ and the number of non-zero parts the \emph{length of the partition} $\lambda$, denoted by $l(\lambda)$. Moreover we call $|\lambda|:=\sum_{i\geq 1}\lambda_{i}$, the \emph{weight of the partition} $\lambda$. If $|\lambda|=k$, then we say that $\lambda$ partitions $k$ and we write $\lambda \vdash k$. The set of partitions of length at most $n$ is denoted by $\Lambda_{n}$.

We define the natural ordering on the space of partitions called the \emph{dominance order}. For two partitions $\lambda, \mu$ we write $\lambda \geq \mu$ if and only if
\[\lambda_{1}+...+\lambda_{i}\geq \mu_{1}+...+\mu_{i}, \text{ for all }i\geq 1.\]

We also define the notion of interlacing of partitions. Let $\lambda$ and $\mu$ be two partitions with $|\lambda|\geq |\mu|$. They are said to be \emph{interlaced}, and we write $\mu\preceq \lambda$ if and only if
\[\lambda_{1}\geq \mu_{1}\geq \lambda_{2}\geq \mu_{2}\geq ...\quad  .\]

A partition $\lambda$ can be represented in the plane by an arrangement of boxes called a \emph{Young diagram}. This arrangement is top and left justified with $\lambda_{i}$ boxes in the $i$-th row. We then say that the shape of the diagram, denoted by $sh$, is $\lambda$. For example the partition $\lambda = (3,2,2,1)$ can be represented as in figure \ref{fig:diagram}.\newpage

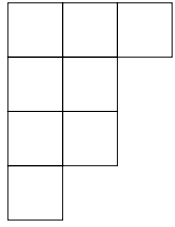
\begin{figure}[h]
\begin{center}
\scalebox{0.9}{
\begin{tikzpicture}
\draw [black] (0,0) rectangle (0.8,0.8);
\draw [black] (0.8,0) rectangle (1.6,0.8);
\draw [black] (1.6,0) rectangle (2.4,0.8);
\draw [black] (0,-0.8) rectangle (0.8,0);
\draw [black] (0.8,-0.8) rectangle (1.6,0);
\draw [black] (0,-1.6) rectangle (0.8,-0.8);
\draw [black] (0.8,-1.6) rectangle (1.6,-0.8);
\draw [black] (0,-2.4) rectangle (0.8,-1.6);
\end{tikzpicture}
}
\caption{A Young diagram of shape $(3,2,2,1)$.}
\label{fig:diagram}
\end{center} 
\end{figure}

Given two partitions/diagrams $\lambda$ and $\mu$ such that $\mu \subset \lambda$, we call the set difference $\lambda \setminus \mu$ a \emph{skew Young diagram}. A skew Young diagram $\lambda \setminus \mu$ is a horizontal strip if in each column it has at most one box. If we denote by $\lambda^\top$ the \emph{transpose partition} of $\lambda$, i.e. the partition with entries given by $\lambda_{i}^\top=|\{1\leq j \leq l(\lambda):\lambda_{j}\geq i\}|$, then $\lambda \setminus \mu$ is a horizontal strip if and only if $\lambda^\top_{i}-\mu^\top_{i}\in \{0,1\}$ for all $i$. It holds that $\mu \preceq \lambda$ if and only if the corresponding skew Young diagram $\lambda \setminus \mu$ is a horizontal strip.

Let us now consider a filling of a Young diagram.
\begin{definition}
\label{tableau_def} 
A Young tableau is a filling of a Young diagram with entries from an alphabet $\{1,...,n\}$ satisfying the conditions:\\
\textbf{(C1)} the entries increase weakly along the rows;\\
\textbf{(C2)} the entries increase strictly down the columns. \\
A Young tableau is called standard if the entries are distinct. Otherwise it is called semistandard.
\end{definition}
We will also define another filling of a Young diagram, namely the symplectic Young tableau. The following definition is due to King.
\begin{definition}[\cite{King_1971}]
\label{symp_tableau_def}
A symplectic Young tableau is a filling of a Young diagram with entries from the alphabet $\{1<\bar{1}<...<n< \bar{n}\}$ satisfying the following conditions.\\
\textbf{(S1)} The entries increase weakly along the rows;\\
\textbf{(S2)} The entries increase strictly down the columns;\\
\textbf{(S3)} No entry $<i$ occurs in row $i$, for any i.
\end{definition}
\begin{example}The filling of the diagram of shape $(4,3)$
\begin{figure*}[h]
\begin{center}
\scalebox{0.9}{
\begin{tikzpicture}
\draw [black] (0,0) rectangle (0.8,0.8);
\node at (0.4,0.4) {1};
\draw [black] (0.8,0) rectangle (1.6,0.8);
\node at (1.2,0.4) {$\bar{1}$};
\draw [black] (1.6,0) rectangle (2.4,0.8);
\node at (2,0.4) {$\bar{1}$};
\draw [black] (2.4,0) rectangle (3.2,0.8);
\node at (2.8,0.4) {$\bar{2}$};
\draw [black] (0,-0.8) rectangle (0.8,0);
\node at (0.4,-0.4) {2};
\draw [black] (0.8,-0.8) rectangle (1.6,0);
\node at (1.2,-0.4) {2};
\draw [black] (1.6,-0.8) rectangle (2.4,0);
\node at (2,-0.4) {$\bar{2}$};
\end{tikzpicture}
}
\end{center} 
\end{figure*}\\
is a symplectic Young tableau, whereas the filling\\
\begin{figure*}[h]
\begin{center}
\scalebox{0.9}{
\begin{tikzpicture}
\draw [black] (0,0) rectangle (0.8,0.8);
\node at (0.4,0.4) {1};
\draw [black] (0.8,0) rectangle (1.6,0.8);
\node at (1.2,0.4) {$\bar{1}$};
\draw [black] (1.6,0) rectangle (2.4,0.8);
\node at (2,0.4) {$\bar{1}$};
\draw [black] (2.4,0) rectangle (3.2,0.8);
\node at (2.8,0.4) {$\bar{2}$};
\draw [black] (0,-0.8) rectangle (0.8,0);
\node at (0.4,-0.4) {$\bar{1}$};
\draw [black] (0.8,-0.8) rectangle (1.6,0);
\node at (1.2,-0.4) {2};
\draw [black] (1.6,-0.8) rectangle (2.4,0);
\node at (2,-0.4) {$\bar{2}$};
\end{tikzpicture}
}
\end{center} 
\end{figure*}\\
is not since there exists an entry $\bar{1}$ at the second row violating condition $\textbf{(S3)}$.
\end{example}
Let $\lambda^{(k)}$ be the sub-tableau of $\lambda$ with entries $\leq k$. It then holds that $\lambda^{(k)}\preceq \lambda^{(k+1)}$ for all $k$ in the alphabet. So for the alphabet $\{1,...,n\}$ it holds that $\lambda^{(k)}\preceq \lambda^{(k+1)}$, for $1\leq k \leq n-1$. For the alphabet $\{1,\bar{1},...,n,\bar{n}\}$ the following interlacing conditions holds: $\lambda^{(2l-1)}\preceq \lambda^{(2l)}$, for $1\leq l \leq n$ and $\lambda^{(2l)}\preceq \lambda^{(2l+1)}$, for $1\leq l \leq n-1$.

\section{Symmetric functions}
The ring of Laurent polynomials in $n$ independent variables with coefficients from some field $\mathbb{F}$ is denoted by $\mathbb{F}[a_{1},...,a_{n},a_{1}^{-1},...,a_{n}^{-1}]$. The hyperoctahedral group $W_{n}=S_{n}\ltimes \mathbb{Z}^n_{2}$ acts on such polynomials by permuting and inverting its variables.

For a partition $\lambda \in \Lambda_{n}$, the \textit{hyperoctahedral monomial symmetric (Laurent) polynomials} are defined as

\begin{equation}
m_{\lambda}(a_{1},...,a_{n}) = \sum_{\mu \in W_{n} \lambda}\prod _{j=1}^n a_{j}^{\mu_{j}}.
\label{eq:monomial_symmetric}
\end{equation}
where $W_{n} \lambda$ denotes the orbit of $\lambda$ with respect to the action of $W_{n}$. The collection of polynomials $m_{\lambda}$ for $\lambda$ running over all elements of $\Lambda_{n}$ form a basis for the ring of Laurent polynomials that are \emph{$W_{n}$-invariant}, i.e. invariant under the action of permutation and inversion of the variables.

Another special $W_{n}$-invariant Laurent polynomial is the symplectic Schur function. This is the character of an irreducible finite-dimensional representation of the symplectic group and the definition follows from the Weyl character formula (see e.g. \cite{Fulton_Harris_2004})
\begin{equation}
Sp_{\lambda}^{(n)}(a_{1},...,a_{n}) = \dfrac{\det \Big(a_{i}^{\lambda_{j}+n-j+1}-a_{i}^{-(\lambda_{j}+n-j+1)} \Big)}{\det \Big(a_{i}^{n-j+1}-a_{i}^{-(n-j+1)}\Big)}.
\label{eq:symp_Schur_classical}
\end{equation}

We also give an equivalent combinatorial definition.
\begin{definition}[\cite{King_1971}]
\label{Symp_Schur}
The symplectic Schur function, parametrized by a partition $\lambda \in \Lambda_{n}$ is defined as
\[Sp_{\lambda}^{(n)}(a_{1},...,a_{n})=\sum_{T}weight(T)\]
where the summation is over the set of symplectic Young tableaux of shape $\lambda$, filled with entries from $\{1,\bar{1},...,n,\bar{n}\}$, and the weight is given by
\[weight(T)=\prod_{i=1}^n (a_{i})^{\#\{i's \text{ in }T\}-\#\{\bar{i}'s \text{ in }T\}}.\]
\end{definition}
Although it is not directly obvious, the expression $\sum_{T}weight(T)$ is $W_{n}$-invariant. A proof of the invariance is given in [\cite{Sundaram_1986}, th. 6.12].

The equivalence of the two definitions for the symplectic Schur function can be easily verified when $n=1$. Indeed, the right hand side of \eqref{eq:symp_Schur_classical}, for $n=1$, $a_{1}\in \mathbb{R}_{>0}$ and $\lambda = l \in \mathbb{Z}_{\geq 0}$, equals
\begin{equation*}
\begin{split}
\dfrac{a_{1}^{l+1}-a_{1}^{-(l+1)}}{a_{1}-a_{1}^{-1}} = a_{1}^{-l}\dfrac{(a_{1}^2)^{l+1}-1}{a_{1}^2 - 1}=a_{1}^{-l}\sum_{k=0}^l (a_{1}^2)^k=\sum_{k=0}^l a_{1}^{2k-l}
\end{split}
\end{equation*}
where the second equality follows from the formula of the sum of geometric series. We finally conclude the equivalence of the two formulas observing that $a_{1}^{2k-l}$ is the weight of a symplectic tableau with $k$ entries equal to $1$ and $l-k$ entries equal to $\bar{1}$.

\section{Some useful $q$-deformations}
We record some $q$-deformations of classical functions (\cite{Kac_Cheung_2002}). We assume throughout that $0<q<1$.\\

\noindent The \emph{$q$-Pochhammer symbol} is written as $(a;q)_{n}$ and defined via the product
\[(a;q)_{n} = \prod_{k=0}^{n-1}(1-aq^k), \qquad (a;q)_{\infty} = \prod_{k\geq 0} (1-aq^k).\]
We also write $(a_{1},...,a_{l};q)_{n}=(a_{1};q)_{n}...(a_{l};q)_{n}$.

\noindent The \emph{$q$-factorial} is written as $n!_{q}$ and is defined as 
\[n!_{q} = \dfrac{(q;q)_{n}}{(1-q)^n}=\prod_{k=1}^n[k]_{q},\text{ where }[k]_{q}:=\dfrac{1-q^k}{1-q}.\]
Observe that at the $q \to 1$ limit one recovers the classical factorial since
\[\lim_{q \to 1}n!_{q}=\lim_{q \to 1}\prod_{k=1}^n \dfrac{1-q^k}{1-q}=\prod_{k=1}^n \lim_{q \to 1}\dfrac{1-q^k}{1-q}=\prod_{k=1}^n k=n!. \]

\noindent The \emph{$q$-binomial coefficients} are defined in terms of $q$-factorials as
\[\dbinom{n}{k}_{q} = \dfrac{n!_{q}}{k!_{q}(n-k)!_{q}}.\]
We record some properties of the $q$-binomial coefficient
\begin{equation}
\begin{split}
\dbinom{n+1}{k}_{q}=\dbinom{n}{k}_{q}\dfrac{1-q^{n+1}}{1-q^{n-k+1}} \qquad \dbinom{n-1}{k}_{q}=\dbinom{n}{k}_{q}\dfrac{1-q^{n-k}}{1-q^n}\\
\dbinom{n}{k+1}_{q}=\dbinom{n}{k}_{q}\dfrac{1-q^{n-k}}{1-q^{k+1}} \qquad \dbinom{n}{k-1}_{q}=\dbinom{n}{k}_{q}\dfrac{1-q^{k}}{1-q^{n-k+1}}.
\end{split}
\label{eq:q_binomial}
\end{equation}
Let us check the first property, the others can be checked in a similar way. For $0\leq k \leq n$ integers, we have
\[\dbinom{n+1}{k}_{q}=\dfrac{(q;q)_{n+1}}{(q;q)_{k}(q;q)_{n-k+1}} =\dfrac{(q;q)_{n}(1-q^{n+1})}{(q;q)_{k}(q;q)_{n-k}(1-q^{n-k+1})}=\dbinom{n}{k}_{q}\dfrac{1-q^{n+1}}{1-q^{n-k+1}}. \]

\section{Deformed hyperoctahedral $q$-Whittaker functions}
In this section we will make a review of the deformed hyperoctahedral $q$-Whittaker functions, which correspond to a special case of a six-parameter family of polynomials, called the Macdonald-Koornwinder polynomials introduced in \cite{Koornwinder_1992}. Our main references for the deformed hyperoctahedral $q$-Whittaker functions are \cite{vanDiejen_Emsiz_2015a}, \cite{vanDiejen_Emsiz_2015b} and \cite{vanDiejen_Emsiz_2016}. 

In the case of a single variable the Macdonald-Koornwinder polynomials reduce to a special five-parameter family of orthogonal polynomials called the Askey-Wilson polynomials. 

\begin{definition}The function ${}_{r}\phi_{s}$ is the basic hypergeometric series which for $r=s+1$ is given by
\[{}_{s+1} \phi_s\left(\begin{matrix}&a_{1},\,...\,,a_{s+1}&\\&b_{1}\,,...,\,b_{s}&\end{matrix}\Big|q;z\right)=\sum_{k=0}^{\infty}\dfrac{(a_{1},...,a_{s+1};q)_{k}}{(b_{1},...,b_{s};q)_{k}}\dfrac{z^k}{(q;q)_{k}}.\]
Then the Askey-Wilson polynomials with parameters $q$, $\mathbf{t}=(t_{0},t_{1},t_{2},t_{3})$ are defined as follows
\[p_{l}(a;q,\mathbf{t})=\dfrac{(t_{0}t_{1},t_{0}t_{2},t_{0}t_{3};q)_{l}}{t_{0}^l}\setlength\arraycolsep{1pt}
{}_4 \phi_3\left(\begin{matrix}&q^{-l},t_{0}t_{1}t_{2}t_{3}q^{l-1},t_{0}a,t_{0}a^{-1}&\\&t_{0}t_{1},t_{0}t_{2},t_{0}t_{3}&\end{matrix}\Big|q;q\right).\]
\end{definition}
If $l>k$, then $(q^{-l};q)_{k}=0$, therefore in the definition of the Askey-Wilson polynomial the series terminates at $k=l$.

The Askey-Wilson polynomials contain other hypergeometric polynomials as special cases. We refer the reader to [\cite{Koekoek_Lesky_Swarttouw_2010}, ch. 14] for a detailed list of these special polynomials and their connections. In \cite{Koekoek_Lesky_Swarttouw_2010} (section 14.4) it is stated that when $t_{1}=t_{2}=t_{3}=0$, the Askey-Wilson polynomials correspond to the continuous big $q$-Hermite polynomials given by
\[H_{l}(a;q,t_{0})=t_{0}^{-l}{}_{3} \phi_2\left(\begin{matrix}&q^{-l},t_{0}x,t_{0}x^{-1}&\\&0,0&\end{matrix}\Big|q;q\right).\]
Moreover, the limit $t_{0}\to 0$ exists and we obtain the continuous $q$-Hermite polynomial, defined as 
\begin{equation}
H_{l}(a|q)=\sum_{m=0}^l \dbinom{l}{m}_{q}a^{2m-l}.
\label{eq:q-Hermite}
\end{equation}
Observe that when $q\to 0$, the $q$-binomial converges to $1$, therefore the continuous $q$-Hermite polynomial recovers the symplectic Schur function $Sp^{(1)}_{(l)}(a)$, we defined in \ref{Symp_Schur}.

In a similar way, the Macdonald-Koornwinder polynomials contain other special polynomials as specials cases. Our main interest are the $q$-deformed $\mathfrak{so}_{2n+1}$-Whittaker functions. These correspond to the following choice of parameters $t_{0}\to 0$ and $t_{r}=0$, for $1\leq r \leq 3$. 

\begin{definition}
Let $q,\mathbf{t}=(t_{0},t_{1},t_{2},t_{3})$ such that $0< q, |t_{l}|<1$. The deformed hyperoctahedral $q$-Whittaker function $P_{\lambda}^{(n)}(\cdot;q,\mathbf{t})$, parametrized by partition $\lambda \in \Lambda_{n}$ is the unique $W_{n}$-invariant Laurent polynomial with coefficients rational functions of $q, \mathbf{t}$ satisfying the following two conditions.
\begin{enumerate}[1)]
\item 
\[P_{\lambda}^{(n)}(\cdot;q, \mathbf{t}) = m_{\lambda}+\sum_{\substack{\mu \in \Lambda_{n}\\ \mu<\lambda}}c_{\lambda,\mu}m_{\mu}\]
where $c_{\lambda, \mu}$ are functions of $q$ and $\mathbf{t}$ and $m_{\lambda}$ are the hyperoctahedral monomial symmetric polynomials we defined in \eqref{eq:monomial_symmetric}.
\item They are orthogonal with respect to the inner product
\begin{equation}
\label{eq:torus_scalar}
\langle f,g\rangle_{\hat{\Delta}^{(n)}} = \frac{1}{(2\pi i)^n n!}\int_{\mathbb{T}^n} f(a)\overline{g(a)}\hat{\Delta}^{(n)}(a) \prod_{j=1}^n \frac{da_{j}}{a_{j}}
\end{equation}
where $\mathbb{T}^n$ is the $n$-dimensional torus $|a_{j}|=1$, $j=1,...,n$ and 
\begin{equation*}
\hat{\Delta}^{(n)}(a)=\prod_{1\leq j \leq n} \dfrac{(a_{j}^2,a_{j}^{-2};q)_{\infty}}{\prod_{0\leq l \leq 3}(t_{l}a_{j},t_{l}a_{j}^{-1};q)_{\infty}}\prod_{1\leq j<k \leq n} (a_{j}a_{k},a_{j}^{-1}a_{k},a_{j}a_{k}^{-1},a_{j}^{-1}a_{k}^{-1};q)_{\infty}.
\end{equation*}
\end{enumerate}
\end{definition}

In this thesis we will mainly focus on the $q$-deformed $\mathfrak{so}_{2n+1}$-Whittaker functions. These can be obtained from the deformed hyperoctahedral $q$-Whittaker functions if we set $t_{0}\to 0$ and $t_{1}=t_{2}=t_{3}=0$. For completeness we give a definition for the $q$-deformed $\mathfrak{so}_{2n+1}$-Whittaker functions as it is implied form the definition of the deformed hyperoctahedral $q$-Whittaker functions.
\begin{definition}
Let $0< q<1$. The $q$-deformed $\mathfrak{so}_{2n+1}$-Whittaker function $P_{\lambda}^{(n)}(\cdot;q)$, parametrized by partition $\lambda \in \Lambda_{n}$ is the unique $W_{n}$-invariant Laurent polynomial with coefficients that are rational functions of $q$ satisfying the following two conditions.
\begin{enumerate}[1)]
\item 
\[P_{\lambda} ^{(n)}(\cdot;q)= m_{\lambda}+\sum_{\substack{\mu \in \Lambda_{n}\\ \mu<\lambda}}c_{\lambda,\mu}m_{\mu}\]
where $c_{\lambda, \mu}$ are functions of $q$ and $m_{\lambda}$ are the hyperoctahedral monomial symmetric polynomials we defined in \eqref{eq:monomial_symmetric}.
\item They are orthogonal with respect to the inner product
\begin{equation}
\label{eq:torus_scalar_qWhittaker}
\langle f,g\rangle_{\hat{\Delta}^{(n)}} = \frac{1}{(2\pi i)^n n!}\int_{\mathbb{T}^n} f(a)\overline{g(a)}\hat{\Delta}^{(n)}(a) \prod_{j=1}^n \frac{da_{j}}{a_{j}}
\end{equation}
where $\mathbb{T}^n$ is the $n$-dimensional torus $|a_{j}|=1$, $j=1,...,n$ and 
\begin{equation*}
\hat{\Delta}^{(n)}(a)=\prod_{1\leq j \leq n} (a_{j}^2,a_{j}^{-2};q)_{\infty}\prod_{1\leq j<k \leq n} (a_{j}a_{k},a_{j}^{-1}a_{k},a_{j}a_{k}^{-1},a_{j}^{-1}a_{k}^{-1};q)_{\infty}.
\end{equation*}
\end{enumerate}
\end{definition}

In figure \ref{fig:polynomials_diagram} we list the polynomials that appear in this chapter and their connections.
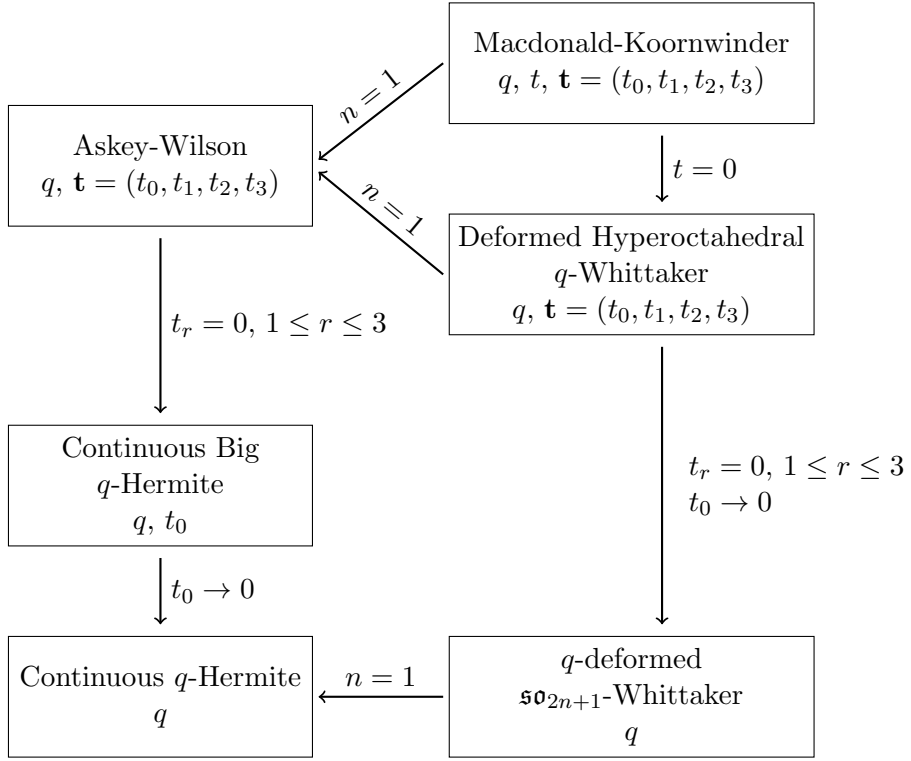
\begin{figure}[h]
\begin{center}
\begin{tikzpicture}[scale=0.8]

\draw  (2,-5) rectangle (8,-3)node[pos=.5, black] {\begin{tabular}{c} Macdonald-Koornwinder \\ $q$, $t$, $\mathbf{t} = (t_{0},t_{1},t_{2},t_{3})$ \end{tabular}};

\draw [thick, ->] (1.9,-4)-- node [above,midway, sloped] {$n=1$}(-0.15,-5.6);

\draw [thick,->] (5.5,-5.2)--node [right, midway] {$t=0$}(5.5,-6.3);

\draw  (2,-8.5) rectangle (8,-6.5)node[pos=.5, black] {\begin{tabular}{c}  Deformed Hyperoctahedral \\ $q$-Whittaker \\ $q$, $\mathbf{t} = (t_{0},t_{1},t_{2},t_{3})$ \end{tabular}};

\draw [thick, ->] (1.9,-7.5) -- node [above,midway, sloped] {$n=1$} (-0.15,-5.8);

\draw [thick,->] (5.5,-8.7)-- node [right, midway]{\begin{tabular}{l} $t_{r}=0, \, 1\leq r \leq 3$ \\ $t_{0}\to 0$ \end{tabular}}(5.5,-13.3);

\draw (2,-15.5) rectangle (8,-13.5)node[pos=.5, black] {\begin{tabular}{c}  $q$-deformed \\ $\mathfrak{so}_{2n+1}$-Whittaker \\ $q$ \end{tabular}};

\draw [thick, ->] (1.9,-14.5) -- node [above,midway, sloped] {$n=1$} (-0.15,-14.5);

%%%%%%%%%%%%%%%%%%%%%%%%%%%
\draw (-5.25,-6.7) rectangle (-0.25,-4.7)node[pos=.5, black] {\begin{tabular}{c}  Askey-Wilson  \\ $q$, $\mathbf{t} = (t_{0},t_{1},t_{2},t_{3})$ \end{tabular}};

\draw [thick,->] (-2.75,-6.9)-- node [right, midway]{$t_{r}=0, \, 1\leq r \leq 3$}(-2.75,-9.8);

\draw (-5.25,-12) rectangle (-0.25,-10)node[pos=.5, black] {\begin{tabular}{c} Continuous Big \\ $q$-Hermite  \\ $q$, $t_{0}$ \end{tabular}};

\draw [thick,->] (-2.75,-12.2)-- node [right, midway]{$t_{0}\to 0$}(-2.75,-13.3);

\draw (-5.25,-15.5) rectangle (-0.25,-13.5)node[pos=.5, black] {\begin{tabular}{c} Continuous $q$-Hermite  \\ $q$\end{tabular}};
\end{tikzpicture}
\end{center} 
\caption{Special cases of Askey-Wilson and Macdonald-Koornwinder polynomials for different choices of parameters. When $n=1$, the Koornwinder polynomial gives the corresponding Askey-Wilson polynomial.} 
\label{fig:polynomials_diagram}
\end{figure}

\newpage
In the rest of the section we review the main results about these special $W_{n}$-invariant functions, as stated in \cite{vanDiejen_Emsiz_2015a}, \cite{vanDiejen_Emsiz_2015b} and \cite{vanDiejen_Emsiz_2016}.We remark that although this choice of parameters does not satisfy the condition $|t_{r}|>0$, according to van Diejen and Emsiz, their results as stated in \cite{vanDiejen_Emsiz_2015a} are still valid. 
\subsection{Orthogonality and Completeness}
In the last section we defined the $q$-deformed $\mathfrak{so}_{2n+1}$-Whittaker functions to be orthogonal with respect to the inner product $\langle \cdot, \cdot \rangle_{\hat{\Delta}^{(n)}}$, given in \eqref{eq:torus_scalar_qWhittaker}. More specifically the following orthogonality relation holds
\begin{equation}
\langle P_{\lambda}^{(n)}, P^{(n)}_{\mu}\rangle_{\hat{\Delta}^{(n)}}=\mathbbm{1}_{\lambda = \mu}/\Delta^{(n)}_{\lambda}
\label{eq:hyperoctahedral_orthogonality}
\end{equation}
where
\[\Delta_{\lambda}^{(n)}=\dfrac{(q;q)_{\infty}^n}{(q,q)_{\lambda_{n}}\prod_{1\leq j <n}(q;q)_{\lambda_{j}-\lambda_{j+1}}}.\]
The $q$-deformed $\mathfrak{so}_{2n+1}$-Whittaker functions form an orthogonal basis for the space of $W_{n}$-invariant Laurent polynomials. More specifically, we have the following result.
\begin{theorem}[\cite{vanDiejen_Emsiz_2015a}]
The transformation
\[(\mathbf{F}f)(a)=\sum_{\lambda \in \Lambda_{n}}f(\lambda)P^{(n)}_{\lambda}(a;q)\Delta_{\lambda}^{(n)}\]
defines an isometry from $L^2(\Lambda_{n}, \Delta^{(n)})$ to $L^2(\mathbb{T}^n, \hat{\Delta}^{(n)}\frac{da}{a})$ with inverse transform given by
\[(\mathbf{F}^{-1}\hat{f})(\lambda)=\dfrac{1}{(2\pi i)^n n!}\int_{\mathbb{T}^n}\hat{f}(a)\overline{P_{\lambda}^{(n)}(a;q)}\hat{\Delta}^{(n)}(a)\frac{da}{a}.\]
The weight function $\hat{\Delta}^{(n)}$ in this case is
\begin{equation*}
\hat{\Delta}^{(n)}(a)=\prod_{1\leq j \leq n} (a_{j}^2,a_{j}^{-2};q)_{\infty}\prod_{1\leq j<k \leq n} (a_{j}a_{k},a_{j}^{-1}a_{k},a_{j}a_{k}^{-1},a_{j}^{-1}a_{k}^{-1};q)_{\infty}.
\end{equation*}
\end{theorem}
The theorem implies the following completeness result for the $q$-deformed $\mathfrak{so}_{2n+1}$-Whittaker functions.
\begin{corollary}
\label{orthogonality_rates}
For $g\in L_{2}(\mathbb{T}^n, \hat{\Delta}^{(n)}\frac{db}{b})$ the following relation holds
\[\dfrac{1}{(2\pi i)^n n!}\int_{\mathbb{T}^n}\hat{\Delta}^{(n)}(b)\sum_{\lambda \in \Lambda_{n}}P^{(n)}_{\lambda}(a;q)\overline{P^{(n)}_{\lambda}(b;q)}\Delta_{\lambda}^{(n)}g(b)\dfrac{db}{b}=g(a)\]
for every $a \in \mathbb{T}^n$.
\end{corollary}
\begin{proof}
We first observe that the required relation can be re-written as follows
\[(FF^{-1}g)(a)=g(a).\]
By the fact that $\mathbf{F}^{-1}$ is the inverse transform of $\mathbf{F}$ we have that for almost every $b$ on the torus $\mathbb{T}^n$
\[(FF^{-1}g)(b)=g(b).\]
The result follows considering a sequence of points $\{b_{n}\}_{n\geq 1}$ on $\mathbb{T}^n$ such that $b_{n}\to a$, as $n\to \infty$ which satisfy the relation
\[(FF^{-1}g)(b_{n})=g(b_{n})\]
for every $n\geq 1$.
\end{proof}
\subsection{Difference operators and Pieri formula}
From the results of \cite{Koornwinder_1992},\cite{Macdonald_2003} ,\cite{vanDiejen_Emsiz_2015a} an equivalent definition is implied for the deformed hyperoctahedral $q$-Whittaker functions via the Koornwinder difference operator, which for the choice of parameters $t_{0}\to 0$ and $t_{1}=t_{2}=t_{3}=0$, is given by
\begin{equation}
\label{eq:Koornwinder_operator}
D_{n}=\sum_{i=1}^nA_{i}(a;q)(\mathcal{T}_{q,i}-I)+\sum_{i=1}^nA_{i}(a^{-1};q)(\mathcal{T}_{q^{-1},i}-I)
\end{equation}
where
\begin{equation}
\label{eq:Koornwinder_operator_A}
A_{i}(a;q)=\dfrac{1}{(1-a_{i}^2)(1-qa_{i}^2)}\prod_{j\neq i}\dfrac{1}{(1-a_{i}a_{j})(1-a_{i}a_{j}^{-1})}
\end{equation}
and the \textit{shift operator} $\mathcal{T}_{u,i}$ is defined as
\[(\mathcal{T}_{u,i}f)(a_{1},...,a_{n})=f(a_{1},...,a_{i-1},ua_{i},a_{i+1},...,a_{n}).\]
\begin{proposition}[\cite{Koornwinder_1992},\cite{Macdonald_2003} ,\cite{vanDiejen_Emsiz_2015a}]
\label{q_eigenrelation_rates}
For any partition $\lambda \in \Lambda_{n}$, it holds
\[D_{n}P_{\lambda}^{(n)}(a_{1},...,a_{n};q)=(q^{-\lambda_{1}}-1)P_{\lambda}^{(n)}(a_{1},...,a_{n};q).\]
\end{proposition}

We define another operator $H^n$ which acts on functions $g:\Lambda_{n}\to \mathbb{C}$ as follows
\begin{equation}
H^ng(\lambda)=\sum_{i=1}^nf_{n}(\lambda, \lambda + e_{i})g(\lambda+e_{i})+\sum_{i=1}^nf_{n}(\lambda, \lambda - e_{i})g(\lambda-e_{i})
\label{eq:q-Toda}
\end{equation}
with
\begin{equation}
\label{eq:us}
f_{n}(\lambda, \lambda')=\left\{ \begin{array}{ll}
1 & \text{, if }\lambda'=\lambda + e_{1}\\
1-q^{\lambda_{i-1}-\lambda_{i}}& \text{, if }\lambda' = \lambda + e_{i} \text{ for some }1<i\leq n\\
1-q^{\lambda_{i}-\lambda_{i+1}}& \text{, if }\lambda' = \lambda - e_{i} \text{ for some }1\leq i< n\\
1-q^{\lambda_{n}} & \text{, if }\lambda'=\lambda - e_{n}
\end{array} \right. 
\end{equation}
Note that if for some $2\leq i \leq n$, $\lambda_{i}=\lambda_{i-1}$, then $f_{n}(\lambda,\lambda + e_{i})=0$ and the term $f_{n}(\lambda, \lambda + e_{i})g(\lambda + e_{i})$ will vanish. Similarly, if for $1\leq i \leq n-1$, $\lambda_{i}=\lambda_{i+1}$ or if $\lambda_{n}=0$, the term $f_{n}(\lambda, \lambda - e_{i})g(\lambda - e_{i})$ will also vanish. 
\begin{proposition}[\cite{Koornwinder_1992},\cite{Macdonald_2003},\cite{vanDiejen_1999}]
\label{q_eigenrelation_partition}
Let $n\geq 1$, $\lambda \in \Lambda_{n}$ and $q\in (0,1)$. The q-deformed $\mathfrak{so}_{2n+1}$-Whittaker functions, satisfy the Pieri identity
\begin{equation}
H^nP_{\lambda}^{(n)}(a_{1},...,a_{n};q)=\sum_{i=1}^n(a_{i}+a_{i}^{-1})P_{\lambda}^{(n)}(a_{1},...,a_{n};q)
\end{equation}
for $(a_{1},...,a_{n})\in \mathbb{C}^n\setminus \{0\}$.
\end{proposition}

\subsection{Branching rule}
\label{branching_vanDiejen}
In \cite{Rains_2005}, it was shown that the Koornwinder polynomials exhibit a recursive structure with respect to the order $n$ which implies that the deformed hyperoctahedral $q$-Whittaker functions satisfy the branching rule
\[P_{\lambda}^{(n)}(a_{1},...,a_{n};q,\mathbf{t})=\sum _{\nu} P_{\lambda\setminus \nu}(a_{n};q,\mathbf{t})P_{\nu}^{(n-1)}(a_{1},...,a_{n-1};q,\mathbf{t})\]
where the summation is over $\nu \in \Lambda_{n-1}$ such that $\lambda, \nu$ differ by two horizontal strips. 

\begin{definition}
The one-variable polynomial $P_{\lambda\setminus \nu}(a_{n};q,\mathbf{t})$ is called the \emph{branching polynomial}.
\end{definition}

In \cite{vanDiejen_Emsiz_2015b}, the authors gave a closed form for the branching polynomial for the Macdonald-Koornwinder polynomial for any choice of parameters. The branching polynomial for the deformed hyperoctahedral $q$-Whittaker functions can be found in \cite{vanDiejen_Emsiz_2016}. Let us only present here the branching rule for the $q$-deformed hyperoctahedral Whittaker function with $t_{1}=t_{2}=t_{3}=0$ and $0<|t_{0}|<1$.

Since the two partitions $\lambda, \nu$ differ by two horizontal strips it follows that $\lambda^\top_{i}-\nu^\top_{i}\in \{0,1,2\}$ for every $i$ such that $1\leq i \leq l(\lambda^\top )=m$. We set $d=d(\lambda, \nu)=|\{1\leq i \leq m: \lambda_{i}^\top-\nu_{i}^\top=1\}|$. For $1\leq r \leq d$, we define the following quantity

\[B^r_{\lambda \setminus \nu}(q,t_{0})=c(\lambda, \nu;q)\sum_{\substack{I_{+},I_{-}\subset J^{\mathsf{c}}\\ I_{+}\cap I_{-}=\emptyset \\ |I_{+}|+|I_{-}|=d-r}}A(\nu^*;I_{+},I_{-},t_{0})\]
where $\nu^* = n^m-\nu^\top=(n-\nu_{m-j+1}^\top)_{1\leq j \leq m}$ and if we also define the partition $\lambda^* = (n+1)^m-\lambda^\top=(n+1-\lambda_{m-j+1}^\top)_{1\leq j \leq m}$ then the subset $J^{\mathsf{c}}$ of $[m]:=\{1,...,m\}$ is the set $J^{\mathsf{c}}=J^{\mathsf{c}}(\nu^*, \lambda^*)=\{1\leq j \leq m: \nu^*_{j}= \lambda^*_{j}\}$ and
\[c(\lambda, \nu;q)=\prod_{\substack{1\leq j < k \leq m \\ \nu^\top_{k}<\nu^\top_{j}\\ \lambda^\top_{k}=\lambda^\top_{j}}}\dfrac{1-q^{1+k-j}}{1-q^{k-j}}\]
Finally the contribution of the partition $(I_{+},I_{-},I_{0}=J^{\mathsf{c}}\setminus (I_{+}\cup I_{-}))$ of $J^{\mathsf{c}}$ to $B_{\lambda \setminus \nu}^{r}(q;t_{0})$ equals
\begin{equation}
\begin{split}
A(\nu^*;I_{+},I_{-},t_{0})
&=\prod_{\substack{j,k \in J^{\mathsf{c}} \\  \nu^*_{j}=\nu^*_{k} \\ \epsilon_{j}>\epsilon_{k}}}\dfrac{1-q^{1+k-j}}{1-q^{k-j}} \prod_{\substack{j \in I_{-},k\in I_{+} \\ \nu^*_{j}=\nu^*_{k}+1}}\dfrac{1-q^{1+k-j}}{1-q^{k-j}}\\
 & \hspace{110pt} \prod_{j \in J^{\mathsf{c}}}t_{0}^{-\epsilon_{j}}\prod_{\substack{j,k \in J^{\mathsf{c}}\\j<k\\\epsilon_{j}\neq \epsilon_{k}=0}}q^{-\epsilon_{j}}\prod_{\substack{j,k \in J^{\mathsf{c}}\\j<k\\ \nu^*_{j}=\nu^*_{k}\\ \epsilon_{k}-\epsilon_{j}=1}}q^{-1}\
\end{split}
\label{eq:partition_contribution}
\end{equation}
with
\begin{equation*}
\epsilon_{j}=\epsilon_{j}(I_{+},I_{-})=\left\{ \begin{array}{ll}
1& j \in I_{+}\\
-1 & j \in I_{-}\\
0 & j \in J^{\mathsf{c}}\backslash (I_{+}\cup I_{-})
\end{array} \right..
\end{equation*}

\begin{theorem}[\cite{vanDiejen_Emsiz_2016}]
The branching polynomials for the deformed hyperoctahedral $q$-Whittaker function with $\mathbf{t}=(t_{0},0,0,0)$, for $0<|t_{0}|<1$, are given by
\[P_{\lambda \setminus \nu}(a;q,t_{0})=\sum_{r=0}^dB^r_{\lambda \setminus \nu}(q,t_{0})\langle a;t_{0}\rangle_{q,r}\] 
where the expanding Laurent polynomials $\langle x;t_{0}\rangle_{q,r}$ are defined as
\begin{equation}
\begin{split}
\langle a;t_{0}\rangle_{q,0}&=1\\
\langle a;t_{0}\rangle_{q,r}&=\prod_{l=1}^r(a+a^{-1}-t_{0}q^{l-1}-t_{0}^{-1}q^{-(l-1)}), \text{ for }r\geq 1
\end{split}
\label{eq:expanding_polynomials}
\end{equation}
\end{theorem}

In \cite{Komori_Noumi_Shiraishi_2009} an expansion formula for the Askey-Wilson polynomials in terms of the polynomials $\langle x;t_{0}\rangle_{q;r}$ is provided, which agrees with the formula provided by  for $n=1$. More specifically, the following result is proved.
\begin{theorem}[\cite{Komori_Noumi_Shiraishi_2009}, th. 5.2]
\label{Askey_expansion}
The continuous big $q$-Hermite polynomials, which are Askey-Wilson polynomials with $t_{1}=t_{2}=t_{3}=0$, are expressed as follows in terms of the polynomials $\langle x;t_{0}\rangle_{q;r}
$
\[H_{j}(x;t_{0}|q)=t_{0}^{-j} \sum_{r=0}^{j}\dbinom{j}{r}_{q}t_{0}^r q^{r^2-jr}\langle x;t_{0}\rangle_{q,r}.\]
\end{theorem}

Ideally we would like to obtain a branching rule for the $q$-deformed $\mathfrak{so}_{2n+1}$-Whittaker functions which correspond to the limit $t_{0}\to 0$, but because the expanding polynomials $\langle a;t_{0}\rangle_{q,r}$ contain negative powers of $t_{0}$ we cannot set $t_{0}\to 0$ in the above formula. In the next section we will investigate this case further. 

\section{Branching rule for $q$-deformed $\mathfrak{so}_{2n+1}$-Whittaker functions}
In section \ref{branching_vanDiejen} we presented a branching formula for the $q$-deformed hyperoctahedral Whittaker function with parameters $q\in (0,1)$, $0<|t_{0}|<1$ and $t_{r}=0$, for $1\leq r \leq 3$. In this section we will conjecture a combinatorial formula for the case $t_{0}\to 0$ which similarly to \ref{branching_vanDiejen} also admits a recursive structure and we will prove that for some special cases of the partitions $\lambda, \nu$ the branching polynomial $P_{\lambda \setminus \nu}(a;q,t_{0})$ of \ref{branching_vanDiejen} converge to the one we conjecture when $t_{0}\to 0$. 
\begin{definition}
\label{recursion_q_whittaker}
For $n\geq 1$, $\lambda \in \Lambda_{n}$ and $a=(a_{1},...,a_{n})\in \mathbb{R}_{>0}^n$, define the Laurent polynomial $\mathcal{P}^{(n)}_{\lambda}(a;q)$ to be a solution to the recursion relation
\[\mathcal{P}_{\lambda}^{(n)}(a_{1},...,a_{n};q)=\sum_{\nu \in \Lambda_{n-1}}Q_{a_{n},q}^{(n-1, n)}(\nu, \lambda)\mathcal{P}_{\nu}^{(n-1)}(a_{1},...,a_{n-1};q)\]
with
\[Q_{a_{n},q}^{(n-1, n)}(\nu, \lambda)=\sum_{\substack{\mu \in \Lambda_{n}: \\ \nu \preceq \mu \preceq \lambda}}a_{n}^{2|\mu|-|\nu|-|\lambda|} \prod_{i=1}^{n-1}\dbinom{\lambda_{i}-\lambda_{i+1}}{\lambda_{i}-\mu_{i}}_{q}\dbinom{\mu_{i}-\mu_{i+1}}{\mu_{i}-\nu_{i}}_{q}\dbinom{\lambda_{n}}{\lambda_{n}-\mu_{n}}_{q}.\]
When $n=1$, $\mathcal{P}^{(n)}$ correspond to the continuous $q$-Hermite polynomials defined in \eqref{eq:q-Hermite}.
\end{definition}
We remark that if there exists no $\mu \in \Lambda_{n}$ such that $\nu \preceq \mu \preceq \lambda$ then $Q_{a_{n},q}^{(n-1,n)}(\nu, \lambda)=0$, therefore the summation in the recursion formula is over $\nu \in \Lambda_{n-1}$ such that $\lambda, \nu$ differ by two horizontal strips.

Later, in chapter \ref{Berele} we will show that the polynomials $\mathcal{P}^{(n)}$ are eigenfunctions of the operator $H^n$ given in \eqref{eq:q-Toda}. 
\begin{conjecture}
\label{conjecture}
As $t_{0}\to 0$, the deformed hyperoctahedral $q$-Whittaker function converges to the polynomial $\mathcal{P}^{(n)}$, defined in \ref{recursion_q_whittaker}.
\end{conjecture}
\begin{proposition}
\label{leading_term}
The polynomial $Q^{(n-1,n)}_{x,q}(\nu, \lambda)$ is of the form
\[Q_{x,q}^{(n-1,n)}(\nu, \lambda)=c(\lambda, \nu;q)(x^d+x^{-d})+\sum_{l=-(d-1)}^{d-1}c_{q}(l)x^l\]
for some coefficients $c_{q}(l), \, -(d-1)\leq l \leq d-1$. So the leading term of $Q^{(n-1,n)}_{x,q}(\nu, \lambda)$ is the same as the leading term of $P_{\lambda \setminus \nu}(x;q,t_{0})$.
\end{proposition}

\begin{proof}
Let us first start by calculating the most extreme values the power of $x$ can take in $Q_{x;q}^{(n-1,n)}(\nu, \lambda)$. For $\mu \in \Lambda_{n}$ such that $\nu \preceq \mu \preceq \lambda$ we have
\begin{equation*}
\begin{split}
2|\mu|-|\nu|-|\lambda| &= 2|\mu^\top|-|\nu^\top|-|\lambda^\top| \\
&= \Big( \sum_{\substack{1\leq j \leq m: \\ \lambda_{j}^\top-\nu_{j}^\top=0 }}+\sum_{\substack{1\leq j \leq m: \\ \lambda_{j}^\top-\nu_{j}^\top=1 }}+\sum_{\substack{1\leq j \leq m: \\ \lambda_{j}^\top-\nu_{j}^\top=2 }}\Big)(2\mu_{j}^\top-\nu_{j}^\top-\lambda_{j}').
\end{split}
\end{equation*}
When $\lambda_{j}^\top-\nu_{j}^\top=0$, then necessarily $\mu_{j}^\top=\lambda_{j}^\top=\nu_{j}^\top$ and therefore the summand for the first term vanishes.\\
When $\lambda_{j}^\top-\nu_{j}^\top=2$, then $\mu_{j}^\top=\lambda^\top_{j}-1=\nu_{j}^\top+1$, therefore
\[2\mu_{j}^\top-\nu_{j}^\top-\lambda_{j}^\top=2(\lambda^\top_{j}-1)-(\lambda^\top_{j}-2)-\lambda_{j}^\top.\]
Finally, if $\lambda_{j}^\top-\nu_{j}^\top=1$, then $\mu_{j}^\top$ can either equal $\lambda^\top_{j}$ or $\nu^\top_{j}$. Therefore we have the following sharp bounds for $2|\mu|-|\nu|-|\lambda|$
\begin{equation*}
\begin{split}
2|\mu|-|\nu|-|\lambda| &\leq \sum_{\substack{1\leq j \leq m: \\ \lambda_{j}^\top-\nu_{j}^\top=1 }} (2\lambda_{j}^\top-(\lambda^\top_{j}-1)-\lambda^\top_{j})=d\\
2|\mu|-|\nu|-|\lambda| &\geq \sum_{\substack{1\leq j \leq m: \\ \lambda_{j}^\top-\nu_{j}^\top=1 }} (2\nu_{j}^\top-\nu^\top_{j}-(\nu^\top_{j}+1))=-d.
\end{split}
\end{equation*}
The coefficient of $x^d$ corresponds to $\mu$ where each $\mu_{i}$ takes its largest value, therefore since
\[0\leq \mu_{n}\leq \nu_{n-1}\leq ... \leq \nu_{1}\leq \mu_{1}\]
\[0\leq \mu_{n}\leq \lambda_{n}\leq ... \leq \mu_{1}\leq \lambda_{1}\]
the coefficient of $x^d$ corresponds to $\mu_{i} = \min\{\nu_{i-1},\lambda_{i}\}$, for $1\leq i \leq n$ with the convention $\mu_{1}=\lambda_{1}$ and equals
\[\prod_{i=1}^{n-1}\dbinom{\lambda_{i}-\lambda_{i+1}}{\lambda_{i}-\min\{\nu_{i-1},\lambda_{i}\}}_{q}\dbinom{\min\{\nu_{i-1},\lambda_{i}\}-\min\{\nu_{i},\lambda_{i+1}\}}{\min\{\nu_{i-1},\lambda_{i}\}-\nu_{i}}_{q}\dbinom{\lambda_{n}}{\lambda_{n}-\min\{\nu_{n-1},\lambda_{n}\}}_{q}.\]
On the other hand the coefficient of $x^{-d}$ correspond to the choice of $\mu$ where each coordinate takes its smallest value, i.e. it corresponds to $\mu$ with $\mu_{i} = \max \{\nu_{i}, \lambda_{i+1}\}$, for $1\leq i \leq n$ with the convention $\mu_{n}=0$ and equals
\[\prod_{i=1}^{n-1}\dbinom{\lambda_{i}-\lambda_{i+1}}{\lambda_{i}-\max\{\nu_{i},\lambda_{i+1}\}}_{q}\dbinom{\max\{\nu_{i},\lambda_{i+1}\}-\max\{\nu_{i+1},\lambda_{i+2}\}}{\max\{\nu_{i},\lambda_{i+1}\}-\nu_{i}}_{q}\dbinom{\lambda_{n}}{\lambda_{n}-\max\{\nu_{n},\lambda_{n+1}\}}_{q}.\]
Let us now try to rewrite $c(\lambda, \nu;q)$ in a similar form.\\
By the definition of $\lambda^\top$ as the transpose of $\lambda$ it holds that $\lambda^\top_{j}=i$ for $j \in \{\lambda_{i+1}+1,...,\lambda_{i}\}$, for $1\leq i \leq n$. In the case of $i=n$, we set $\lambda_{i+1}=\lambda_{n+1}=0$.
Therefore
\[c(\lambda, \nu;q)=\prod_{i=1}^n \prod_{\substack{j,k\in \{\lambda_{i+1}+1,...,\lambda_{i}\}\\\nu^\top_{k}<\nu^\top_{j}}}\dfrac{1-q^{1+k-j}}{1-q^{k-j}}.\]
Let us focus on a block of indices $\{\lambda_{i+1}+1,...,\lambda_{i}\}$. Since $\lambda, \nu$ differ by two horizontal strips, it holds that $\nu^\top_{j} \in \{i,i-1,i-2\}$ for every $j\in \{\lambda_{i+1}+1,...,\lambda_{i}\}$. Let us consider the following cases.
\begin{enumerate}[I.]
\item If $\nu_{i}\geq \lambda_{i+1}$ and $\nu_{i-1}\leq \lambda_{i}$, then
\begin{equation*}
\begin{split}
\prod_{\substack{j,k\in \{\lambda_{i+1}+1,...,\lambda_{i}\}\\\nu^\top_{k}<\nu^\top_{j}}}&\dfrac{1-q^{1+k-j}}{1-q^{k-j}}\\
&= \prod_{j=\lambda_{i+1}+1}^{\nu_{i}}\prod_{k = \nu_{i}+1}^{\lambda_{i}}\dfrac{1-q^{1+k-j}}{1-q^{k-j}}\prod_{j=\nu_{i}+1}^{\nu_{i-1}}\prod_{k = \nu_{i-1}+1}^{\lambda_{i}}\dfrac{1-q^{1+k-j}}{1-q^{k-j}}\\
&=\prod_{j=\lambda_{i+1}+1}^{\nu_{i}}\dfrac{1-q^{1+\lambda_{i}-j}}{1-q^{\nu_{i}+1-j}} \prod_{j=\nu_{i}+1}^{\nu_{i-1}}\dfrac{1-q^{1+\lambda_{i}-j}}{1-q^{\nu_{i-1}+1-j}}\\
&= \dfrac{(1-q^{\lambda_{i}-\lambda_{i+1}})...(1-q^{\lambda_{i}-\nu_{i}+1})}{(1-q^{\nu_{i}-\lambda_{i+1}})...(1-q)}\cdot \dfrac{(1-q^{\lambda_{i}-\nu_{i}})...(1-q^{\lambda_{i}-\nu_{i-1}+1})}{(1-q^{\nu_{i-1}-\nu_{i}})...(1-q)}\\
&=\,\dbinom{\lambda_{i}-\lambda_{i+1}}{\lambda_{i}-\nu_{i}}_{q}\dbinom{\lambda_{i}-\nu_{i}}{\lambda_{i}-\nu_{i-1}}_{q}
\end{split}
\end{equation*}
where the last expression equals
\[\dbinom{\lambda_{i}-\lambda_{i+1}}{\lambda_{i}-\max\{\nu_{i},\lambda_{i+1}\}}_{q}\dbinom{\max\{\nu_{i-1},\lambda_{i}\}-\max\{\nu_{i},\lambda_{i+1}\}}{\max\{\nu_{i-1},\lambda_{i}\}-\nu_{i-1}}_{q}.\]
Moreover we observe that
\begin{equation*}
\begin{split}
\dbinom{\lambda_{i}-\lambda_{i+1}}{\lambda_{i}-\nu_{i}}_{q}\dbinom{\lambda_{i}-\nu_{i}}{\lambda_{i}-\nu_{i-1}}_{q}&=\dfrac{(q;q)_{\lambda_{i}-\lambda_{i+1}}}{(q;q)_{\lambda_{i}-\nu_{i}}(q;q)_{\nu_{i}-\lambda_{i+1}}}
\dfrac{(q;q)_{\lambda_{i}-\nu_{i}}}{(q;q)_{\lambda_{i}-\nu_{i-1}}(q;q)_{\nu_{i-1}-\nu_{i}}}\\
&=\dfrac{(q;q)_{\lambda_{i}-\lambda_{i+1}}}{(q;q)_{\lambda_{i}-\nu_{i-1}}(q;q)_{\nu_{i-1}-\lambda_{i+1}}}
\dfrac{(q;q)_{\nu_{i-1}-\lambda_{i+1}}}{(q;q)_{\nu_{i-1}-\nu_{i}}(q;q)_{\nu_{i}-\lambda_{i+1}}}\\
&= \dbinom{\lambda_{i}-\lambda_{i+1}}{\lambda_{i}-\nu_{i-1}}_{q}\dbinom{\nu_{i-1}-\lambda_{i+1}}{\nu_{i-1}-\nu_{i}}_{q}
\end{split}
\end{equation*}
which equals
\[\dbinom{\lambda_{i}-\lambda_{i+1}}{\lambda_{i}-\min\{\nu_{i-1},\lambda_{i}\}}_{q}\dbinom{\min\{\nu_{i-1},\lambda_{i}\}-\min\{\nu_{i},\lambda_{i+1}\}}{\min\{\nu_{i-1},\lambda_{i}\}-\nu_{i}}_{q}.\]
\item If $\nu_{i}< \lambda_{i+1}$ and $\nu_{i-1}\leq  \lambda_{i}$, then $\forall j \in \{\lambda_{i+1}+1,...,\lambda_{i}\}$ it holds that $\nu^\top_{j}\neq l$ and
\begin{equation*}
\begin{split}
\prod_{\substack{j,k\in \{\lambda_{i+1}+1,...,\lambda_{i}\}\\\nu^\top_{k}<\nu^\top_{j}}}\dfrac{1-q^{1+k-j}}{1-q^{k-j}}&=\prod_{j=\lambda_{i+1}+1}^{\nu_{i-1}}\prod_{k=\nu_{i-1}+1}^{\lambda_{i}}\dfrac{1-q^{1+k-j}}{1-q^{k-j}}\\
&=\dbinom{\lambda_{i}-\lambda_{i+1}}{\lambda_{i}-\nu_{i-1}}_{q}.
\end{split}
\end{equation*}
\item If $\nu_{i}\geq \lambda_{i+1}$ and $\nu_{i-1}>\lambda_{i}$, then $\forall j \in \{\lambda_{i+1}+1,...,\lambda_{i}\}$ it holds that $\nu^\top_{j}\neq i-2$ and
\begin{equation*}
\begin{split}
\prod_{\substack{j,k\in \{\lambda_{i+1}+1,...,\lambda_{i}\}\\\nu^\top_{k}<\nu^\top_{j}}}\dfrac{1-q^{1+k-j}}{1-q^{k-j}}&=\prod_{j=\lambda_{i+1}+1}^{\nu_{i}}\prod_{k=\nu_{i}+1}^{\lambda_{i}}\dfrac{1-q^{1+k-j}}{1-q^{k-j}}\\
&=\dbinom{\lambda_{i}-\lambda_{i+1}}{\lambda_{i}-\nu_{i}}_{q}.
\end{split}
\end{equation*}
\item If $\nu_{i}< \lambda_{i+1}$ and $\nu_{i-1}>\lambda_{i}$, then $\forall j \in \{\lambda_{i+1}+1,...,\lambda_{i}\}$ it holds that $\nu^\top_{j}=i-1$ therefore
\begin{equation*}
\prod_{\substack{j,k\in \{\lambda_{i+1}+1,...,\lambda_{i}\}\\\nu^\top_{k}<\nu^\top_{j}}}\dfrac{1-q^{1+k-j}}{1-q^{k-j}}=1.
\end{equation*}
\end{enumerate}
Therefore, we obtain the following
\[c(\lambda, \nu;q)=\prod_{i=1}^n \dbinom{\lambda_{i}-\lambda_{i+1}}{\lambda_{i}-\min\{\nu_{i-1},\lambda_{i}\}}_{q}\dbinom{\min\{\nu_{i-1},\lambda_{i}\}-\min\{\nu_{i},\lambda_{i+1}\}}{\min\{\nu_{i-1},\lambda_{i}\}-\nu_{i}}_{q}\]
which gives the coefficient of $x^d$, or the equivalent form
\[c(\lambda, \nu;q) = \prod_{i=1}^n \dbinom{\lambda_{i}-\lambda_{i+1}}{\lambda_{i}-\max\{\nu_{i},\lambda_{i+1}\}}_{q}\dbinom{\max\{\nu_{i-1},\lambda_{i}\}-\max\{\nu_{i},\lambda_{i+1}\}}{\max\{\nu_{i-1},\lambda_{i}\}-\nu_{i-1}}_{q}\]
which gives the coefficient of $x^{-d}$.
\end{proof}

We will now prove that for two special cases of partitions $\lambda, \nu$, as $t_{0}\to 0$, the branching polynomial $P_{\lambda \setminus \nu}(x;q,t_{0})$ converges to $Q^{(n-1,n)}_{x,q}(\nu, \lambda)$.

\begin{proposition}
Assume that the partitions $\lambda \in \Lambda_{n}$, $\nu \in \Lambda_{n-1}$ are such that there exists a unique index $1\leq s \leq n$ such that
\[\max\{\nu_{s}, \lambda_{s+1}\}<\min \{\nu_{s-1},\lambda_{s}\}\]
and
\[\max\{\nu_{j}, \lambda_{j+1}\}=\min \{\nu_{j-1},\lambda_{j}\}, \text{ for }j\neq s.\]
Then the limit $\lim_{t_{0}\to 0}P_{\lambda\setminus \nu}(x;q,t_{0})$ exists and $Q^{(n-1, n)}_{x,q}(\nu, \lambda)=\lim_{t_{0}\to 0}P_{\lambda\setminus \nu}(x;q,t_{0})$.
\end{proposition}
\begin{proof}
Assume that the partitions $\lambda, \nu$ are as in the statement of the Proposition with
\[\min \{\nu_{s-1},\lambda_{s}\} - \max\{\nu_{s}, \lambda_{s+1}\}=d\]
for some $0\leq d \leq m=l(\lambda^\top)$, then it holds that $|\{1\leq i \leq l(\lambda^\top):\lambda^\top_{i}-\nu^\top_{i}=1\}|=d$ and for all $j \in \{1\leq i \leq l(\lambda^\top):\lambda^\top_{i}-\nu^\top_{i}=1\}$, $\nu^\top_{j}=s-1$, (as in figure \ref{fig:comparisonI}).

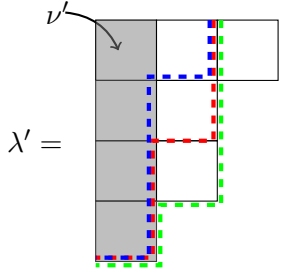
\begin{figure}[h]
\begin{center}
\scalebox{1}{
\begin{tikzpicture}
\node at (-0.8,-0.8) {$\lambda'=$};
\node at (-0.5,0.9) {$\nu'$};
\draw [thick] [->] (-0.4,0.9) to [bend left] (0.3,0.4);
\draw [fill, gray, opacity = 0.5] (0,0) rectangle (0.8,0.8);
\draw [black] (0,0) rectangle (0.8,0.8);
\draw [black] (0.8,0) rectangle (1.6,0.8);
\draw [black] (1.6,0) rectangle (2.4,0.8);
\draw [black] (0,-0.8) rectangle (0.8,0);
\draw [fill, gray, opacity = 0.5] (0,-0.8) rectangle (0.8,0);
\draw [black] (0.8,-0.8) rectangle (1.6,0);
\draw [black] (0,-1.6) rectangle (0.8,-0.8);
\draw [fill, gray, opacity =0.5] (0,-1.6) rectangle (0.8,-0.8);
\draw [black] (0.8,-1.6) rectangle (1.6,-0.8);
\draw [black] (0,-2.4) rectangle (0.8,-1.6);
\draw [fill, gray, opacity =0.5] (0,-2.4) rectangle (0.8,-1.6);
\draw [ultra thick, dashed, green] (1.65,0.8)--(1.65,-0.8)--(1.65,-1.65)--(0.85,-1.65)--(0.85,-2.45)--(0,-2.45);
\draw [ultra thick, dashed, red] (1.55,0.8)--(1.55,-0.8)--(0.75,-0.8)--(0.75,-1.6)--(0.75,-2.35)--(0,-2.35);
\draw [ultra thick, dashed, blue] (1.5,0.8)--(1.5,0.05)--(0.7,0.05)--(0.7,-1.6)--(0.7,-2.35)--(0,-2.35);
\end{tikzpicture}
}
\caption{For the partitions $\lambda=(4,3,1)$ and $\nu = (4,0)$ (marked with gray), we have $d=d(\lambda, \nu)=|\{1\leq j \leq 4: \lambda_{j}^\top-\nu_{j}^\top = 1\}|=2$. The \enquote{single} boxes, i.e. the boxes that are such that $\lambda^\top_{j}=\nu^\top_{j}+\{\square\}$, are stacked the one above the other. The dashed lines show the possible $\mu: \nu \preceq \mu \preceq \lambda$.}
\label{fig:comparisonI}
\end{center} 
\end{figure}
Then the set $J^{\mathsf{c}}=\{1\leq j \leq m: \nu_{j}^*=\lambda_{j}^*\}$, where we recall that $\nu_{j}^* = n-\nu^\top_{m-j+1}$ and $\lambda_{j}^* = n+1-\lambda^\top_{m-j+1}$, consists of subsequent indices, i.e. there exists some $0\leq i \leq m-d$ such that $J^{\mathsf{c}}=\{i+1,...,i+d\}$. 

Let us calculate the contribution of $A(\nu^*;I_{+},I_{-},t_{0})$, defined in \eqref{eq:partition_contribution}, for each partition $(I_{+},I_{-},I_{0})$ of $J^{\mathsf{c}}$. \\
First we observe that $A(\nu^*;I_{+},I_{-},t_{0})$ is zero unless the subsets $I_{+},I_{-},I_{0}$ are ordered such that all the indices in $I_{+}$ are smaller of the indices in $I_{0}$ which are smaller than the indices in $I_{-}$. For example the partition $I_{+}=\{i+1\}$, $I_{0}=\{i+2,...,i+d-1\}$, $I_{-}=\{i+d\}$ has non zero contribution, whereas the partition  $I_{-}=\{i+1\}$, $I_{0}=\{i+2,...,i+d-1\}$, $I_{+}=\{i+d\}$ has zero contribution. Let us explain why this is true.\\
We assume that there exists a pair of indices $(l,l+1)$ which is \enquote{out of order}, for example $l\in I_{0}$ and $l+1\in I_{+}$, then in the product 
\[\prod_{\substack{j,k \in J^{\mathsf{c}} \\  \nu^*_{j}=\nu^*_{k} \\ \epsilon_{j}>\epsilon_{k}}}\dfrac{1-q^{1+k-j}}{1-q^{k-j}}\]
the following term appears
\[\dfrac{1-q^{1+l-(l+1)}}{1-q^{l-(l+1)}}=0.\]
Moreover, we observe that for those partition where $A(\nu^*;I_{+},I_{-},t_{0})$ is non-zero, the set of pairs of indices $\{j,k \in J^{\mathsf{c}}: \, j<k \text{ s.t. } \nu^*_{j}=\nu^*_{k}, \epsilon_{k}-\epsilon_{j}=1\}$ is empty, therefore we end up with the following form
\[A(\nu^*;I_{+},I_{-},t_{0})=\prod_{\substack{j,k \in J^{\mathsf{c}} \\  \nu^*_{j}=\nu^*_{k} \\ \epsilon_{j}>\epsilon_{k}}}\dfrac{1-q^{1+k-j}}{1-q^{k-j}}  \prod_{j \in J^{\mathsf{c}}}t_{0}^{-\epsilon_{j}}\prod_{\substack{j,k \in J^{\mathsf{c}}\\j<k\\\epsilon_{j}\neq \epsilon_{k}=0}}q^{-\epsilon_{j}}.\]

Let us now fix a partition with non-zero contribution such that $|I_{+}|+|I_{-}|=d-r$ for some $0\leq r \leq d$. If $|I_{+}|=l$, for some $0\leq l \leq d-r$, it necessarily holds that the partition $(I_{+},I_{-},I_{0})$ of $J^{\mathsf{c}}$ is the following
\begin{equation*}
\begin{split}
I_{+}&= \{i+1,...,i+l\}\\
I_{0}\,&= \{i+l+1,...,i+l+r\}\\
I_{-}&= \{i+l+r+1,...,i+d\}
\end{split}
\end{equation*}
we then calculate
\begin{equation*}
\begin{split}
\prod_{\substack{j,k \in J^{\mathsf{c}} \\  \nu^*_{j}=\nu^*_{k} \\ \epsilon_{j}>\epsilon_{k}}}\dfrac{1-q^{1+k-j}}{1-q^{k-j}} &= \prod_{j=i+1}^{i+l}\prod_{k=i+l+1}^{i+d}\dfrac{1-q^{1+k-j}}{1-q^{k-j}}\prod_{j=i+l+1}^{i+l+r}\prod_{k=i+l+r+1}^{i+d}\dfrac{1-q^{1+k-j}}{1-q^{k-j}}\\
&= \dbinom{d}{l}_{q}\dbinom{d-l}{r}_{q}.
\end{split}
\end{equation*}
and hence
\[A(\nu^*;I_{+},I_{-},t_{0}) = \dbinom{d}{l}_{q}\dbinom{d-l}{r}_{q}t_{0}^{d-2l-r}q^{-rl}.\]
Therefore the branching polynomial $P_{\lambda \setminus \nu}(x;q,t_{0})$ equals
\[P_{\lambda \setminus \nu}(x;q,t_{0}) = c(\lambda, \nu;q)\sum_{0\leq r \leq d}\Big( \sum_{l=0}^{d-r}\dbinom{d}{l}_{q}\dbinom{d-l}{r}_{q}t_{0}^{d-2l-r}q^{-rl}\Big)\langle x;t_{0}\rangle_{q,r}.\]
Setting $j=r+l$ we have
\begin{equation*}
\begin{split}
\sum_{0\leq r \leq d}\Big( \sum_{l=0}^{d-r}\dbinom{d}{l}_{q}&\dbinom{d-l}{r}_{q}t_{0}^{d-2l-r}q^{-rl}\Big)\langle x;t_{0}\rangle_{q,r}\\
&= \sum_{r=0}^d\Big( \sum_{j=r}^{d}\dbinom{d}{j-r}_{q}\dbinom{d+r-j}{r}_{q}t_{0}^{d+r-2j}q^{-r(j-r)}\Big)\langle x;t_{0}\rangle_{q,r}\\
&=\sum_{j=0}^d \sum_{r=0}^j \dbinom{d}{j-r}_{q}\dbinom{d+r-j}{r}_{q}t_{0}^{d+r-2j}q^{r^2-rj}\langle x;t_{0}\rangle_{q,r}\\
&=\sum_{j=0}^d \dbinom{d}{j}_{q}t_{0}^{d-j}t_{0}^{-j} \sum_{r=0}^{j}\dfrac{\binom{d}{j-r}_{q}\binom{d+r-j}{k}_{q}}{\binom{d}{j}_{q}}t_{0}^r q^{r^2-jr}\langle x;t_{0}\rangle_{q,r}\\
&=\sum_{j=0}^d \dbinom{d}{j}_{q}t_{0}^{d-j}t_{0}^{-j} \sum_{r=0}^{j}\dbinom{j}{r}_{q}t_{0}^r q^{r^2-jr}\langle x;t_{0}\rangle_{q,r}\\
&=\sum_{j=0}^d \dbinom{d}{j}_{q}t_{0}^{d-j}H_{j}(x;t_{0}|q)
\end{split}
\end{equation*}
where, for the last equality we used the expansion of the continuous big $q$-Hermite polynomials, $H_{j}(x;t_{0}|q)$, stated in Theorem \ref{Askey_expansion}.
Therefore the polynomial $P_{\lambda \setminus \nu}(x;q,t_{0})$ can be rewritten as follows
\[P_{\lambda \setminus \nu}(x;q,t_{0})=c(\lambda, \nu;q)\sum_{j=0}^d \dbinom{d}{j}t_{0}^{d-j}H_{j}(x;t_{0}|q).\]
As $t_{0}\to 0$ the continuous big $q$-Hermite polynomials converge to the continuous $q$-Hermite polynomials, see Askey-Wilson scheme in \cite{Koekoek_Lesky_Swarttouw_2010},  given by
\[H_{j}(x;q)=\sum_{k=0}^j \dbinom{j}{k}_{q}x^{2k-j}\]
therefore we conclude that
\[\lim_{t_{0}\to 0}P_{\lambda\setminus \nu}(x;q,t_{0})=c(\lambda, \nu;q)H_{d}(x;q)=c(\lambda, \nu;q)\sum_{k=0}^d \dbinom{d}{k}_{q}x^{2k-d}.\]

Let us now calculate $Q_{x,q}^{(n-1,n)}(\nu, \lambda)$ for this specific case of partitions $\nu, \lambda$. The only choice for $\mu \in \Lambda_{n}$ such that $\nu \preceq \mu \preceq \lambda$ is the following
\[\mu_{s}\in [\max\{\nu_{s},\lambda_{s+1}\}, \min\{\nu_{s-1},\lambda_{s}\}]=[\max\{\nu_{s},\lambda_{s+1}\},\max\{\nu_{s},\lambda_{s+1}\}+d]\]
and
\[\mu_{j}= \max\{\nu_{j},\lambda_{j+1}\}=\min\{\nu_{j-1},\lambda_{j}\}, \text{ for }j\neq s.\]
Let us denote by $\mu^*$ the partition with the smallest possible coordinates, i.e. $\mu_{j}^*=\max\{\nu_{j},\lambda_{j+1}\}$, for $1\leq j \leq n$. Using Proposition \ref{leading_term} we have that
\[x^{2|\mu^*|-|\nu|-|\lambda|}=x^{-d}\]
therefore for $\mu \in \Lambda_{n}$ with $\mu_{s}=\max\{\nu_{s},\lambda_{s+1}\}+k$, for some $0\leq k \leq d$, it holds that
\[x^{2|\mu|-|\nu|-|\lambda|}=x^{2(|\mu^*|+k)-|\nu|-|\lambda|}=x^{2k-d}.\]
Moreover we have
\begin{equation*}
\begin{split}
\prod_{i=1}^{n-1}\dbinom{\lambda_{i}-\lambda_{i+1}}{\lambda_{i}-\mu_{i}}_{q}\dbinom{\mu_{i}-\mu_{i+1}}{\mu_{i}-\nu_{i}}_{q}&\dbinom{\lambda_{n}}{\lambda_{n}-\mu_{n}}_{q}\\
 = \prod_{i=1}^{n-1}&\dbinom{\lambda_{i}-\lambda_{i+1}}{\lambda_{i}-\mu^*_{i}}_{q}\dbinom{\mu^*_{i}-\mu^*_{i+1}}{\mu_{i}-\nu_{i}}_{q}\dbinom{\lambda_{n}}{\lambda_{n}-\mu^*_{n}}_{q}\\
&\times \dfrac{\dbinom{\lambda_{s}-\lambda_{s+1}}{\lambda_{s}-\mu_{s}}_{q}}{\dbinom{\lambda_{s}-\lambda_{s+1}}{\lambda_{s}-\mu^*_{s}}_{q}}\dfrac{\dbinom{\mu_{s-1}-\mu_{s}}{\mu_{s-1}-\nu_{s-1}}_{q}}{\dbinom{\mu^*_{s-1}-\mu^*_{s}}{\mu^*_{s-1}-\nu_{s-1}}_{q}}\dfrac{\dbinom{\mu_{s}-\mu_{s+1}}{\mu_{s}-\nu_{s}}_{q}}{\dbinom{\mu^*_{s}-\mu^*_{s+1}}{\mu^*_{s}-\nu_{s}}_{q}}.
\end{split}
\end{equation*}
By Proposition \ref{leading_term} we have that
\[\prod_{i=1}^{n-1}\dbinom{\lambda_{i}-\lambda_{i+1}}{\lambda_{i}-\mu^*_{i}}_{q}\dbinom{\mu^*_{i}-\mu^*_{i+1}}{\mu_{i}-\nu_{i}}_{q}\dbinom{\lambda_{n}}{\lambda_{n}-\mu^*_{n}}_{q}=c(\lambda, \nu;q).\]
Regarding the remaining part we calculate
\begin{equation*}
\begin{split}
\dbinom{\lambda_{s}-\lambda_{s+1}}{\lambda_{s}-\mu_{s}}_{q}&=\dbinom{\lambda_{s}-\lambda_{s+1}}{\lambda_{s}-\mu_{s}^*-k}_{q}\\
& = \dfrac{(q;q)_{\lambda_{s}-\lambda_{s+1}}}{(q;q)_{\lambda_{s}-\mu_{s}^*-k}(q;q)_{\mu_{s}^*-\lambda_{s+1}+k}}\\
& = \dfrac{(q;q)_{\lambda_{s}-\lambda_{s+1}}}{(q;q)_{\lambda_{s}-\mu_{s}^*}(q;q)_{\mu_{s}^*-\lambda_{s+1}}} \dfrac{(q;q)_{\lambda_{s}-\mu_{s}^*}(q;q)_{\mu_{s}^*-\lambda_{s+1}}}{(q;q)_{\lambda_{s}-\mu_{s}^*-k}(q;q)_{\mu_{s}^*-\lambda_{s+1}+k}}\\
&=\dbinom{\lambda_{s}-\lambda_{s+1}}{\lambda_{s}-\mu_{s}^*}_{q} \dfrac{(q;q)_{\lambda_{s}-\mu_{s}^*}(q;q)_{\mu_{s}^*-\lambda_{s+1}}}{(q;q)_{\lambda_{s}-\mu_{s}^*-k}(q;q)_{\mu_{s}^*-\lambda_{s+1}+k}}
\end{split}
\end{equation*}
and similarly we find that
\[\dbinom{\mu_{s-1}-\mu_{s}}{\mu_{s-1}-\nu_{s-1}}_{q}=\dbinom{\mu_{s-1}^*-\mu_{s}^*}{\mu_{s-1}^*-\nu_{s-1}}_{q} \dfrac{(q;q)_{\nu_{s-1}-\mu_{s}^*}(q;q)_{\mu_{s-1}^*-\mu_{s}^*-k}}{(q;q)_{\mu_{s-1}^*-\mu_{s}^*}(q;q)_{\nu_{s-1}-\mu_{s}^*-k}}\]
and
\[\dbinom{\mu_{s}-\mu_{s+1}}{\mu_{s}-\nu_{s}}_{q}=\dbinom{\mu_{s}^*-\mu_{s+1}^*}{\mu_{s}^*-\nu_{s}}_{q} \dfrac{(q;q)_{\mu_{s}^*-\nu_{s}}(q;q)_{\mu_{s}^*-\mu_{s+1}^*+k}}{(q;q)_{\mu_{s}^*-\mu_{s+1}^*}(q;q)_{\mu_{s}^*-\nu_{s}+k}}.\]
We claim that the coefficient of $x^{2k-d}$ equals 
\[c(\lambda, \nu;q)\dbinom{d}{k}_{q}\]
for every possible case of $\min\{\nu_{s-1},\lambda_{s}\}$, $\max\{\nu_{s},\lambda_{s+1}\}$ with $\min\{\nu_{s-1},\lambda_{s}\}-\max\{\nu_{s},\lambda_{s+1}\} = d$. Let us check the case $\min\{\nu_{s-1},\lambda_{s}\}=\nu_{s-1}$, $\max\{\nu_{s},\lambda_{s+1}\}=\lambda_{s+1}$. All the other cases can be checked in a similar way. If $\min\{\nu_{s-1},\lambda_{s}\}=\nu_{s-1}$, $\max\{\nu_{s},\lambda_{s+1}\}=\lambda_{s+1}$ then 
\begin{equation*}
\begin{split}
&\mu_{s}^* = \max\{\nu_{s},\lambda_{s+1}\}=\lambda_{s+1}\\
&\mu_{s-1}^* = \max\{\nu_{s-1},\lambda_{s}\}=\lambda_{s}\\
&\mu_{s+1}^* = \max\{\nu_{s+1},\lambda_{s+2}\}=\min\{\nu_{s},\lambda_{s+1}\}=\nu_{s}
\end{split}
\end{equation*}
where for the value of $\mu^*_{s+1}$ we recall that the equality $\max\{\nu_{j},\lambda_{j+1}\}=\min\{\nu_{j-1},\lambda_{j}\}$ holds for every $j \neq s$. Therefore
\[\dfrac{(q;q)_{\lambda_{s}-\mu_{s}^*}(q;q)_{\mu_{s}^*-\lambda_{s+1}}}{(q;q)_{\lambda_{s}-\mu_{s}^*-k}(q;q)_{\mu_{s}^*-\lambda_{s+1}+k}} = \dfrac{(q;q)_{\lambda_{s}-\lambda_{s+1}}}{(q;q)_{\lambda_{s}-\lambda_{s+1}-k}(q;q)_{k}}\]
\[\dfrac{(q;q)_{\nu_{s-1}-\mu_{s}^*}(q;q)_{\mu_{s-1}^*-\mu_{s}^*-k}}{(q;q)_{\mu_{s-1}^*-\mu_{s}^*}(q;q)_{\nu_{s-1}-\mu_{s}^*-k}} = \dfrac{(q;q)_{\nu_{s-1}-\lambda_{s+1}}(q;q)_{\lambda_{s}-\lambda_{s+1}-k}}{(q;q)_{\lambda_{s}-\lambda_{s+1}}(q;q)_{\nu_{s-1}-\lambda_{s+1}-k}}\]
\[\dfrac{(q;q)_{\mu_{s}^*-\nu_{s}}(q;q)_{\mu_{s}^*-\mu_{s+1}^*+k}}{(q;q)_{\mu_{s}^*-\mu_{s+1}^*}(q;q)_{\mu_{s}^*-\nu_{s}+k}} =  \dfrac{(q;q)_{\lambda_{s+1}-\nu_{s}}(q;q)_{\lambda_{s+1}-\nu_{s}+k}}{(q;q)_{\lambda_{s+1}-\nu_{s}}(q;q)_{\lambda_{s+1}-\nu_{s}+k}}=1\]
Combining the three ratios we conclude that the coefficient of $x^{2k-d}$ is given by
\begin{equation*}
\begin{split}
c(\lambda, \nu;q)\dbinom{\nu_{s-1}-\lambda_{s+1}}{k}_{q}&= c(\lambda, \nu;q)\dbinom{\min\{\nu_{s-1},\lambda_{s}\}-\max\{\nu_{s},\lambda_{s+1}\}}{k}_{q}
\end{split}
\end{equation*}
which equals $ c(\lambda, \nu;q)\dbinom{d}{k}_{q}$ as required.
\end{proof}

\begin{proposition}
Assume that the partitions $\lambda \in \Lambda_{n}$, $\nu \in \Lambda_{n-1}$ are such that there exist $d$ indices, for some $1\leq d \leq l(\lambda^\top)$, $1\leq s_{1}<...<s_{d} \leq n$ satisfying $s_{i+1}>s_{i}+1$ such that for $1\leq i \leq d$
\[\min \{\nu_{s_{i}-1},\lambda_{s_{i}}\}-\max\{\nu_{s_{i}}, \lambda_{s_{i}+1}\}=1\]
and
\[\max\{\nu_{j}, \lambda_{j+1}\}=\min \{\nu_{j-1},\lambda_{j}\}, \text{ for }j\neq s_{1},...,s_{d}.\]
Then $P_{\lambda\setminus \nu}(x;q,t_{0})$ does not depend on  $t_{0}$ and $Q^{(n-1, n)}_{x,q}(\nu, \lambda)=P_{\lambda\setminus \nu}(x;q,t_{0})$.
\end{proposition}
\begin{proof}
Assume that $\lambda$, $\nu$ are as in the statement of the Proposition, then the set $J^{\mathsf{c}}=\{1\leq j \leq l(\lambda^\top)=m:\nu_{j}^* = \lambda_{j}^*\}$, where we recall that $\nu_{j}^* = n-\nu^\top_{m-j+1}$ and $\lambda^*_{j}=n+1-\lambda^\top_{m-j+1}$, consists of $d$ indices $J^{\mathsf{c}}=\{i_{1},...,i_{d}\}$ and it holds that $\nu^*_{i_{j}}> \nu^*_{i_{j+1}}+1$ for every $j=1,...,d-1$. This case corresponds to partitions $\lambda, \mu$ as in figure \ref{fig:fig2}
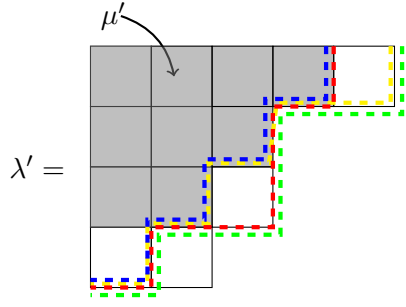
\begin{figure}[h]
\begin{center}
\scalebox{1}{
\begin{tikzpicture}
\node at (-1.5,-0.8) {$\lambda'=$};
\node at (-0.5,1.2) {$\mu'$};
\draw [thick] [->] (-0.4,1.2) to [bend left] (0.3,0.4);
\draw [fill, gray, opacity = 0.5] (0,0) rectangle (2.4,0.8);
\draw [black] (0,0) rectangle (0.8,0.8);
\draw [black] (0.8,0) rectangle (1.6,0.8);
\draw [black] (1.6,0) rectangle (2.4,0.8);
\draw [black] (2.4,0) rectangle (3.2,0.8);
\draw [black] (0,-0.8) rectangle (0.8,0);
\draw [fill, gray, opacity = 0.5] (0,-0.8) rectangle (1.6,0);
\draw [black] (0.8,-0.8) rectangle (1.6,0);
\draw [black] (0,-1.6) rectangle (0.8,-0.8);
\draw [fill, gray, opacity =0.5] (0,-1.6) rectangle (0.8,-0.8);
\draw [black] (0.8,-1.6) rectangle (1.6,-0.8);
\draw [black] (0,-2.4) rectangle (0.8,-1.6);
\draw [black] (0,0.8) rectangle (-0.8,0);
\draw [black] (0,0) rectangle (-0.8,-0.8);
\draw [black] (0,-0.8) rectangle (-0.8,-1.6);
\draw [black] (0,-1.6) rectangle (-0.8,-2.4);
\draw [fill, gray, opacity =0.5] (0,0.8) rectangle (-0.8,-1.6);
\draw [ultra thick, dashed, green] (3.3,0.8)--(3.3,-0.1)--(1.7,-0.1)--(1.7,-1.7)--(0.1,-1.7)--(0.1,-2.5)--(-0.8,-2.5);
\draw [ultra thick, dashed, yellow] (3.15,0.8)--(3.15,0.05)--(1.65,0.05)--(1.55,-0.75)--(0.75,-0.75)--(0.75,-1.55)--(-0.05,-1.55)--(-0.05,-2.35)--(-0.8,-2.35);
\draw [ultra thick, dashed, red] (2.4,0.8)--(2.4,0)--(1.6,0)--(1.6,-1.6)--(0,-1.6)--(0,-2.4)--(-0.8,-2.4);
\draw [ultra thick, dashed, blue] (2.3,0.8)--(2.3,0.1)--(1.5,0.1)--(1.5,-0.7)--(0.7,-0.7)--(0.7,-1.5)--(-0.05,-1.5)--(-0.05,-2.3)--(-0.8,-2.3);
\end{tikzpicture}
}
\caption{For the partitions $\lambda=(4,4,3,1,1)$ and $\mu = (3,3,2,1)$ (marked with gray), we have $d=d(\lambda, \mu)=|\{1\leq j \leq 4: \lambda_{j}^\top-\mu_{j}^\top = 1\}|=2$. The \enquote{single} boxes are well separated, i.e. there exists at least one column separating them. The dashed lines show the possible $\nu: \mu \preceq \nu \preceq \lambda$.}
\label{fig:fig2}
\end{center} 
\end{figure}\\
In this case, for a partition $(I_{+},I_{-},I_{0})$ of $J^{\mathsf{c}}$ we calculate
\[A(\nu^*;I_{+},I_{-},t_{0})=\prod_{j \in J^{\mathsf{c}}}t_{0}^{-\epsilon_{j}}\prod_{\substack{j,k\in J^{\mathsf{c}}\\j<k\\ \epsilon_{j}\neq \epsilon_{k}=0}}q^{-\epsilon_{j}}\]
so $A(\nu^*;I_{+},I_{-},t_{0})$ doesn't actually depend on the relation among the coordinates of $\nu^*$ or on the specific form of the set $J^{\mathsf{c}}$ but only on the number of indices it contains. Let us then consider the following quantity
\begin{equation*}
\begin{split}
U_{d,r}(q,t_{0})&=\sum_{\substack{I_{+},I_{-}\subset J^{\mathsf{c}}\\ I_{+}\cap I_{-}=\emptyset \\ |I_{+}|+|I_{-}|=d-r}}A(\nu^*;I_{+},I_{-},t_{0})\\
&=\sum_{\substack{I_{+},I_{-}\subset [d]\\ I_{+}\cap I_{-}=\emptyset \\ |I_{+}|+|I_{-}|=d-r}}\Bigg \{\prod_{j\in [d]}t_{0}^{-\epsilon_{j}}\prod_{\substack{j,k \in [d]\\j<k \\ \epsilon_{j}\neq \epsilon_{k}=0}}q^{-\epsilon_{j}}\Bigg\}
\end{split}
\end{equation*}
where $[d]=\{1,...,d\}$. Therefore we conclude that
\[P_{\lambda \setminus \nu}(x;q,t_{0})=c(\lambda, \nu;q)F_{d}(x;q,t_{0})\]
with
\begin{equation}
F_{d}(x;q,t_{0}) = \sum_{r=0}^d U_{d,r}(q,t_{0})\langle x;t_{0} \rangle_{q,r}.
\label{eq:branching_vDE_special2}
\end{equation}
The function $F_{d}(x;q,t_{0})$ satisfies the recursion
\[F_{d+1}(x;q,t_{0})=(x+x^{-1})F_{d}(x;q,t_{0})\]
with $F_{0}(x;q,t_{0})=1$. \\
Let us check the validity of this statement. We have
\begin{equation*}
\begin{split}
(x+x^{-1})F_{d}(x;&q,t_{0})\\
&=\sum_{r=0}^d(x+x^{-1}-t_{0}q^{r}-t_{0}^{-1}q^{-r})U_{d,r}(q,t_{0})\langle x;t_{0}\rangle_{q,r}\\
&\hspace{128pt}+\sum_{r=0}^d(t_{0}q^{r}+t_{0}^{-1}q^{-r})U_{d,r}(q,t_{0})\langle x;t_{0}\rangle_{q,r}\\
&= \sum_{r=0}^d U_{d,r}(q,t_{0})\langle x;t_{0}\rangle_{q,r+1}+\sum_{r=0}^d(t_{0}q^{r}+t_{0}^{-1}q^{-r})U_{d,r}(q,t_{0})\langle x;t_{0}\rangle_{q,r}\\
&= \sum_{r=1}^{d+1} U_{d,r-1}(q,t_{0})\langle x;t_{0}\rangle_{q,r}+\sum_{r=0}^d(t_{0}q^{r}+t_{0}^{-1}q^{-r})U_{d,r}(q,t_{0})\langle x;t_{0}\rangle_{q,r}.
\end{split}
\end{equation*}
We observe that $U_{d,d}(q,t_{0})=U_{d+1,d+1}(q,t_{0})=1$. Moreover for $r\geq 1$ we have
\begin{equation*}
\begin{split}
U_{d+1,r}(q,t_{0})&=\sum_{\substack{I_{+},I_{-}\subset [d+1]\\ I_{+}\cap I_{-}=\emptyset \\ |I_{+}|+|I_{-}|=d+1-r}}\Bigg \{\prod_{j\in [d+1]}t_{0}^{-\epsilon_{j}}\prod_{\substack{j,k \in [d+1]\\j<k \\ \epsilon_{j}\neq \epsilon_{k}=0}}q^{-\epsilon_{j}}\Bigg\}\\
&= \sum_{\substack{I_{+},I_{-}\subset [d+1]\setminus \{1\}\\ I_{+}\cap I_{-}=\emptyset \\ |I_{+}|+|I_{-}|=d+1-r\\ \epsilon_{1}=0}}\Bigg \{\prod_{j\in [d+1]}t_{0}^{-\epsilon_{j}}\prod_{\substack{j,k \in [d+1]\\j<k \\ \epsilon_{j}\neq \epsilon_{k}=0}}q^{-\epsilon_{j}}\Bigg\}\\
& \hspace{50pt}+\sum_{\substack{I_{+},I_{-}\subset [d+1]\setminus \{1\}\\ I_{+}\cap I_{-}=\emptyset \\ |I_{+}|+|I_{-}|=d-r\\ \epsilon_{1}=\pm 1}}\Bigg \{\prod_{j\in [d+1]}t_{0}^{-\epsilon_{j}}\prod_{\substack{j,k \in [d+1]\\j<k \\ \epsilon_{j}\neq \epsilon_{k}=0}}q^{-\epsilon_{j}}\Bigg\}\\
&= \sum_{\substack{I_{+},I_{-}\subset [d+1]\setminus \{1\}\\ I_{+}\cap I_{-}=\emptyset \\ |I_{+}|+|I_{-}|=d-(r+1)}}\Bigg \{\prod_{j\in [d+1]\setminus \{1\}}t_{0}^{-\epsilon_{j}}\prod_{\substack{j,k \in [d+1]\setminus \{1\}\\j<k \\ \epsilon_{j}\neq \epsilon_{k}=0}}q^{-\epsilon_{j}}\Bigg\}\\
& \hspace{50pt}+\sum_{\substack{I_{+},I_{-}\subset [d+1]\setminus \{1\}\\ I_{+}\cap I_{-}=\emptyset \\ |I_{+}|+|I_{-}|=d-r\\ \epsilon_{1}=\pm 1}}\Bigg \{t_{0}^{-\epsilon_{1}}\prod_{j\in [d+1]\setminus \{1\}}t_{0}^{-\epsilon_{j}}\underbrace{\prod_{\substack{k\in [d+1]\setminus \{1\}\\ \epsilon_{k}=0}}q^{-\epsilon_{1}}}_{q^{-k\epsilon_{1}}}\prod_{\substack{j,k \in [d+1]\setminus \{1\}\\j<k \\ \epsilon_{j}\neq \epsilon_{k}=0}}q^{-\epsilon_{j}}\Bigg\}\\
&=U_{d,r-1}(q,t_{0})+(t_{0}q^r+t_{0}^{-1}q^{-r})U_{d,r}(q,t_{0})
\end{split}
\end{equation*}
which proves the recursion. 

From the initial condition together with the recursion we conclude that the polynomials $F_{d}(x;q,t_{0})$ do not depend on $t_{0}$, therefore the branching polynomials $P_{\lambda \setminus\nu}(x;q,t_{0})$ do not depend  on $t_{0}$ either.

Let us now calculate the polynomial $Q_{x,q}^{(n-1,n)}(\nu,\lambda)$ for this case of $\nu, \lambda$. We have that the only choice for $\mu \in \Lambda_{n}$ such that $\nu \preceq \mu \preceq \lambda$ is the following:
\[\mu_{s_{i}}\in \{\max\{\nu_{s_{i}},\lambda_{s_{i}+1}\},\max\{\nu_{s_{i}},\lambda_{s_{i}+1}\}+1\}, \text{ for }i=1,...,d\]
and
\[\mu_{j}=\max\{\nu_{j},\lambda_{j+1}\}=\min\{\nu_{j-1},\lambda_{j}\}, \text{ for }j\neq s_{1},...,s_{d}.\]
Let us denote by $\mu^*$ the partition that corresponds to the smallest value for its coordinates, i.e.
\[\mu_{j}^*=\max\{\nu_{j},\lambda_{j+1}\}, \text{ for }1\leq j \leq n.\]
We proved in Proposition \ref{leading_term} that this corresponds to 
\[x^{2|\mu^*|-|\nu|-|\lambda|}=x^{-d}.\]
Fix $0\leq r \leq d$ and let us choose a subset $\mathcal{S}_{r}$ of the indices $\{s_{1},...,s_{d}\}$ such that $|\mathcal{S}_{r}|=r$ and consider the partition $\mu$ as follows
\begin{equation*}
\mu_{j}=\left\{ \begin{array}{ll}
\max\{\nu_{j},\lambda_{j+1}\}+1 & \text{ for }j\in \mathcal{S}_{r}\\
\max\{\nu_{j},\lambda_{j+1}\} & \text{ for }j \in \{s_{1},...,s_{d}\}\setminus \mathcal{S}_{r}\\
\max\{\nu_{j},\lambda_{j+1}\} = \min\{\nu_{j-1},\lambda_{j}\} & \text{ for }j \not \in \{s_{1},...,s_{d}\}
\end{array} \right. .
\end{equation*}
Then
\[x^{2|\mu|-|\nu|-|\lambda|}=x^{2(|\mu^*|+r)-|\nu|-|\lambda|}=x^{2r-d}.\]
Moreover we have
\begin{equation*}
\begin{split}
\prod_{i=1}^{n-1}\dbinom{\lambda_{i}-\lambda_{i+1}}{\lambda_{i}-\mu_{i}}_{q}\dbinom{\mu_{i}-\mu_{i+1}}{\mu_{i}-\nu_{i}}_{q}&\dbinom{\lambda_{n}}{\lambda_{n}-\mu_{n}}_{q}\\
 = \prod_{i=1}^{n-1}&\dbinom{\lambda_{i}-\lambda_{i+1}}{\lambda_{i}-\mu^*_{i}}_{q}\dbinom{\mu^*_{i}-\mu^*_{i+1}}{\mu_{i}-\nu_{i}}_{q}\dbinom{\lambda_{n}}{\lambda_{n}-\mu^*_{n}}_{q}\\
&\times \prod_{j \in \mathcal{S}_{r}}\dfrac{\dbinom{\lambda_{j}-\lambda_{j+1}}{\lambda_{j}-\mu_{j}}_{q}}{\dbinom{\lambda_{j}-\lambda_{j+1}}{\lambda_{j}-\mu^*_{j}}_{q}}\dfrac{\dbinom{\mu_{j-1}-\mu_{j}}{\mu_{j-1}-\nu_{j-1}}_{q}}{\dbinom{\mu^*_{j-1}-\mu^*_{j}}{\mu^*_{j-1}-\nu_{j-1}}_{q}}\dfrac{\dbinom{\mu_{j}-\mu_{j+1}}{\mu_{j}-\nu_{j}}_{q}}{\dbinom{\mu^*_{j}-\mu^*_{j+1}}{\mu^*_{j}-\nu_{j}}_{q}}.
\end{split}
\end{equation*}
By Proposition \ref{leading_term} we have that
\[\prod_{i=1}^{n-1}\dbinom{\lambda_{i}-\lambda_{i+1}}{\lambda_{i}-\mu^*_{i}}_{q}\dbinom{\mu^*_{i}-\mu^*_{i+1}}{\mu_{i}-\nu_{i}}_{q}\dbinom{\lambda_{n}}{\lambda_{n}-\mu^*_{n}}_{q}=c(\lambda, \nu;q).\]
Using the properties of the $q$-binomial recorded in \eqref{eq:q_binomial}, we calculate for each $j\in \mathcal{S}_{r}$ the following
\[\dbinom{\lambda_{j}-\lambda_{j+1}}{\lambda_{j}-\mu_{j}}_{q}=\dbinom{\lambda_{j}-\lambda_{j+1}}{\lambda_{j}-\mu^*_{j}-1}_{q}=\dbinom{\lambda_{j}-\lambda_{j+1}}{\lambda_{j}-\mu^*_{j}}_{q}\dfrac{1-q^{\lambda_{j}-\mu^*_{j}}}{1-q^{\mu^*_{j}-\lambda_{j+1}+1}}\]
\[\dbinom{\mu_{j-1}-\mu_{j}}{\mu_{j-1}-\nu_{j-1}}_{q}=\dbinom{\mu^*_{j-1}-\mu^*_{j}-1}{\mu^*_{j-1}-\nu_{j-1}}_{q}=\dbinom{\mu^*_{j-1}-\mu^*_{j}}{\mu^*_{j-1}-\nu_{j-1}}_{q}\dfrac{1-q^{\nu_{j-1}-\mu^*_{j}}}{1-q^{\mu_{j-1}^*-\mu^*_{j}}}\]
and
\[\dbinom{\mu_{j}-\mu_{j+1}}{\mu_{j}-\nu_{j}}_{q}=\dbinom{\mu^*_{j}-\mu^*_{j+1}+1}{\mu^*_{j}-\nu_{j}+1}_{q}=\dbinom{\mu^*_{j}-\mu^*_{j+1}}{\mu^*_{j}-\nu_{j}}_{q}\dfrac{1-q^{\mu_{j}^*-\mu_{j+1}^*+1}}{1-q^{\mu^*_{j}-\nu_{j}+1}}.\]
We claim that for any possible case of $\min\{\nu_{j-1},\lambda_{j}\}$, $\max\{\nu_{j},\lambda_{j+1}\}$, for $j\in \mathcal{S}_{r}$, the term corresponding to the partition $\mu$ equals
\[c(\lambda, \nu;q)x^{2r-d}.\]
Let us assume that for a fixed $j\in \mathcal{S}_{r}$, $\min\{\nu_{j-1},\lambda_{j}\}=\lambda_{j}$ and $\max\{\nu_{j},\lambda_{j+1}\}=\nu_{j}$, then 
\begin{equation*}
\begin{split}
&\mu^*_{j}=\max\{\nu_{j},\lambda_{j+1}\}=\nu_{j}\\
& \mu^*_{j-1}=\max\{\nu_{j-1},\lambda_{j}\}=\nu_{j-1}\\
&\mu_{j+1}^*=\max\{\nu_{j+1},\lambda_{j+2}\}=\min\{\nu_{j},\lambda_{j+1}\}=\lambda_{j+1}.
\end{split}
\end{equation*}
We remark that the equality $\max\{\nu_{j+1},\lambda_{j+2}\}=\min\{\nu_{j},\lambda_{j+1}\}$ holds since whenever $j \in \{s_{1},...,s_{d}\}$, $j+1\not \in \{s_{1},...,s_{d}\}$ (recall that $s_{j+1}>s_{j}+1$). We then compute the following
\[\dfrac{1-q^{\lambda_{j}-\mu^*_{j}}}{1-q^{\mu^*_{j}-\lambda_{j+1}+1}}=\dfrac{1-q^{\lambda_{j}-\nu_{j}}}{1-q^{\nu_{j}-\lambda_{j+1}+1}}\]
\[\dfrac{1-q^{\nu_{j-1}-\mu^*_{j}}}{1-q^{\mu_{j-1}^*-\mu^*_{j}}}=\dfrac{1-q^{\nu_{j-1}-\nu_{j}}}{1-q^{\nu_{j-1}-\nu_{j}}}=1\]
and
\[\dfrac{1-q^{\mu_{j}^*-\mu_{j+1}^*+1}}{1-q^{\mu^*_{j}-\nu_{j}+1}}=\dfrac{1-q^{\nu_{j}-\lambda_{j+1}+1}}{1-q^{\nu_{j}-\nu_{j}+1}}.\]
Combining all three ratios we conclude to the following expression 
\[\dfrac{1-q^{\lambda_{j}-\nu_{j}}}{1-q}=\dfrac{1-q^{\min\{\nu_{j-1},\lambda_{j}\}-\max\{\nu_{j},\lambda_{j+1}\}}}{1-q}=1\]
since for $j\in \{s_{1},...,s_{d}\}$, we assumed that
\[\min\{\nu_{j-1},\lambda_{j}\}-\max\{\nu_{j},\lambda_{j+1}\} = 1.\]
All the other cases can be calculated in a similar manner. 

Therefore $Q^{(n-1,n)}_{x,q}(\nu, \lambda)$ takes the form 
\[Q^{(n-1,n)}_{x,q}(\nu, \lambda)=\sum_{r=0}^d \sum_{\substack{\mathcal{S}_{r}\subset \{s_{1},...,s_{d}\}:\\ |\mathcal{S}_{r}|=r}}c(\lambda;\nu;q)x^{2r-d}=c(\lambda, \nu;q)\sum_{r=0}^d \dbinom{d}{r}x^{2r-d}.\]
Using the identity 
\[\dbinom{d+1}{r}= \dbinom{d}{r}+\dbinom{d}{r-1}\]
we observe that the function $\sum_{r=0}^d \dbinom{d}{r}x^{2r-d}$ satisfies the same recursion as the function $F_{d}(x;q,t_{0})$ with the same initial condition. Therefore we conclude that
\[P_{\lambda \setminus \nu}(x;q,t_{0})=Q^{(n-1,n)}_{x,q}(\nu, \lambda).\]
\end{proof}

%% file: tex/intertwining.tex
In this chapter we will present some of the background knowledge required for the study of the processes in the subsequent chapters.

\section{Intertwining of Markov processes}
In the introduction we presented some two-dimensional models evolving on Gelfand-Tsetlin cone, where under certain initial conditions, each single level evolves as a Markov process. In this chapter we will present the main theory that provides sufficient conditions for a function of a Markov process to be Markovian itself.

Let $X=(X_{t};t\geq 0)$ be a continuous-time Markov process with state space $S$ and let $f$ be a function from  $S$ to some other space $\hat{S}$. For $t\geq 0$, let $Y_{t}=f(X_{t})$. The result of Rogers-Pitman \cite{Rogers_Pitman_1981} provides sufficient conditions such that $Y=(Y_{t};t\geq 0)$ is Markovian. Here we will use a version of the Rogers-Pitman result proved in \cite{Jansen_Kurt_2014}.

Before we state the result we need to define the notion of Markov kernels.
\begin{definition}
\label{MarkovKernel}
Let $(S,\mathcal{S})$, $(\hat{S}, \hat{\mathcal{S}})$ be measurable spaces. A Markov kernel from $(S, \mathcal{S})$ to $(\hat{S}, \hat{\mathcal{S}})$ is a mapping $\mathcal{K}:S\times \mathcal{\hat{S}}\mapsto [0,1]$ with the following properties
\begin{enumerate}[1.]
\item $x\mapsto \mathcal{K}(x,B)$ is $\mathcal{S}$-measurable for every $B\in \hat{\mathcal{S}}$;
\item $B \mapsto \mathcal{K}(x,B)$ is a probability measure on $(\hat{S},\hat{\mathcal{S}})$ for every $x \in S$.
\end{enumerate}
\end{definition}
\begin{theorem} [\cite{Jansen_Kurt_2014}]
\label{rogers_pitman} 
Let $X=(X_{t},t \geq 0)$ be a continuous-time Markov process with state space $S$, initial distribution $\mu$ and transition kernel $P=(P_{t};t \geq 0)$. Let $f:S \to \hat{S}$, be a surjective function from $S$ to $\hat{S}$. Suppose there exists a Markov kernel $\mathcal{K}$ from $\hat{S}$ to $S$ supported on the set $\{y \in \hat{S}:f(x)=y\}$. Also assume there exists a second transition kernel $\hat{P}=(\hat{P}_{t};t \geq 0)$ on $\hat{S}$ satisfying for every $t> 0$ the intertwining relation 
\[\hat{P}_{t}\circ\mathcal{K}=\mathcal{K}\circ P_{t} .\]
Finally, we assume that the initial distribution of $X$ satisfies $\mu(\cdot) = \mathcal{K}(y,\cdot)$ for some $y\in \hat{S}$. Then, $Y$ is a Markov process on $\hat{S}$, with transition kernel $\hat{P}$, starting from $y$.
\end{theorem}
%\noindent{\color{red}The kernel $\mathcal{K}$ can be interpreted as a conditional probability, i.e. for every $A\in \mathcal{S}$, it holds
%\[\mathbb{P}(X_{t}\in A|f(X_{s}),0\leq s \leq t) = \mathcal{K}(f(X_{t}),A).\]}

Very often, the evolution of a continuous-time Markov process with discrete state space will be described by its transition rate matrix instead of its semigroup. The following result of Warren-Windridge provides conditions such that the intertwining of the transition rate matrices imply intertwining of the corresponding semigroups.
\begin{lemma}[\cite{Warren_Windridge_2009}]
\label{Q_intertwining}
Suppose $Q$ and $\hat{Q}$ are uniformly bounded conservative Q-matrices on discrete spaces $S$ and $\hat{S}$, i.e. there exists  
$c>0$ such that for all $s, s'\in S$ and $\hat{s}, \hat{s}'\in \hat{S}$
\[|Q(s,s')|,|\hat{Q}(\hat{s},\hat{s}')|<c. \] 
Moreover, we assume that $Q, \hat{Q}$ are intertwined by the Markov kernel $\mathcal{K}:\hat{S}\times S \to [0,1]$, i.e.
\[\hat{Q}\mathcal{K}=\mathcal{K}Q.\]
Then the transition kernels for the Markov processes with transition rate matrices $Q$ and $\hat{Q}$ are also intertwined. 
\end{lemma} 

A martingale formulation for the Rogers-Pitman result can be found in the work of Kurtz \cite{Kurtz_1998}.
Let us first introduce some notation. If $E$ is a topological space we denote by $B(E)$ the set of Borel measurable functions on $E$, by $C_{b}(E)$ the set of bounded continuous functions on $E$ and by $\mathcal{P}(E)$ the set of Borel probability measures on $E$. If $A$ is an operator, we denote by $\mathcal{D}(A)\subset C_{b}(E)$, the set of bounded functions $A$ acts on. 

\begin{theorem}[\cite{Kurtz_1998}, cor. 3.5]
\label{Kurtz}
Assume that $E$ is locally compact, that $A : \mathcal{D}(A) \subset C_b(E) \mapsto C_b(E)$, and that $\mathcal{D}(A)$ is closed under multiplication, separates points and is convergence determining. Let $F$ be another complete, separable metric space, $\gamma :E \mapsto F$ continuous and $\Lambda(y, dx)$ a Markov transition kernel from $F$ to $E$ such that
$\Lambda(g \circ \gamma) = g$ for all $g \in B(F)$, where $\Lambda f(y)=\int_{E}f(x)\Lambda(y,dx)$, for $f \in B(E)$. Let $B:\mathcal{D}(B)\subset B(F)\mapsto B(F)$ where $\Lambda(\mathcal{D}(A))\subset \mathcal{D}(B)$ and suppose
\[B\Lambda f = \Lambda A f, \qquad f \in \mathcal{D}(A).\]
Let $\mu \in \mathcal{P}(F)$ and set $\nu = \int_{F}\mu(dy)\Lambda(y,dx)\in \mathcal{P}(E)$. Suppose that the martingale problems for $(A, \nu)$ and $(B, \mu)$ are well-posed, and that $X$ is a solution to the martingale problem for $(A, \nu)$. Then $Y = \gamma \circ X$ is a Markov process and a solution to the
martingale problem for $(B, \mu)$. Furthermore, for each $t \geq 0$ and $g \in B(F)$ we have, almost surely,
\[E(g(X_{t})|Y_{s},0\leq s \leq t)=\int_{E}g(x)\Lambda(Y_{t},dx).\]
\end{theorem}

\section{Doob's $h$-processes}
In the introduction we referred to stochastic processes obtained from processes with killing via a transformation. In this section we will first present this transformation in the discrete time setting to get some intuition and then we will formalise the continuous time. For the discrete-time setting we refer to O' Connell \cite{O'Connell_2003b} and for the continuous-time to  draft notes of Bloemendal \cite{Bloemendal_2010}.

Let $\Sigma$ be a countably infinite set and let $\Pi: \Sigma \times \Sigma \mapsto [0,1]$ be a substochastic matrix, that is
\[\sum_{y\in \Sigma}\Pi(x,y)\leq 1\]
for all $x \in \Sigma$.\\
A function $h:\Sigma \mapsto \mathbb{R}$ is harmonic for $\Pi$ if for every $x \in \Sigma$
\[\sum_{y\in \Sigma}\Pi(x,y)h(y)=h(x).\]
Let us define a new matrix $\Pi_{h}$ as follows
\[\Pi_{h}(x,y)=\dfrac{h(y)}{h(x)}\Pi(x,y)\]
then, due to the fact that $h$ is harmonic for $\Pi$, it follows that $\Pi_{h}$ is stochastic
\[\sum_{y\in \Sigma}\Pi_{h}(x,y)=\dfrac{1}{h(x)}\sum_{y\in \Sigma}h(y)\Pi(x,y)=1.\]
The matrix $\Pi_{h}$ is called the \emph{Doob $h$-transform} of $\Pi$.

In the language of Markov chains we have the following interpretation. Let $X=(X_{n},n\geq 0)$ be a Markov chain on an extended space $\Sigma \cup \{\Delta\}$ which contains some special state here denoted by $\Delta$. Moreover assume that the transition probabilities for  $X$ are given by
\begin{equation*}
\begin{split}
\mathbb{P}[X_{n+1}=y|X_{n}=x]&=\Pi(x,y), \text{ for }x,y \in \Sigma\\
\mathbb{P}[X_{n+1}=\Delta|X_{n}=x]&=1-\sum_{y \in \Sigma}\Pi(x,y), \text{ for }x \in \Sigma\\
\mathbb{P}[X_{n+1}=\Delta|X_{n}=\Delta]&=1.
\end{split}
\end{equation*}
As we can see from the transition probabilities, $\Delta$ is an absorbing state. Suppose now we want to restrict the chain $X$ to $\Sigma$. Then the transition matrix for $X$ restricted to $\Sigma$ is $\Pi$. Then the Doob $h$-transform of $X$, denoted by $X^h$, is a Markov chain with state space $\Sigma$ and transition probabilities given by
\[\mathbb{P}[X^h_{n+1}=y|X^h_{n}=x]=\Pi_{h}(x,y)=\dfrac{h(y)}{h(x)}\Pi(x,y), \text{ for }x,y \in \Sigma\]

Let us now proceed to the continuous-time setting. We consider a Markov process $X=(X_{t})_{t \geq 0}$ with state space $S$ and transition kernel $P^t(x,dy)$. We want to consider a process obtained from $X$ which is conditioned to stay within a subset $A\subset S$ for every $t\geq 0$.

Let $h:S \to [0, \infty)$ be a function which is strictly positive and harmonic in $A$ and vanishes outside $A$, i.e.
\begin{equation*}
\begin{split}
P^th(x):= \int_{A}P^t(x,dy)h(y)=h(x)>0,& \text{  for }x\in A\\
h(x)=0, & \text{  for }x\in S \setminus A.
\end{split}
\end{equation*}

\begin{definition}
A function $h\in C_{0}(S)$ belongs to the domain, $\mathcal{D}(\mathcal{L})$, of the infinitesimal generator of a Feller process, if the limit
\[\mathcal{L}h:=\lim_{t \to 0}\dfrac{(P^t-I)h}{t}\]
exists, under the uniform topology on $C_{0}(S)$.\\
The operator $\mathcal{L}:\mathcal{D}(\mathcal{L})\mapsto C_{0}(S)$ is called the \emph{infinitesimal generator }of the process.
\end{definition}

\begin{remark}The condition $P^th=h$ can be rewritten as $\mathcal{L}h=0$, for $h \in \mathcal{D}(\mathcal{L})$. Then we will say that the function $h$ is \emph{$\mathcal{L}$-harmonic}.
\end{remark}

\begin{definition}
The Doob's $h$-process of $X$ is the process $\hat{X}$, with state space $A$ and transition kernel given by
\begin{equation}
\hat{P}^t(x,dy)=\dfrac{h(y)}{h(x)}P^t(x,dy).
\label{eq:Doob}
\end{equation}
\end{definition}

Let us now consider the case where the process $X$ is a diffusion on $\mathbb{R}^N$, possibly with some killing, i.e. $X$ has generator 
\[\mathcal{L} = \dfrac{1}{2}\sum_{i,j=1}^N a_{i,j}\dfrac{\partial^2}{\partial x_{i}\partial x_{j}} + \sum_{i=1}^N b_{i}\dfrac{\partial}{\partial x_{i}}-c\]
where the matrix $a(x)=(a_{ij}(x))_{i,j=1}^N$ is symmetric and non-negative, $b(x)$ denotes the drift and $c(x)\geq 0$ corresponds to the killing.

Let $h$ be a $\mathcal{L}$-harmonic function on $A\subset \mathbb{R}^N$ vanishing on its boundary. The Doob's $h$-transform of $X$ is then a diffusion on $A$ with generator given by
\[\hat{\mathcal{L}}=h^{-1}\mathcal{L}h=\dfrac{1}{2}\sum_{i,j=1}^N a_{i,j}\dfrac{\partial^2}{\partial x_{i}\partial x_{j}} + \sum_{i=1}^N b_{i}\dfrac{\partial}{\partial x_{i}}+\sum_{i,j=1}^N a_{ij}\Big(\dfrac{\partial}{\partial x_{j}}\log h\Big) \dfrac{\partial}{\partial x_{i}}.\]

%% file: tex/berele.tex
In chapter \ref{Koornwinder} we introduced the Young diagram which is a way of representing a partition via a collection of left justified boxes. We also introduced, see Definition \ref{tableau_def}, the Young tableau which is a Young diagram filled with entries from some alphabet $\{1,...,n\}$. It is well known (see \cite{Robinson_1938, Scensted_1961, Stanley_2001}) that there is a one-to-one correspondence, namely the Robinson-Schensted correspondence, between the set of words of length $k$ in the alphabet $\{1,...,n\}$ and the set of pairs of Young tableaux $(P,Q)$, where $P$ is a semistandard Young tableau with entries from the given alphabet and $Q$ is a standard tableau with entries $\{1,2,...,k\}$. The two tableaux have the same shape which is a partition of the length of the word.

A. Berele in \cite{Berele_1986} proved that similarly to the Robinson-Schensted algorithm, there exists a one-to-one correspondence between words of length $k$ from the alphabet $\{1<\bar{1}<...<n< \bar{n}\}$  and the sets of pairs $(P,(f^0,...,f^k))$, where $P$ is a symplectic tableau, as defined in \ref{symp_tableau_def}, of shape given by a partition of $k$ and $(f^0,...,f^k)$ is a recording sequence of up-down diagrams. The Berele correspondence can be used to prove identities for the symplectic Schur function. For example, Sundaram in \cite{Sundaram_1990} gave a combinatorial proof for the Cauchy identity for the symplectic group
\[\prod_{1\leq i<j \leq n}(1-b_{i}b_{j})\prod_{i,j=1}^n (1-b_{i}a_{j})^{-1}(1-b_{i}a_{j}^{-1})^{-1}=\sum_{\mu : l(\mu)\leq n}Sp^{(n)}_{\mu}(a_{1},...,a_{n})S^{(n)}_{\mu}(b_{1},...,b_{n}).\]
In section 4.1, we will describe the Berele insertion algorithm. 

A different way of representing a symplectic Young tableau is a symplectic Gelfand-Tsetlin pattern. In section 4.2 we will describe dynamics obtained from the Berele algorithm for the Gelfand Tsetlin pattern and we will deduce that for appropriate initial conditions, the marginal distribution of each even-indexed level of the pattern is a Markov process. In section 4.3, we will $q$-deform the dynamics described in section 4.2. All the proofs can be found in section 4.4.

\section{The Berele insertion algorithm}
We recall that the Robinson-Schensted algorithm has two versions; the row insertion and the column insertion. The Berele algorithm is similar to the row insertion version. Before we describe the insertion algorithm we need to introduce a sliding algorithm called the \textit{jeu de taquin} algorithm.
\begin{definition}
A punctured tableau of shape $z$ is a Young diagram of shape $z$ in which every box except one is filled. We will refer to this special box as the empty box. 
\end{definition}
\begin{definition}[jeu de taquin]
Let $T$ be a punctured tableau with $(\alpha, \beta)$ entry $t_{\alpha \beta}$, where $\alpha$ denotes a row and $\beta$ a column, and with empty box in position $(i,j)$. We consider the transformation $jdt: T \to jdt(T)$ defined as follows
\begin{itemize}
\item if $T$ is an ordinary tableau then $jdt(T)=T$;
\item while $T$ is a punctured tableau
\begin{equation*}
T \rightarrow \left\{ \begin{array}{ll}
 T\text{ switching the empty box and } t_{i,j+1} & \text{, if } t_{i,j+1}<t_{i+1,j}\\
 T\text{ switching the empty box and }t_{i+1,j} & \text{, if }t_{i,j+1}\geq t_{i+1,j}.
\end{array} \right.
\end{equation*}
\end{itemize}
Here we will use the convention that if the empty box has only one right/down neighbouring box $(\alpha, \beta)$, then
\[T \to T \text{ switching the empty box and }t_{\alpha \beta}.\]
\end{definition}
\begin{example} \hspace{5pt }
\begin{figure}[h]
\begin{center}
\scalebox{0.9}{
\begin{tikzpicture}
\node at (-0.5,0) {$T =$ };
\draw [black] (0,0) rectangle (0.8,0.8);
\draw [black] (0.8,0) rectangle (1.6,0.8);
\draw [black] (1.6,0) rectangle (2.4,0.8);
\draw [black] (0,-0.8) rectangle (0.8,0);
\draw [black] (0.8,-0.8) rectangle (1.6,0);
\node at (1.2,0.4) {1};
\node at (2,0.4) {2};
\node at (0.4,-0.4) {2};
\node at (1.2,-0.4) {2};

\draw [thick] [->] (3,0) -- (4.5,0) node[above] {$t_{12}<t_{21} \qquad \qquad$};

\draw [black] (5,0) rectangle (5.8,0.8);
\draw [black] (5.8,0) rectangle (6.6,0.8);
\draw [black] (6.6,0) rectangle (7.4,0.8);
\draw [black] (5,-0.8) rectangle (5.8,0);
\draw [black] (5.8,-0.8) rectangle (6.6,0);
\node at (5.4,0.4) {1};
\node at (7,0.4) {2};
\node at (5.4,-0.4) {2};
\node at (6.2,-0.4) {2};

\draw [thick] [->] (8,0) -- (9.5,0) node[above] {$t_{13}=t_{22}\qquad \qquad$ };

\draw [black] (10,0) rectangle (10.8,0.8);
\draw [black] (10.8,0) rectangle (11.6,0.8);
\draw [black] (11.6,0) rectangle (12.4,0.8);
\draw [black] (10,-0.8) rectangle (10.8,0);
\node at (10.4,0.4) {1};
\node at (12,0.4) {2};
\node at (10.4,-0.4) {2};
\node at (11.2,0.4) {2};

\node at (13.5,0) {$=jdt(T)$};
\end{tikzpicture}
}
\end{center} 
\end{figure}
\end{example}

We are now ready to describe the Berele insertion algorithm. To insert a letter $i$ from the alphabet $\{1<\bar{1}<...<n< \bar{n}\}$  to a symplectic tableau $P$, we begin by trying to place the letter at the end of the first row. If the result is a symplectic tableau we are done. Otherwise, the smallest entry which is larger than $i$ is bumped and we proceed by inserting the bumped letter to the second row and so on. If at some instance of the insertion process condition \textbf{(S3)} in Definition \ref{symp_tableau_def} is violated then we proceed as follows: assume that we tried to insert letter $l$ to the $l$-th row and bumped an $\bar{l}$. Since we cannot insert $\bar{l}$ to the $(l+1)$-th row we erase both $l$ and $\bar{l}$, leaving the position formerly occupied by the $\bar{l}$ as hole. Note that this is the only way condition \textbf{S3} is  violated as we can never try to insert a letter smaller than $l$ to the $l$-th row and if we try to insert a letter greater than $l$ this letter would not bump an $\bar{l}$. We then apply jeu de taquin slide algorithm to cover the hole. We now give an example of the Berele insertion algorithm.

\begin{example}
Insert $\bar{1}$ to \\
\begin{figure}[h]
\begin{center}
\scalebox{0.9}{
\begin{tikzpicture}
\node at (-0.5,-0.4) {P = };
\draw [black] (0,0) rectangle (0.8,0.8);
\draw [black] (0.8,0) rectangle (1.6,0.8);
\draw [black] (1.6,0) rectangle (2.4,0.8);
\draw [black] (2.4,0) rectangle (3.2,0.8);
\draw [black] (0,-0.8) rectangle (0.8,0);
\draw [black] (0.8,-0.8) rectangle (1.6,0);
\draw [black] (1.6,-0.8) rectangle (2.4,0);
\draw [black] (0,-1.6) rectangle (0.8,-0.8);
\draw [black] (0.8,-1.6) rectangle (1.6,-0.8);
\node at (0.4,0.4) {1};
\node at (1.2,0.4) {1};
\node at (2,0.4) {2};
\node at (2.8,0.4) {$\bar{2}$};
\node at (0.4,-0.4) {2};
\node at (1.2,-0.4) {$\bar{2}$};
\node at (2,-0.4) {3};
\node at (0.4,-1.2) {3};
\node at (1.2,-1.2) {$\bar{3}$};
\node at (3.6,0.4) {+$\bar{1}$};

\draw [thick] [->] (4,-0.4) -- (5.5,-0.4) node[above] {bump 2 \hspace{40pt} };

\draw [black] (6,0) rectangle (6.8,0.8);
\draw [black] (6.8,0) rectangle (7.6,0.8);
\draw [black] (7.6,0) rectangle (8.4,0.8);
\draw [black] (8.4,0) rectangle (9.2,0.8);
\draw [black] (6,-0.8) rectangle (6.8,0);
\draw [black] (6.8,-0.8) rectangle (7.6,0);
\draw [black] (7.6,-0.8) rectangle (8.4,0);
\draw [black] (6,-1.6) rectangle (6.8,-0.8);
\draw [black] (6.8,-1.6) rectangle (7.6,-0.8);
\node at (6.4,0.4) {1};
\node at (7.2,0.4) {1};
\node at (8,0.4) {$\mathbf{\bar{1}}$};
\node at (8.8,0.4) {$\bar{2}$};
\node at (6.4,-0.4) {2};
\node at (7.2,-0.4) {$\bar{2}$};
\node at (8,-0.4) {3};
\node at (6.4,-1.2) {3};
\node at (7.2,-1.2) {$\bar{3}$};
\node at (9.6,-0.4) {+2};

\draw [thick] [->] (10,-0.4) -- (11.5,-0.4) node[above] {bump $\bar{2}$ \hspace{40pt} };

\draw [black] (0,-3) rectangle (0.8,-2.2);
\draw [black] (0.8,-3) rectangle (1.6,-2.2);
\draw [black] (1.6,-3) rectangle (2.4,-2.2);
\draw [black] (2.4,-3) rectangle (3.2,-2.2);
\draw [black] (0,-3.8) rectangle (0.8,-3);
\draw [black] (0.8,-3.8) rectangle (1.6,-3);
\draw [black] (1.6,-3.8) rectangle (2.4,-3);
\draw [black] (0,-4.6) rectangle (0.8,-3.8);
\draw [black] (0.8,-4.6) rectangle (1.6,-3.8);
\node at (0.4,-2.6) {1};
\node at (1.2,-2.6) {1};
\node at (2,-2.6) {$\bar{1}$};
\node at (2.8,-2.6) {$\bar{2}$};
\node at (0.4,-3.4) {2};
\node at (1.2,-3.4) {$\mathbf{2}$};
\node at (2,-3.4) {3};
\node at (0.4,-4.2) {3};
\node at (1.2,-4.2) {$\bar{3}$};
\node at (3.6,-4.2) {+$\bar{2}$};

\draw [thick] [->] (4,-3.4) -- (5.5,-3.4) node[above] {cancel 2,$\bar{2}$ \hspace{40pt} };

\draw [black] (6,-3) rectangle (6.8,-2.2);
\draw [black] (6.8,-3) rectangle (7.6,-2.2);
\draw [black] (7.6,-3) rectangle (8.4,-2.2);
\draw [black] (8.4,-3) rectangle (9.2,-2.2);
\draw [black] (6,-3.8) rectangle (6.8,-3);
\draw [black] (6.8,-3.8) rectangle (7.6,-3);
\draw [black] (7.6,-3.8) rectangle (8.4,-3);
\draw [black] (6,-4.6) rectangle (6.8,-3.8);
\draw [black] (6.8,-4.6) rectangle (7.6,-3.8);
\node at (6.4,-2.6) {1};
\node at (7.2,-2.6) {1};
\node at (8,-2.6) {$\bar{1}$};
\node at (8.8,-2.6) {$\bar{2}$};
\node at (6.4,-3.4) {2};
\node at (8,-3.4) {3};
\node at (6.4,-4.2) {3};
\node at (7.2,-4.2) {$\bar{3}$};

\draw [thick] [->] (10,-3.4) -- (11.5,-3.4) node[above] {jeu de taquin \hspace{40pt} };

\draw [black] (0,-6) rectangle (0.8,-5.2);
\draw [black] (0.8,-6) rectangle (1.6,-5.2);
\draw [black] (1.6,-6) rectangle (2.4,-5.2);
\draw [black] (2.4,-6) rectangle (3.2,-5.2);
\draw [black] (0,-6.8) rectangle (0.8,-6);
\draw [black] (0.8,-6.8) rectangle (1.6,-6);
\draw [black] (0,-7.6) rectangle (0.8,-6.8);
\draw [black] (0.8,-7.6) rectangle (1.6,-6.8);
\node at (0.4,-5.6) {1};
\node at (1.2,-5.6) {1};
\node at (2,-5.6) {$\bar{1}$};
\node at (2.8,-5.6) {$\bar{2}$};
\node at (0.4,-6.4) {2};
\node at (1.2,-6.4) {3};
\node at (0.4,-7.2) {3};
\node at (1.2,-7.2) {$\bar{3}$};
\node at (4.5,-6.4) { = (P $\leftarrow \bar{1}$)};
\end{tikzpicture}
}
\end{center} 
\end{figure}
\end{example}
The insertion algorithm can be applied to a word $w=w_{1},...,w_{m}$, with $w_{i}\in \{1,\bar{1},...,n ,\bar{n}\}$, starting with an empty tableau. The output, denoted by $\mathcal{B}(w)$, is a symplectic Young tableau along with a sequence $(f^{0},...,f^{m})$ that records the shapes of the tableau for all the intermediate steps. More specifically, if we denote by $P(i)$ the tableau after the insertion of the $i$-th letter then $shP(i)=f^{i}$ for $i=1,..,m$. Since we start with an empty tableau, it follows that $f^{0}=\emptyset$. We also remark that any two consecutive shapes must differ by exactly one box, i.e. $f^{i}\setminus f^{i-1}=(1)$ or $f^{i-1}\setminus f^{i}=(1)$ since the insertion of a letter can lead to either the addition of a box or to a deletion of a box if the condition \textbf{S3} is violated.\\
\begin{example}
Applying the Berele algorithm to the word $w=\bar{3}2\bar{1}\bar{3}121$ yields the pair $\mathcal{B}(w)=(P,(f^{0},...,f^{7}))$, where

\begin{figure}[h]
\begin{center}
\scalebox{0.9}{
\begin{tikzpicture}
\node at (-0.4,-0.4) {$P(i):$};
\node at (-0.4,-3) {$f^i:$};

\node at (0.4,-0.4) {$\emptyset$};
\node at (0.4,-3) {$\emptyset$};

\draw [black] (1.3,0) rectangle (2.1,-0.8);
\node at (1.7,-0.4) {$\bar{3}$};
\node at (1.7,-3) {$(1)$};

\draw [black] (2.6,0) rectangle (3.4,-0.8);
\node at (3,-0.4) {2};
\draw [black] (2.6,-0.8) rectangle (3.4,-1.6);
\node at (3,-1.2) {$\bar{3}$};
\node at (3,-3) {$(1,1)$};

\draw [black] (3.9,0) rectangle (4.7,-0.8);
\node at (4.3,-0.4) {$\bar{1}$};
\draw [black] (3.9,-0.8) rectangle (4.7,-1.6);
\node at (4.3,-1.2) {2};
\draw [black] (3.9,-1.6) rectangle (4.7,-2.4);
\node at (4.3,-2) {$\bar{3}$};
\node at (4.3,-3) {$(1,1,1)$};

\draw [black] (5.2,0) rectangle (6,-0.8);
\node at (5.6,-0.4) {$\bar{1}$};
\draw [black] (5.2,-0.8) rectangle (6,-1.6);
\node at (5.6,-1.2) {2};
\draw [black] (5.2,-1.6) rectangle (6,-2.4);
\node at (5.6,-2) {$\bar{3}$};
\draw [black] (6,0) rectangle (6.8,-0.8);
\node at (6.4,-0.4) {$\bar{3}$};
\node at (6.2,-3) {$(2,1,1)$};

\draw [black] (7.3,0) rectangle (8.1,-0.8);
\node at (7.7,-0.4) {$2$};
\draw [black] (7.3,-0.8) rectangle (8.1,-1.6);
\node at (7.7,-1.2) {$\bar{3}$};
\draw [black] (8.1,0) rectangle (8.9,-0.8);
\node at (8.5,-0.4) {$\bar{3}$};
\node at (8.1,-3) {$(2,1)$};

\draw [black] (9.4,0) rectangle (10.2,-0.8);
\node at (9.8,-0.4) {$2$};
\draw [black] (9.4,-0.8) rectangle (10.2,-1.6);
\node at (9.8,-1.2) {$\bar{3}$};
\draw [black] (10.2,0) rectangle (11,-0.8);
\node at (10.6,-0.4) {$2$};
\draw [black] (10.2,-0.8) rectangle (11,-1.6);
\node at (10.6,-1.2) {$\bar{3}$};
\node at (10.2,-3) {$(2,2)$};

\draw [black] (11.5,0) rectangle (12.3,-0.8);
\node at (11.9,-0.4) {$1$};
\draw [black] (11.5,-0.8) rectangle (12.3,-1.6);
\node at (11.9,-1.2) {$2$};
\draw [black] (11.5,-1.6) rectangle (12.3,-2.4);
\node at (11.9,-2) {$\bar{3}$};
\draw [black] (12.3,0) rectangle (13.1,-0.8);
\node at (12.7,-0.4) {$2$};
\draw [black] (12.3,-0.8) rectangle (13.1,-1.6);
\node at (12.7,-1.2) {$\bar{3}$};
\node at (12.4,-3) {$(2,2,1)$ .};
\end{tikzpicture}
}
\end{center} 
\end{figure}
\end{example}
Analogously to the Robinson-Schensted correspondence, Berele proved the following result.
\begin{theorem}[\cite{Berele_1986}]
$\mathcal{B}$ is a bijection between words $w_{1},...,w_{m}$ in $\{1,\bar{1},...,n ,\bar{n}\}$ and pairs $(P,(f^0,...,f^m))$ in which $P$ is a symplectic tableau and $(f^0,...,f^{m})$ is a sequence of up/down shapes.
\end{theorem}
\section{Dynamics on Gelfand Tsetlin pattern}

As we already mentioned, a different way to represent a symplectic Young tableau is via a symplectic Gelfand-Tsetlin pattern. Let us make this statement more precise. For a symplectic tableau $P$ with entries from $\{1<\bar{1}<...<n<\bar{n}\}$ let $z^{2l-1}=shP^l$ and $z^{2l}=shP^{\bar{l}}$, where $P^k$ is the sub-tableau of $P$ that contains only entries that are less or equal to $k$. For example, if
\begin{figure}[h]
\begin{center}
\scalebox{0.9}{
\begin{tikzpicture}
\node at (-0.5,0) {P = };
\draw [black] (0,0) rectangle (0.8,0.8);
\draw [black] (0.8,0) rectangle (1.6,0.8);
\draw [black] (1.6,0) rectangle (2.4,0.8);
\draw [black] (2.4,0) rectangle (3.2,0.8);
\draw [black] (3.2,0) rectangle (4,0.8);
\draw [black] (0,-0.8) rectangle (0.8,0);
\draw [black] (0.8,-0.8) rectangle (1.6,0);
\node at (0.4,0.4) {1};
\node at (1.2,0.4) {$\bar{1}$};
\node at (2,0.4) {2};
\node at (2.8,0.4) {2};
\node at (3.6,0.4) {$\bar{2}$};
\node at (0.4,-0.4) {$\bar{2}$};
\node at (1.2,-0.4) {$\bar{2}$};
\end{tikzpicture}
}
\end{center} 
\end{figure}\\
\noindent then $z^1 = shP^{1}=(1)$, $z^{2}=shP^{\bar{1}}=(2)$, $z^3=shP^{2}=(4,0)$ and $z^4 = shP^{\bar{2}}=(5,2)$. By the definition of the symplectic Young tableau it follows that $z^{k-1}\preceq z^k$, for $1\leq k \leq N$. We recall the definition of the symplectic Gelfand-Tsetlin cone.
\begin{definition}
\label{GT_cone(Berele)}
Let $N$ be a positive integer. We denote by $\mathbb{K}^N_{\mathbb{Z}_{\geq 0}}=\mathbb{K}_{0}^N$ the symplectic Gelfand-Tsetlin cone, i.e. the set of integer-valued symplectic Gelfand Tsetlin patterns defined by
\[\mathbb{K}_{0}^N = \{\mathbf{z}=(z^1,...,z^N): z^{2l-1},z^{2l}\in \mathbb{Z}^l_{\geq 0}, \,0\leq l \leq \big[\frac{N+1}{2}\big] \text{ s.t. }z^{k-1}\preceq z^k, \, 2 \leq k \leq N\}\]
where the symbol $\big[ \frac{N+1}{2}\big]$ denotes the integer part of $\frac{N+1}{2}$.
\end{definition}
Schematically, a symplectic Gelfand-Tsetlin pattern is represented as in figure \ref{symp_GT_cone(Berele)}.
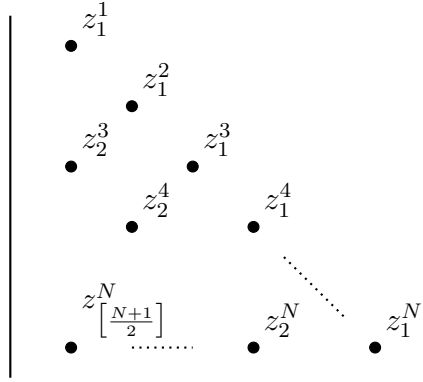
\begin{figure}[h]
\begin{center}
\begin{tikzpicture}[scale = 0.8]

\draw [thick] (0,0) -- (0,6) ;
\draw [fill] (1,5.5) circle [radius=0.09];
\node [above right, black] at (1,5.5) {\large{$z_{1}^1$}};
\draw [fill] (2,4.5) circle [radius=0.09];
\node [above right, black] at (2,4.5) {\large{$z_{1}^2$}};
\draw [fill] (3,3.5) circle [radius=0.09];
\node [above right, black] at (3,3.5) {\large{$z_{1}^3$}};
\draw [fill] (1,3.5) circle [radius=0.09];
\node [above right, black] at (1,3.5) {\large{$z_{2}^3$}};
\draw [fill] (4,2.5) circle [radius=0.09];
\node [above right, black] at (4,2.5) {\large{$z_{1}^4$}};
\draw [fill] (2,2.5) circle [radius=0.09];
\node [above right, black] at (2,2.5) {\large{$z_{2}^4$}};
\draw [fill] (6,0.5) circle [radius=0.09];
\node [above right, black] at (6,.5) {\large{$z_{1}^N$}};
\draw [fill] (4,0.5) circle [radius=0.09];
\node [above right, black] at (4,.5) {\large{$z_{2}^N$}};
\draw [fill] (1,0.5) circle [radius=0.09];
\node [above right, black] at (1,.5) {$z_{\tiny{\big[\frac{N+1}{2}\big]}}^N$};
\draw [dotted, thick] (4.5,2) -- (5.5,1);
\draw [dotted, thick] (2,0.5) -- (3,0.5);
\end{tikzpicture}
\end{center} 
\caption{This figure shows an element of $\mathbb{K}^{N}_0$. Both levels with indices $2l$, $2l-1$ contain $l$ particles. At each level the particles satisfy the interlacing property, i.e. $z^{2l-1}_{l}\leq z^{2l}_{l}\leq z^{2l-1}_{l-1}\leq ... \leq z^{2l-1}_{1}\leq z^{2l}_{1}$ and $z^{2l+1}_{l+1}\leq z^{2l}_{l}\leq z^{2l+1}_{l}\leq ... \leq z^{2l}_{1}\leq z^{2l+1}_{1}$.} 
\label{symp_GT_cone(Berele)}
\end{figure}
In this chapter we will only focus on patterns with even number of levels.

For $t\geq 0$, let $\mathbf{Z}(t)$ denote the configuration of the pattern at time $t$. We construct a process $(\mathbf{Z}(t),t \geq 0)$ on $\mathbb{K}_{0}^{2n}$ whose dynamics are dictated by Berele's insertion algorithm. The main characteristic of the dynamics is that any transition propagates all the way to the bottom of the pattern. So, whenever a particle attempts a jump, this will cause the movement of exactly one of its lower nearest neighbours, which will trigger the movement of one of its own nearest neighbours and so on. 

We fix $a = (a_{1},...,a_{n})\in \mathbb{R}^n_{>0}$. We will describe the dynamics of the process and then give an example.

The top particle $Z^1_{1}$ attempts a right jump after waiting an exponential time with parameter $a_{1}$. Let us denote by $T$ a jump time and by $T-$ the time before the jump. If $Z^1_{1}(T-)=Z^2_{1}(T-)$ then the jump is performed and $Z^2_{1}$ is simultaneously pushed to the right. Otherwise, the jump is suppressed and $Z^2_{1}$ jumps to the left instead.\\
The particles at the right edge, are the only ones that can jump to the right of their own volition. More specifically, the particle $Z^{2l-1}_{1}$ take steps to the right at rate $a_{l}$ and $Z^{2l}_{1}$ at rate $a_{l}^{-1}$. Such a jump corresponds to the insertion of letter $l$ and $\bar{l}$ respectively.\\
If at some jump time $T$ a particle $Z^k_{i}$, for some $1\leq i \leq [\frac{k+1}{2}]$, $1\leq k \leq 2n$, attempts a rightward jump, then
\begin{enumerate}[i)]
\item if $Z^k_{i}(T-)=Z^{k+1}_{i}(T-)$ the jump is performed and $Z^{k+1}_{i}$ is simultaneously pushed one step to the right. This means that a letter is either added at the end of the $i$-th row or it is added somewhere in the interior of the row causing a greater letter to be bumped to the next row of the corresponding tableau;
\item if $Z^k_{i}(T-)<Z^{k+1}_{i}(T-)$ we have two different cases
\begin{enumerate}[a.]
\item if $i=\frac{k+1}{2}$, with $k>1$ odd, the jump is suppressed and $Z^{k+1}_{i}$ is pulled to the left instead. The transition described here corresponds to the cancellation step in Berele's algorithm;
\item for all the other particles, the jump is performed and the particle $Z^{k+1}_{i+1}$ is pulled to the right. This means that a letter is added at the $i$-th row of the corresponding tableau and bumped another letter to the row beneath.
\end{enumerate}
\end{enumerate}

If at some jump time $T$, $Z^k_{i}$ performs a left jump, then it triggers the leftward move of exactly one of its nearest lower neighbours; the left, $Z^{k+1}_{i+1}$, if $Z^{k}_{i}(T-)=Z^{k+1}_{i+1}(T-)$, and the right, $Z^{k+1}_{i}$, otherwise. If $k=2n$ then the procedure ends with the left jump of $Z^{2n}_{i}$. The transitions we describe here corresponds to the jeu de taquin step of Berele's algorithm. Let us explain why.\\
We assume without loss of generality that $k=2l$ and at some jump time $T$, $Z^{2l}_{i}$, for some $1\leq i \leq l$, jumps to the left. This means that a letter $\leq \bar{l}$ is removed from a box in the $i$-th row of the tableau. Note that the empty box can now have only letters $\geq l+1$ to its right. If $Z^{2l}_{i}(T-)=Z^{2l+1}_{i+1}(T-)$, then beneath the empty box there is a letter $\leq l+1$, therefore according to the jeu de taquin algorithm we should swap the empty box with the box beneath. Therefore, the $(i+1)$-th row now contains one less box with entry $\leq l+1$, causing $Z^{2l+1}_{i+1}$ to jump to the left.

\begin{example}
At a jump time $T$, in step a) the particle $Z^2_{1}$ performs a rightward jump. Since $Z^2_{1}(T-)<Z^3_{1}(T-)$, $Z^2_{1}$ triggers the move of $Z^3_{2}$. In b), $Z^3_{2}$ attempts to jump to the right, but since $Z^3_{2}(T-)<Z^4_{2}(T-)$, the jump is suppressed. Finally, in step c) since the jump of the particle $Z^3_{2}$ was suppressed , $Z^4_{2}$ performs a leftward jump instead.
\begin{figure}[h]
\begin{center}
\begin{tikzpicture}[scale = 0.8]
% step 1
\node [right, black] at (-1,6) {\textit{a)}};
\draw [thick] (0,2) -- (0,6.5) ;
\draw [fill] (1,5.5) circle [radius=0.09];
\node [above, black] at (1,5.5) {1};
\draw [fill] (2,4.5) circle [radius=0.09];
\node [above, black] at (2,4.5) {2};
\draw [thick, red] [->](2.15,4.55) to [bend left = 60] (2.85,4.55);
\draw [fill, blue] (3,4.5) circle [radius=0.09];
\node [above, blue] at (3,4.5) {3};
\draw [fill] (3,3.5) circle [radius=0.09];
\node [above, black] at (3,3.5) {3};
\draw [fill] (1,3.5) circle [radius=0.09];
\node [above, black] at (1,3.5) {1};
\draw [fill] (4,2.5) circle [radius=0.09];
\node [above, black] at (4,2.5) {4};
\draw [fill] (2,2.5) circle [radius=0.09];
\node [above, black] at (2,2.5) {2};
% step 2
\node [right, black] at (5,6) {\textit{b)}};
\draw [thick] (6,2) -- (6,6.5) ;
\draw [fill] (7,5.5) circle [radius=0.09];
\node [above, black] at (7,5.5) {1};
\draw [fill, blue] (9,4.5) circle [radius=0.09];
\node [above, blue] at (9,4.5) {3};
\draw [fill] (9,3.5) circle [radius=0.09];
\node [above, black] at (9,3.5) {3};
\draw [fill] (7,3.5) circle [radius=0.09];
\node [above, black] at (7,3.5) {1};
\draw [thick, red] [->](7.1,3.6) to [bend left = 60] (7.9,3.6);
\draw [thick, red] (7.3,3.4)--(7.5,4);
\draw [fill] (10,2.5) circle [radius=0.09];
\node [above, black] at (10,2.5) {4};
\draw [fill] (8,2.5) circle [radius=0.09];
\node [above, black] at (8,2.5) {2};
% step 3
\node [right, black] at (11,6) {\textit{c)}};
\draw [thick] (12,2) -- (12,6.5) ;
\draw [fill] (13,5.5) circle [radius=0.09];
\node [above, black] at (13,5.5) {1};
\draw [fill, blue] (15,4.5) circle [radius=0.09];
\node [above, blue] at (15,4.5) {3};
\draw [fill] (15,3.5) circle [radius=0.09];
\node [above, black] at (15,3.5) {3};
\draw [fill] (13,3.5) circle [radius=0.09];
\node [above, black] at (13,3.5) {1};
\draw [fill] (16,2.5) circle [radius=0.09];
\node [above, black] at (16,2.5) {4};
\draw [fill] (14,2.5) circle [radius=0.09];
\node [above, black] at (14,2.5) {2};
\draw [fill, blue] (13,2.5) circle [radius=0.09];
\node [above, blue] at (13,2.5) {1};
\draw [thick, red] [->](13.9,2.6) to [bend right = 60] (13.1,2.6);
\end{tikzpicture}
\end{center} 
\end{figure}
\end{example}

We recall the notation $\mathcal{W}_{0}^n = \{z \in \mathbb{Z}^n:z_{1}\geq ... \geq z_{n}\geq 0\}$ and the set of symplectic patterns $\mathbf{z}$ in $\mathbb{K}_{0}^{2n}$ with bottom row $z^{2n}$ equal to $z\in \mathcal{W}_{0}^n$, denoted by $\mathbb{K}_{0}^{2n}[z]$. The geometric weight $w_{2n}^a$ on $\mathbb{K}_{0}^{2n}$ is
\[w_{2n}^a(\mathbf{z})=\prod_{i=1}^n a_{i}^{2|z^{2i-1}|-|z^{2i}|-|z^{2i-2}|}=\prod_{i=1}^n (a_{i}^{-1})^{|z^{2i}|-|z^{2i-1}|}(a_{i})^{|z^{2i-1}|-|z^{2i-2}|}\]
using the convention that $|z^0|=0$. 

We recall that the symplectic Schur function $Sp^{(n)}_{z}$ parametrized by $z \in \mathcal{W}^n_{0}$ is given by (see Definition \ref{Symp_Schur} for the combinatorial formula of the symplectic Schur function due to King \cite{King_1971})
\[Sp_{z}^{(n)}(a_{1},...,a_{n})=\sum_{\mathbf{z}\in \mathbb{K}_{0}^{2n}[z]}w_{2n}^a(\mathbf{z}).\]
For $z\in \mathcal{W}_{0}^{n}$ we define $M^{n}_{a}(\cdot ; z):\mathbb{K}_{0}^{2n}[z] \to [0,1]$ by
\[M^{n}_{a}(\cdot;z)=\dfrac{w_{2n}^a(\cdot)}{Sp_{z}^{(n)}(a)}.\]
It then follows that $M^{n}_{a}(\cdot;z)$ gives a probability distribution on patterns in $\mathbb{K}_{0}^{2n}[z]$.\\

In the Introduction we presented a continuous-time Markov process with symplectic Schur dynamics, specified by the transition rate matrix $Q_{n}:\mathcal{W}_{0}^n \times \mathcal{W}_{0}^n \to \mathbb{R}$ as follows. For $z \in \mathcal{W}_{0}^n$ and $z\pm e_{i}\in \mathcal{W}_{0}^n$ for some $1\leq i \leq n$
\[Q_{n}(z, z \pm e_{i})= \dfrac{Sp_{z\pm e_{i}}^{(n)}(a_{1},...,a_{n})}{Sp_{z}^{(n)}(a_{1},...,a_{n})}.\]
All the other off-diagonal entries of $Q_{n}$ are zero and the diagonals are given by
\[Q_{n}(z,z) = -\sum_{i=1}^n (a_{i}+a_{i}^{-1}), \qquad z\in \mathcal{W}_{0}^{n}.\]

We then have the following result for the shape of the Gelfand-Tsetlin pattern.
\begin{theorem}Suppose the process $\mathbf{Z} = (\mathbf{Z}(t);t \geq 0)$ has initial distribution given by $M^{n}_{a}(\cdot;z)$, for some $z\in \mathcal{W}_{0}^n$, and evolves according to Berele's dynamics. Then $(Z^{2n}(t);t \geq 0)$ is distributed as a process with transition rate matrix $Q_{n}$, started from $z$. Moreover, for each $t>0$, the conditional law of $\mathbf{Z}(t)$ given $\{Z^{2n}(s),s\leq t\}$ is given by $M^n_{a}(\cdot;Z^{2n}(t))$.
\end{theorem}
The proof of the theorem is omitted as it is similar to the proof of the result for the $q$-deformed process that we will study in the following section. 

\section{A $q$-deformation of Berele's insertion algorithm}
In this section we consider a generalization of Berele's insertion algorithm that depends on a parameter $q\in (0,1)$. The $q$-deformed algorithm can be thought as a randomization of the usual algorithm, as inserting a letter to a given tableau results in a distribution over a set of tableaux. Since a symplectic tableau is equivalent to a symplectic Gelfand-Tsetlin pattern we will present the algorithm in terms of dynamics on the pattern.

For $q\in (0,1)$ and $(x,y) \in \mathbb{R}^n \times \mathbb{R}^n$ or $(x,y) \in \mathbb{R}^{n-1} \times \mathbb{R}^n$, we define the quantities
\begin{equation}
r_{i}(y;x)=q^{y_{i}-x_{i}}\dfrac{1-q^{x_{i-1}-y_{i}}}{1-q^{x_{i-1}-x_{i}}} \text{,\hspace{5pt} }l_{i}(y;x)=q^{x_{i}-y_{i+1}}\dfrac{1-q^{y_{i+1}-x_{i+1}}}{1-q^{x_{i}-x_{i+1}}}
\label{eq:probabilities(Berele)}
\end{equation}
for $1\leq i \leq n$, with the convention that if $x=(x_{1},...,x_{n}) \in \mathbb{R}^n$, we set $x_{n+1}\equiv 0$ and $x_{0}\equiv \infty$ (similarly for $y$) and the quantities $l_{i}(y;x),r_{i}(y;x)$ are modified accordingly. The probabilities $r_{i},l_{i}$ we introduce here have appeared in the literature before. For example, O'Connell-Pei in \cite{O'Connell_Pei_2013} use the probabilities $r_{i}$ (compare with \eqref{eq:q_push_probability_intro}) for the $q$-deformation of the RSK algorithm.

We fix $n\in \mathbb{N}$ and set $N=2n$. Let us also consider a vector $a=(a_{1},...,a_{n})\in \mathbb{R}^n_{>0}$. We then construct a process $(\mathbf{Z}(t),t\geq 0)$ in the symplectic Gelfand Tsetlin cone which evolves according to the rules that follow.

The top particle $Z^1_{1}$ attempts a right jump after waiting an exponential time with parameter $a_{1}$. With probability $r_{1}(Z^2;Z^1)$ the jump is performed and $Z^2_{1}$ is simultaneously pushed to the right and with probability $1-r_{1}(Z^2;Z^1)$ the jump is suppressed and $Z^2_{1}$ jumps to the left instead.\\
The particles at the right edge, are the only ones that can jump to the right of their own volition. More specifically, the particle $Z^{2l}_{1}$ take steps to the right at rate $a_{l}^{-1}$ and $Z^{2l-1}_{1}$ at rate $a_{l}$.\\
If a particle $Z^k_{i}$ attempts a rightward jump, then
\begin{enumerate}[i)]
\item if $i=\frac{k+1}{2}$, with $k>1$ odd, then with probability $r_{i}(Z^{k+1};Z^k)$, the jump is performed and $Z^{k+1}_{i}$ is pushed to the right. Otherwise, with probability $1-r_{i}(Z^{k+1};Z^k)$ the jump is suppressed and $Z^{k+1}_{i}$ is pulled to the left instead;
\item for all the other particles, the jump is performed and either the particle $Z^{k+1}_{i}$ is pushed to the right with probability $r_{i}(Z^{k+1};Z^k)$ or the particle $Z^{k+1}_{i+1}$ is pulled to the right with probability $1-r_{i}(Z^{k+1};Z^k)$.
\end{enumerate}

If $Z^k_{i}$ performs a left jump, then it triggers the leftward move of exactly one of its nearest lower neighbours; the left, $Z^{k+1}_{i+1}$, with probability $l_{i}(Z^{k+1};Z^k)$ and the right, $Z^{k+1}_{i}$, with probability $1-l_{i}(Z^{k+1};Z^k)$.

Observe that if $x\preceq y$ with $x_{i}=y_{i}$, for some $i$, then $r_{i}(y;x)=1$, which means that if $x_{i}\to x_{i}+1$, then necessarily $y_{i}\to y_{i}+1$ in order to maintain the interlacing condition among $x,y$. Similarly, if $x_{i}=y_{i+1}$, then $l_{i}(y;x)=1$ causing the transition $y_{i+1}\to y_{i+1}-1$ whenever $x_{i}\to x_{i}-1$. Moreover, as $q\to 0$, the probabilities $r_{i}(y;x)$ and $l_{i}(y;x)$ converge to $\mathbbm{1}_{\{x_{i}=y_{i}\}}$ and $\mathbbm{1}_{\{x_{i}=y_{i+1}\}}$ and we therefore recover the dynamics from the usual Berele algorithm.

For each symplectic Gelfand-Tsetlin pattern $\mathbf{z} \in \mathbb{K}^{2n}_{0}$, we consider a real-valued weight $w^{2n}_{a,q}(\mathbf{z})$ given by
\begin{equation}
w^{2n}_{a,q}(\mathbf{z}) = \prod_{k=1}^n \Lambda_{k-1,k}^{a_{k},q} (z^{2k-2},z^{2k-1}, z^{2k})
\label{eq:q_weights(Berele)}
\end{equation}
where 
\begin{equation*}
\begin{split}
 \Lambda_{k-1,k}^{a_{k},q} (z^{2k-2},z^{2k-1}, z^{2k})=& a_{k}^{2|z^{2k-1}|-|z^{2k-2}|-|z^{2k}|}\\
& \times\mathlarger{\mathlarger{‎‎\prod}}_{i=1}^{k-1} \dbinom{z^{2k-1}_{i}-z^{2k-1}_{i+1}}{z^{2k-1}_{i}-z^{2k-2}_{i}}_{q} \dbinom{z^{2k}_{i}-z^{2k}_{i+1}}{z^{2k}_{i}-z^{2k-1}_{i}}_{q} \dbinom{z^{2k}_{k}}{z^{2k}_{k}-z^{2k-1}_{k}}_{q} .
\end{split}
\end{equation*}
If $k=1$, then 
\[\Lambda_{0,1}^{a_{1},q}(z^0,z^1,z^2) \equiv \Lambda_{0,1}^{a_{1},q}(z^1,z^2) = a_{1}^{2|z^1|-|z^2|}\dbinom{z^2_{1}}{z^2_{1}-z^1_{1}}_{q}.\]

Fix $z \in \mathcal{W}_{0}^n$, then the function $\mathcal{P}_{z}^{(n)}(\cdot ; q)$ defined in \ref{recursion_q_whittaker} is given by
\begin{equation}
\mathcal{P}_{z}^{(n)}(a;q) := \sum_{\mathbf{z}\in \mathbb{K}^{2n}_{0}[z]}w^{2n}_{a,q}(\mathbf{z}).
\label{eq:symp_q_whittaker(Berele)}
\end{equation}
Therefore, a natural probability mass function on patterns $\mathbf{z}\in \mathbb{K}^{2n}_0[z]$ can be given by
\begin{equation}
M_{a,q}^{n}(\cdot;z)=\dfrac{w^{n}_{a,q}(\cdot)}{\mathcal{P}^{(n)}_{z}(a;q)}.
\label{eq:initial_distribution(Berele)}
\end{equation}

The main result of the chapter is the following.
\begin{theorem}
\label{q_main_theorem(Berele)}
Let $(\mathbf{Z}(t),t \geq 0)$ be a Markov process with state space $\mathbb{K}^{2n}_{0}$, initial distribution $M_{a,q}^{n}(\cdot;z)$, for some $z\in \mathcal{W}_{0}^{n}$, that evolves according to the $q$-deformed Berele dynamics. Then $(Z^{2n}(t),t \geq 0)$ is a Markov process on $\mathcal{W}_{0}^n$, started at $z$, with transition rates $\{Q^q_{n}(z,z'), z,z' \in \mathcal{W}_{0}^{n}\}$ given for $z'\neq z$ by
\begin{equation}
Q^q_{n}(z,z')=\dfrac{\mathcal{P}^{(n)}_{z'}(a;q)}{\mathcal{P}^{(n)}_{z}(a;q)}f_{n}(z,z')
\label{eq:rate_matrix_q_berele(Berele)}
\end{equation}
where the function $f_{n}:\mathcal{W}_{0}^n \times \mathcal{W}_{0}^n \to \mathbb{R}_{+}$ is given by
\begin{equation*}
f_{n}(z,z') = \left\{ \begin{array}{ll}
1-q^{z_{i-1}-z_{i}} & \text{ if } z'-z = e_{i}, 1\leq i \leq n\\
1-q^{z_{i}-z_{i+1}} & \text{ if } z'-z =- e_{i}, 1\leq  i \leq n\\
0 & \text{ otherwise}
\end{array} \right.
\end{equation*}
where we use the convention $z_{0}=+\infty$ and $z_{n+1}=0$. The diagonal entries are
\[Q^q_{n}(z,z) =-\sum_{i=1}^{n}(a_{i}+a_{i}^{-1}).\]
Moreover, for each $t>0$ the conditional law of $\mathbf{Z}(t)$ given $\{Z^{2n}(s),s\leq t\}$ is given by $M^n_{a,q}(\cdot;Z^{2n}(t))$.
\end{theorem}

While proving theorem \ref{q_main_theorem(Berele)}, we will conclude that the matrix $\{Q^q_{n}(z,z'), z,z' \in \mathcal{W}_{0}^{n}\}$ is conservative, i.e. it holds that $\sum_{z'\in \mathcal{W}^n_{0}}Q^q_{n}(z,z')=0$, for every $z \in \mathcal{W}^n_{0}$. Hence the following corollary is direct.
\begin{corollary}
\label{Pieri_mypolynomials(Berele)}
Let $n\geq 1$, $z \in \mathcal{W}^n_{0}$ and $q \in (0,1)$. The function $\mathcal{P}_{z}^{(n)}(\cdot ; q)$ satisfies the identity
\[\sum_{i=1}^n(a_{i}+a_{i}^{-1})\mathcal{P}_{z}^{(n)}(a;q)=\sum_{i=1}^n \Big(\mathcal{P}_{z+e_{i}}^{(n)}(a;q)f_{n}(z,z+e_{i})+\mathcal{P}_{z-e_{i}}^{(n)}(a;q)f_{n}(z,z-e_{i})\Big)\]
for $a=(a_{1},...,a_{n})\in \mathbb{C}^n\setminus\{0\}$.
\end{corollary}
 Although for the scope of this section we assume that the indeterminate $a$ is real-valued, the proof of the Corollary does not depend on this assumption and the identity in the Corollary holds for all $a \in \mathbb{C}^n\setminus \{0\}$. Therefore we conclude that the function $\mathcal{P}^{(n)}(\cdot;q)$ satisfies the Pieri identity for the $q$-deformed $\mathfrak{so}_{2n+1}$-Whittaker function we stated in Proposition \ref{q_eigenrelation_partition} supporting our conjecture that the $q$-deformed $\mathfrak{so}_{2n+1}$-Whittaker functions coincide with the polynomials $\mathcal{P}^{(n)}$ defined in \ref{recursion_q_whittaker}.
\section{Proof of Theorem \ref{q_main_theorem(Berele)}}

Let us denote by $\hat{Q}^q_{n} = \{\hat{Q}^q_{n}(\mathbf{z},\mathbf{w}), \mathbf{z},\mathbf{w}\in \mathbb{K}^{2n}_{0}\}$ the transition kernel of the process $\mathbf{Z}$ which evolves according to the $q$-deformed Berele dynamics. We also introduce the kernel $\mathcal{L}_{n}$ from $\mathcal{W}^n_{0}$ to $\mathbb{K}^{2n}_{0}$ given by 
\[\mathcal{L}_{n}(z,\mathbf{z})=\dfrac{w_{a,q}^{n}(\mathbf{z})}{\mathcal{P}^{(n)}_{z}(a;q)}\mathbbm{1}_{z^{2n}=z}.\]
Using the formula given in \eqref{eq:symp_q_whittaker(Berele)} for the function $\mathcal{P}^{(n)}_{z}(\cdot ;q)$, we see that for each $z\in \mathcal{W}^n_{0}$, $\mathcal{L}_{n}(z,\cdot)$ is a probability measure on the set of patterns $\mathbb{K}^{2n}_{0}$. Let $f(\mathbf{z})=z^{2n}$ be the projection of $\mathbf{z}$ to the bottom level. We see that the kernel $\mathcal{L}_{n}(z,\cdot)$ is supported on the set $\{\mathbf{z}\in \mathbb{K}^{2n}_{0}:f(\mathbf{z})=z\}$, for every $z\in \mathcal{W}^n_{0}$.\\
We will prove that $Q^q_{n}$ is intertwined with $\hat{Q}^q_{n}$ via the kernel $\mathcal{L}_{n}$, i.e. it holds that
\begin{equation}
\label{eq:main_intert(Berele)}
Q_{n}^q \mathcal{L}_{n}=\mathcal{L}_{n}\hat{Q}^q_{n}.
\end{equation}

We observe that the intertwining relation implies that $Q^q_{n}$ is conservative, i.e. it holds that for every $z \in \mathcal{W}^n_{0}$
\[\sum_{z' \in \mathcal{W}^n_{0}}Q^q_{n}(z,z')=0.\]
Indeed, we have
\[\sum_{\mathbf{z}\in \mathbb{K}^{2n}_{0}}(Q^q_{n}\mathcal{L}_{n})(z,\mathbf{z})=\sum_{z'\in \mathcal{W}^n_{0}}Q^q_{n}(z,z')\sum_{\mathbf{z}\in \mathbb{K}^{2n}_{0}}\mathcal{L}_{n}(z',\mathbf{z})=\sum_{z'\in \mathcal{W}^n_{0}}Q^q_{n}(z,z')\]
where we used that, by the definition of $\mathcal{P}^{(n)}_{z'}(\cdot;q)$ in \ref{recursion_q_whittaker}, $\mathcal{L}_{n}(z',\cdot)$ sums to one for every $z'\in \mathcal{W}^n_{0}$. \\
On the other hand, we have
\[\sum_{\mathbf{z}\in \mathbb{K}^{2n}_{0}}(\mathcal{L}{_n}\hat{Q}^q_{n})(z,\mathbf{z})=\sum_{\mathbf{w}\in \mathbb{K}^{2n}_{0}}\mathcal{L}_{n}(z,\mathbf{w})\sum_{\mathbf{z}\in \mathbb{K}^{2n}_{0}}\hat{Q}^q_{n}(\mathbf{w},\mathbf{z})=0\]
where we used the fact that $\hat{Q}^q_{n}$, as the transition rate matrix of a continuous-time Markov chain, is conservative. Therefore, it follows that $Q^q_{n}$ is conservative, as required.

We then conclude, using Theorem \ref{rogers_pitman} combined with Lemma \ref{Q_intertwining}, that the projection of the process $\mathbf{Z}$ to the bottom level i.e. for $f(\mathbf{z})=z^{2n}$, is a Markov chain on $\mathcal{W}^n_{0}$ with transition rate matrix $Q^q_{n}$. Finally, using the intertwining relation along with the assumption on the initial condition we conclude that the conditional law of $\mathbf{Z}(t)$ is given by $\mathcal{L}_{n}(Z^{2n}(t),\cdot) = M^n_{a,q}(\cdot ;Z^{2n}(t))$.

Let us now prove the intertwining relation \eqref{eq:main_intert(Berele)} by induction on the dimension $n$.

\noindent \textbf{\underline{Base case:}} Let us first establish the base case. When $n=1$, the transition rate matrix $\hat{Q}^q_{1}$ is given, for $(x,y),(x',y')\in \mathbb{K}^2_{0}$ with $(x',y')\neq (x,y)$, by
\begin{equation*}
\hat{Q}^q_{1}((x,y),(x',y')) = \left\{ \begin{array}{ll}
a_{1} r_{1}(y;x) & \text{ if }(x',y')=(x+1,y+1)\\
a_{1}[1- r_{1}(y;x)] & \text{ if }(x',y')=(x,y-1)\\
a_{1}^{-1} & \text{ if }(x',y')=(x,y+1)\\
0 & \text{ otherwise}
\end{array} \right. 
\end{equation*}
where $a_{1}\in \mathbb{R}_{>0}$ and $r_{1}(y;x)$ is as in \eqref{eq:probabilities(Berele)}. The diagonal elements are given by
\[\hat{Q}^q_{1}((x,y),(x,y))=-a_{1}-a_{1}^{-1} \text{, for }(x,y)\in \mathbb{K}^2_{0}.\]

We define the mapping $m_{1}:\mathbb{K}^2_{0}\mapsto [0,1]$
\[m_{1}(x,y)=a_{1}^{2x-y}\dbinom{y}{y-x}_{q} \, \dfrac{1}{\mathcal{P}^{(1)}_{y}(a_{1};q)}\]
where 
\[\mathcal{P}^{(1)}_{y}(a_{1};q) = \sum_{0\leq x \leq y}a_{1}^{2x-y}\dbinom{y}{y-x}_{q}.\]
The Markov kernel $\mathcal{L}_{1}$ then takes the form
\[\mathcal{L}_{1}(y,(x',y'))=m_{1}(x',y')\mathbbm{1}_{y=y'}\]
and the intertwining relation \eqref{eq:main_intert(Berele)} reads
\begin{equation}
Q^q_{1}(y,y')m_{1}(x',y')=\sum_{\substack{x \in \mathcal{W}^1_{0}:\\(x,y)\in \mathbb{K}^2_{0}}}m_{1}(x,y)\hat{Q}^q_{1}((x,y),(x',y'))
\label{eq:intert_base(Berele)}
\end{equation}
for all $y\in \mathcal{W}^1_{0}$ and $(x',y')\in \mathbb{K}^2_{0}$.

The $y$-particle may only take unit step to either direction, hence we will check the equality \eqref{eq:intert_base(Berele)} for $y'=y+1$, $y'=y-1$ and $y'=y$.

The transition $y'=y+1$ may have occurred in one of the following ways; either $(x',y')=(x,y+1)$ in which case particle $y$ jumped to the right on its own at rate $a_{1}^{-1}$ or $(x',y')=(x+1,y+1)$ and the particle $y$ jumped to the right due to pushing from the particle $x$. The latter transition occurred at rate $a_{1}r_{1}(y'-1;x'-1)$. \\
Using the properties of the $q$-binomial coefficient, we recorded in \eqref{eq:q_binomial}, we calculate the contribution of each case to the sum at the right hand side of \eqref{eq:intert_base(Berele)}. \\
\begin{equation*}
\begin{split}
m_{1}(x',y'-1)\hat{Q}^q_{1}((x',y'-1),&(x',y'))\\
&=a_{1}^{2x'-y'+1}\dbinom{y'-1}{y'-x'-1}_{q}\dfrac{1}{\mathcal{P}^{(1)}_{y'-1}(a_{1};q)}a_{1}^{-1}\\
&=a_{1}^{2x'-y'}\dbinom{y'}{y'-x'}_{q}\dfrac{1-q^{y'-x'}}{1-q^{y'}}\dfrac{1}{\mathcal{P}^{(1)}_{y'}(a_{1};q)}\dfrac{\mathcal{P}^{(1)}_{y'}(a_{1};q)}{\mathcal{P}^{(1)}_{y'-1}(a_{1};q)}\\
&=m_{1}(x',y')\dfrac{1-q^{y'-x'}}{1-q^{y'}}\dfrac{\mathcal{P}^{(1)}_{y'}(a_{1};q)}{\mathcal{P}^{(1)}_{y'-1}(a_{1};q)}.
\end{split}
\end{equation*}
For the second case, we have
\begin{equation*}
\begin{split}
m_{1}(x'-1,y'-1)\hat{Q}^q_{1}(&(x'-1,y'-1),(x',y'))\\
&=a_{1}^{2x'-y'-1}\dbinom{y'-1}{y'-x'}_{q}\dfrac{1}{\mathcal{P}^{(1)}_{y'-1}(a_{1};q)}a_{1}[1-r_{1}(y'-1;x'-1)]\\
&=a_{1}^{2x'-y'}\dbinom{y'}{y'-x'}_{q}\dfrac{(1-q^{x'})q^{y'-x'}}{1-q^{y'}}\dfrac{1}{\mathcal{P}^{(1)}_{y'}(a_{1};q)}\dfrac{\mathcal{P}^{(1)}_{y'}(a_{1};q)}{\mathcal{P}^{(1)}_{y'-1}(a_{1};q)}\\
&=m_{1}(x',y')\dfrac{(1-q^{x'})q^{y'-x'}}{1-q^{y'}}\dfrac{\mathcal{P}^{(1)}_{y'}(a_{1};q)}{\mathcal{P}^{(1)}_{y'-1}(a_{1};q)}.
\end{split}
\end{equation*}
Adding up both cases leads to
\[m_{1}(x',y')\dfrac{\mathcal{P}^{(1)}_{y'}(a_{1};q)}{\mathcal{P}^{(1)}_{y'-1}(a_{1};q)}\]
which agrees with the left hand side of \eqref{eq:intert_base(Berele)}.

Next we consider the case $y'=y-1$. The right hand side of \eqref{eq:intert_base(Berele)} involves a single term corresponding to the transition $(x',y')=(x,y-1)$ occurring at rate $a_{1}[1-r_{1}(y'+1;x')]$ and equals
\[a_{1}^{2x'-y'}\dbinom{y'+1}{y'-x'+1}_{q}\dfrac{1}{\mathcal{P}^{(1)}_{y'+1}(a_{1};q)}(1-q^{y'-x'+1})=m_{1}(x',y')(1-q^{y'+1})\dfrac{\mathcal{P}_{y'}^{(1)}(a_{1};q)}{\mathcal{P}_{y'+1}^{(1)}(a_{1};q)}\]
which gives the left-hand side of \eqref{eq:intert_base(Berele)} for $y'=y-1$.

Finally, if $y'=y$ we have that both the right and the left hand side of \eqref{eq:intert_base(Berele)} equal
\[m_{1}(x',y')(-a_{1}-a_{1}^{-1}).\]

\noindent \textbf{\underline{General case:}} 
We assume that the result holds for $n-1$, i.e. we have that the intertwining relation
\[Q^q_{n-1}\mathcal{L}_{n-1}=\mathcal{L}_{n-1}\hat{Q}^q_{n-1}\]
holds.

Proving the intertwining relation \eqref{eq:main_intert(Berele)} can be very challenging due to the complexity of the transition rate matrix $\hat{Q}^q_{n}$. We will therefore prove a helper intertwining relation that focuses only on a part of the pattern and then deduce the result for the whole pattern.

Let $\mathcal{S}_{n}:=\{(x,y,z)\in \mathcal{W}_{0}^{n-1}\times\mathcal{W}_{0}^{n}\times\mathcal{W}_{0}^{n}:x\preceq y \preceq z\}$ and consider the matrix $\mathcal{A}_{n}$ with off diagonal entries for $(x,y,z)$, $(x',y',z')$ as found in table \ref{tab:transitions(Berele)}.
\begin{table}[ht]
\begin{center}
\scalebox{0.9}{
\begin{tabular}{|l|l|}\hline
$(x',y',z')$ & $\mathcal{A}_{n}((x,y,z),(x',y',z'))$\\ \hline
$(x+e_{i},y+e_{i},z+e_{i})$, $1\leq i \leq n-1$ & $Q^q_{n-1}(x,x+e_{i})r_{i}(y;x)r_{i}(z;y)$\\
$(x+e_{i},y+e_{i},z+e_{i+1})$, $1\leq i \leq n-1$ & $Q^q_{n-1}(x,x+e_{i})r_{i}(y;x)(1-r_{i}(z;y))$\\
$(x+e_{i},y+e_{i+1},z+e_{i+1})$, $1\leq i \leq n-1$ & $Q^q_{n-1}(x,x+e_{i})(1-r_{i}(y;x))r_{i+1}(z;y)$\\
$(x+e_{i},y+e_{i+1},z+e_{i+2})$, $1\leq i \leq n-2$ & $Q^q_{n-1}(x,x+e_{i})(1-r_{i}(y;x))(1-r_{i+1}(z;y))$\\
$(x+e_{n-1},y,z-e_{n})$ & $Q^q_{n-1}(x,x+e_{n-1})(1-r_{n-1}(y;x))(1-r_{n}(z;y))$\\
$(x,y+e_{1},z+e_{1})$ & $a_{n}r_{1}(z;y)$\\
$(x,y+e_{1},z+e_{2})$ & $a_{n}(1-r_{1}(z;y))$\\
$(x,y,z+e_{1})$ & $a_{n}^{-1}$\\
$(x-e_{i},y-e_{i},z-e_{i})$, $1\leq i \leq n-1$ & $Q^q_{n-1}(x,x-e_{i})(1-l_{i}(y;x))(1-l_{i}(z;y))$\\
$(x-e_{i},y-e_{i},z-e_{i+1})$, $1\leq i \leq n-1$ & $Q^q_{n-1}(x,x-e_{i})(1-l_{i}(y;x))l_{i}(z;y)$\\
$(x-e_{i},y-e_{i+1},z-e_{i+1})$, $1\leq i \leq n-2$ & $Q^q_{n-1}(x,x-e_{i})l_{i}(y;x)(1-l_{i+1}(z;y))$\\
$(x-e_{i},y-e_{i+1},z-e_{i+2})$, $1\leq i \leq n-2$ & $Q^q_{n-1}(x,x-e_{i})l_{i}(y;x)l_{i+1}(z;y)$\\
$(x-e_{n-1},y-e_{n},z-e_{n})$ & $Q^q_{n-1}(x,x-e_{n-1})l_{n-1}(y;x)$\\
\hline
\end{tabular}
}
\end{center}
\caption{Off-diagonal entries of $\mathcal{A}_{n}$. Any off-diagonal entry not listed above equals zero. The matrix $Q^q_{n-1}$ is given in \eqref{eq:rate_matrix_q_berele(Berele)} and the probabilities $r_{i}, \, l_{i}$ are as in \eqref{eq:probabilities(Berele)}.}
\label{tab:transitions(Berele)}
\end{table}\\
The diagonal entries of $\mathcal{A}_{n}$ are given by
\[\mathcal{A}_{n}((x,y,z),(x,y,z))=-\sum_{i=1}^n(a_{i}+a_{i}^{-1}).\]

Using the recursive structure of the weights in \eqref{eq:q_weights(Berele)} we see that $\mathcal{P}^{(n)}_{z}(a;q)$ satisfies the relation
\begin{equation*}
\mathcal{P}^{(n)}_{z}(a;q)= \sum_{x \preceq y\preceq z}\Lambda_{n-1,n}^{a_{n},q}(x,y,z) \mathcal{P}_{x}^{(n-1)}(\tilde{a};q)
\end{equation*}
where the summation is over $(x,y,z)\in \mathcal{S}_{n}$ and the vector $\tilde{a} $ contains the first $n-1$ entries of $a$. We will then define a mapping $m_{n}:\mathcal{S}_{n}\mapsto [0,1]$ as follows
\[m_{n}(x,y,z)=\Lambda_{n-1,n}^{a_{n},q}(x,y,z) \dfrac{\mathcal{P}_{x}^{(n-1)}(\tilde{a};q)}{\mathcal{P}^{(n)}_{z}(a;q)}\]
along with a kernel $\mathcal{K}_{n}$ from $\mathcal{W}_{0}^n$ to $\mathcal{S}_{n}$ defined by
\[\mathcal{K}_{n}(z,(x',y',z'))=m_{n}(x',y',z')\mathbbm{1}_{z'=z}.\]

\begin{proposition}
\label{helper_proposition(Berele)}
The intertwining relation
\begin{equation}
Q^q_{n}\mathcal{K}_{n}=\mathcal{K}_{n}\mathcal{A}_{n}
\label{eq:intert_helper(Berele)}
\end{equation}
holds, for every $n>1$.
\end{proposition}

\begin{proof}(of Proposition \ref{helper_proposition(Berele)}) We will prove for $z\in \mathcal{W}_{0}^n$ and $ (x',y',z')\in \mathcal{S}_{n}$ the following identity
\begin{equation}
\label{eq:intert(Berele)}
Q^q_{n}(z,z')=\sum_{\substack{(x,y)\in \mathcal{W}_{0}^{n-1}\times \mathcal{W}_{0}^n:\\(x,y,z)\in \mathcal{S}_{n}}}\dfrac{m_{n}(x,y,z)}{m_{n}(x',y',z')}\mathcal{A}_{n}((x,y,z),(x',y',z')).
\end{equation}

The case $z=z'$ follows directly since the right-hand side of \eqref{eq:intert(Berele)} consists of a single term which corresponds to $(x,y,z)=(x',y',z')$ and equals 
\[\mathcal{A}_{n}((x',y',z'),(x',y',z'))=-\sum_{i=1}^n (a_{i}+a_{i})^{-1}.\]

When $z\neq z'$, expression \eqref{eq:intert(Berele)} simplifies to
\begin{equation}
\begin{split}
\Lambda_{n-1,n}^{a_{n},q}(x',y',z')&\mathcal{P}_{x'}^{(n-1)}(\tilde{a};q)f_{n}(z,z')\\
&=\sum_{\substack{(x,y)\in \mathcal{W}_{0}^{n-1}\times \mathcal{W}^n_{0}:\\(x,y,z)\in \mathcal{S}_{n}}}\Lambda_{n-1,n}^{a_{n},q}(x,y,z)\mathcal{P}_{x}^{(n-1)}(\tilde{a};q)\mathcal{A}_{n}((x,y,z),(x',y',z')).
\label{eq:intert_off_diag(Berele)}
\end{split}
\end{equation}
As the particles can only make unit jumps, both sides of the expression \eqref{eq:intert_off_diag(Berele)} vanish unless $z'=z \pm e_{i}$, for some $1\leq i \leq n$.

Let us consider the case $z'=z+e_{i}$. If $i=1$, the right-hand side of \eqref{eq:intert_off_diag(Berele)} consists of three terms corresponding to the following transitions
\begin{enumerate}[i)]
\item{$(x,y,z)=(x',y',z'-e_{1})$:} in this case $Z_{1}$ jumps to the right of its own volition at rate $a_{n}^{-1}$, then its contribution to the right hand side of \eqref{eq:intert_off_diag(Berele)} equals 
\begin{equation*}
\Lambda_{n-1,n}^{a_{n},q}(x',y',z'-e_{1})\mathcal{P}_{x'}^{(n-1)}(\tilde{a};q)a_{n}^{-1}=\Lambda_{n-1,n}^{a_{n},q}(x',y',z')\mathcal{P}_{x'}^{(n-1)}(\tilde{a};q)\dfrac{1-q^{z_{1}'-y_{1}'}}{1-q^{z_{1}'-z_{2}'}}
\end{equation*}
where we use the equality
\begin{equation*}
\begin{split}
\Lambda_{n-1,n}^{a_{n},q}(x',y',z'-e_{1})&= a_{n}^{2|y'|-|x'|-(|z'|-1)}\prod_{i=1}^{n-1}\dbinom{y_{i}'-y_{i+1}'}{y_{i}'-x_{i}'}_{q}\\
&\hspace{40pt} \times \dbinom{(z'_{1}-1)-z'_{2}}{(z'_{1}-1)-y'_{1}}_{q}\prod_{i=2}^{n-1}\dbinom{z_{i}'-z_{i+1}'}{z_{i}'-y_{i}'}_{q}\dbinom{z_{n}'}{z_{n}'-y_{n}'}_{q}\\ &= \Lambda_{n-1,n}^{a_{n},q}(x',y',z')\,a_{n}\,\dfrac{1-q^{z_{1}'-y_{1}'}}{1-q^{z_{1}'-z_{2}'}}
\end{split}
\end{equation*}
which  follows by the properties of the $q$-binomial recorded in \eqref{eq:q_binomial};
\item{$(x,y,z)=(x',y'-e_{1},z'-e_{1})$:} this transition corresponds to a right jump of $Y_{1}$ at rate $a_{n}$ which then pushes $Z_{1}$ with probability $r_{1}(z'-e_{1};y'-e_{1})$, therefore the corresponding term to the right hand side of \eqref{eq:intert_off_diag(Berele)} equals 
\begin{equation*}
\begin{split}
\Lambda_{n-1,n}^{a_{n},q}(x',y'-e_{1}&,z'-e_{1})\mathcal{P}_{x'}^{(n-1)}(\tilde{a};q)\,a_{n}\,r_{1}(z'-e_{1};y'-e_{1})\\
& = \Lambda_{n-1,n}^{a_{n},q}(x',y',z')\mathcal{P}_{x'}^{(n-1)}(\tilde{a};q)\dfrac{q^{z_{1}'-y_{1}'}(1-q^{y_{1}'-x_{1}'})(1-q^{y_{1}'-z_{2}'})}{(1-q^{z_{1}'-z_{2}'})(1-q^{y_{1}'-y_{2}'})};
\end{split}
\end{equation*}

\item{$(x,y,z)=(x'-e_{1},y'-e_{1},z'-e_{1})$:} in this case $X_{1}$ performs as right jump at rate $Q^q_{n-1}(x'-e_{1},x')=\dfrac{\mathcal{P}^{(n-1)}_{x'}(\tilde{a};q)}{\mathcal{P}^{(n-1)}_{x'-e_{1}}(\tilde{a};q)}f_{n-1}(x'-e_{1},x')$ and pushes $Y_{1}$ with probability $r_{1}(y'_{1};x'_{1})$ which then pushes $Z_{1}$ with probability $r_{1}(z'_{1};y'_{1})$, therefore this transition contributes 
\begin{equation*}
\begin{split}
\Lambda_{n-1,n}^{a_{n},q}(x'-e_{1}&,y'-e_{1},z'-e_{1})\mathcal{P}^{(n-1)}_{x'-e_{1}}Q^q_{n-1}(x'-e_{1},x')r_{1}(y'_{1};x'_{1})r_{1}(z'_{1};y'_{1})\\
&=\Lambda_{n-1,n}^{a_{n},q}(x',y',z')\mathcal{P}_{x'}^{(n-1)}(\tilde{a};q)\dfrac{q^{z_{1}'-x_{1}'}(1-q^{x_{1}'-y_{2}'})(1-q^{y_{1}'-z_{2}'})}{(1-q^{z_{1}'-z_{2}'})(1-q^{y_{1}'-y_{2}'})}.
\end{split}
\end{equation*}

\end{enumerate}
Adding the three terms leads to $\Lambda_{n-1,n}^{a_{n},q}(x',y',z')\mathcal{P}_{x'}^{(n-1)}(\tilde{a};q)$ which equals the left-hand side of \eqref{eq:intert_off_diag(Berele)}.

If $z'=z+e_{i}$, with $2\leq i \leq n$, the only non-zero terms come from the following transitions
\begin{enumerate}[i)]
\item 
	\begin{enumerate}[a.]
	\item$(x,y,z)=(x'-e_{i-2},y'-e_{i-1},z'-e_{i})$ if $i>2$,
	\item $(x,y,z)=(x',y'-e_{1},z'-e_{2})$ if $i=2$;
	\end{enumerate}
\item $(x,y,z)=(x'-e_{i-1},y'-e_{i-1},z'-e_{i})$;
\item $(x,y,z)=(x'-e_{i-1},y'-e_{i},z'-e_{i})$;
\item $(x,y,z)=(x'-e_{i},y'-e_{i},z'-e_{i})$ for $i<n$.
\end{enumerate}
The corresponding terms are
\begin{enumerate}[i)]
\item $\Lambda_{n-1,n}^{a_{n},q}(x',y',z')\mathcal{P}_{x'}^{n-1}(\tilde{a};q)(1-q^{z_{i-1}'-z_{i}'+1})\dfrac{(1-q^{y_{i-1}'-x_{i-1}'})(1-q^{z_{i}'-y_{i}'})}{(1-q^{z_{i}'-z_{i+1}'})(1-q^{y_{i-1}'-y_{i}'})}$;
\item $\Lambda_{n-1,n}^{a_{n},q}(x',y',z')\mathcal{P}_{x'}^{n-1}(\tilde{a};q)(1-q^{z_{i-1}'-z_{i}'+1})\dfrac{q^{y_{i-1}'-x_{i-1}'}(1-q^{x_{i-1}'-y_{i}'})(1-q^{z_{i}'-y_{i}'})}{(1-q^{z_{i}'-z_{i+1}'})(1-q^{y_{i-1}'-y_{i}'})}$;
\item $\Lambda_{n-1,n}^{a_{n},q}(x',y',z')\mathcal{P}_{x'}^{n-1}(\tilde{a};q)(1-q^{z_{i-1}'-z_{i}'+1})\dfrac{q^{z_{i}'-y_{i}'}(1-q^{y_{i}'-x_{i}'})(1-q^{y_{i}'-z_{i+1}'})}{(1-q^{z_{i}'-z_{i+1}'})(1-q^{y_{i}'-y_{i+1}'})}$;
\item $\Lambda_{n-1,n}^{a_{n},q}(x',y',z')\mathcal{P}_{x'}^{n-1}(\tilde{a};q)(1-q^{z_{i-1}'-z_{i}'+1})\dfrac{q^{z_{i}'-x_{i}'}(1-q^{y_{i}'-z_{i+1}'})(1-q^{x_{i}'-y_{i+1}'})}{(1-q^{z_{i}'-z_{i+1}'})(1-q^{y_{i}'-y_{i+1}'})}$.
\end{enumerate}
Gathering all the terms together we conclude that if $z'=z+e_{i}$, the right-hand side of \eqref{eq:intert_off_diag(Berele)} equals $\Lambda_{n-1,n}^{a_{n},q}(x',y',z')\mathcal{P}_{x'}^{(n-1)}(\tilde{a};q)(1-q^{z_{i-1}'-z_{i}'+1})$.

Next we consider the case $z'=z-e_{i}$, for some $1\leq i \leq n$. If $i=1$ the sum consists of a single term corresponding to $(x,y,z)=(x'+e_{1},y'+e_{1},z'+e_{1})$ and equals $\Lambda_{n-1,n}(x',y',z')\mathcal{P}_{x'}^{(n-1)}(\tilde{a};q)(1-q^{z_{1}'-z_{2}'+1})$. 

If $2\leq i \leq n$, we have the contribution of four terms corresponding to the following transitions.
\begin{enumerate}[i)]
\item $(x,y,z)=(x'+e_{i-2},y'+e_{i-1},z'+e_{i})$, for $i>2$;
\item $(x,y,z)=(x'+e_{i-1},y'+e_{i-1},z'+e_{i})$;
\item $(x,y,z)=(x'+e_{i-1},y'+e_{i},z'+e_{i})$;
\item \begin{enumerate}[a)] \item $(x,y,z)=(x'+e_{i},y'+e_{i},z'+e_{i})$, if $i<n$ \item $(x,y,z)=(x'-e_{n-1},y',z'+e_{n})$, if $i=n$.\end{enumerate}
\end{enumerate}
The corresponding terms are given by
\begin{enumerate}[i)]
\item $\Lambda_{n-1,n}^{a_{n},q}(x',y',z')\mathcal{P}_{x'}^{(n-1)}(\tilde{a};q)(1-q^{z_{i}'-z_{i+1}'+1})\dfrac{q^{x_{i-2}'-z_{i}'}(1-q^{y_{i-2}'-x_{i-2}'})(1-q^{z_{i-1}'-y_{i-1}'})}{(1-q^{z_{i-1}'-z_{i}'})(1-q^{y_{i-2}'-y_{i-1}'})}$;
\item $\Lambda_{n-1,n}^{a_{n},q}(x',y',z')\mathcal{P}_{x'}^{(n-1)}(\tilde{a};q)(1-q^{z_{i}'-z_{i+1}'+1})\dfrac{q^{y_{i-1}'-z_{i}'}(1-q^{x_{i-2}'-y_{i-1}'})(1-q^{z_{i-1}'-y_{i-1}'})}{(1-q^{z_{i-1}'-z_{i}'})(1-q^{y_{i-2}'-y_{i-1}'})}$;
\item $\Lambda_{n-1,n}^{a_{n},q}(x',y',z')\mathcal{P}_{x'}^{(n-1)}(\tilde{a};q)(1-q^{z_{i}'-z_{i+1}'+1})\dfrac{q^{x_{i-1}'-y_{i}'}(1-q^{y_{i-1}'-x_{i-1}'})(1-q^{y_{i-1}'-z_{i}'})}{(1-q^{z_{i-1}'-z_{i}'})(1-q^{y_{i-1}'-y_{i}'})}$;
\item $\Lambda_{n-1,n}^{a_{n},q}(x',y',z')\mathcal{P}_{x'}^{(n-1)}(\tilde{a};q)(1-q^{z_{i}'-z_{i+1}'+1})\dfrac{(1-q^{y_{i-1}'-z_{i}'})(1-q^{x_{i-1}'-y_{i}'})}{(1-q^{z_{i-1}'-z_{i}'})(1-q^{y_{i-1}'-y_{i}'})}$.
\end{enumerate}
We conclude that if $z'=z-e_{i}$, for $2\leq i \leq n$, the right-hand side of \eqref{eq:intert_off_diag(Berele)} equals $\Lambda_{n-1,n}^{a_{n},q}(x',y',z')\mathcal{P}_{x'}^{(n-1)}(\tilde{a};q)(1-q^{z_{i}'-z_{i+1}'+1})$.

\end{proof}
Let us now proceed to the proof of the full intertwining relation \eqref{eq:intert(Berele)}. For $\mathbf{z}\in \mathbb{K}^{2n}_{0}$, let $\mathbf{z}^{1:2n-2}$ denote the top $2n-2$ levels of $\mathbf{z}$. Then the Markov kernel $\mathcal{L}_{n}$ can be decomposed as follows. Let $z\in \mathcal{W}^n_{0}$ and $\mathbf{z}\in \mathbb{K}^{2n}_{0}$, then
\[\mathcal{L}_{n}(z,\mathbf{z})=\mathcal{L}_{n-1}(z^{2n-2},\mathbf{z}^{1:2n-2})\mathcal{K}_{n}(z,(z^{2n-2},z^{2n-1},z^{2n})).\]

Due to the fact that any transition at the pattern propagates to the bottom of the pattern and any transition initiated at the lower levels does not affect the upper levels of the pattern we may re-write the matrix $\hat{Q}^q_{n}$ as follows. For $\mathbf{z},\mathbf{w}\in \mathbb{K}^{2n}_{0}$, we have
\[\hat{Q}^q_{n}(\mathbf{z},\mathbf{w})=\hat{Q}^q_{n}(\mathbf{z}^{1:2n-2},\mathbf{w}^{1:2n-2})\dfrac{\mathcal{A}_{n}((z^{2n-2},z^{2n-1},z^{2n}),(w^{2n-2},w^{2n-1},w^{2n}))}{Q^{q}_{n-1}(z^{2n-2},w^{2n-2})}.\]

\noindent Therefore, for $z\in \mathcal{W}^n_{0}$ and $\mathbf{w}\in \mathbb{K}^{2n}_{0}$ we have
\begin{equation*}
\begin{split}
(\mathcal{L}_{n}\hat{Q}^q_{n})(z,\mathbf{w})&=\sum_{\mathbf{z}\in \mathbb{K}^{2n}_{0}}\mathcal{L}_{n}(z,\mathbf{z})\hat{Q}^q_{n}(\mathbf{z},\mathbf{w})\\
&=\sum_{(z^{2n-2},z^{2n-1},z^{2n})\in \mathcal{S}_{n}}\mathcal{K}_{n}(z,(z^{2n-2},z^{2n-1},z^{2n}))\\
& \\
& \hspace{50pt} \times \dfrac{\mathcal{A}_{n}((z^{2n-2},z^{2n-1},z^{2n}),(w^{2n-2},w^{2n-1},w^{2n}))}{Q^{q}_{n-1}(z^{2n-2},w^{2n-2})}\\
& \\
& \hspace{50pt}  \times  \sum_{\mathbf{z}^{1:2n-2}\in \mathbb{K}^{2n-2}_{0}}\mathcal{L}_{n-1}(z^{2n-2},\mathbf{z}^{1:2n-2})\hat{Q}^q_{n-1}(\mathbf{z}^{1:2n-2},\mathbf{w}^{1:2n-2}).
\end{split}
\end{equation*}
The inner sum equals $(\mathcal{L}_{n-1}\hat{Q}^q_{n-1})(z^{2n-2},\mathbf{w}^{1:2n-2})$, which by induction and due to the special structure of the kernel $\mathcal{L}_{n-1}$, equals
\[Q^q_{n-1}(z^{2n-2},w^{2n-2})\mathcal{L}_{n-1}(w^{2n-2},\mathbf{w}^{1:2n-2}).\]
For the outer sum, we will use the intertwining relation \eqref{eq:intert_helper(Berele)} to conclude 
\[(Q^q_{n}\mathcal{K}_{n})(z,(w^{2n-2},w^{2n-1},w^{2n}))=Q^q_{n}(z,w^{2n})\mathcal{K}_{n}(w^{2n},(w^{2n-2},w^{2n-1},w^{2n})).\]
Combining the above calculations we conclude that
\begin{equation*}
\begin{split}
(\mathcal{L}_{n}\hat{Q}^q_{n})(z,\mathbf{w})&=Q^q_{n}(z,w^{2n})\mathcal{L}_{n}(w^{2n-2},\mathbf{w}^{1:2n-2})\mathcal{K}_{n}(w^{2n},(w^{2n-2},w^{2n-1},w^{2n}))\\
&=Q^q_{n}(z,w^{2n})\mathcal{L}_{n}(w^{2n},\mathbf{w})\\
&=(Q^q_{n}\mathcal{L}_{n})(z,\mathbf{w})
\end{split}
\end{equation*}
as required.

%% file: tex/discrete.tex
In chapter \ref{Berele} we described a Markov process on the integer-valued symplectic Gelfand-Tsetlin cone where only the particles at the edge of the pattern may jump of their own volition, meaning that only the particles at the edge have their own independent exponential clocks initiating their movement whereas all the other particles can move only as a result of a pushing or pulling from another particle. We also proposed a $q$-deformation of the process and we proved that under certain initial conditions the even-indexed levels evolve as Markov processes. In this chapter we will construct a second process which in contrast with the process in \ref{Berele} is fully randomized, in the sense that all the particles have their own independent exponential clocks that drive their jumps. The process we will describe is a $q$-deformation of the one proposed by Warren-Windridge in \cite{Warren_Windridge_2009} and we briefly described in the Introduction.

\section{A $q$-deformed process on $\mathbb{K}^N_{0}$}
\label{q_deformed_dynamics}
In the Introduction we presented a continuous-time $\mathbb{K}^N_{0}$-valued Markov process introduced in \cite{Warren_Windridge_2009}. Here we will modify this process in the following way. Assume a particle attempts a jump. The probability of the jump to succeed will now depend on the distance of the particle from the boundary of the interval $\mathcal{I}^k_{i}$, where
\begin{equation*}
\mathcal{I}^k_{i}:=\left\{ \begin{array}{ll}
\lbrack z^{k-1}_{1}, \infty \rbrack & \text{ if }i=1\\
\lbrack0,z^{k-1}_{i-1}\rbrack & \text{ if } i= \frac{k+1}{2} \text{ with }k\geq 1 \text{ odd}\\
\lbrack z^{k-1}_{i}, z^{k-1}_{i-1} \rbrack & \text{ otherwise}
\end{array} \right. .
\end{equation*}
Let us now formally describe the process. We fix $n \in \mathbb{N}$ and assume that $N=2n$ or $2n-1$. We also fix a parameter $q\in (0,1)$ and a vector $a=(a_{1},...,a_{n})$ of real positive numbers. It is also useful to extend the vector $a$ to $\bar{a}=(\bar{a_{1}},...,\bar{a}_{N})$ such that $\bar{a}_{2l}=a_{l}^{-1}$ and $\bar{a}_{2l-1}=a_{l}$. We define a continuous-time process $\mathbf{Z}=(\mathbf{Z}(t);t\geq 0)$ with state space $\mathbb{K}^N_{0}$. The stochastic evolution of the process is as follows.
\begin{itemize}
\item \textbf{Right jumps.} Each particle $Z_{j}^k$ independently jumps to the right at rate $\bar{a}_{k}R_{j}(Z^k;Z^{k-1})$. The quantity $R_{j}$ is given by
\begin{equation}
R_{j}(Z^k;Z^{k-1})=(1-q^{Z^{k-1}_{j-1}-Z^{k}_{j}})\dfrac{1-q^{Z^{k}_{j}-Z^k_{j+1}+1}}{1-q^{Z^{k}_{j}-Z^{k-1}_{j}+1}}.
\label{eq:RightProbabilities}
\end{equation}
If $Z_{j}^k = Z^{k+1}_{j}$ and $Z^k_{j}$ jumps to the right, then $Z^{k+1}_{j}$ is simultaneously pushed one step to the right.

\item \textbf{Left jumps.} Each particle $Z^{k}_{j}$ independently jumps to the left at rate $\bar{a}_{k}^{-1}L_{j}(Z^k;Z^{k-1})$. The quantity $L_{j}$ is defined as
\begin{equation}
L_{j}(Z^k;Z^{k-1})=(1-q^{Z^{k}_{j}-Z^{k-1}_{j}})\dfrac{1-q^{Z^{k}_{j-1}-Z^k_{j}+1}}{1-q^{Z^{k-1}_{j-1}-Z^{k}_{j}+1}}.
\label{eq:LeftProbabilities}
\end{equation}
If $Z_{j}^k = Z^{k+1}_{j+1}$ and $Z^k_{j}$ jumps to the left, then $Z^{k+1}_{j+1}$ is simultaneously pushed one step to the left.
\end{itemize}
Regarding the particles at the edges of the pattern we consider the conventions $Z^k_{0}=\infty$ and $Z^k_{[\frac{k+1}{2}]+1}=0$ and the probabilities $R$ and $L$ are modified appropriately. Schematically, the probabilities $R$ and $L$ are as in figure \ref{jumping_rates}.

\begin{figure}[h]
\begin{center}
\begin{tikzpicture}[scale = 0.8]
% right jumps
\draw [fill] (0.5,2.2) circle [radius=0.09];
\node [above, black] at (0.5,2.3) {\large{$Z^{k-1}_j$}};
\draw [fill] (3.5,2.2) circle [radius=0.09];
\node [above, black] at (3.5,2.3) {\large{$Z^{k-1}_{j-1}$}};
\draw [fill] (-1,0.8) circle [radius=0.09];
\node [above, black] at (-1,0.9) {\large{$Z_{j+1}^k$}};
\draw [fill, red] (2,0.8) circle [radius=0.09];
\node [above, red] at (2,0.9) {\large{$Z_{j}^k$}};
\draw [fill] (5,0.8) circle [radius=0.09];
\node [above, black] at (5,0.9) {\large{$Z_{j-1}^k$}};

\draw [ultra thick, red] [->](2,0.6) to [bend right = 60] (3,0.8);
\draw [thick, ->](2.1,1) to node [above]{$a^r$} (3.4,2);
\draw [thick, <-](1.8,0.8) to node [above]{$b^r$} (-0.8,0.8);
\draw [thick, <-](1.9,1) to node [above]{$c^r$} (0.6,2);

\node at (2,-1) {a)};

% left jumps

\draw [fill] (9,2.2) circle [radius=0.09];
\node [above, black] at (9,2.3) {\large{$Z^{k-1}_j$}};
\draw [fill] (12,2.2) circle [radius=0.09];
\node [above, black] at (12,2.3) {\large{$Z^{k-1}_{j-1}$}};
\draw [fill] (7.5,0.8) circle [radius=0.09];
\node [above, black] at (7.5,0.9) {\large{$Z_{j+1}^k$}};
\draw [fill, red] (10.5,0.8) circle [radius=0.09];
\node [above, red] at (10.5,0.9) {\large{$Z_{j}^k$}};
\draw [fill] (13.5,0.8) circle [radius=0.09];
\node [above, black] at (13.5,0.9) {\large{$Z_{j-1}^k$}};

\draw [ultra thick, red] [->](10.5,0.6) to [bend left = 60] (9.5,0.8);
\draw [thick, ->](10.6,1) to node [above]{$c^l$} (11.9,2);
\draw [thick, ->](10.7,0.8) to node [above]{$b^l$} (13.3,0.8);
\draw [thick, <-](10.4,1) to node [above]{$a^l$} (9.1,2);
\node at (10.5,-1) {b)};
\end{tikzpicture}
\end{center} 
\caption{The particle $Z^k_{j}$ performs a right jump at rate $\bar{a}_{k}R_{j}(Z^k;Z^{k-1})=\bar{a}_{k}(1-q^{|a^r|})\frac{1-q^{|b^r|+1}}{1-q^{|c^r|+1}}$ where $a^r, b^r$ and $c^r$ are the arrows in part a).The particle $Z^k_{j}$ performs a left jump at rate $\bar{a}_{k}^{-1}L_{j}(Z^k;Z^{k-1})=\bar{a}_{k}^{-1}(1-q^{|a^l|})\frac{1-q^{|b^l|+1}}{1-q^{|c^l|+1}}$ where $a^l, b^l$ and $c^l$ are the arrows in part b). For an arrow $v:\alpha \to \beta$, we write $|v|:=\beta - \alpha$.} 
\label{jumping_rates}
\end{figure}
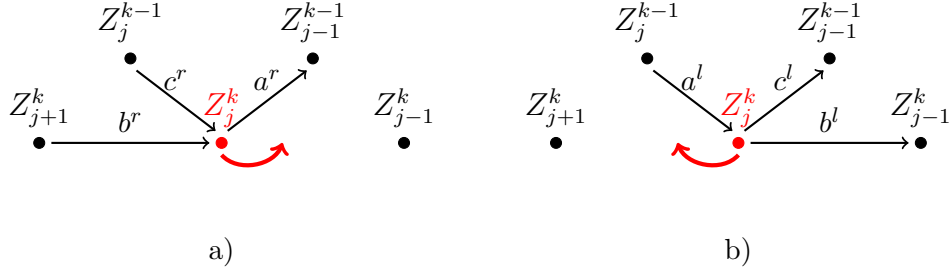

For each $\mathbf{z} \in \mathbb{K}^{N}_{0}$, we consider a real-valued weight $w^{(N)}_{a,q}(\mathbf{z})$ given by
\begin{equation}
w^{(N)}_{a,q}(\mathbf{z}) = \prod_{k=1}^N \Lambda_{k-1,k}^{\bar{a}_{k},q} (z^{k-1}, z^{k})
\label{eq:q_weights}
\end{equation}
where for $k=2l$ or $2l-1$, the weight of the "slice" $(z^{k-1},z^k)$ is calculated as
\begin{equation*}
\Lambda_{k-1,k}^{\bar{a}_{k},q} (z^{k-1}, z^{k})= (\bar{a}_{k})^{|z^{k}|-|z^{k-1}|}\mathlarger{\mathlarger{‎‎\prod}}_{i=1}^{l-1} \dbinom{z^{k}_{i}-z^{k}_{i+1}}{z^{k}_{i}-z^{k-1}_{i}}_{q}
\Big ( \mathbbm{1}_{\{k=2l-1\}} + \dbinom{z^{k}_{l}}{z^k_{l}-z^{k-1}_{l}}_{q} \mathbbm{1}_{\{k=2l\}}\Big ).
\end{equation*}
For example, if $n=1$, $a_{1} \in \mathbb{R}_{>0}$ and $q\in (0,1)$, the weight of $z^1 \in \mathbb{K}^{1}_0\equiv \mathbb{Z}_{\geq 0}$ is
\[w^{(1)}_{a_{1},q}(z^1)=a_{1}^{|z^1|}\]
and the weight of $(z^1,z^2)\in \mathbb{K}^2_{0}$ is
\[w^{(2)}_{a_{1},q}(z^1,z^2) = a_{1}^{2|z^1|-|z^2|}\dbinom{z^{2}_{1}}{z^{2}_{1}-z^1_{1}}_{q}.\]
Schematically, the weight $\Lambda^{\bar{a}_{k},q}_{k-1,k}(z^{k-1},z^{k})$ is shown in figure \ref{weight_slice} below.

\begin{figure}[h]
\begin{center}
\begin{tikzpicture}[scale = 0.7]

% odd to even
\draw [thick] (-3,0) -- (-3,3) ;
\draw [fill] (5,2) circle [radius=0.09];
\node [above right, black] at (5,2) {\large{$z_{1}^{k-1}$}};
\draw [fill] (-2,2) circle [radius=0.09];
\node [above right, black] at (-2,2) {\large{$z_{l}^{k-1}$}};
\draw [fill] (0,2) circle [radius=0.09];
\node [above right, black] at (0,2) {\large{$z_{l-1}^{k-1}$}};
\draw [fill] (6,0.5) circle [radius=0.09];
\node [above right, black] at (6,.5) {\large{$z_{1}^k$}};
\draw [fill] (4,0.5) circle [radius=0.09];
\node [above right, black] at (4,.5) {\large{$z_{2}^k$}};
\draw [fill] (1,0.5) circle [radius=0.09];
\node [above right, black] at (1,.5) {\large{$z_{l-1}^k$}};
\draw [fill] (-1,0.5) circle [radius=0.09];
\node [above right, black] at (-1,.5) {\large{$z_{l}^k$}};
\draw [dotted,thick] (2,0.5) -- (3,0.5);
\draw [dotted,thick] (2,2) -- (3,2);

\draw [thick][<->] (-2.8,0.5) -- (-1.2,0.5);
\draw [thick][<->] (-0.8,0.5) -- (0.8,0.5);
\draw [thick][<->] (4.2,0.5) -- (5.8,0.5);

\draw [gray, thick][<->] (-2,0.7) -- (-1.2,0.7);
\draw [gray, thick][<->] (0,0.7) -- (0.8,0.7);
\draw [gray,thick][<->] (5,0.7) -- (5.8,0.7);

\draw [dashed, thick](-2,1.8) -- (-2,0.7);
\draw [dashed, thick](0,1.8) -- (0,0.7);
\draw [dashed, thick](5,1.8) -- (5,0.7);

\node [below] at (2.5,-0.5) {a)};

% even to odd
\draw [thick] (8,0) -- (8,3) ;
\draw [fill] (15,2) circle [radius=0.09];
\node [above right, black] at (15,2) {\large{$z_{1}^{k-1}$}};
\draw [fill] (10,2) circle [radius=0.09];
\node [above right, black] at (10,2) {\large{$z_{l-1}^{k-1}$}};
\draw [fill] (16,0.5) circle [radius=0.09];
\node [above right, black] at (16,.5) {\large{$z_{1}^k$}};
\draw [fill] (14,0.5) circle [radius=0.09];
\node [above right, black] at (14,.5) {\large{$z_{2}^k$}};
\draw [fill] (11,0.5) circle [radius=0.09];
\node [above right, black] at (11,.5) {\large{$z_{l-1}^k$}};
\draw [fill] (9,0.5) circle [radius=0.09];
\node [above right, black] at (9,.5) {\large{$z_{l}^k$}};
\draw [dotted,thick] (12,0.5) -- (13,0.5);
\draw [dotted,thick] (12,2) -- (13,2);

\draw [thick][<->] (9.2,0.5) -- (10.8,0.5);
\draw [thick][<->] (14.2,0.5) -- (15.8,0.5);

\draw [gray, thick][<->] (10,0.7) -- (10.8,0.7);
\draw [gray,thick][<->] (15,0.7) -- (15.8,0.7);

\draw [dashed, thick](10,1.8) -- (10,0.7);
\draw [dashed, thick](15,1.8) -- (15,0.7);

\node [below] at (12.25,-0.5) {b)};
\end{tikzpicture}
\end{center} 
\caption{Weight for the \enquote{slice} $(z^{k-1},z^k)$ with a) $k=2l$ and b) $k=2l-1$. Each pair of arrows correspond to a $q$-binomial coefficient $\binom{\longleftrightarrow}{\, \,\textcolor{gray}{\leftrightarrow}}_{q}$ of $\Lambda^{\bar{a}_{k},q}_{k-1,k}(z^{k-1},z^{k})$.} 
\label{weight_slice}
\end{figure}
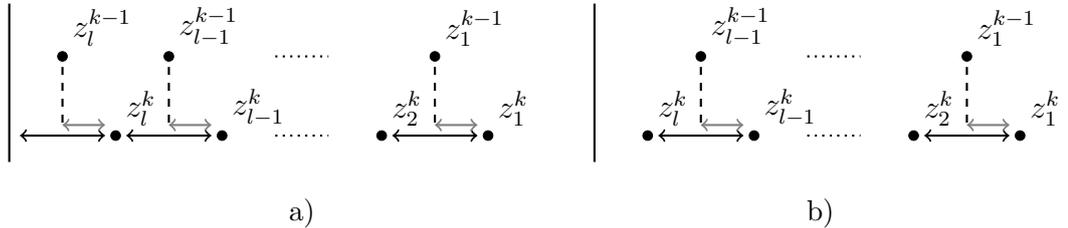
For each $z \in \mathcal{W}_{0}^n$, we define the function 
\begin{equation}
\hat{\mathcal{P}}_{z}^{(N)}(a;q) = \sum_{\mathbf{z}\in \mathbb{K}^{N}_0[z]}w^{(N)}_{a,q}(\mathbf{z}).
\label{eq:symp_q_whittaker}
\end{equation}
When $N=2n$ the function $\hat{\mathcal{P}}^{(N)}_{x}(a;q)$ gives the function $\mathcal{P}^{(n)}_{x}(a;q)$ we defined via the recursion in \ref{recursion_q_whittaker}.

For each $z \in \mathcal{W}^n_{0}$, the function $M^{(N)}_{a,q}(\cdot ; z): \mathbb{K}^N_{0}[z]\to [0,1]$ given by
\[M^{(N)}_{a,q}(\cdot ; z) = \dfrac{w^{(N)}_{a,q}(\cdot)}{\hat{\mathcal{P}}^{(N)}_{z}(a;q)}\]
defines a probability measure on $\mathbb{K}^N_{0}[z]$.

We are now ready to state the main result of the chapter.
\begin{theorem}
\label{q_main_theorem}
Let $n \in \mathbb{N}$ and $N=2n$ or $2n-1$. Suppose that the process $\mathbf{Z}=(\mathbf{Z}(t);t\geq 0)$ has initial distribution $M_{a,q}^{(N)}(\cdot;z)$, for some $z\in \mathcal{W}_{0}^{n}$. Then $Z^N=(Z^N(t),t \geq 0)$ is a Markov process in $\mathcal{W}_{0}^n$, started at $z$, with transition rates $\{Q_{N}^q(z,z'), z,z' \in \mathcal{W}_{0}^{n}\}$ given, for $z \neq z'$, by
\begin{equation}
Q_{N}^q(z,z')=\dfrac{\hat{\mathcal{P}}^{(N)}_{z'}(a;q)}{\hat{\mathcal{P}}^{(N)}_{z}(a;q)}f_{N}(z,z')
\label{eq:rate_matrix_q}
\end{equation}
where the function $f_{N}:\mathcal{W}_{0}^n \times \mathcal{W}_{0}^n \to \mathbb{R}$ is given by
\begin{equation*}
f_{N}(z,z') = \left\{ \begin{array}{lll}
1-q^{z_{i-1}-z_{i}} & \text{ if } z'-z = e_{i} &, 1\leq i \leq n\\
1-q^{z_{i}-z_{i+1}} & \text{ if } z'-z =- e_{i} &, 1\leq  i \leq n\\
0 & \text{ otherwise}
\end{array} \right.
\end{equation*}
with the convention $z_{0}=+\infty$ and $z_{n+1}=0$. The diagonal entries are
\begin{equation*}
Q_{N}^q(z,z) = \left\{ \begin{array}{ll}
-\sum_{i=1}^{n}(a_{i}+a_{i}^{-1}) & , N=2n\\
-\sum_{i=1}^{n-1}(a_{i}+a_{i}^{-1})-a_{n}-a_{n}^{-1}(1-q^{z_{n}}) & , N=2n-1.
\end{array} \right.
\end{equation*}
Moreover, for $t>0$ the conditional law of $\mathbb{Z}(t)$ given $\{Z^N(s),s\leq t\}$ is given by $M^{(N)}_{a,q}(\cdot;Z^N(t))$.
\end{theorem}

% proofs
\newpage
\section{Proof of Theorem \ref{q_main_theorem}}
Let us denote by $\hat{Q}^q_{N} = \{\hat{Q}^q_{N}(\mathbf{z},\mathbf{w}),\mathbf{z},\mathbf{w}\in \mathbb{K}^{N}_{0}\}$ the transition kernel of the process $\mathbf{Z}$ introduced in section 5.1. We also introduce the kernel $\mathcal{L}_{N}$ from $\mathcal{W}^{[\frac{N+1}{2}]}_{0}$ to $\mathbb{K}^N_{0}$
\[\mathcal{L}_{N}(z,\mathbf{z})=\dfrac{w^{N}_{a,q}(\mathbf{z})}{\hat{\mathcal{P}}^{(N)}_{z}(a;q)}\mathbbm{1}_{z^N=z}.\]
By the definition of the function $\hat{\mathcal{P}}^{(N)}_{z}(\cdot;q)$ in \eqref{eq:symp_q_whittaker} we see that, for each $z\in\mathcal{W}^{[\frac{N+1}{2}]}_{0}$, $\mathcal{L}_{N}(z,\cdot)$ is a probability measure on the set of patterns $\mathbb{K}^N_{0}$ which is supported on the set $\{\mathbf{z}\in \mathbb{K}^N_{0}:f(\mathbf{z})=z\}$, where $f(\mathbf{z})=z^N$ is the projection of $\mathbf{z}$ to the bottom level. 

The key result for the proof of Theorem \ref{q_main_theorem} is the following intertwining relation
\begin{equation}
Q^q_{N}\mathcal{L}_{N}=\mathcal{L}_{N}\hat{Q}^q_{N}.
\label{eq:main_intert(discrete)}
\end{equation}

Suppose relation \eqref{eq:main_intert(discrete)} holds. The matrix $\hat{Q}^q_{N}$, as the transition kernel of a continuous-time Markov chain is conservative, i.e. $\sum_{\mathbf{z}\in \mathbb{K}^N_{0}}\hat{Q}^q_{N}(\mathbf{w},\mathbf{z})=0$, for every $\mathbf{w}\in \mathbb{K}^N_{0}$. Therefore, for each $z\in \mathcal{W}^{[\frac{N+1}{2}]}_{0}$ we have
\[\sum_{\mathbf{z}\in \mathbb{K}^N_{0}}(\mathcal{L}_{N}\hat{Q}^q_{N})(z,\mathbf{z})=\sum_{\mathbf{w}\in \mathbb{K}^N_{0}}\mathcal{L}_{N}(z,\mathbf{w})\sum_{\mathbf{z}\in \mathbb{K}^N_{0}}\hat{Q}^q_{N}(\mathbf{w},\mathbf{z})=0.\]
On the other hand
\[\sum_{\mathbf{z}\in \mathbb{K}^N_{0}}(Q^q_{N}\mathcal{L}_{N})(z,\mathbf{z})=\sum_{w\in \mathcal{W}^{[\frac{N+1}{2}]}_{0}}Q^q_{N}(z,w)\sum_{\mathbf{z}\in \mathbb{K}^N_{0}}\mathcal{L}_{N}(w,\mathbf{z})=\sum_{w\in \mathcal{W}^{[\frac{N+1}{2}]}_{0}}Q^q_{N}(z,w)\]
where we used the fact that for each $w\in \mathcal{W}^{[\frac{N+1}{2}]}_{0}$, $\mathcal{L}_{N}(w,\cdot)$ is a probability measure on $\mathbb{K}^N_{0}$. Therefore, the intertwining relation \eqref{eq:main_intert(discrete)} implies that the matrix $Q^q_{N}$ is also conservative. Along with the fact that all the off diagonal entries are non-negative, we conclude that $\{Q^q_{N}(z,z'),z,z'\in \mathcal{W}^{[\frac{N+1}{2}]}_{0}\}$ is the transition kernel of a continuous-time Markov chain on $\mathcal{W}^{[\frac{N+1}{2}]}_{0}$. 

We then conclude, using Theorem \ref{rogers_pitman} along with Lemma \ref{Q_intertwining}, that the projection of $\mathbf{Z}$ to the bottom level is a Markov chain on $\mathcal{W}^{[\frac{N+1}{2}]}_{0}$ with transition rate matrix $Q^q_{N}$ and the conditional law of $\mathbf{Z}(t)$ given $\{Z^N(s),s\leq t\}$ is given by $\mathcal{L}_{N}(Z^N(t),\cdot) = M^{(N)}_{a,q}(\cdot ; Z^N(t))$.

Let us now prove the intertwining relation \eqref{eq:main_intert(discrete)} by induction on $N$. When $N=1$ the pattern consists of a single particle. In this case we have that the transition kernels $Q^q_{1}$, $\hat{Q}^q_{1}$ coincide. Moreover, the Markov kernel $\mathcal{L}_{1}$ is given by $\mathcal{L}_{1}(z,z^1_{1})=\mathbbm{1}_{z^1_{1}=z}$, therefore relation \eqref{eq:main_intert(discrete)} holds trivially.  
 
Due to the different behaviour of the process at odd and even levels we will check the cases where $N$ is even and $N$ is odd separately. 

\paragraph{Part I: Iterating from odd to even row.}
Suppose that $N=2n$ for some $n\geq 1$ and
\[Q^q_{N-1}\mathcal{L}_{N-1}=\mathcal{L}_{N-1}\hat{Q}^q_{N-1}.\]
The transition kernel $\hat{Q}^q_{N}$ is very complicated to handle directly. Therefore we will first prove an intertwining relation that focuses on the two bottom levels of the pattern and then proceed to the intertwining relation for the whole pattern.

We introduce a matrix $\mathcal{A}_{N}$ on $\mathcal{W}_{0}^{n,n}=\{(x,y)\in \mathcal{W}_{0}^n \times \mathcal{W}_{0}^n: x \preceq y\}$ with off-diagonal entries given for $(x,y), (x',y') \in \mathcal{W}_{0}^{n,n}$ as in the following table\\ 
\begin{table}[ht]
\begin{center}
\begin{tabular}{|l|l|}\hline
$(x',y')$ & $\mathcal{A}_{N}((x,y),(x',y')))$\\ \hline
$(x+e_{i},y+e_{i}), x_{i}=y_{i}, \text{ for some }i\in [n]$ & $Q_{N-1}^q(x,x+e_{i})$\\
$(x+e_{i},y), x_{i}<y_{i}, \text{ for some } i\in [n]$ & $Q_{N-1}^q(x,x+e_{i})$\\
$(x-e_{i},y-e_{i+1}), x_{i}=y_{i+1},\text{ for some } i\in [n-1]$ & $Q_{N-1}^q(x,x-e_{i})$\\
$(x-e_{i},y), x_{i}>y_{i+1},\text{ for some } i\in [n-1]$ & $Q_{N-1}^q(x,x-e_{i})$\\
$(x-e_{n},y)$ & $Q_{N-1}^q(x,x-e_{n})$\\
$(x,y-e_{i}),\text{ for some }i\in [n]$ & $a_{n}L_{i}(y;x)$\\
$(x,y+e_{i}),\text{ for some }i\in [n]$ & $a_{n}^{-1}R_{i}(y;x)$\\
\hline
\end{tabular}
\end{center}
\end{table}

\noindent where $R_{i}, L_{i}$ are the probabilities defined in \eqref{eq:RightProbabilities}, \eqref{eq:LeftProbabilities}. All the other off-diagonal entries equal to zero. The diagonal entries are for $(x,y)\in \mathcal{W}_{0}^{n,n}$ equal to
\begin{small}
\begin{equation}
\mathcal{A}_{N}((x,y),(x,y)))=-\sum_{i=1}^{n-1}(a_{i}+a_{i}^{-1})-a_{n}-a_{n}^{-1}(1-q^{x_{n}})-\sum_{i=1}^n[a_{n}L_{i}(y;x)+a_{n}^{-1}R_{i}(y;x)].
\label{eq:Odd2EvenADiag}
\end{equation}
\end{small}

Let us define the mapping $m_{N}:\mathcal{W}_{0}^{n,n}\mapsto [0,1]$
\[m_{N}(x,y)=\Lambda_{N-1,N}^{a_{n}^{-1},q}(x,y) \dfrac{\hat{\mathcal{P}}_{x}^{(N-1)}(a;q)}{\hat{\mathcal{P}}^{(N)}_{y}(a;q)}.\]
Using the recursive structure of the weights in \eqref{eq:q_weights} we obtain the following relation for $\hat{\mathcal{P}}^{(N)}_{y}(a;q)$, with $N=2n$,
\begin{equation*}
\begin{split}
\hat{\mathcal{P}}^{(N)}_{y}(a;q)&= \sum_{\mathbf{z} \in \mathbb{K}^{N}_{0}[y]}w^{(N)}_{a,q}(\mathbf{z})\\
&= \sum_{x \in \mathcal{W}_{0}^n: x \preceq y} \Lambda_{N-1,N}^{a_{n}^{-1},q}(x,y)\sum_{\mathbf{w} \in \mathbb{K}^{N-1}_0[x]}w^{(N-1)}_{a,q}(\mathbf{w})\\
&= \sum_{x\in \mathcal{W}_{0}^n: x \preceq y}\Lambda_{N-1,N}^{a_{n}^{-1},q}(x,y) \hat{\mathcal{P}}_{x}^{(N-1)}(a;q)
\end{split}
\end{equation*}
therefore, we conclude that the mapping $m_{N}$ defines a Markov kernel $\mathcal{K}_{N}$ from $\mathcal{W}_{0}^n$ to $\mathcal{W}_{0}^{n,n}$ defined by
\[\mathcal{K}_{N}(y,(x',y'))=m_{N}(x',y')\mathbbm{1}_{y'=y}.\]

\begin{proposition}
\label{helper_intert_even(discrete)}
The transition kernel $Q_{N}^q$ is intertwined with $\mathcal{A}_{N}$ via the kernel $\mathcal{K}_{N}$, i.e. the relation
\[Q_{N}^q\mathcal{K}_{N} = \mathcal{K}_{N}\mathcal{A}_{N}\]
holds. 
\end{proposition}
\begin{proof}
The intertwining relation of Proposition \ref{helper_intert_even(discrete)} is equivalent to 
\begin{equation}
\label{eq:intert_even}
Q_{N}^q(y,y')=\sum_{x\preceq y}\dfrac{m_{N}(x,y)}{m_{N}(x',y')}\mathcal{A}_{N}((x,y),(x',y')), \quad y \in \mathcal{W}_{0}^n, (x',y')\in \mathcal{W}_{0}^{n,n}
\end{equation} 
where the summation is over $x \in \mathcal{W}_{0}^n$ such that $(x,y)\in \mathcal{W}_{0}^{n,n}$.

Both sides of \eqref{eq:intert_even} equal zero, except from the cases $y=y'$ and $y=y'+e_{i}$, $y=y'-e_{i}$, for some $1\leq i \leq n$. So our goal is to prove that \eqref{eq:intert_even} holds for these three cases, since every other choice is trivially true.\\

\noindent \textbf{STEP 1:} For the \textbf{diagonal case $y=y'$}, we note that $\mathcal{A}_{N}((x,y'),(x',y'))$ is non-zero only for \\
\textbf{a)} $x=x'+e_{i}$, for some $1\leq i \leq n$, assuming $x_{i}'<y_{i}'$;\\
\textbf{b)} $x=x'-e_{i}$, for some $1\leq i \leq n$, assuming $x_{i}'>y_{i+1}'$;\\
\textbf{c)} $x=x'$.

We will calculate the contribution of each case to the summation in the right-hand side of \eqref{eq:intert_even} separately and then we will combine the results together to obtain the left-hand side of \eqref{eq:intert_even}.\\
For \textbf{a)} let $x = x'+e_{i}$ for some $1\leq i \leq n$ which moreover satisfies $x_{i}'<y_{i}'$. Note that, if $x_{i}'\geq y_{i}'$, this would imply that $x_{i}=y_{i}+1$ which is not a valid configuration.

\noindent For $1\leq i<n$ we have 
\begin{equation*}
\begin{split}
\mathcal{A}_{N}((x'+e_{i},y'),(x',y'))&=Q_{N-1}^q(x'+e_{i},x')\\
&=\dfrac{\hat{\mathcal{P}}^{(N-1)}_{x'}(a;q)}{\hat{\mathcal{P}}^{(N-1)}_{x'+e_{i}}(a;q)}f_{N-1}(x'+e_{i},x')\\
&=\dfrac{\hat{\mathcal{P}}^{(N-1)}_{x'}(a;q)}{\hat{\mathcal{P}}^{(N-1)}_{x'+e_{i}}(a;q)}(1-q^{x'_{i}-x'_{i+1}+1})
\end{split}
\end{equation*}
and
\begin{equation*}
\begin{split}
m_{N}(x'+e_{i},y')&=\Lambda^{a_{n}^{-1},q}_{N-1,N}(x'+e_{i},y')\dfrac{\hat{\mathcal{P}}^{(N-1)}_{x'+e_{i}}(a;q)}{\hat{\mathcal{P}}^{(N)}_{y'}(a;q)}\\
&=m_{N}(x',y')\dfrac{\Lambda^{a_{n}^{-1},q}_{N-1,N}(x'+e_{i},y')}{\Lambda^{a_{n}^{-1},q}_{N-1,N}(x',y')}\dfrac{\hat{\mathcal{P}}^{(N-1)}_{x'+e_{i}}(a;q)}{\hat{\mathcal{P}}^{(N-1)}_{x'}(a;q)}
\end{split}
\end{equation*}
with
\begin{equation*}
\begin{split}
\Lambda^{a_{n}^{-1},q}_{N-1,N}(x'+e_{i},y')&=a_{n}^{\sum_{j=1}^nx_{j}'+1-\sum_{j=1}^n y'_{j}}\prod_{j\neq i}\dbinom{y'_{j}-y'_{j+1}}{y'_{j}-x'_{j}}_{q}\dbinom{y'_{i}-y'_{i+1}}{y'_{i}-(x'_{i}+1)}_{q}\dbinom{y'_{n}}{y'_{n}-x'_{n}}_{q}\\
&=\Lambda^{a_{n}^{-1},q}_{N-1,N}(x',y')\,a_{n}\dfrac{1-q^{y'_{i}-x'_{i}}}{1-q^{x'_{i}-y'_{i+1}+1}}
\end{split}
\end{equation*}
where for the last equality we used the following property of the $q$-binomial coefficient 
\[\dbinom{n}{k-1}_{q}=\dbinom{n}{k}_{q}\dfrac{1-q^k}{1-q^{n-k+1}}.\]
Therefore, we conclude that the contribution of the transition $x=x'+e_{i}$, for $1\leq i<n$ to the summation in the right-hand side of \eqref{eq:intert_even} is
\begin{equation*}
\begin{split}
\dfrac{m_{N}(x'+e_{i},y')}{m_{N}(x',y')}\mathcal{A}_{N}((x'+e_{i},y'),(x',y'))&=a_{n}(1-q^{y_{i}'-x_{i}'})\dfrac{1-q^{x_{i}'-x_{i+1}'+1}}{1-q^{x_{i}'-y_{i+1}'+1}}\\
&=a_{n}(L_{i+1}(y';x')+q^{y_{i+1}'-x_{i+1}'}-q^{y_{i}'-x_{i}'})
\end{split}
\end{equation*}
with 
\[L_{i+1}(y';x')=\dfrac{(1-q^{y'_{i+1}-x'_{i+1}})(1-q^{y'_{i}-y'_{i+1}+1})}{1-q^{x'_{i}-y'_{i+1}+1}}.\]
For the last equality we observe that the following holds
\[(1-q^{y_{i}'-x_{i}'})\dfrac{1-q^{x_{i}'-x_{i+1}'+1}}{1-q^{x_{i}'-y_{i+1}'+1}}=\dfrac{(1-q^{y'_{i+1}-x'_{i+1}})(1-q^{y'_{i}-y'_{i+1}+1})}{1-q^{x'_{i}-y'_{i+1}+1}}+ q^{y'_{i+1}-x'_{i+1}}-q^{y'_{i}-x'_{i}}.\]

\noindent When $i=n$, following the same reasoning, we have that
\[\dfrac{m_{N}(x'+e_{i},y')}{m_{N}(x',y')}\mathcal{A}_{N}((x'+e_{i},y'),(x',y')) = a_{n}(1-q^{y_{n}'-x_{n}'}).\]

\noindent For \textbf{b)} let $x=x'-e_{i}$, for some $1\leq i < n$ such that $x_{i}'>y_{i+1}'$. Then, calculations similar to the case \textbf{a)} lead to 
\begin{equation*}
\begin{split}
\dfrac{m_{N}(x'-e_{i},y')}{m_{N}(x',y')}\mathcal{A}_{N}((x'-e_{i},y'),(x',y'))&=a_{n}^{-1}(1-q^{x_{i}'-y_{i+1}'})\dfrac{1-q^{x_{i-1}'-x_{i}'+1}}{1-q^{y_{i}'-x_{i}'+1}}\\
&=a_{n}^{-1}(R_{i}(y';x')+q^{x_{i-1}'-y_{i}'}-q^{x_{i}'-y_{i+1}'})
\end{split}
\end{equation*}
where 
\[R_{i}(y';x') = \dfrac{(1-q^{x_{i-1}'-y_{i}'})(1-q^{y_{i}'-y_{i+1}'+1})}{1-q^{y_{i}'-x_{i}'+1}}\]
Similarly, if $x = x'-e_{n}$ such that $x_{n}'>0$, then
\[\dfrac{m_{N}(x'-e_{n},y')}{m_{N}(x',y')}\mathcal{A}_{N}((x'-e_{n},y'),(x',y')) = a_{n}^{-1}(R_{n}(y';x')+q^{x_{n-1}'-y_{n}'}-q^{x_{n}'}).\]

We observe that gathering all terms corresponding to case \textbf{a)}, we have the following
\begin{equation*}
\begin{split}
\sum_{i=1}^n \dfrac{m_{N}(x'+e_{i},y')}{m_{N}(x',y')}&\mathcal{A}_{N}((x;+e_{i},y'),(x',y'))\\
&= a_{n} \Big(\sum_{i=1}^{n-1}\big(L_{i+1}(y';x')+q^{y_{i+1}'-x_{i+1}'}-q^{y'_{i}-x'_{i}}\big)+1-q^{y_{n}'-x_{n}'} \Big)\\
& = a_{n}\Big(\sum_{i=1}^{n-1}L_{i+1}(y';x') + \sum_{i=1}^{n-1} q^{y'_{i+1}-x'_{i+1}}+1 - \sum_{i=1}^nq^{y'_{i}-x'_{i}} \Big)\\
&= a_{n}\Big(\sum_{i=2}^n L_{i}(y';x')+(1 - q^{y'_{1}-x'_{1}})\Big)\\
& = a_{n}\sum_{i=1}^n L_{i}(y';x').
\end{split}
\end{equation*}
Similarly for all the terms in case \textbf{b)} one can calculate
\[\sum_{i=1}^n \dfrac{m_{N}(x'-e_{i},y')}{m_{N}(x',y')}\mathcal{A}_{N}((x;-e_{i},y'),(x',y'))=a_{n}^{-1}\Big(\sum_{i=1}^n R_{i}(y';x')-q^{x_{n}'}\Big).\]
\noindent Combining the previous two cases with the case \textbf{c)}, which equals $\mathcal{A}_{N}((x',y'),(x',y'))$ defined as in \eqref{eq:Odd2EvenADiag}, we conclude that
\[\sum_{x \preceq y'}\dfrac{m_{N}(x,y')}{m_{N}(x',y')}\mathcal{A}_{N}((x,y'),(x',y'))=-\sum_{i=1}^n (a_{i}+a_{i}^{-1})=Q_{N}(y',y')\]
as required.\\

\noindent \textbf{STEP 2:} For the \textbf{off-diagonal cases $y\neq y'$}, we first recall that if $y \neq y'$
\[Q^q_{N}(y,y') = \dfrac{\hat{\mathcal{P}}^{(N)}_{y'}(a;q)}{\hat{\mathcal{P}}^{(N)}_{y}(a;q)}f_{N}(y,y')\]
and
\[\dfrac{m_{N}(x,y)}{m_{N}(x',y')} = \dfrac{\Lambda_{N-1,N}^{a_{n}^{-1},q}(x,y)}{\Lambda_{N-1,N}^{a_{n}^{-1},q}(x',y')}\dfrac{\hat{\mathcal{P}}^{(N)}_{y'}(a;q)}{\hat{\mathcal{P}}^{(N)}_{y}(a;q)}\dfrac{\hat{\mathcal{P}}^{(N-1)}_{x}(a;q)}{\hat{\mathcal{P}}^{(N-1)}_{x'}(a;q)}\]
therefore proving \eqref{eq:intert_even} for $y\neq y'$ is equivalent to proving the following identity
 \begin{equation}
\Lambda_{N-1,N}^{a_{n}^{-1},q}(x',y')f_{N}(y,y')= \sum_{x \preceq y}\Lambda_{N-1,N}^{a_{n}^{-1},q}(x,y)\dfrac{\hat{\mathcal{P}}^{(N-1)}_{x}(a;q)}{\hat{\mathcal{P}}^{(N-1)}_{x'}(a;q)}\mathcal{A}_{N}((x,y),(x',y')).
\label{eq:intert_even_off_diagonal}
\end{equation}

Let us first prove the relation \eqref{eq:intert_even_off_diagonal} for $y=y'-e_{i}$, for some $1\leq i \leq n$. We consider the dichotomy $y_{i}'=x_{i}'$ or $y_{i}'>x_{i}'$. Suppose we are in the former case, i.e. $y=y'-e_{i}$ and $y_{i}'=x_{i}'$, then necessarily $x=x'-e_{i}$, otherwise we would have $y_{i}=x_{i}-1$ and this would violate the interlacing condition. Therefore the right-hand side of \eqref{eq:intert_even_off_diagonal} equals

\begin{equation*}
\begin{split}
\Lambda_{N-1,N}^{a_{n}^{-1},q}(x'-e_{i},&y'-e_{i})\dfrac{\hat{\mathcal{P}}^{(N-1)}_{x'-e_{i}}(a;q)}{\hat{\mathcal{P}}^{(N-1)}_{x'}(a;q)}\mathcal{A}_{N}((x'-e_{i},y'-e_{i}),(x',y'))\\
&= \Lambda_{N-1,N}^{a_{n}^{-1},q}(x'-e_{i},y'-e_{i})\dfrac{\hat{\mathcal{P}}^{(N-1)}_{x'-e_{i}}(a;q)}{\hat{\mathcal{P}}^{(N-1)}_{x'}(a;q)}Q^q_{N-1}(x'-e_{i},x')\\
&=\Lambda_{N-1,N}^{a_{n}^{-1},q}(x',y')(1-q^{y_{i-1}'-y_{i}'+1})\dfrac{(1-q^{x_{i}'-y_{i+1}'})(1-q^{x_{i-1}'-x_{i}'+1})}{(1-q^{x_{i-1}'-y_{i}'+1})(1-q^{y_{i}'-y_{i+1}'})}
\end{split}
\end{equation*}
where for the last equality we used the properties of the $q$-binomial we recorded in \eqref{eq:q_binomial}. Using the assumption $x_{i}'=y_{i}'$ the last expression simplifies to $\Lambda_{N-1,N}^{a_{n}^{-1},q}(x',y')(1-q^{y_{i-1}'-y_{i}'+1})$.\\

\noindent On the other hand, if $y_{i}'>x_{i}'$, then it is not possible that the movement on the $i$-th component of $Y$ was performed due to pushing therefore the right--hand side of \eqref{eq:intert_even_off_diagonal} is given by
\begin{equation*}
\begin{split}
\Lambda_{N-1,N}^{a_{n}^{-1},q}&(x',y'-e_{i})\dfrac{\hat{\mathcal{P}}^{(N-1)}_{x'}(a;q)}{\hat{\mathcal{P}}^{(N-1)}_{x'}(a;q)}\mathcal{A}_{N}((x',y'-e_{i}),(x',y'))\\
&=\Lambda_{N-1,N}^{a_{n}^{-1},q}(x',y')a_{n}\dfrac{(1-q^{y_{i-1}'-y_{i}'+1})(1-q^{y_{i}'-x_{i}'+1})}{(1-q^{x_{i-1}'-y_{i}'})(1-q^{y_{i}'-y_{i+1}'+1})}a_{n}^{-1}R_{i}(y'-e_{i};x').
\end{split}
\end{equation*}
Using the definition of $R_{i}(y;x)$, in \eqref{eq:RightProbabilities}, we immediately see that the last quantity equals $\Lambda_{N-1,N}^{a_{n}^{-1},q}(x',y')(1-q^{y_{i-1}'-y_{i}'+1})$.\\

Let us now turn to the case $y=y'+e_{i}$ for some $1\leq i \leq n$. If $i=1$, the rightmost component of $Y$ doesn't have an upper right neighbour therefore the only valid transition is from $(x',y'+e_{1})$ to $(x',y')$, which occurs at rate $a_{n}L_{1}(y'+e_{i},x')$. The right-hand side of \eqref{eq:intert_even_off_diagonal} is then equal to 
\begin{equation*}
\begin{split}
\Lambda_{N-1,N}^{a_{n}^{-1},q}&(x',y'+e_{1})\dfrac{\hat{\mathcal{P}}^{(N-1)}_{x'}(a;q)}{\hat{\mathcal{P}}^{(N-1)}_{x'}(a;q)}\mathcal{A}_{N}((x',y'+e_{1}),(x',y'))\\
&=\Lambda_{N-1,N}^{a_{n}^{-1},q}(x',y')a_{n}^{-1}\dfrac{1-q^{y_{1}'-y_{2}'+1}}{1-q^{y_{1}'-x_{1}'+1}}a_{n}L_{1}(y'+e_{1};x').
\end{split}
\end{equation*}
where we used the properties of the $q$-binomial from \eqref{eq:q_binomial}. By the definition of $L_{1}(y;x)$, in \eqref{eq:LeftProbabilities}, the last expression equals $\Lambda_{N-1,N}^{a_{n}^{-1},q}(x',y')(1-q^{y_{1}'-y_{2}'+1})$, which is exactly the expression in the left-hand side of \eqref{eq:intert_even_off_diagonal}.

If $1<i \leq n$ we consider two cases; either $y_{i}'=x_{i-1}'$ or $y_{i}'<x_{i-1}'$. In the former case, i.e. $y_{i}=y_{i}+1$ and $y_{i}'=x_{i-1}'$, the transition occurred due to pushing since an independent jump of the $i$-th component of $Y$ would imply that $y_{i}-1=x_{i-1}$ and this would lead to violation of the intertwining property. So in this case the right-hand side of \eqref{eq:intert_even_off_diagonal} equals

\begin{equation*}
\begin{split}
\Lambda_{N-1,N}^{a_{n}^{-1},q}(x'+e_{i-1},& y'+e_{i})\dfrac{\hat{\mathcal{P}}^{(N-1)}_{x'+e_{i-1}}(a;q)}{\hat{\mathcal{P}}^{(N-1)}_{x'}(a;q)}\mathcal{A}_{N}((x'+e_{i-1},y'+e_{i}),(x',y'))\\
&=\Lambda_{N-1,N}^{a_{n}^{-1},q}(x',y')(1-q^{y_{i}'-y_{i+1}'+1})\dfrac{(1-q^{y_{i-1}'-x_{i-1}'})(1-q^{x_{i-1}'-x_{i}'+1})}{(1-q^{y_{i}'-x_{i}'+1})(1-q^{y_{i-1}'-y_{i}'})}
\end{split}
\end{equation*}
and using the assumption $y_{i}'=x_{i-1}'$, we conclude that the last quantity equals $\Lambda_{N-1,N}^{a_{n}^{-1},q}(x',y')(1-q^{y_{i}'-y_{i+1}'+1})$.\\

\noindent If we assume instead that $y_{i}'<x_{i-1}'$, then the only valid transition is from $(x',y'+e_{i})$ to $(x',y')$ which occurs at rate $a_{n}L_{i}(y'+e_{i};x')$. Then the right-hand side of \eqref{eq:intert_even_off_diagonal} is
\begin{equation*}
\begin{split}
\Lambda_{N-1,N}^{a_{n}^{-1},q}&(x',y'+e_{i})\dfrac{\hat{\mathcal{P}}^{(N-1)}_{x'}(a;q)}{\hat{\mathcal{P}}^{(N-1)}_{x'}(a;q)}\mathcal{A}_{N}((x',y'+e_{i}),(x',y'))\\
&=\Lambda_{N-1,N}^{a_{n}^{-1},q}(x',y')a_{n}^{-1}\dfrac{(1-q^{y_{i}'-y_{i+1}'+1})(1-q^{x_{i-1}'-y_{i}'})}{(1-q^{y_{i}'-x_{i}'+1})(1-q^{y_{i-1}'-y_{i}'})}a_{n}L_{i}(y'+e_{i};x').
\end{split}
\end{equation*}
which by the definition of $L_{i}(y;x)$, in \eqref{eq:LeftProbabilities}, equals $\Lambda_{N-1,N}^{a_{n}^{-1},q}(x',y')(1-q^{y_{i}'-y_{i+1}'+1})$.
\end{proof}

Let us now proceed to proving relation \eqref{eq:main_intert(discrete)}. For $\mathbf{z}\in \mathbb{K}^N_{0}$ we write $\mathbf{z}^{1:N-1}$ for the sub-pattern of the top $N-1$ levels of $\mathbf{z}$. The kernel $\mathcal{L}_{N}$ can be decomposed as follows, for $z\in \mathcal{W}^{n}_{0}$ and $\mathbf{z}\in \mathbb{K}^N_{0}$
\[\mathcal{L}_{N}(z,\mathbf{z})=\mathcal{L}_{N-1}(z^{N-1},\mathbf{z}^{1:N-1})\mathcal{K}_{N}(z,(z^{N-1},z^N)).\]
We also observe that any transition initiated at the top $N-1$ levels can affect $z^N$ only via $z^{N-1}$ and any transition initiated at the bottom level does not affect the sub-pattern $\mathbf{z}^{1:N-1}$, therefore the transition kernel $\hat{Q}^q_{N}$ of the process $\mathbf{Z}$ has off-diagonal entries, given for $\mathbf{z},\mathbf{w}\in \mathbb{K}^N_{0}$ by\\
\begin{table}[ht]
\begin{center}
\scalebox{0.9}{

\begin{tabular}{|l|l|}\hline
$\mathbf{w}$ & $\hat{Q}^q_{N}(\mathbf{z},\mathbf{w})$\\ \hline
$(z^{N-1}+e_{i},z^N+e_{i}), z^{N-1}_{i}=z^N_{i}, \text{ for some }i\in [n]$ & $\hat{Q}^q_{N}(\mathbf{z}^{1:N-1},\mathbf{w}^{1:N-1})$\\
$(z^{N-1}+e_{i},z^N), z^{N-1}_{i}<z^N_{i}, \text{ for some } i\in [n]$ & $\hat{Q}^q_{N}(\mathbf{z}^{1:N-1},\mathbf{w}^{1:N-1})$\\
$(z^{N-1}-e_{i},z^N-e_{i+1}), z^{N-1}_{i}=z^N_{i+1},\text{ for some } i\in [n-1]$ & $\hat{Q}^q_{N}(\mathbf{z}^{1:N-1},\mathbf{w}^{1:N-1})$\\
$(z^{N-1}-e_{i},z^N), z^{N-1}_{i}>z^N_{i+1},\text{ for some } i\in [n-1]$ & $\hat{Q}^q_{N}(\mathbf{z}^{1:N-1},\mathbf{w}^{1:N-1})$\\
$(z^{N-1}-e_{n},z^N)$ & $\hat{Q}^q_{N}(\mathbf{z}^{1:N-1},\mathbf{w}^{1:N-1})$\\
$(z^{1:N-1},z^N \pm e_{i}),\text{ for some } i \in [n]$ & $\mathcal{A}_{N}((z^{N-1},z^N),(w^{N-1},w^N))$\\
\hline
\end{tabular}

}
\end{center}
\end{table}

\noindent where $R_{i}, L_{i}$ are the probabilities defined in \eqref{eq:RightProbabilities}, \eqref{eq:LeftProbabilities}. All the other off-diagonal entries equal to zero. The diagonal entries are for $\mathbf{z}\in \mathbb{K}^{N}_{0}$ equal to
\begin{small}
\begin{equation}
\hat{Q}^q_{N}(\mathbf{z},\mathbf{z})=\hat{Q}^q_{N-1}(\mathbf{z}^{1:N-1},\mathbf{z}^{1:N-1})-\sum_{i=1}^n[a_{n}L_{i}(z^N;z^{N-1})+a_{n}^{-1}R_{i}(z^N;z^{N-1})].
\label{eq:Odd2EvenDiag}
\end{equation}
\end{small}
Then the intertwining relation \eqref{eq:main_intert(discrete)} is equivalent to proving for each $z\in \mathcal{W}^n_{0}$  and $\mathbf{w}\in \mathbb{K}^N_{0}$ the following
\begin{equation}
\begin{split}
Q_{N}^q(z,w^N)\mathcal{L}_{N}(w^N,\mathbf{w})&= \sum_{(z^{N-1},z^N)\in \mathcal{W}^{n,n}_{0}}\mathcal{K}_{N}(z,(z^{N-1},z^N))\\
&\hspace{50pt}\sum_{\mathbf{z}^{1:N-1}\in \mathbb{K}^{N-1}_{0}}\mathcal{L}_{N-1}(z^{N-1},\mathbf{z}^{1:N-1})\hat{Q}^q_{N}(\mathbf{z},\mathbf{w})
\end{split}
\label{eq:main_intert_decomposed_odd2even(discrete)}
\end{equation}
where $\mathbf{z}=(\mathbf{z}^{1:N-1},z^N)$.\\

\noindent \textbf{Diagonal case:} If $w^N = z$, then we necessarily have $w^N = z^N$ therefore the inner sum in the right hand side of \eqref{eq:main_intert_decomposed_odd2even(discrete)} equals
\begin{equation*}
\begin{split}
\sum_{\substack{\mathbf{z}^{1:N-1}\in \mathbb{K}^{N-1}_{0}:\\ \mathbf{z}^{1:N-1} \neq \mathbf{w}^{1:N-1}}}
\mathcal{L}_{N-1}&(z^{N-1},\mathbf{z}^{1:N-1})\hat{Q}^q_{N-1}(\mathbf{z}^{1:N-1},\mathbf{w}^{1:N-1})\\
+\mathcal{L}_{N-1}&(z^{N-1},\mathbf{w}^{1:N-1})\Big[\hat{Q}^q_{N-1}(w^{N-1},w^{N-1})\\
&\hspace{55pt}-\sum_{i=1}^n \big(a_{n}L_{i}(w^N;w^{N-1}) +a_{n}^{-1}R_{i}(w^N;w^{N-1})\big) \Big]\\
= \sum_{\mathbf{z}^{1:N-1}\in \mathbb{K}^{N-1}_{0}}&
\mathcal{L}_{N-1}(z^{N-1},\mathbf{z}^{1:N-1})\hat{Q}^q_{N-1}(\mathbf{z}^{1:N-1},\mathbf{w}^{1:N-1})\\
-\mathcal{L}_{N-1}&(z^{N-1},\mathbf{w}^{1:N-1})\sum_{i=1}^n \big(a_{n}L_{i}(w^N;w^{N-1}) +a_{n}^{-1}R_{i}(w^N;w^{N-1})\big)\\
= Q^q_{N-1}(z^{N-1}&,w^{N-1})\mathcal{L}_{N-1}(w^{N-1},\mathbf{w}^{1:N-1})\\
-\mathcal{L}_{N-1}&(z^{N-1},\mathbf{w}^{1:N-1})\sum_{i=1}^n \big(a_{n}L_{i}(w^N;w^{N-1}) +a_{n}^{-1}R_{i}(w^N;w^{N-1})\big)
\end{split}
\end{equation*}
where in the last equality we used the induction hypothesis.\\
Therefore, the right hand side of  \eqref{eq:main_intert_decomposed_odd2even(discrete)} equals
\begin{equation*}
\begin{split}
\mathcal{L}_{N-1}(w^{N-1},\mathbf{w}^{1:N-1})&\Big[\sum_{z^{N-1}\in \mathcal{W}^n_{0}}\mathcal{K}_{N}(z,(z^{N-1},z))Q^q_{N-1}(z^{N-1},w^{N-1})\\
-&\mathcal{K}_{N}(z,(w^{N-1},z))\sum_{i=1}^n \big(a_{n}L_{i}(w^N;w^{N-1}) +a_{n}^{-1}R_{i}(w^N;w^{N-1})\big)\Big].
\end{split}
\end{equation*}
We observe that the expression inside the square brackets equals
\[\sum_{z^{N-1}\in \mathcal{W}^n_{0}}\mathcal{K}_{N}(z,(z^{N-1},z))\mathcal{A}_{N}((z^{N-1},z),(w^{N-1},w^N))\]
for $z=w^N$.\\
Therefore, using Proposition \ref{helper_intert_even(discrete)}, we conclude that the right hand side of  \eqref{eq:main_intert_decomposed_odd2even(discrete)} equals
\[\mathcal{L}_{N-1}(w^{N-1},\mathbf{w}^{1:N-1})Q^q_{N}(z,w^N)\mathcal{K}_{N}(w^N,(w^{N-1},w^N)) = \mathcal{L}_{N}(z,\mathbf{w})Q^q_{N}(z,w^N)\]
as required.\\

\noindent \textbf{Off-diagonal case:} Assume that $z\neq w^N$, then the right hand side of  \eqref{eq:main_intert_decomposed_odd2even(discrete)} equals
\begin{equation*}
\begin{split}
\sum_{\substack{z^{N-1}\in \mathcal{W}^n_{0}:\\ z^{N-1}\neq w^{N-1}}}&\mathcal{K}_{N}(z,(z^{N-1},z))\sum_{\mathbf{z}^{1:N-1}\in \mathbb{K}^{N-1}_{0}}\mathcal{L}_{N-1}(z^{N-1},\mathbf{z}^{1:N-1})\hat{Q}^q_{N}(\mathbf{z},\mathbf{w})\\
+&\mathcal{K}_{N}(z,(w^{N-1},z))\sum_{\mathbf{z}^{1:N-1}\in \mathbb{K}^{N-1}_{0}}\mathcal{L}_{N-1}(w^{N-1},\mathbf{z}^{1:N-1})\hat{Q}^q_{N}(\mathbf{z},\mathbf{w}).
\end{split}
\end{equation*}
We observe that if $z^{N-1}=w^{N-1}$ and $z\neq w^N$ then necessarily $\mathbf{z}^{1:N-1}=\mathbf{w}^{1:N-1}$, therefore the second term equals
\[\mathcal{K}_{N}(z,(w^{N-1},z))\mathcal{L}_{N-1}(w^{N-1},\mathbf{w}^{1:N-1})\mathcal{A}_{N}((w^{N-1},z),(w^{N-1},w^N)).\]
The inner sum of the first term equals
\begin{equation*}
\begin{split}
\sum_{\mathbf{z}^{1:N-1}\in \mathbb{K}^{N-1}_{0}}\mathcal{L}_{N-1}(z^{N-1},\mathbf{z}^{1:N-1})&\hat{Q}^{q}_{N-1}(\mathbf{z}^{1:N-1},\mathbf{w}^{1:N-1})\\
& = (\mathcal{L}_{N-1}\hat{Q}^q_{N-1})(z^{N-1},\mathbf{w}^{1:N-1})\\
&= (Q^q_{N-1}\mathcal{L}_{N-1})(z^{N-1},\mathbf{w}^{1:N-1})
\end{split}
\end{equation*}
where the last equality is true from the induction hypothesis.\\
Therefore the right hand side of \eqref{eq:main_intert_decomposed_odd2even(discrete)} equals
\begin{equation*}
\begin{split}
\mathcal{L}_{N-1}(w^{N-1},\mathbf{w}^{1:N-1})\Big[&\sum_{\substack{z^{N-1}\in \mathcal{W}^{n}_{0}:\\ z^{N-1}\neq w^{N-1}}}\mathcal{K}_{N}(z,(z^{N-1},z))Q^q_{N-1}(z^{N-1},w^{N-1})\\
&+\mathcal{K}_{N}(z,(w^{N-1},z))\mathcal{A}_{N}((w^{N-1},z),(w^{N-1},w^N)).\Big]
\end{split}
\end{equation*}
Observe that the term inside the square brackets corresponds to 
\[(\mathcal{K}_{N}\mathcal{A}_{N})(z,(w^{N-1},w^N))\]
for $z\neq w^N$, therefore using the intertwining relation of Proposition \ref{helper_intert_even(discrete)} we conclude that the right hand side of \eqref{eq:main_intert_decomposed_odd2even(discrete)} equals
\begin{equation*}
\begin{split}
\mathcal{L}_{N-1}(w^{N-1},\mathbf{w}^{1:N-1})&(Q^q_{N}\mathcal{K}_{N})(z,(w^{N-1},w^N))\\
 &= \mathcal{L}_{N-1}(w^{N-1},\mathbf{w}^{1:N-1})Q^q_{N}(z,w^N)\mathcal{K}_{N}(w^N,(w^{N-1},w^N))\\
&= \mathcal{L}_{N}(z,\mathbf{w})Q^q_{N}(z,w^N)
\end{split}
\end{equation*}
as required.
% even to odd proof
\paragraph{Part II: Iterating from even to odd row.}
Suppose that $N=2n-1$, for some $n \geq 1$ and assume that
\[Q^q_{N-1}\mathcal{L}_{N-1}=\mathcal{L}_{N-1}\hat{Q}^q_{N-1}.\]
We again introduce a matrix $\mathcal{A}_{N}$ on $\mathcal{W}_{0}^{n-1,n}=\{(x,y)\in \mathcal{W}_{0}^{n-1} \times \mathcal{W}_{0}^n: x \preceq y\}$ which focuses only on the bottom two levels. The matrix $\mathcal{A}_{N}$ has off-diagonal entries given for $(x,y),(x',y')\in \mathcal{W}^{n-1,n}_{0}$ as in the following table
\\ 
\begin{table}[ht]
\begin{center}
\begin{tabular}{|l|l|}\hline
$(x',y')$ & $\mathcal{A}_{N}((x,y),(x',y')))$\\ \hline
$(x+e_{i},y+e_{i}), x_{i}=y_{i}, \text{ for some }i\in [n]$ & $Q_{N-1}^q(x,x+e_{i})$\\
$(x+e_{i},y), x_{i}<y_{i}, \text{ for some } i\in [n]$ & $Q_{N-1}^q(x,x+e_{i})$\\
$(x-e_{i},y-e_{i+1}), x_{i}=y_{i+1},\text{ for some } i\in [n-1]$ & $Q_{N-1}^q(x,x-e_{i})$\\
$(x-e_{i},y), x_{i}>y_{i+1},\text{ for some } i\in [n-1]$ & $Q_{N-1}^q(x,x-e_{i})$\\
$(x,y-e_{i}),\text{ for some }i\in [n]$ & $a_{n}^{-1}L_{i}(y;x)$\\
$(x,y+e_{i}),\text{ for some }i\in [n]$ & $a_{n}R_{i}(y;x)$\\
\hline
\end{tabular}
\end{center}
\end{table}

\noindent where $R_{i}, L_{i}$ are defined in \eqref{eq:RightProbabilities}, \eqref{eq:LeftProbabilities}. All the other off-diagonal entries equal to zero. The diagonal entries are for $(x,y)\in \mathcal{W}_{0}^{n-1,n}$ given by
\begin{small}
\begin{equation}
\mathcal{A}_{N}((x,y),(x,y)))=-\sum_{i=1}^{n-1}(a_{i}+a_{i}^{-1})-\sum_{i=1}^n[a_{n}^{-1}L_{i}(y;x)+a_{n}R_{i}(y;x)].
\label{eq:Even2OddADiag}
\end{equation}
\end{small}

We now need to define a Markov kernel from $\mathcal{W}_{0}^n$ to $\mathcal{W}^{n-1,n}_{0}$, i.e. a mapping from $\mathcal{W}^n_{0}\times \mathcal{W}^{n-1,n}_{0}$ to $[0,1]$ satisfying the properties of definition \ref{MarkovKernel}. It is easy to see that the function $\hat{\mathcal{P}}^{(N)}_{y}(a;q)$, with $N=2n-1$, defined in \eqref{eq:symp_q_whittaker} satisfies the following recursion
\[\hat{\mathcal{P}}^{(N)}_{y}(a;q)= \sum_{x\in \mathcal{W}_{0}^{n-1}: x \preceq y}\Lambda_{N-1,N}^{a_{n},q}(x,y) \hat{\mathcal{P}}_{x}^{(N-1)}(\tilde{a};q)\]
where for $a=(a_{1},...,a_{n})$, $\tilde{a}$ is a vector containing its first $n-1$ components. Therefore, $m:\mathcal{W}_{0}^{n-1,n}\mapsto [0,1]$, defined by
\[m_{N}(x,y)=\Lambda_{N-1,N}^{a_{n},q}(x,y) \dfrac{\hat{\mathcal{P}}_{x}^{(N-1)}(\tilde{a};q)}{\hat{\mathcal{P}}^{(N)}_{y}(a;q)}\]
gives a Markov kernel $\mathcal{K}_{N}$ from $\mathcal{W}_{0}^n$ to $\mathcal{W}_{0}^{n-1,n}$ defined by
\[\mathcal{K}_{N}(y,(x',y'))=m_{N}(x',y')\mathbbm{1}_{y'=y}.\]

\begin{proposition}
\label{helper_intert_odd(discrete)}
The transition kernel $Q^q_{N}$ is intertwined with $\mathcal{A}_{N}$ via the kernel $\mathcal{K}_{N}$, i.e. the relation
\[Q^q_{N}\mathcal{K}_{N} = \mathcal{K}_{N}\mathcal{A}_{N}\]
holds.
\end{proposition}

\begin{proof}
The intertwining relation of Proposition \ref{helper_intert_odd(discrete)} is equivalent to
\begin{equation}
\label{eq:intert_odd}
Q_{N}^q(y,y')=\sum_{x\preceq y}\dfrac{m_{N}(x,y)}{m_{N}(x',y')}\mathcal{A}_{N}((x,y),(x',y')), \quad y \in \mathcal{W}_{0}^n, (x',y')\in \mathcal{W}_{0}^{n-1,n}
\end{equation} 
where the summation is over $x \in \mathcal{W}_{0}^{n-1}$ such that $(x,y)\in \mathcal{W}_{0}^{n-1,n}$.

Proving \eqref{eq:intert_odd} is similar to the calculations we performed for the odd-to-even intertwining relation \eqref{eq:intert_even}. Therefore, we will skip the calculations for most of the transitions and we will only prove the intertwining relation for the diagonal case, i.e. $y=y'$ and the rightward transition of $Y^n$.

\noindent \textbf{STEP 1:} For the \textbf{diagonal case $y=y'$}, we note that $\mathcal{A}_{N}((x,y'),(x',y'))$ is non-zero only for the following cases\\
\textbf{a)} $x=x'+e_{i}$, for some $1\leq i \leq n-1$ assuming $x_{i}'<y_{i}'$ ;\\
\textbf{b)} $x=x'-e_{i}$, for some $1\leq i \leq n-1$ assuming $x_{i}'>y_{i+1}'$;\\
\textbf{c)} $x=x'$.

Let us calculate the contribution of each case to the right-hand side of \eqref{eq:Even2OddADiag}.\\
\noindent For $\textbf{a)}$, let $x = x' + e_{i}$, for some $1\leq i \leq n-1$ where $x_{i}'<y_{i}'$. We then calculate
\begin{equation*}
\begin{split}
\dfrac{m_{N}(x'+e_{i},y')}{m_{N}(x',y')}\mathcal{A}_{N}((x'+e_{i},y'),(x',y'))&=\dfrac{m_{N}(x'+e_{i},y')}{m_{N}(x',y')}Q^q_{N-1}(x'+e_{i},x')\\
&=a_{n}^{-1}(1-q^{y_{i}'-x_{i}'})\dfrac{1-q^{x_{i}'-x_{i+1}'+1}}{1-q^{x_{i}'-y_{i+1}'+1}}.
\end{split}
\end{equation*}
It is easy to check that the last quantity equals
\[a_{n}^{-1}(L_{i+1}(y';x')+q^{y_{i+1}'-x_{i+1}'}-q^{y_{i}'-x_{i}'}).\]
For case \textbf{b)}, which corresponds to $x=x'-e_{i}$ for some $1\leq i \leq n-1$ with $x_{i}'>y_{i+1}'$, we have
\begin{equation*}
\begin{split}
\dfrac{m_{N}(x'-e_{i},y')}{m_{N}(x',y')}\mathcal{A}_{N}((x'-e_{i},y'),(x',y'))&=\dfrac{m_{N}(x'-e_{i},y')}{m_{N}(x',y')}Q^q_{N-1}(x'-e_{i},x')\\
&=a_{n}(1-q^{x_{i}'-y_{i+1}'})\dfrac{1-q^{x_{i-1}'-x_{i}'+1}}{1-q^{y_{i}'-x_{i}'+1}}
\end{split}
\end{equation*}
which is equal to
\[a_{n}(R_{i}(y';x')+q^{x_{i-1}'-y_{i}'}-q^{x_{i}'-y_{i+1}'}).\]
Finally, for case \textbf{c)} we have that this equals $\mathcal{A}_{N}((x',y'),(x',y'))$ as defined in \eqref{eq:Even2OddADiag}. Therefore, combining the cases \textbf{a)}, \textbf{b)} with the case corresponding to $x=x'$ we conclude that
\[\sum_{x \preceq y'}\dfrac{m_{N}(x,y')}{m_{N}(x',y')}\mathcal{A}_{N}((x,y'),(x',y'))=-\sum_{i=1}^{n-1} (a_{i}+a_{i}^{-1})-a_{n}-a_{n}^{-1}(1-q^{y_{n}'}).\]
as required.\\

\noindent \textbf{STEP 2:} For the rightward transition of $Y^n$, i.e. the case $y=y'-e_{n}$ we have that
\[Q^q_{N}(y'-e_{N},y') = \dfrac{\hat{\mathcal{P}}^{(N)}_{y'}(a;q)}{\hat{\mathcal{P}}^{(N)}_{y'-e_{n}}(a;q)}f_{N}(y'-e_{n},y')\]
and
\[\dfrac{m_{N}(x,y)}{m_{N}(x',y')} = \dfrac{\Lambda_{N-1,N}^{a_{n}^{-1},q}(x,y'-e_{n})}{\Lambda_{N-1,N}^{a_{n}^{-1},q}(x',y')}\dfrac{\hat{\mathcal{P}}^{(N)}_{y'}(a;q)}{\hat{\mathcal{P}}^{(N)}_{y'-e_{n}}(a;q)}\dfrac{\hat{\mathcal{P}}^{(N-1)}_{x}(\tilde{a};q)}{\hat{\mathcal{P}}^{(N-1)}_{x'}(\tilde{a};q)}.\]
Therefore, proving \eqref{eq:intert_odd}, for $y = y'-e_{n}$ is equivalent to proving the following relation
 \begin{equation}
\Lambda_{N-1,N}^{a_{n},q}(x',y')f_{N}(y'-e_{n},y')= \sum_{x \preceq y'-e_{n}}\Lambda_{N-1,N}^{a_{n},q}(x,y'-e_{n})\dfrac{\hat{\mathcal{P}}^{(N-1)}_{x}(\tilde{a};q)}{\hat{\mathcal{P}}^{(N-1)}_{x'}(\tilde{a};q)}\mathcal{A}_{N}((x,y'-e_{n}),(x',y')).
\label{eq:intert_odd_off_diagonal}
\end{equation}
The transition $y'=y+e_{n}$ can only have occurred due to jump of the leftmost particle of $Y$ of its own volition hence in this case $x=x'$ and the right hand side of \eqref{eq:intert_odd_off_diagonal} equals
\begin{equation*}
\begin{split}
\Lambda_{N-1,N}^{a_{n},q}&(x',y'-e_{n})\dfrac{\hat{\mathcal{P}}^{(N-1)}_{x'}(\tilde{a};q)}{\hat{\mathcal{P}}^{(N-1)}_{x'}(\tilde{a};q)}\mathcal{A}_{N}((x',y'-e_{n}),(x',y'))\\
&=\Lambda_{N-1,N}^{a_{n},q}(x',y')a_{n}^{-1}\dfrac{1-q^{y_{n-1}'-y_{n}'+1}}{1-q^{x_{n-1}'-y_{n}'+1}}a_{n}R_{n}(y'-e_{n};x').
\end{split}
\end{equation*}
Using the definition of $R_{i}(y;x)$ we immediately see that the last quantity equals $\Lambda_{N-1,N}^{a_{n},q}(x',y')(1-q^{y_{n-1}'-y_{n}'+1})$.
\end{proof}

Following the same arguments as for the odd-to-even case, one can prove the main intertwining relation \eqref{eq:main_intert(discrete)} for $N=2n-1$ from the intertwining relation of Proposition \ref{helper_intert_odd(discrete)}.

%% file: tex/q-whittaker_processes.tex
In Chapters \ref{Berele} and \ref{discrete} we studied two processes whose state is space the set of integer-valued Gelfand-Tsetlin patterns denoted by $\mathbb{K}^N_{0}$ and proved that under appropriate initial conditions, if $N=2n$ for some $n\geq 1$ then the shape of the pattern for each process  is a partition that evolves as a continuous-time Markov process. More specifically, if $a=(a_{1},...,a_{n})\in \mathbb{R}^n$ and $q \in (0,1)$ then the shape evolves as a Markov process on the set of partitions with at most $n$ parts, $\Lambda_{n}\equiv \mathcal{W}^n_{0}$, with transition rate matrix 
with entries given by
\begin{equation*}
Q^q_{n}(z,z'):=\left\{ \begin{array}{ll}
\dfrac{\mathcal{P}^{(n)}_{z'}(a;q)}{\mathcal{P}^{(n)}_{z}(a;q)}f_{n}(z,z') & \text{ if } z'=z \pm e_{i},\text{ for some } 1\leq i \leq n\\
-\sum_{i=1}^n(a_{i}+a_{i}^{-1}) & \text{ if }z'=z\\
0 & \text{ otherwise}
\end{array} \right. 
\end{equation*}
where $f_{n}(z,z')$ is defined in \eqref{eq:us} and $\mathcal{P}^{(n)}_{z}(a;q)$ is defined via the recursion in \ref{recursion_q_whittaker}. 

In this chapter we will attempt to study the law of the process of the shape of the patterns. Then we can obtain the law of the whole pattern using the intertwining kernel. In order to do that we need an orthogonality and completeness result associated with the polynomials $\mathcal{P}^{(n)}$. Unfortunately, the only results that appear in the literature are in terms of the $q$-deformed $\mathfrak{so}_{2n+1}$-Whittaker functions defined in Chapter \ref{Koornwinder}. Therefore, the majority of the results in this Chapter are valid assuming the conjecture that the $q$-deformed $\mathfrak{so}_{2n+1}$-Whittaker functions satisfy the recursion in \ref{recursion_q_whittaker} is true. 

For $t\geq 0$, $z \in \mathcal{W}^n_{0}$ and $a \in \mathbb{R}^n_{>0}$, we define the function
\begin{equation}
p_{t}^{(n)}(z;a,q)= \dfrac{\mathcal{P}^{(n)}_{z}(a;q)Q_{z}^{(n)}(t;q)}{\Pi(a;t)}
\label{eq:law_discrete}
\end{equation}
where $\mathcal{P}^{(n)}_{z}(a;q)$ is the function defined in \ref{recursion_q_whittaker}
\begin{equation}
\label{eq:norm_constant}
\Pi(a;t):=e^{\sum_{i=1}^n(a_{i}+a_{i}^{-1})t}
\end{equation}
and
\begin{equation*}
Q_{z}^{(n)}(t;q):= \dfrac{\langle \Pi(\cdot;t),\mathcal{P}^{(n)}_{z}(\cdot;q)\rangle_{\hat{\Delta}^{(n)}}}{\langle \mathcal{P}^{(n)}_{z}(\cdot, q),\mathcal{P}^{(n)}_{z}(\cdot;q)\rangle_{\hat{\Delta}^{(n)}}}
\end{equation*}
where $\langle \cdot, \cdot \rangle_{\hat{\Delta}^{(n)}}$ is the inner product defined in \eqref{eq:torus_scalar_qWhittaker}. 

Assuming that Conjecture \ref{conjecture} holds, the polynomials $\mathcal{P}^{(n)}$ are orthogonal with respect to the inner product $\langle \cdot, \cdot \rangle_{\hat{\Delta}^{(n)}}$ and by \eqref{eq:hyperoctahedral_orthogonality} it holds that
\[\langle \mathcal{P}^{(n)}_{z}(\cdot;q),\mathcal{P}^{(n)}_{z}(\cdot;q) \rangle_{\hat{\Delta}^{(n)}}=1/\Delta_{z}^{(n)}\]
where
\[\Delta^{(n)}_{z}=\dfrac{(q;q)^n_{\infty}}{(q;q)_{z_{n}}\prod_{1\leq j <n}(q;q)_{z_{j}-z_{j+1}}}.\]
Therefore, we may re-write $Q^{(n)}_{z}(t;q)$ as follows
\begin{equation}
Q^{(n)}_{z}(t;q) = \Delta^{(n)}_{z}\langle \Pi(\cdot;t),\mathcal{P}^{(n)}_{z}(\cdot;q)\rangle_{\hat{\Delta}^{(n)}}.
\label{eq:Q(t;q)}
\end{equation}

\begin{lemma}
Under the assumption of Conjecture \ref{conjecture}, $p_{t}^{(n)}(\cdot;a,q)$ defines a probability measure on the set of partitions $\Lambda_{n}$.
\end{lemma}
\begin{proof}

Since all the terms that appear in the recursion \ref{recursion_q_whittaker} defining the function $\mathcal{P}^{(n)}$ are non-negative, we have that $p_{t}^{(n)}(z;a,q)\geq 0$ for every $t>0$, $z \in \mathcal{W}^n_{0}$. Moreover, we calculate
\begin{equation*}
\begin{split}
\sum_{z \in \mathcal{W}^n_{0}}\mathcal{P}^{(n)}_{z}(a;q)Q_{z}^{(n)}(t;q)&=\sum_{z \in \mathcal{W}^n_{0}}\mathcal{P}^{(n)}_{z}(a;q)\Delta^{(n)}_{z}\langle \Pi(\cdot;t),\mathcal{P}^{(n)}_{z}(\cdot;q)\rangle_{\hat{\Delta}^{(n)}}\\
&= \Big\langle \Pi(\cdot;t), \sum_{z\in \mathcal{W}^n_{0}}\Delta_{z}^{(n)}\mathcal{P}^{(n)}_{z}(a;q)P^{(n)}_{z}(\cdot;q)\Big\rangle_{\hat{\Delta}^{(n)}}.
\end{split}
\end{equation*}

If Conjecture \ref{conjecture} holds, we may use the orthogonality result of Corollary \ref{orthogonality_rates} we have that
\[\sum_{z \in \mathcal{W}^n_{0}}\mathcal{P}^{(n)}_{z}(a;q)Q_{z}^{(n)}(t;q) = \Pi(a;t)\]
for every $a\in \mathbb{T}^n$.

We would like to prove the identity for $a\in \mathbb{R}^n_{>0}$. Let us consider the function
\[f(a)=f(a_{1},...,a_{n})=\sum_{z \in \mathcal{W}^n_{0}}\mathcal{P}^{(n)}_{z}(a;q)Q_{z}^{(n)}(t;q)\]
for $a\in \mathbb{C}^n\setminus \{0\}$.

On the $n$-dimensional torus $\mathbb{T}^n$, we have that $f(a)=\Pi(a;t)$. If the function $f$ is holomorphic on $\mathbb{C}^n\setminus \{0\}$, then using the identity theorem we may conclude that $f$ coincides with $\Pi(\cdot;t)$, everywhere on $\mathbb{C}^n\setminus \{0\}$.

We will only prove that $f$ is holomorphic on $\mathbb{C}^n\setminus \{0\}$ when $n=1$. Although we don't have a proof for the general case, we believe that the result holds for every $n\geq 1$. 

The function $\mathcal{P}^{(1)}_{(k)}(a;q)Q_{(k)}^{(1)}(t;q)$ is holomorphic on $\mathbb{C}^n\setminus \{0\}$. In order to prove that 
\[f(a)=\sum_{k=0}^{\infty}\mathcal{P}_{(k)}^{(1)}(a;q)Q_{(k)}^{(1)}(t;q)\]
is holomorphic for every $a\in \mathbb{C}$ away from the origin we need to show that for any compact region $A\subset \mathbb{C}\setminus  \{0\}$ and any $\epsilon > 0$ there exists $k_{0}\in \mathbb{N}$ such that 
\begin{equation}
\label{eq:control_k_large}
\sum_{k=k_{0}}^{\infty}|\mathcal{P}_{(k)}^{(1)}(a;q)Q_{(k)}^{(1)}(t;q)|\leq \epsilon
\end{equation}
for all $a\in A$.

Fix $d>1$ and assume that $A\equiv A_{d}=\{a\in \mathbb{C}: \frac{1}{d}\leq |a|\leq d\}$. Then for every $k\geq 0$, we have the following bound
\begin{equation}
\label{eq:polynomial_upper_bound(calculate_moments)}
|\mathcal{P}^{(1)}_{k}(a;q)|\leq \dfrac{\sqrt{2}+1}{\sqrt{2}}((\sqrt{2}+1)d)^k.
\end{equation}
Let us prove the bound by induction on $k$. When $k=0$, the result holds trivially since $\mathcal{P}^{(1)}_{(0)}(a;q)=1$. When $k=1$, we have
\[\mathcal{P}^{(1)}_{(1)}(a;q)=a+a^{-1}\]
therefore
\[|\mathcal{P}^{(1)}_{(1)}(a;q)|\leq 2d < \dfrac{\sqrt{2}+1}{\sqrt{2}}((\sqrt{2}+1)d). \]
We assume now that the result holds for $l\leq k$. Using the Pieri formula we proved in Corollary \ref{Pieri_mypolynomials(Berele)} we have that $\mathcal{P}^{(1)}_{(k)}(\cdot;q)$ satisfies the recurrence relation
\[(a+a^{-1})\mathcal{P}^{(1)}_{(k)}(a;q)= \mathcal{P}^{(1)}_{(k+1)}(a;q)+(1-q^k)\mathcal{P}^{(1)}_{(k-1)}(a;q)\]
then, using the induction  hypothesis, we have
\begin{equation*}
\begin{split}
|\mathcal{P}^{(1)}_{(k+1)}(a;q)|& = |a+a^{-1}||\mathcal{P}^{(1)}_{(k)}(a;q)| + (1-q^k)|\mathcal{P}^{(1)}_{(k-1)}(a;q)|\\
&\leq 2d \dfrac{\sqrt{2}+1}{\sqrt{2}}((\sqrt{2}+1)d)^k + (1-q^k)\dfrac{\sqrt{2}+1}{\sqrt{2}}((\sqrt{2}+1)d)^{k-1}\\
& \leq \dfrac{\sqrt{2}+1}{\sqrt{2}}((\sqrt{2}+1)d)^{k+1}
\end{split}
\end{equation*}
as required. 

Let us now calculate $Q_{(k)}(t;q)$ for large $k$. Ismail in \cite{Ismail_1986} obtain asymptotics for the Askey-Wilson polynomials, as $k \to \infty$. For the continuous $q$-Hermite polynomials, $\mathcal{P}^{(1)}_{k}(b;q)$, he proved that, as $k \to \infty$, the leading term of $\mathcal{P}^{(1)}_{(k)}(b;q)$, for $b\in \mathbb{T}$ and $q\in (0,1)$ fixed, is asymptotically equivalent to
\[\dfrac{(q;q)_{k}}{(q;q)_{\infty}}\Big[\dfrac{b^k}{(b^{-2};q)_{\infty}} + \dfrac{b^{-k}}{(b^{2};q)_{\infty}}\Big] = \dfrac{(q;q)_{k}}{(q;q)_{\infty}}\dfrac{b^k(b^2;q)_{\infty}+b^{-k}(b^{-2};q)_{\infty}}{(b^2,b^{-2};q)_{\infty}}.\]

We now need to calculate upper bounds for the modulus of the integrals
\[I = \frac{1}{2\pi i}\int_{\mathbb{T}}e^{(b+b^{-1})t}b^k(b^2;q)_{\infty}\frac{db}{b}\]
and
\[II = \frac{1}{2\pi i}\int_{\mathbb{T}}e^{(b+b^{-1})t}b^{-k}(b^{-2};q)_{\infty}\frac{db}{b}.\]
For $I$, let us make the change of variables $a=br$, for some $r>0$, then if we denote by $\mathbb{T}_{r}=\{z\in \mathbb{C}:|z|=r\}$, the torus of radius $r$, we can write $I$ as follows
\[I = r^{-k}\int_{\mathbb{T}_{r}}e^{(ar^{-1}+a^{-1}r)t}a^k(a^2r^{-2};q)_{\infty}\frac{da}{a}.\]
Note that, for every $r>0$, both $\mathbb{T}_{r}$ and $\mathbb{T}$ enclose the same poles of the integrand, therefore, using Cauchy's theorem we have
\[I = r^{-k}\int_{\mathbb{T}}e^{(ar^{-1}+a^{-1}r)t}a^k(a^2r^{-2};q)_{\infty}\frac{da}{a}.\]
After making the change of variable $a=e^{i\theta}$ we find that
\begin{equation*}
\begin{split}
|I|&\leq r^{-n}\frac{1}{2\pi}\int_{0}^{2\pi}e^{t(r+r^{-1})\cos \theta}|(e^{2i\theta}r^{-2};q)_{\infty}|d\theta\\
&\leq r^{-n}(-r^{-2};q)_{\infty}\frac{1}{2\pi}\int_{0}^{2\pi}e^{t(r+r^{-1})\cos \theta}d\theta\\\\
& =r^{-n}(-r^{-2};q)_{\infty}I_{0}(t(r+r^{-1}))
\end{split}
\end{equation*}
where $I_{0}(z)$ is the modified Bessel function of first kind given by
\[I_{0}(z) = \frac{1}{\pi}\int_{0}^{\pi}e^{z\cos \theta}d\theta.\]
Similarly, making the change of variable $a=b/r$, we can show that 
\[|II |\leq r^{-n}(-r^{-2};q)_{\infty}I_{0}(t(r+r^{-1})).\]
Therefore, for $k$ sufficiently large we may bound $|Q^{(1)}_{(k)}(t;q)|$ by $Cr^{-k}$, where $C$ depends on $r,t$ but not on $k$. Choosing $r=((\sqrt{2}+1)d)^2$ we may conclude that for $k$ sufficiently large and $a\in A_{d}$
\[|P^{(1)}_{(k)}(a;q)Q_{(k)}^{(1)}(t;q)|\leq C((\sqrt{2}+1)d)^{-k}\]
for some $C$ that depends only on $t,d$ which completes the proof of the inequality \eqref{eq:control_k_large}.

\end{proof}
\begin{proposition}
\label{law_shape(law)}
Let $Z=(Z(t),t\geq 0)$ be the process with transition rate matrix $\{Q^q_{n}(z,z'),z,z' \in \mathcal{W}^n_{0}\}$, started from the origin. Then, under Conjecture \ref{conjecture}, $Z$ has law at time $t\geq 0$ given by $p^{(n)}_{t}(\cdot;a,q)$.
\end{proposition}
\begin{proof}
We have to show two things; first that $p_{t}^{(n)}(z;a,q)$ solves the forward equation for the process $Z$ and second that at time $t=0$, the function concentrates at the zero configuration.\\
For the forward equation we want to show that
\begin{equation*}
\dfrac{\partial}{\partial t}p_{t}^{(n)}(z;a,q)=\sum_{z'\neq z}\Big(p_{t}^{(n)}(z';a,q)Q^q_{n}(z',z)-p_{t}^{(n)}(z;a,q)Q^q_{n}(z,z')\Big).
\end{equation*}
Now, we have
\begin{equation*}
\begin{split}
\sum_{z'\neq z}p_{t}^{(n)}(z';a,q)Q^q_{n}(z',z) &= \sum_{z'\neq z}\dfrac{\mathcal{P}^{(n)}_{z'}(a;q)Q^{(n)}_{z'}(t;q)}{\Pi(a;t)}\dfrac{\mathcal{P}^{(n)}_{z}(a;q)}{\mathcal{P}^{(n)}_{z'}(a;q)}f_{n}(z',z)\\
&=\dfrac{\mathcal{P}^{(n)}_{z}(a;q)}{\Pi(a;t)}\sum_{z'\neq z}Q^{(n)}_{z'}(t;q)f_{n}(z',z)\\
&= \dfrac{\mathcal{P}^{(n)}_{z}(a;q)}{\Pi(a;t)}\sum_{i=1}^n \Big(Q_{z-e_{i}}^{(n)}(t;q)(1-q^{z_{i-1}-z_{i}+1})\\
& \hspace{120pt}+Q_{z+e_{i}}^{(n)}(t;q)(1-q^{z_{i}-z_{i+1}+1})\Big).
\end{split}
\end{equation*}
We calculate
\begin{equation*}
\begin{split}
\Delta_{z+e_{i}}^{(n)}& = \dfrac{(q;q)^n_{\infty}}{(q;q)_{z_{n}}\Big(\prod_{\substack{1\leq j<n\\ j\neq i-1,i}}(q;q)_{z_{j}-z_{j+1}}\Big)(q;q)_{z_{i-1}-(z_{i}+1)}(q;q)_{(z_{i}+1)-z_{i+1}}}\\
& = \Delta_{z}^{(n)}\dfrac{(q;q)_{z_{i-1}-z_{i}}(q;q)_{z_{i}-z_{i+1}}}{(q;q)_{z_{i-1}-z_{i}-1}(q;q)_{z_{i}-z_{i+1}+1}}\\
&=\Delta_{z}^{(n)}\dfrac{1-q^{z_{i-1}-z_{i}}}{1-q^{z_{i}-z_{i+1}+1}}
\end{split}
\end{equation*}
therefore
\begin{equation*}
\begin{split}
Q^{(n)}_{z+e_{i}}(t;q)&=\Delta_{z+e_{i}}^{(n)}\langle \Pi(\cdot;t),\mathcal{P}_{z+e_{i}}^{(n)}(\cdot;q) \rangle_{\hat{\Delta}^{(n)}}\\
&=\Delta_{z}^{(n)}\dfrac{1-q^{z_{i-1}-z_{i}}}{1-q^{z_{i}-z_{i+1}+1}}\langle \Pi(\cdot;t),\mathcal{P}_{z+e_{i}}^{(n)}(\cdot;q) \rangle_{\hat{\Delta}^{(n)}}\\
&=\Delta_{z}^{(n)}\dfrac{1}{1-q^{z_{i}-z_{i+1}+1}}\langle \Pi(\cdot;t),f_{n}(z,z+e_{i})\mathcal{P}_{z+e_{i}}^{(n)}(\cdot;q) \rangle_{\hat{\Delta}^{(n)}}.
\end{split}
\end{equation*}
Similarly we have
\[\Delta_{z-e_{i}}^{(n)}=\Delta_{z}^{(n)}\dfrac{1-q^{z_{i}-z_{i+1}}}{1-q^{z_{i-1}-z_{i}+1}}\]
and
\[Q^{(n)}_{z-e_{i}}(t;q) = \Delta_{z}^{(n)}\dfrac{1}{1-q^{z_{i-1}-z_{i}+1}}\langle \Pi(\cdot;t),f_{n}(z,z-e_{i})\mathcal{P}_{z-e_{i}}^{(n)}(\cdot;q) \rangle_{\hat{\Delta}^{(n)}}.\]
Combining the above calculations we conclude that
\begin{equation*}
\begin{split}
\sum_{z'\neq z}p_{t}^{(n)}&(z';a,q)Q^q_{n}(z',z) \\
&= \dfrac{\mathcal{P}^{(n)}_{z}(a;q)}{\Pi(a;t)}\Delta_{z}^{(n)}\dfrac{1}{(2\pi i)^n n!}\int_{\mathbb{T}^n}\Pi(b;t)\sum_{i=1}^n \Big(f_{n}(z,z+e_{i})\overline{\mathcal{P}^{(n)}_{z+e_{i}}(b;q)}\\
& \hspace{160pt}+f_{n}(z,z-e_{i})\overline{\mathcal{P}^{(n)}_{z-e_{i}}(b;q)} \Big)\hat{\Delta}^{(n)}(b)\prod_{j=1}^n \dfrac{db_{j}}{b_{j}}\\
&=\dfrac{\mathcal{P}^{(n)}_{z}(a;q)}{\Pi(a;t)}\Delta_{z}^{(n)}\dfrac{1}{(2\pi i)^n n!}\int_{\mathbb{T}^n}\Pi(b;t)\sum_{i=1}^n (b_{i}+b_{i}^{-1})\mathcal{P}^{(n)}_{z}(b;q)\hat{\Delta}^{(n)}(b)\prod_{j=1}^n \dfrac{db_{j}}{b_{j}}\\
&=\dfrac{\mathcal{P}^{(n)}_{z}(a;q)}{\Pi(a;t)} \frac{\partial}{\partial t}Q^{(n)}_{z}(t;q)
\end{split}
\end{equation*}
where for the second equality we used the eigenrelation of Proposition \ref{q_eigenrelation_partition}. Therefore, it follows that
\begin{equation*}
\begin{split}
\dfrac{\partial}{\partial t}p_{t}^{(n)}(z;a,q)&=\sum_{z'\neq z}p_{t}^{(n)}(z';a,q)Q^q_{n}(z',z)-p_{t}^{(n)}(z;a,q)\sum_{i=1}^n(a_{i}+a_{i}^{-1})\\
&=\sum_{z'\neq z}\Big(p_{t}^{(n)}(z';a,q)Q^q_{n}(z',z)-p_{t}^{(n)}(z;a,q)Q^q_{n}(z,z')\Big)
\end{split}
\end{equation*}
where for the second equality we used again Proposition \ref{q_eigenrelation_partition}.\\
Let us finally calculate the behaviour of the law at time $t=0$. By its definition it holds that $\Pi(\cdot;0)\equiv 1$ therefore
\[p_{0}^{(n)}(z;a,q)=P_{z}^{(n)}(a;q)\Delta_{z}^{(n)}\langle 1, \mathcal{P}^{(n)}_{z}(\cdot;q)\rangle_{\hat{\Delta}^{(n)}}.\]
Moreover, we observe that $P^{(n)}_{\mathbf{0}}(\cdot;q)=1$, since the set of symplectic Gelfand-Tsetlin patterns whose shape is the zero partition, which we denote by $\mathbf{0}$, contains only the pattern with all coordinates equal to $0$. Therefore
\begin{equation*}
\begin{split}
p_{0}^{(n)}(z;a,q)&=\mathcal{P}_{z}^{(n)}(a;q)\Delta_{z}^{(n)}\langle  \mathcal{P}^{(n)}_{\mathbf{0}}(\cdot;q), \mathcal{P}^{(n)}_{z}(\cdot;q)\rangle_{\hat{\Delta}^{(n)}}\\
&=\mathcal{P}_{z}^{(n)}(a;q)\mathbbm{1}_{z = \mathbf{0}}
\end{split}
\end{equation*}
where the last equality follows from \eqref{eq:hyperoctahedral_orthogonality}. Hence we conclude that 
\[p_{0}^{(n)}(z;a,q)=\mathbbm{1}_{z =\mathbf{0}}.\]
\end{proof}

Let $\mathbf{Z}=(\mathbf{Z}(t),t\geq 0)$ be the process on the set of integer-valued symplectic Gelfand Tsetlin patterns, $\mathbb{K}^N_{0}$, with $N=2n$, that evolves according to $q$-Berele dynamics from chapter \ref{Berele} or the fully randomised version of chapter \ref{discrete} and starts at the origin. Proposition \ref{law_shape(law)} together with Theorems \ref{q_main_theorem(Berele)} and \ref{q_main_theorem} imply the following for the law of the process on the pattern.
\begin{corollary}
\label{law_pattern(law)}
Suppose that $\mathbf{Z}$ starts at the origin. Then under Conjecture \ref{conjecture}, for $t\geq 0$
\[\mathbb{P}(\mathbf{Z}(t)=\mathbf{z})=\dfrac{w^N_{a,q}(\mathbf{z})Q_{z^N}(t;q)}{\Pi(a;t)}\]
where $w^N_{a,q}(\mathbf{z})$ is the kernel of the function $\mathcal{P}$ defined in \ref{eq:q_weights(Berele)}.
\end{corollary}
\section{Calculating moments}
For a function $f:\mathcal{W}^n_{0}\to \mathbb{R}$, let $\langle f(Z)\rangle_{n,t}^{a,q}:=\sum_{z \in \mathcal{W}^n_{0}}f(z)p_{t}^{(n)}(z;a,q)$ denote the expectation of $f(Z)$ under $p_{t}^{(n)}(\cdot;a,q)$.
\begin{proposition}
\label{q_moments}
For $k \in \mathbb{N}$, assuming Conjecture \ref{conjecture} holds,it holds that
\[\langle q^{-kZ_{1}}\rangle_{n,t}^{a,q}=\sum_{j=0}^k\dbinom{k}{j}\dfrac{D^j_{n}\Pi(a;t)}{\Pi(a;t)}\]
where $D_{n}^j$ denotes the convolution of the Koornwinder operator $D_{n}$, we defined in \eqref{eq:Koornwinder_operator}, with itself $j$ times and $\Pi(a;t)$ is defined in \eqref{eq:norm_constant}.
\end{proposition}
\begin{proof}
By Proposition \ref{q_eigenrelation_rates} we have that for every $j\geq 1$, it holds that
\[D_{n}^j\mathcal{P}_{z}^{(n)}(a;q)=(q^{-z_{1}}-1)^j\mathcal{P}_{z}^{(n)}(a;q)\]
therefore we have the following
\begin{equation*}
\begin{split}
\langle (q^{-Z_{1}}-1)^j\rangle_{n,t}^{a,q}&= \sum_{z \in \mathcal{W}^n_{0}}(q^{-z_{1}}-1)^j\dfrac{\mathcal{P}_{z}^{(n)}(a;q)Q_{z}^{(n)}(t;q)}{\Pi(a;t)}\\
&= \dfrac{1}{\Pi(a;t)}\sum_{z \in \mathcal{W}^n_{0}}D_{n}^j\mathcal{P}^{(n)}_{z}(a;q)Q_{z}^{(n)}(t;q)\\
&= \dfrac{1}{\Pi(a;t)}D_{n}^j\sum_{z \in \mathcal{W}^n_{0}}\mathcal{P}^{(n)}_{z}(a;q)Q_{z}^{(n)}(t;q)\\
& = \dfrac{D_{n}^j\Pi(a;t)}{\Pi(a;t)}.
\end{split}
\end{equation*}
We then find
\begin{equation*}
\begin{split}
\langle q^{-kZ_{1}}\rangle_{n,t}^{a,q}&=\langle \sum_{j=0}^k \dbinom{k}{j}(q^{-Z_{1}}-1)^j\rangle_{n,t}^{a,q}\\
&=\sum_{j=0}^k \dbinom{k}{j}\langle (q^{-Z_{1}}-1)^j\rangle_{n,t}^{a,q}\\
&= \sum_{j=0}^k\dbinom{k}{j}\dfrac{D^j_{n}\Pi(a;t)}{\Pi(a;t)}.
\end{split}
\end{equation*}
\end{proof}

The action of the operators $D_{n}$ on some special symmetric functions can be represented via contour integrals.
\begin{proposition}
For $F(u_{1},...,u_{n})=f(u_{1})...f(u_{n})$ such that $f(u)=f(u^{-1})$ 
\begin{equation*}
\begin{split}
\dfrac{((D_{n})^kF)(a)}{F(a)}=\dfrac{(-1)^k}{(2\pi i)^k}\oint ... \oint &\prod_{j=1}^k \dfrac{1}{1-qs^2_{j}}\Big(\prod_{i=1}^n \dfrac{1}{(s_{j}-a_{i})(s_{j}-a_{i}^{-1})}\Big)\\
& \times  \Big(\prod_{l=1}^{j-1}\dfrac{(s_{l}-s_{j})(s_{l}-s_{j}^{-1})}{(s_{l}-qs_{j})(s_{l}-q^{-1}s_{j}^{-1})}\dfrac{f(qs_{j})}{f(s_{j})}-1 \Big)\dfrac{ds_{j}}{s_{j}}
\end{split}
\label{eq:q-contour}
\end{equation*}
where the $s_{j}$-contour encircles the poles $q^{-1/2}$, $a_{1}^{\pm 1},...,a_{n}^{\pm 1}$, $(qs_{j+1})^{\pm 1},...,(qs_{k})^{\pm 1}$ and no other singularities of the integrand.
\end{proposition}
\begin{proof}
For $k=1$ using the residue theorem we have that the right-hand side of \eqref{eq:q-contour} equals
\[\sum_{p}Res\Big(\dfrac{-s^{-1}}{1-qs^2}\prod_{i=j}^n \dfrac{1}{(s-a_{j})(s-a_{j}^{-1})}\Big(\dfrac{f(qs)}{f(s)}-1\Big);p\Big)\]
where the summation is over $\{a_{1}^{\pm 1},...,a_{n}^{\pm 1},q^{-1/2}\}$, the set of singularities of the integrand.

The residue corresponding to the pole $p=a_{i}$, for some $1\leq i \leq n$ is given by
\begin{equation*}
\begin{split}
\dfrac{-s^{-1}}{1-qs^2}\dfrac{1}{s-a_{i}^{-1}}\prod_{j\neq i} &\dfrac{1}{(s-a_{j})(s-a_{j}^{-1})}\Big(\dfrac{f(qs)}{f(s)}-1\Big)\Bigg|_{s=a_{i}}\\
&=\dfrac{1}{(1-a_{i}^{2})(1-qa_{i}^{2})}\prod_{j\neq i} \dfrac{1}{(a_{i}-a_{j})(a_{i}-a_{j}^{-1})}\Big(\dfrac{f(qa_{i})}{f(a_{i})}-1\Big)\\
&=A_{i}(a;q)\Big(\dfrac{f(qa_{i})}{f(a_{i})}-1\Big)
\end{split}
\end{equation*}
where $A_{i}(a;q)$ is given in \eqref{eq:Koornwinder_operator_A}.\\
Similarly, the residue for $p=a_{i}^{-1}$, for some $1\leq i \leq n$, is calculated to be
\[A_{i}(a^{-1};q)\Big(\dfrac{f(qa_{i}^{-1})}{f(a_{i}^{-1})}-1\Big).\]
Finally, for the residue corresponding to $p=q^{-1/2}$ we have
\begin{equation*}
\begin{split}
\dfrac{-s^{-1}}{-2qs}\prod_{j=1}^n &\dfrac{1}{(s-a_{j})(s-a_{j}^{-1})}\Big(\dfrac{f(qs)}{f(s)}-1\Big)\Bigg|_{s=q^{-1/2}}\\
&=\dfrac{1}{2}\prod_{j=1}^n \dfrac{1}{(q^{-1/2}-a_{j})(q^{-1/2}-a_{j}^{-1})}\Big(\dfrac{f(q^{1/2})}{f(q^{-1/2})}-1\Big)\\
&=0
\end{split}
\end{equation*}
due to symmetry of the function $f$, $f(u^{-1})=f(u)$.

Therefore the right-hand side of \eqref{eq:q-contour} equals
\[\sum_{i=1}^n\Big( A_{i}(a;q)\big(\dfrac{f(qa_{i})}{f(a_{i})}-1\big)+A_{i}(a^{-1};q)\big(\dfrac{f(qa^{-1}_{i})}{f(a^{-1}_{i})}-1\big)\Big)\]
which equals the left-hand side due to the symmetry of the function $f$.

Assuming that the hypothesis is true for $k\geq 1$, we have that
\begin{equation*}
\begin{split}
\dfrac{((D_{n})^{k+1}F)(a)}{F(a)}&=\dfrac{1}{F(a)}D_{n}\Big[ F(a) \dfrac{(-1)^k}{(2\pi i)^k}\oint ... \oint \prod_{j=1}^k \dfrac{1}{1-qs^2_{j}}\Big(\prod_{i=1}^n \dfrac{1}{(s_{j}-a_{i})(s_{j}-a_{i}^{-1})}\Big)\\
& \hspace{50pt}\times  \Big(\prod_{l=1}^{j-1}\dfrac{(s_{l}-s_{j})(s_{l}-s_{j}^{-1})}{(s_{l}-qs_{j})(s_{l}-q^{-1}s_{j}^{-1})}\dfrac{f(qs_{j})}{f(s_{j})}-1 \Big)\dfrac{ds_{j}}{s_{j}}\Big]\\
&=\dfrac{1}{F(a)}D_{n}\Big[ \dfrac{(-1)^k}{(2\pi i)^k}\oint ... \oint \prod_{i=1}^n \Big(f(a_{i})\prod_{j=1}^k \dfrac{1}{(s_{j}-a_{i})(s_{j}-a_{i}^{-1})} \Big)\\
& \hspace{50pt}\times \prod_{j=1}^k \dfrac{1}{1-qs_{j}^2}\Big(\prod_{l=1}^{j-1}\dfrac{(s_{l}-s_{j})(s_{l}-s_{j}^{-1})}{(s_{l}-qs_{j})(s_{l}-q^{-1}s_{j}^{-1})}\dfrac{f(qs_{j})}{f(s_{j})}-1 \Big)\dfrac{ds_{j}}{s_{j}}\Big]\\
&=\dfrac{1}{F(a)} \dfrac{(-1)^k}{(2\pi i)^k}\oint ... \oint D_{n}\Big[\prod_{i=1}^n \Big(f(a_{i})\prod_{j=1}^k \dfrac{1}{(s_{j}-a_{i})(s_{j}-a_{i}^{-1})} \Big)\Big]\\
& \hspace{50pt}\times \prod_{j=1}^k \dfrac{1}{1-qs_{j}^2}\Big(\prod_{l=1}^{j-1}\dfrac{(s_{l}-s_{j})(s_{l}-s_{j}^{-1})}{(s_{l}-qs_{j})(s_{l}-q^{-1}s_{j}^{-1})}\dfrac{f(qs_{j})}{f(s_{j})}-1 \Big)\dfrac{ds_{j}}{s_{j}}
\end{split}
\end{equation*}
applying $D_{n}$ to the function 
\[G(a_{1},...,a_{n})=\prod_{i=1}^ng(a_{i})=\prod_{i=1}^n f(a_{i})\prod_{j=1}^k\dfrac{1}{(s_{j}-a_{i})(s_{j}-a_{i}^{-1})}\]
we find that
\begin{equation*}
\begin{split}
\dfrac{1}{F(a)}(D_{n}G)(a)&=\dfrac{1}{F(a)}\sum_{i=1}^n\Big( A_{i}(a;q)\big(\mathcal{T}_{q,i}G(a)-G(a)\big)+A_{i}(a^{-1};q)\big(T_{q^{-1},i}G(a)-G(a) \big)\Big)\\
&=\dfrac{G(a)}{F(a)}\sum_{i=1}^n \Big( A_{i}(a;q)\Big( \dfrac{g(qa_{i})}{g(a_{i})}-1\Big)+A_{i}(a^{-1};q)\Big( \dfrac{g(q^{-1}a_{i})}{g(a_{i})}-1\Big)\Big)\\
&=\prod_{j=1}^k\dfrac{1}{(s_{j}-a_{i})(s_{j}-a_{i}^{-1})}\sum_{i=1}^n \Big( A_{i}(a;q)\Big( \dfrac{g(qa_{i})}{g(a_{i})}-1\Big)+A_{i}(a^{-1};q)\Big( \dfrac{g(q^{-1}a_{i})}{g(a_{i})}-1\Big)\Big)
\end{split}
\end{equation*}
and hence we conclude that
\[\dfrac{((D_{n})^{k+1}F)(a)}{F(a)} = \sum_{i=1}^n (\mathcal{I}^+_{i}+\mathcal{I}^-_{i})\]
where
\begin{equation*}
\begin{split}
\mathcal{I}^+_{i}=A_{i}(a;q)\dfrac{(-1)^k}{(2\pi i)^k}&\oint ... \oint \prod_{j=1}^k \dfrac{1}{1-qs^2_{j}}\Big(\prod_{i=1}^n \dfrac{1}{(s_{j}-a_{i})(s_{j}-a_{i}^{-1})}\Big)\\
& \times  \Big(\prod_{l=1}^{j-1}\dfrac{(s_{l}-s_{j})(s_{l}-s_{j}^{-1})}{(s_{l}-qs_{j})(s_{l}-q^{-1}s_{j}^{-1})}\dfrac{f(qs_{j})}{f(s_{j})}-1 \Big)\Big(\dfrac{g(qa_{i})}{g(a_{i})}-1\Big)\dfrac{ds_{j}}{s_{j}}
\end{split}
\end{equation*}
and the $s_{j}$-contour contains the poles $q^{-1/2}, a_{1}^{\pm 1},...,a_{n}^{\pm 1}, (qs_{j+1})^{\pm 1},..., (qs_{k})^{\pm 1}$ as before and in addition the poles $(qa_{i})^{\pm 1}$, for $1\leq i \leq n$ due to the term $\dfrac{g(qa_{i})}{g(a_{i})}$. Similarly, the integrals $\mathcal{I}^-_{i}$ is given by
\begin{equation*}
\begin{split}
\mathcal{I}^-_{i}=A_{i}(a^{-1};q)\dfrac{(-1)^k}{(2\pi i)^k}&\oint ... \oint \prod_{j=1}^k \dfrac{1}{1-qs^2_{j}}\Big(\prod_{i=1}^n \dfrac{1}{(s_{j}-a_{i})(s_{j}-a_{i}^{-1})}\Big)\\
& \times  \Big(\prod_{l=1}^{j-1}\dfrac{(s_{l}-s_{j})(s_{l}-s_{j}^{-1})}{(s_{l}-qs_{j})(s_{l}-q^{-1}s_{j}^{-1})}\dfrac{f(qs_{j})}{f(s_{j})}-1 \Big)\Big(\dfrac{g(qa_{i}^{-1})}{g(a_{i}^{-1})}-1\Big)\dfrac{ds_{j}}{s_{j}}
\end{split}
\end{equation*}
with the $s_{j}$-contour containing the poles $q^{-1/2}, a_{1}^{\pm 1},...,a_{n}^{\pm 1}, (qs_{j+1})^{\pm 1},..., (qs_{k})^{\pm 1}$ and $(qa_{i}^{-1})^{\pm 1}$, for $1\leq i \leq n$ due to the term $\dfrac{g(qa_{i}^{-1})}{g(a_{i}^{-1})}$.

On the other hand let us consider the following multiple integral
\begin{equation*}
\begin{split}
\dfrac{(-1)^{k+1}}{(2\pi i)^{k+1}}\oint ... \oint \prod_{j=1}^{k+1}& \dfrac{1}{1-qs^2_{j}}\Big(\prod_{i=1}^n \dfrac{1}{(s_{j}-a_{i})(s_{j}-a_{i}^{-1})}\Big) \\
&\hspace{50pt}\times\Big(\prod_{l=1}^{j-1}\dfrac{(s_{l}-s_{j})(s_{l}-s_{j}^{-1})}{(s_{l}-qs_{j})(s_{l}-q^{-1}s_{j}^{-1})}\dfrac{f(qs_{j})}{f(s_{j})}-1 \Big)\dfrac{ds_{j}}{s_{j}}
\end{split}
\end{equation*}
if we apply the Residue Theorem to the $s_{k+1}$-contour we will also get $\sum_{i=1}^n (\mathcal{I}^+_{i}+\mathcal{I}^-_{i})$. The calculations are similar to the $k=1$ case. Therefore the hypothesis holds for $k+1$ and the proof concludes.
\end{proof}

\begin{corollary}
\label{q_moments_contour}
For $k \in \mathbb{N}$, assuming Conjecture \ref{conjecture} holds,it holds that
\begin{equation}
\begin{split}
\langle q^{-kZ_{1}}\rangle_{n,t}^{a,q}=\sum_{l=0}^k&\dbinom{k}{l}\dfrac{(-1)^l}{(2\pi i)^l}\oint ... \oint \prod_{j=1}^l \dfrac{1}{1-qs^2_{j}}\Big(\prod_{i=1}^n \dfrac{1}{(s_{j}-a_{i})(s_{j}-a_{i}^{-1})}\Big)\\
& \times  \Big(\prod_{i=1}^{j-1}\dfrac{(s_{i}-s_{j})(s_{i}-s_{j}^{-1})}{(s_{i}-qs_{j})(s_{i}-q^{-1}s_{j}^{-1})}e^{((q-1)s_{j}+(q^{-1}-1)s_{j}^{-1})t}-1 \Big)\dfrac{ds_{j}}{s_{j}}.
\end{split}
\end{equation}
where for the $l$-th term the $s_{j}$-contour encloses the poles $q^{-1/2}$, $a_{i}^{\pm 1}$, for $1\leq i \leq n$ and $(qs_{i})^{\pm 1}$, for $j+1\leq i \leq l$.
\end{corollary}
\begin{proof}
The result follows directly from Proposition \ref{q_moments} using equation \eqref{eq:q-contour} for the function $F(u_{1},...,u_{n})=\prod_{j=1}^n\exp \big((u_{j}+u_{j}^{-1})t\big)$.
\end{proof}

%Later, in chapter 7 we will study a diffusion that appear as a scaling limit of the $q$-deformed $\mathfrak{so}_{2n+1}$-Whittaker process. The formula of the moments via contour integrals will allow us to calculate the moment generating function for the first coordinate of the diffusion.

%% file: tex/whittaker_processes.tex
In chapter \ref{Koornwinder} we defined \ref{recursion_q_whittaker} a family of polynomials parametrized by partitions and conjectured \ref{conjecture} that these polynomials coincide with the $q$-deformed $\mathfrak{so}_{2n+1}$-Whittaker functions. In this chapter we will prove that the classical $\mathfrak{so}_{2n+1}$-Whittaker functions are degenerations of these polynomials.

In section 7.1 we recall the definition of the $\mathfrak{so}_{2n+1}$-Whittaker functions. In section 7.2 we propose a scaling of the polynomials $\mathcal{P}^{(n)}$ and prove the convergence to the $\mathfrak{so}_{2n+1}$-Whittaker functions.

\section{The $\mathfrak{so}_{2n+1}$-Whittaker functions}
The quantum Toda lattice is a quantum integrable system associated with a Lie algebra $\mathfrak{g}$ and has Hamiltonian given by
\begin{equation}
\mathcal{H}^{\mathfrak{g}}=\frac{1}{2}\sum_{i=1}^n\frac{\partial^2}{\partial x_{i}^2}-\sum_{i=1}^nd_{i}e^{-\langle\alpha_{i}, x \rangle}
\label{eq:Hamiltonian_general}
\end{equation}
where $(\alpha_{i}, i=1,...,n)$ are the simple roots corresponding to the algebra $\mathfrak{g}$, $(d_{i},i=1,...,n)$ are appropriate constants and $\langle \cdot , \cdot \rangle$ denotes the dot product.

Our main interest focuses on the $B_{n}$ type root system which corresponds to the Lie algebra $\mathfrak{so}_{2n+1}$. The simple roots are
\[\alpha_{i}=e_{i}-e_{i+1}, \quad 1\leq i \leq n-1, \quad \alpha_{n}=e_{n}\]
where $(e_{1},...,e_{n})$ denotes the orthonormal basis in $\mathbb{R}^n$. The coroots, defined for a root $\alpha$ by $\alpha^{\vee}= \frac{2}{\langle \alpha, \alpha\rangle}\alpha$, are the following
\[\alpha_{i}^{\vee}=e_{i}-e_{i+1}, \quad 1\leq i \leq n-1, \quad \alpha_{n}^{\vee}=2e_{n}.\]
The coefficients, $d_{i}, i=1,...,n$, that appear in the Hamiltonian \eqref{eq:Hamiltonian_general} are chosen such that the matrix $(d_{i}\langle \alpha_{i}^{\vee}, \alpha_{j}\rangle)_{i,j=1}^n$ is symmetric. Therefore, the Hamiltonian corresponding to $\mathfrak{g}=\mathfrak{so}_{2n+1}$ takes the form
\begin{equation}
\mathcal{H}^{\mathfrak{so}_{2n+1}}=\frac{1}{2}\sum_{i=1}^n\frac{\partial^2}{\partial x_{i}^2}-\sum_{i=1}^{n-1}e^{x_{i+1}-x_{i}}-\frac{1}{2}e^{-x_{n}}.
\label{eq:Hamiltonian_so}
\end{equation}

The $\mathfrak{so}_{2n+1}$-Whittaker functions are eigenfunctions of the operator $\mathcal{H}^{\mathfrak{so}_{2n+1}}$. We will defined them via their Givental-type integral representation proved in \cite{Gerasimov_et_al_2008, Gerasimov_et_al_2012}.
\begin{definition}[\cite{Gerasimov_et_al_2008, Gerasimov_et_al_2012}]
\label{so-whittaker-definition}
For $n\geq 1$, $\lambda= (\lambda_{1},...,\lambda_{n}) \in \mathbb{C}^n$ and $x \in \mathbb{R}^n$, the $\mathfrak{so}_{2n+1}$-Whittaker function admits an integral representation as follows
\[\Psi_{\lambda}^{\mathfrak{so}_{2n+1}}(x)=\int_{\Gamma[x]}\prod_{k=1}^{2n-1}\prod_{i=1}^{[\frac{k+1}{2}]}dx^k_{i}e^{\mathcal{F}_{\lambda}(\mathbf{x})},\]
where the exponent $\mathcal{F}_{\lambda}(\mathbf{x})$ is given by
\begin{equation*}
\begin{split}
\mathcal{F}_{\lambda}(\mathbf{x})=\sum_{k=1}^n \lambda_{k}(2|x^{2k-1}|&-|x^{2k}|-|x^{2k-2}|)-\sum_{k=1}^n \Big(\sum_{i=1}^{k-1}(e^{x^{2k-1}_{i}-x^{2k}_{i}}+e^{x^{2k}_{i+1}-x^{2k-1}_{i}})\\
+&(e^{-x^{2k-1}_{k}}+e^{x^{2k-1}_{k}-x^{2k}_{k}})+\sum_{i=1}^{k-1}(e^{x^{2k-1}_{i+1}-x^{2k-2}_{i}}+e^{x^{2k-2}_{i}-x^{2k-1}_{i}})\Big)
\end{split}
\end{equation*}
and $\Gamma[x]$ is a small deformation of the subspace $\{\mathbf{x}=(x^1,...,x^{2n}):x^k \in \mathbb{R}^{[\frac{k+1}{2}]} \text{ with }x^{2n}=x\}$ such that the integral converges.
\end{definition}
We observe that $\Psi^{\mathfrak{so}_{2n+1}}(x)$ exhibits a recursive structure. For $\lambda_{n}\in \mathbb{C}$, $x^{2n}\in \mathbb{R}^n$ and $x^{2n-2}\in \mathbb{R}^{n-1}$ we define the kernel
\begin{equation}
\begin{split}
Q^{n-1,n}_{\lambda_{n}}(x^{2n-2},x^{2n})=\int_{\mathbb{R}^n} &\exp \Big(\lambda_{n}\big(2|x^{2n-1}|-|x^{2n}|-|x^{2n-2}|\big)\\
&-\sum_{i=1}^{n-1}\big(e^{x^{2n-1}_{i}-x^{2n}_{i}}+e^{x^{2n}_{i+1}-x^{2n-1}_{i}}\big)-\big(e^{-x^{2n-1}_{n}}+e^{x^{2n-1}_{n}-x^{2n}_{n}}\big)\\
&-\sum_{i=1}^{n-1}\big(e^{x^{2n-1}_{i+1}-x^{2n-2}_{i}}+e^{x^{2n-2}_{i}-x^{2n-1}_{i}}\big) \Big)\prod_{i=1}^ndx^{2n-1}_{i}.
\end{split}
\label{eq:so_kernel}
\end{equation}
Then
\begin{equation}
\Psi_{\lambda}^{\mathfrak{so}_{2n+1}}(x)=\int_{\mathbb{R}^{n-1}}Q^{n-1,n}_{\lambda_{n}}(x^{2n-2},x^{2n})\Psi_{\tilde{\lambda}}^{\mathfrak{so}_{2n-1}}(x^{2n-2})\prod_{j=1}^{n-1}dx^{2n-2}_{j}
\label{eq:whittaker-recursion}
\end{equation}
with $x\equiv x^{2n}$ and $\tilde{\lambda}=(\lambda_{1},...,\lambda_{n-1})$.\\

\noindent When $n=1$, the $\mathfrak{so}_{2n+1}$-Whittaker function is related with the Macdonald function as follows. For $\lambda \in \mathbb{C}$ and $x \in \mathbb{R}$
\[\Psi^{\mathfrak{so}_{2\cdot 1+1}}_{\lambda}(x)=\int_{\mathbb{R}}\exp \Big(\lambda(2x^1_{1}-x^2_{1})-(e^{-x^1_{1}}+e^{x^1_{1}-x^2_{1}})\Big)dx^1_{1},\, x^2_{1}=x\]
with the change of variable $t = e^{x^1_{1}-x^2_{1}/2}$ leading to the following representation
\[\Psi^{\mathfrak{so}_{2\cdot 1+1}}_{\lambda}(x)=\int_{0}^{\infty}t^{2\lambda - 1}\exp \Big(-e^{-x^2_{1}/2}\big(t+\frac{1}{t}\big)\Big) dt = 2K_{2\lambda}(2e^{-x/2}).\]

As we mentioned before the $\mathfrak{so}_{2n+1}$-Whittaker function are eigenfunction of the Toda operator associated with the $\mathfrak{so}_{2n+1}$-Lie algebra. More specifically, we have the following result.

\begin{theorem}[\cite{Gerasimov_et_al_2008, Gerasimov_et_al_2012}] For $\lambda = (\lambda_{1},...,\lambda_{n})\in \mathbb{C}^n$, $\Psi^{\mathfrak{so}_{2n+1}}_{\lambda}(x)$ solves the differential equation
\begin{equation}
\label{eq:eigenrelation_so}
\mathcal{H}^{\mathfrak{so}_{2n+1}}\Psi_{\lambda}^{\mathfrak{so}_{2n+1}}(x)=\frac{1}{2}\langle \lambda, \lambda\rangle \Psi_{\lambda}^{\mathfrak{so}_{2n+1}}(x).
\end{equation}
\end{theorem}

Baudoin and O'Connell give the following characterization for the $\mathfrak{so}_{2n+1}$-Whittaker function for real-valued parameter in the type-$B$ Weyl chamber.
\begin{proposition}[\cite{Baudoin-O'Connell_2011}, cor. 2.3, prop. 4.1]Let $\lambda \in \mathbb{R}^n$ such that $\lambda_{1}>...>\lambda_{n}>0$. The function $\Psi^{\mathfrak{so}_{2n+1}}_{\lambda}$ is the unique solution to the differential equation \eqref{eq:eigenrelation_so} such that $\Psi_{\lambda}^{\mathfrak{so}_{2n+1}}(x)e^{-\langle \lambda, x\rangle}$ is bounded with growth condition 
\[\Psi_{\lambda}^{\mathfrak{so}_{2n+1}}(x)e^{-\langle \lambda, x\rangle} \xrightarrow[\text{$1\leq i \leq n$}]{\text{$\alpha_{i}(x)\to \infty$}}\prod_{i=1}^n \Gamma(\langle \alpha_{i}^{\vee}, \lambda\rangle)\] 
where $\alpha_{i}, \alpha_{i}^{\vee}$ for $1\leq i \leq n$ are the roots and coroots for the $B_{n}$ root system. 
\label{Whittaker_characterization}
\end{proposition}
\section{Whittaker functions as degenerations of $q$-Whittaker functions}
In this section we will show that the $q$-deformed $\mathfrak{so}_{2n+1}$-Whittaker functions converge to the $\mathfrak{so}_{2n+1}$-Whittaker functions when we consider certain rescaling. In order to do that we will use the recursive structure we conjectured in \ref{recursion_q_whittaker} for the $q$-deformed $\mathfrak{so}_{2n+1}$-Whittaker functions.
\begin{definition}
\label{scaling}
We introduce the following scaling:
\[z^{k}_{i}=\epsilon^{-1} x^{k}_{i}+(k-2(i-1))m(\epsilon),\quad 1\leq i \leq \big[\frac{k+1}{2}\big],\, 1\leq k \leq 2n\]
\[q=e^{-\epsilon},\quad a_{l}=e^{i \epsilon \lambda_{l}},1\leq l \leq n\]
\[m(\epsilon)=-[\epsilon^{-1}\log \epsilon],\quad \mathcal{A}(\epsilon)=-\frac{\pi^2}{6\epsilon}-\frac{1}{2}\log \frac{\epsilon}{2\pi}\]
\[ f_{\alpha}(y,\epsilon)=(q;q)_{[\epsilon^{-1}y]+\alpha m(\epsilon)}, \quad \alpha\in \mathbb{Z}_{\geq 1}.\]
\end{definition}
\noindent We also define the rescaling of the $q$-deformed $\mathfrak{so}_{2n+1}$-Whittaker function as
\[\Psi_{i\lambda,\epsilon}^{(n)}(x)=\epsilon^{n^2}e^{n^2\mathcal{A}(\epsilon)}\mathcal{P}_{z}^{(n)}(a;q).\] 
with $x = x^{2n}$ and $z = z^{2n}$. 

Although the coordinates of the $z$-vector are ordered such that $z_{1}\geq ... \geq z_{n}\geq 0$, the same is not true for the coordinates of the $x$-vector. We need for a given vector $x\in \mathbb{R}^n$ to be able to keep track of the indices that are out of order.
\begin{definition}
For $x\in \mathbb{R}^n$ define the set
\[\sigma(x) = \{1\leq i \leq n: x_{i}\leq x_{i+1}\}\]
with the convention $x_{n+1}=0$.
\end{definition}
\begin{theorem}
\label{convergence}
For all $n\geq 1$ and $\lambda \in \mathbb{R}^n$, the following hold.
\begin{enumerate}[1)]
\item For each $\sigma \subset \{1,...,n\}$, there exists a polynomial $R_{n,\sigma}$ of $n$ variables (chosen independently of $\lambda$ and $\epsilon$) such that for all $x$ with $\sigma(x)=\sigma$ we have the following estimate: for some $c^*>0$
\[|\Psi_{i\lambda, \epsilon}^{(n)}(x)|\leq R_{n,\sigma(x)}(x)\prod_{i\in \sigma(x)}\exp\{-c^*e^{-(x_{i}-x_{i+1})/2}\}\]
with $x_{n+1}=0$.
\item For $x$ varying in a compact domain, $\Psi_{i\lambda,\epsilon}^{(n)}(x)$ converges, as $\epsilon$ goes to zero, uniformly to $\Psi_{i \lambda}^{\mathfrak{so}_{2n+1}}(x)$.  
\end{enumerate}
\end{theorem}

Before we proceed to the proof of the theorem we need to investigate the behaviour of $f_{\alpha}(y,\epsilon)$, as $\epsilon$ goes to zero.

\begin{proposition}[\cite{Gerasimov_Lebedev_Oblezin_2012}, lem. 3.1]
For $\epsilon \to 0+$, the following expansions hold
\[f_{1}(y,\epsilon)=e^{\mathcal{A}(\epsilon)+e^{-y}+O(\epsilon)};\]
\[f_{\alpha > 1}(y,\epsilon)=e^{\mathcal{A}(\epsilon)+O(\epsilon^{\alpha - 1})}.\]
\end{proposition}

We moreover have the following bounds.
\begin{proposition}[\cite{Borodin_Corwin_2011}, prop. 4.1.9] 
\label{pochhammer_bounds}
Assume $\alpha\geq 1$, then for all $y$ such that $\epsilon^{-1}y_{\alpha, \epsilon}$, with $y_{\alpha,\epsilon}=y+\epsilon \alpha m_{\epsilon}$, is a non-negative integer
\[\log f_{\alpha}(y,\epsilon)-\mathcal{A}(\epsilon)\geq -c+\epsilon^{-1}e^{-y_{\alpha, \epsilon}}\]
where $c=c(\epsilon)>0$ is independent of all the other variables besides $\epsilon$ and tends to zero with $\epsilon$. On the other hand for all $b^{*}<1$ 
\[\log f_{\alpha}(y,\epsilon)-\mathcal{A}(\epsilon)\leq c'+(b^*)^{-1}\epsilon^{-1}\sum_{r=1}^{\infty}\dfrac{e^{-r(\epsilon+y_{\alpha, \epsilon})}}{r^2}\]
where $c' = c'(y_{\alpha, \epsilon})<C$ for $C<\infty$ fixed and independent of all the variables and where $c'(y_{\alpha, \epsilon})$ can be taken to zero as $y_{\alpha, \epsilon}$ increases.
\end{proposition}

\begin{corollary}[\cite{Borodin_Corwin_2011}, corol. 4.1.10]
\label{pochhammer_asymptotics}
Assume $\alpha\geq 1$, then for any $M>0$ and any $\delta >0$, there exists $\epsilon_{0}>0$ such that for all $\epsilon < \epsilon_{0}$ and all $y\geq -M$ such that $\epsilon ^{-1}y_{\alpha,\epsilon}$ is a nonnegative integer,
\[\log f_{\alpha}(y,\epsilon)-\mathcal{A}(\epsilon)-\epsilon ^{\alpha -1}e^{-y} \in [-\delta, \delta].\]
\end{corollary}
We also need upper bounds for some special integrals.
\begin{lemma}[\cite{Borodin_Corwin_2011},lemma 7.3.5]
\label{double_exp_1}
For all $m\geq 0$ there exists constant $c = c(m)\in (0, \infty)$ such that for $a\leq 0$
\[\int_{-\infty}^0|x|^m e^{-e^{-(a+x)}}dx\leq c(m)e^{-e^{-a}}.\]
\end{lemma}

\begin{corollary}
\label{double_exp_2}
For all $m\geq 0$, $c^* \in (0,1]$ and $a,b\in \mathbb{R}$ there exists polynomial of two variables $R_{m,s}$, that depends on $m$ and $s = sign(a-b)$ such that
\[\int_{-\infty}^{\infty}x^m\exp \{-c^*(e^{-(a-x)}+e^{-(x - b)})\}dx \leq R_{m,s}(a,b)\exp\{-c^*e^{-(a-b)/2}\mathbbm{1}_{a\leq b}\}.\]
\end{corollary}
\begin{proof}
The proof we will present here follows the step of the proof of Proposition 4.1.3 of \cite{Borodin_Corwin_2011}. Let us start by proving the case $c^* = 1$. We will consider separately the two cases $a>b$ and $a\leq b$.\\
\textbf{CASE I:} Let $a>b$, then we may split the integral as follows
\[\Big(\int_{-\infty}^b + \int_{b}^a +\int_{a}^{+\infty}\Big)x^m \exp\Big(-e^{-(a-x)}-e^{-(x-b)} \Big)dx.\]
Since $e^{-(a-x)}\geq 0$, we have for the first integral
\begin{equation*}
\begin{split}
\int_{-\infty}^b x^m \exp\Big(-e^{-(a-x)}-e^{-(x-b)} \Big)dx&\leq \int_{-\infty}^b x^m \exp \Big(-e^{-(x-b)} \Big)dx\\
&=\int_{-\infty}^0(z+b)^m \exp \big(-e^{-z} \big)dz\\
&\leq \sum_{r=0}^m \dbinom{m}{r}|b|^{m-r}\int_{-\infty}^0|z|^r\exp \big( -e^{-z}\big)dz
\end{split}
\end{equation*}
using Lemma \ref{double_exp_1}, the last expression is bounded by 
\[\sum_{r=0}^m \dbinom{m}{r}c(r)|b|^{m-r}\]
where $c(r)$ is some constants depending only on $r$.\\
We will bound the second integral by
\[\int_{b}^a x^mdx = \dfrac{a^{m+1}-b^{m+1}}{m+1}.\]
Finally for the third integral we calculate
\begin{equation*}
\begin{split}
\int_a^{+\infty}x^m \exp \Big(-e^{-(a-x)}-e^{-(x-b)} \Big)& \leq \int_{a}^{+\infty} x^m \exp \big(-e^{-(a-x)} \big)dx\\
&=\int_{-\infty}^0(a-z)^m \exp \big(-e^{-z} \big)dz
\end{split}
\end{equation*}
and using the same argument that we used for the first integral, we may bound the third integral by
\[\sum_{r=0}^m \dbinom{m}{r}c(r)|a|^{m-r}.\]
Therefore, when $a>b$ the integral is bounded by a polynomial in $a,b$, as required.\\
\textbf{CASE II:} If $a\leq b$ we will split the integral as follows
\[\Big(\int_{-\infty}^{\frac{a+b}{2}}+\int_{\frac{a+b}{2}}^{+\infty} \Big)x^m \exp \Big(-e^{-(a-x)}-e^{-(x-b)} \Big)dx.\]
We will then bound the first integral by
\begin{equation*}
\begin{split}
\int_{-\infty}^{\frac{a+b}{2}}x^m \exp \big(-e^{-(x-b)} \big)dx&= \int_{-\infty}^0\big(z+\dfrac{a+b}{2}\big)^m\exp \big(-e^{-(z+\frac{a-b}{2})} \big)dz\\
& \leq \sum_{r=0}^m \dbinom{m}{r}\big|\dfrac{a+b}{2}\big|^{m-r}\int_{-\infty}^0|z|^r\exp \big(-e^{-(z+\frac{a-b}{2})} \big)dz
\end{split}
\end{equation*}
and using Lemma \eqref{double_exp_1}, we conclude that the first integral is bounded by
\[\sum_{r=0}^m \dbinom{m}{r}\big(\dfrac{|a|+|b|}{2}\big)^{m-r}c(r)\exp \Big(-e^{-\frac{a-b}{2}} \Big).\]
The second integral is also bounded by the same quantity, therefore when $a \leq b$ we conclude that
\[\int_{-\infty}^{\infty}x^m \exp \big(-e^{-(a-x)}-e^{-(x-b)} \big)dx\leq 2\sum_{r=0}^m \dbinom{m}{r}\big(\dfrac{|a|+|b|}{2}\big)^{m-r}c(r)\exp \Big(-e^{-\frac{a-b}{2}} \Big).\]
When $c^*<1$, we apply the result with $A=a-\log c^*$ and $B=b+\log c^*$.
\end{proof}
We are now ready to prove the convergence of the $q$-deformed $\mathfrak{so}_{2n+1}$-Whittaker function.

\begin{proof}(of Theorem \ref{convergence})
We will prove the result by induction on the order $n$ taking advantage of the recursive structure of the function $\mathcal{P}_{z}^{(n)}(a;q)$ we presented in \ref{recursion_q_whittaker}. 

\noindent \textbf{STEP 1 - Base case:} For $n=1$, $z^2_{1}\in \mathbb{Z}_{\geq 0}$ and $a_{1}\in \mathbb{R}_{>0}$, using the scaling from Definition \ref{scaling}, we have
\[a_{1}^{2z_{1}^1-z^2_{1}}=e^{i\lambda_{1}(2x^1_{1}-x^2_{1})}\]
and
\[\dbinom{z^2_{1}}{z^2_{1}-z^1_{1}}_{q}=\dfrac{f_{2}(x^2_{1}, \epsilon)}{f_{1}(x^1_{1}, \epsilon)f_{1}(x^2_{1}-x^1_1, \epsilon)}.\]
The scaled function then takes the form 
\begin{equation}
\Psi_{i\lambda_{1},\epsilon}^{(1)}(x^2_{1})=\epsilon\sum_{x^1_{1}\in R_{1}^{\epsilon}(x^2_{1})} e^{i\lambda_{1}(2x^1_{1}-x^2_{1})}\dfrac{f_{2}(x^2_{1},\epsilon)e^{-\mathcal{A}(\epsilon)}}{f_{1}(x^1_{1},\epsilon)e^{-\mathcal{A}(\epsilon)}f_{1}(x^2_{1}-x^1_{1},\epsilon)e^{-\mathcal{A}(\epsilon)}}
\label{eq:scaled_n=1}
\end{equation}
where $R_{1}^{\epsilon}(x^2_{1})=\{x^1_{1}\in \epsilon \mathbb{Z}: -\epsilon m(\epsilon) \leq x^1_{1}\leq x^2_{1}+\epsilon m(\epsilon)\}$.

\noindent \textbf{Part 1:} Let us first prove an upper bound for the scaled function $\Psi_{i\lambda_{1},\epsilon}^{(1)}(x^2_{1})$. Using Proposition \ref{pochhammer_bounds} we have that for some $c=c(\epsilon)>0$ that depends only on $\epsilon$
\[|(f_{1}(x^1_{1},\epsilon)e^{-\mathcal{A}(\epsilon)}f_{1}(x^2_{1}-x^1_{1},\epsilon)e^{-\mathcal{A}(\epsilon)})^{-1}|\leq e^c \exp \big(-e^{-x^1_{1}}-e^{x^1_{1}-x^2_{1}}\big)\]
and for any $0<b^*<1$ and $C>0$, both independent of all the variables, 
\[|f_{2}(x^2_{1},\epsilon)e^{-\mathcal{A}(\epsilon)}|\leq \exp \Big(C + (b^*)^{-1}\epsilon^{-1}\sum_{r=1}^{\infty}\dfrac{e^{-r(x^2_{1}+2\epsilon m(\epsilon))}}{r^2} \Big).\]
Let us calculate an upper bound for the series. We have that $z^2_{1}\in \mathbb{Z}_{\geq 0}$ which implies that $x^2_{1}+2\epsilon m(\epsilon)\geq 0$. Therefore $a=e^{-(x^2_{1}+2\epsilon m(\epsilon)\geq 0)}\in (0,1]$ and hence the series converges to the dilogarithm of $a$
\[Li_{2}(a)=\sum_{r=1}^{\infty}\dfrac{a^r}{r^2}.\]
We observe that the function $Li_{2}(a)/a$ is increasing in $(0,1]$, therefore
\[\dfrac{Li_{2}(a)}{a}\leq Li_{2}(1)=\dfrac{\pi^2}{6}.\]
Moreover if $0<a,b\leq 1$, the bound $2ab\leq a+b$ holds, therefore
\[Li_{2}(ab)\leq \dfrac{\pi^2}{12}(a+b).\]
Since $x^1_{1}\in R^{\epsilon}_{1}(x^2_{1})$ it follows that both $x^2_{1}-x^1_{1}+\epsilon m(\epsilon)$ and $x^1_{1}+\epsilon m(\epsilon)$ are non-negative, therefore
\[Li_{2}(e^{-(x^2_{1}-x^1_{1}+\epsilon m(\epsilon))}e^{-(x^1_{1}+\epsilon m(\epsilon))})\leq \dfrac{\pi^2}{12}(e^{-(x^2_{1}-x^1_{1}+\epsilon m(\epsilon))} + e^{-(x^1_{1}+\epsilon m(\epsilon))}).\]
From the last inequality we conclude that
\begin{equation}
\Big| \dfrac{f_{2}(x^2_{1},\epsilon)e^{-\mathcal{A}(\epsilon)}}{f_{1}(x^1_{1},\epsilon)e^{-\mathcal{A}(\epsilon)}f_{1}(x^2_{1}-x^1_{1},\epsilon)e^{-\mathcal{A}(\epsilon)}} \Big|\leq C\exp\{-c^*(e^{-(x^2_{1}-x^1_{1})}+e^{-x^1_{1}})\}
\label{eq:bound_for_qbinom}
\end{equation}
where $C$ is independent of all the variables and $c^* = 1-\frac{\pi^2}{12}(b^*)^{-1}\in (0,1]$, provided we choose $\frac{\pi^2}{12}\leq b^*<1$. Therefore
\[|\Psi^{(1)}_{i\lambda_{1}, \epsilon}(x^2_{1})|\leq \epsilon \sum_{x^1_{1}\in R^{\epsilon}_{1}(x^2_{1})}C\exp\{-c^*(e^{-(x^2_{1}-x^1_{1})}+e^{-x^1_{1}})\}\]
which is bounded above by the integral it converges
\[C\int_{-\infty}^{\infty}\exp\{-c^*(e^{-(x^2_{1}-x^1_{1})}+e^{-x^1_{1}})\}dx^1_{1}.\]
The required bound then follows using Corollary \ref{double_exp_2} with $a=x^2_{1}$ and $b=0$.\\

\noindent \textbf{Part 2:} We assume that $x^2_{1}$ varies in a compact set $D_{2}=[a_{2},b_{2}]\subset \mathbb{R}$. Let us consider a second compact subset $D_{1}=[a_{1},b_{1}]$ of $\mathbb{R}$ where $x^1_{1}$ will vary. The scaled function $\Psi_{i\lambda_{1},\epsilon}^{(1)}(x^2_{1})$, given in \eqref{eq:scaled_n=1}, can be split into three parts as follows
\begin{equation*}
\begin{split}
\Psi_{i\lambda_{1},\epsilon}^{(1)}(x^2_{1})=\Big(\epsilon\sum_{\substack{x^1_{1}\in \epsilon \mathbb{Z}:\\-\epsilon m(\epsilon)\leq x^1_{1}\leq a_{1}}}\,+\epsilon &\sum_{\substack{x^1_{1}\in \epsilon \mathbb{Z}:\\a_{1}\leq x^1_{1}\leq b_{1}}}\,+\epsilon\sum_{\substack{x^1_{1}\in \epsilon \mathbb{Z}:\\b_{1}\leq x^1_{1}\leq x^2_{1}+\epsilon m(\epsilon)}}\Big)\\
& e^{i\lambda_{1}(2x^1_{1}-x^2_{1})}\dfrac{f_{2}(x^2_{1},\epsilon)e^{-\mathcal{A}(\epsilon)}}{f_{1}(x^1_{1},\epsilon)e^{-\mathcal{A}(\epsilon)}f_{1}(x^2_{1}-x^1_{1},\epsilon)e^{-\mathcal{A}(\epsilon)}}.
\end{split}
\end{equation*}
Using Corollary \ref{pochhammer_asymptotics}, we have that in $D_{1}$
\[e^{i\lambda_{1}(2x^1_{1}-x^2_{1})}\dfrac{f_{2}(x^2_{1},\epsilon)e^{-\mathcal{A}(\epsilon)}}{f_{1}(x^1_{1},\epsilon)e^{-\mathcal{A}(\epsilon)}f_{1}(x^2_{1}-x^1_{1},\epsilon)e^{-\mathcal{A}(\epsilon)}}\]
converges uniformly to
\[e^{i\lambda_{1}(2x^1_{1}-x^2_{1})}\exp \Big(-\big(e^{x^1_{1}-x^2_{1}}+e^{-x^1_{1}}\big) \Big)\]
and the Riemann sum converges to the corresponding integral
\[\int_{D_{1}}e^{i\lambda_{1}(2x^1_{1}-x^2_{1})}\exp \Big(-\big(e^{x^1_{1}-x^2_{1}}+e^{-x^1_{1}}\big) \Big)dx^1_{1}\]
uniformly for $x^2_{1}\in D_{2}$.

Let us now calculate upper bounds for the other two sums. Using the bound in equation \eqref{eq:bound_for_qbinom} the first sum can be bounded as follows
\begin{equation*}
\begin{split}
\epsilon \sum_{\substack{x^1_{1}\in \epsilon \mathbb{Z}:\\-\epsilon m(\epsilon)\leq x^1_{1}\leq a_{1}}}C \exp\Big( -c^*(e^{-x^1_{1}}+e^{-(x^2_{1}-x^1_{1})})\Big)&\leq C\int_{-\infty}^{a_{1}}\exp\Big( -c^*(e^{-x^1_{1}}+e^{-(x^2_{1}-x^1_{1})})\Big)\\
&\leq C\int_{-\infty}^{a_{1}}\exp \Big(-e^{-(x^1_{1}-\log c^*)} \Big)dx^1_{1}\\
&=C\int_{-\infty}^0 \exp \Big( -e^{-(x^1_{1}+(a_{1}-\log c^*))}\Big)dx^1_{1}
\end{split}
\end{equation*}
using lemma \ref{double_exp_1}, we may bound the last quantity by $Ce^{-c^*e^{-a_{1}}}$ uniformly in $x^2_{1}\in D_{2}$, provided $a_{1}\leq \log c^*$. Similarly, the third sum can be bounded by $C e^{-c^*e^{-(b_{2}-b_{1})}}$ uniformly for $x^2_{1}\in D_{2}$, provided $b_{1}$ is chosen such that it satisfies $b_{1}\geq b_{2}-\log c^*$.

Let $\eta>0$, choose $D_{1}=[a_{1},b_{1}]$ such that the contribution of the extra two sums is bounded above by $\eta / 3$, i.e. choose $a_{1},b_{1}$ such that
\[Ce^{-c^*(e^{-a_{1}}+e^{-(b_{2}-b_{1})})}\leq \frac{\eta}{3}.\]
Then, using again Lemma \ref{double_exp_1} we can see that
\[\int_{\mathbb{R}\setminus D_{1}}e^{i\lambda_{1}(2x^1_{1}-x^2_{1})}\exp \Big(-(e^{-(x^2_{1}-x^1_{1})}+e^{-x^1_{1}}) \Big)dx^1_{1}\]
is also bounded by $\eta /3$. Therefore we conclude that there exists $\epsilon_{0}>0$ such that for every $\epsilon \leq \epsilon_{0}$
\[|\Psi_{i\lambda_{1},\epsilon}^{(1)}(x^2_{1})-\Psi_{i\lambda_{1}}^{\mathfrak{so}_{3}}(x^2_{1})|\leq \eta\]
uniformly in $x^2_{1}\in D_{2}$.

\noindent \textbf{STEP 2 - Induction step:} Let us now proceed to the proof of the bound for $n>1$. We assume that the result holds for $n-1$. Let $(a_{1},...,a_{n})\in \mathbb{R}^n_{>0}$ and $z^{2n}\in \Lambda_{n}=\{z\in \mathbb{Z}^n_{\geq 0}:z_{i}\geq z_{i+1}, 1\leq i <n\}$. We have conjectured in ... that the $q$-deformed $\mathfrak{so}_{2n+1}$-Whittaker function satisfies the recursion
\[\mathcal{P}_{z^{2n}}^{(n)}(a_{1},...,a_{n};q)=\sum_{z^{2n-2} \in \Lambda_{n-1}}Q_{a_{n},q}^{n-1, n}(z^{2n-2}, z^{2n})\mathcal{P}_{z^{2n-2}}^{(n-1)}(a_{1},...,a_{n-1};q)\]
with
\begin{equation*}
\begin{split}
Q_{a_{n},q}^{n-1,n}(z^{2n-2}, z^{2n})=\sum_{\substack{z^{2n-1} \in \Lambda_{n}: \\ z^{2n-2} \prec z^{2n-1} \prec z^{2n}}}&a_{n}^{2|z^{2n-1}|-|z^{2n}|-|z^{2n-2}|} \\
\prod_{i=1}^{n-1}&\dbinom{z^{2n-1}_{i}-z^{2n-1}_{i+1}}{z^{2n-1}_{i}-z^{2n-2}_{i}}_{q}\dbinom{z^{2n}_{i}-z^{2n}_{i+1}}{z^{2n}_{i}-z^{2n-1}_{i}}_{q}\dbinom{z^{2n}_{n}}{z^{2n}_{n}-z^{2n-1}_{n}}_{q}.
\end{split}
\end{equation*}
Using the scaling from Definition \ref{scaling} we have the following
\[a_{n}^{2|z^{2n-1}|-|z^{2n}|-|z^{2n-2}|}=e^{i\lambda_{n}(2|x^{2n-1}|-|x^{2n}|-|x^{2n-2}|)}\]
for $1\leq i <n$
\begin{equation*}
\begin{split}
\dbinom{z^{2n-1}_{i}-z^{2n-1}_{i+1}}{z^{2n-1}_{i}-z^{2n-2}_{i}}_{q}&=\dfrac{f_{2}(x^{2n-1}_{i}-x^{2n-1}_{i+1},\epsilon)}{f_{1}(x^{2n-1}_{i}-x^{2n-2}_{i},\epsilon)f_{1}(x^{2n-2}_{i}-x^{2n-1}_{i+1},\epsilon)}\\
\dbinom{z^{2n}_{i}-z^{2n}_{i+1}}{z^{2n}_{i}-z^{2n-1}_{i}}_{q}&=\dfrac{f_{2}(x^{2n}_{i}-x^{2n}_{i+1},\epsilon)}{f_{1}(x^{2n}_{i}-x^{2n-1}_{i},\epsilon)f_{1}(x^{2n-1}_{i}-x^{2n}_{i+1},\epsilon)}
\end{split}
\end{equation*}
and
\[\dbinom{z^{2n}_{n}}{z^{2n}_{n}-z^{2n-1}_{n}}_{q}=\dfrac{f_{2}(x^{2n}_{n},\epsilon)}{f_{1}(x^{2n-1}_{n},\epsilon)f_{1}(x^{2n}_{n}-x^{2n-1}_{n},\epsilon)}.\]
Therefore the scaled function takes the following form
\begin{equation}
\begin{split}
\Psi_{i\lambda, \epsilon}^{(n)}(x^{2n})& =\epsilon^n \sum_{x^{2n-1}\in R_{1}^{\epsilon}(x^{2n})} \Big(e^{i\lambda_{n}(|x^{2n-1}|-|x^{2n}|)}\Lambda^{\epsilon}(x^{2n-1},x^{2n})\\
& \hspace{20 pt} \times \epsilon^{n-1} \sum_{x^{2n-2}\in R_{2}^{\epsilon}(x^{2n-1})}\big(e^{i\lambda_{n}(|x^{2n-1}|-|x^{2n}|)}\Lambda^{\epsilon}(x^{2n-2},x^{2n-1})\Psi_{i\tilde{\lambda}, \epsilon}^{(n-1)}(x^{2n-2})\big)\Big)
\end{split}
\label{eq:scaled_n>1}
\end{equation}
where $R^{\epsilon}_{1}(x^{2n})=\{x^{2n-1}\in \epsilon \mathbb{Z}^n: x^{2n}_{i+1}-\epsilon m(\epsilon)\leq x^{2n-1}_{i}\leq x^{2n}_{i}+\epsilon m(\epsilon), \, 1\leq i \leq n\}$, with $x^{2n}_{n+1}=0$, and $R^{\epsilon}_{2}(x^{2n-1})=\{x^{2n-2}\in \epsilon \mathbb{Z}^{n-1}: x^{2n-1}_{i+1}-\epsilon m(\epsilon)\leq x^{2n-2}_{i}\leq x^{2n-1}_{i}+\epsilon m(\epsilon), \, 1\leq i \leq n-1\}$ and
\begin{equation*}
\begin{split}
\Lambda^{\epsilon}(x^{2n-1},x^{2n})&=\prod_{i=1}^{n-1}\dfrac{f_{2}(x^{2n}_{i}-x^{2n}_{i+1},\epsilon)e^{-\mathcal{A}(\epsilon)}}{f_{1}(x^{2n}_{i}-x^{2n-1}_{i},\epsilon)e^{-\mathcal{A}(\epsilon)}f_{1}(x^{2n-1}_{i}-x^{2n}_{i+1},\epsilon)e^{-\mathcal{A}(\epsilon)}}\\
& \hspace{100pt}\times \dfrac{f_{2}(x^{2n}_{n},\epsilon)e^{-\mathcal{A}(\epsilon)}}{f_{1}(x^{2n-1}_{n},\epsilon)e^{-\mathcal{A}(\epsilon)}f_{1}(x^{2n}_{n}-x^{2n-1}_{n},\epsilon)e^{-\mathcal{A}(\epsilon)}}\\
\Lambda^{\epsilon}(x^{2n-2},x^{2n-1})&=\prod_{i=1}^{n-1}\dfrac{f_{2}(x^{2n-1}_{i}-x^{2n-1}_{i+1},\epsilon)e^{-\mathcal{A}(\epsilon)}}{f_{1}(x^{2n-1}_{i}-x^{2n-2}_{i},\epsilon)e^{-\mathcal{A}(\epsilon)}f_{1}(x^{2n-2}_{i}-x^{2n-1}_{i+1},\epsilon)e^{-\mathcal{A}(\epsilon)}}.
\end{split}
\end{equation*}

\noindent \textbf{PART 1:} Let us start by proving the bound for the scaled function. Using the same argument we used for proving inequality \eqref{eq:bound_for_qbinom}, we have the following bounds

\begin{equation*}
\begin{split}
|\Lambda^{\epsilon}(x^{2n-1},x^{2n})|&\leq C \prod_{i=1}^{n-1}\exp\{-c^*(e^{-(x^{2n}_{i}-x^{2n-1}_{i})}+e^{-(x^{2n-1}_{i}-x^{2n}_{i+1})})\}\\
& \hspace{100pt}\times \exp\{-c^*(e^{-(x^{2n}_{n}-x^{2n-1}_{n})}+e^{-x^{2n-1}_{n}})\}\\
|\Lambda^{\epsilon}(x^{2n-2},x^{2n-1})|&\leq C \prod_{i=1}^{n-1}\exp\{-c^*(e^{-(x^{2n-1}_{i}-x^{2n-2}_{i})}+e^{-(x^{2n-2}_{i}-x^{2n-1}_{i+1})})\}
\end{split}
\end{equation*}
for some $C>0$ fixed and independent of all the variables and $c^* \in (0,1]$.

By the inductive step, $|\Psi_{i\tilde{\lambda}, \epsilon}^{(n-1)}(x^{2n-2})|$ is essentially bounded by a polynomial, $R_{n-1,\sigma(x^{2n-2})}(x^{2n-2})$, therefore the inner sum is bounded by the integral it converges
\[\int_{\mathbb{R}^{n-1}}C \prod_{i=1}^{n-1}\exp\{-c^*(e^{-(x^{2n-1}_{i}-x^{2n-2}_{i})}+e^{-(x^{2n-2}_{i}-x^{2n-1}_{i+1})})\}R_{n-1,\sigma(x^{2n-2})}(x^{2n-2})\prod_{i=1}^{n-1}dx^{2n-2}_{i}\]
which, using Corollary \ref{double_exp_2}, is bounded by a polynomial of $n$ variables in $x^{2n-1}$, which we denote by $R_{n,\sigma(x^{2n-1})}(x^{2n-1})$, multiplied by a product of double exponentials, which are bounded by 1. Moving to the outer sum we then have
\begin{equation*}
\begin{split}
|\Psi_{i\lambda, \epsilon}^{(n)}(x^{2n})| &\leq  \int_{\mathbb{R}^n}C \prod_{i=1}^{n-1}\exp\{-c^*(e^{-(x^{2n}_{i}-x^{2n-1}_{i})}+e^{-(x^{2n-1}_{i}-x^{2n}_{i+1})})\}\\
& \hspace{50pt}\times \exp\{-c^*(e^{-(x^{2n}_{n}-x^{2n-1}_{n})}+e^{-x^{2n-1}_{n}})\}R_{n, \sigma(x^{2n-1})}(x^{2n-1}) \prod_{i=1}^n dx^{2n-1}_{i} 
\end{split}
\end{equation*}
and again applying Corollary \ref{double_exp_2} we conclude that
\[|\Psi_{i\lambda, \epsilon}^{(n)}(x^{2n})| \leq R_{n,\sigma(x^{2n})}(x^{2n})\prod_{i\in \sigma(x^{2n})}\exp\{-c^*e^{-(x^{2n}_{i}-x^{2n}_{i+1})/2}\}\]
as required.

\noindent \textbf{PART 2:} Let us now fix a compact subset $D_{2n}$ of $\mathbb{R}^n$ in which $x^{2n}$ may vary. For $\eta >0$ we may find a compact subset $D_{2n-1}$ of $\mathbb{R}^n$ in which $x^{2n-1}$ may vary such that outside of that subset the contribution to the expression \eqref{eq:scaled_n>1}
is bounded by $\eta / 3$. Similarly, the integral for the Whittaker function in the same domain is also bounded by $\eta/3$. Using the same reasoning as for the $n=1$ case along with the induction hypothesis that $\Psi^{\mathfrak{so}_{2n-1}}_{i\tilde{\lambda}}(x^{2n-2})$, we may prove that uniformly in $x^{2n-1}\in D_{2n-1}$, the inner sum in the expression \eqref{eq:scaled_n>1} converges to 
\[\int_{\mathbb{R}^{n-1}}e^{i\lambda_{n}(|x^{2n-1}|-|x^{2n-2}|)}\prod_{i=1}^{n-1}\exp\{-(e^{-(x^{2n}_{i}-x^{2n-1}_{i})}+e^{-(x^{2n-1}_{i}-x^{2n}_{i+1})})\}\Psi^{\mathfrak{so}_{2n-1}}_{i\tilde{\lambda}}(x^{2n-2})\prod_{i=1}^{n-1}dx^{2n-2}_{i}.\]
We conclude the proof using the same strategy for the outer sum.
\end{proof}

%% file: tex/continuous.tex
In the introduction we presented a model on the real-valued symplectic Gelfand-Tsetlin cone proposed in\cite{Borodin_et_al_2009, Sasamoto_2011} where if started from the origin, the even-indexed levels evolve as Dyson's Brownian motion of type $B$ and the odd-indexed levels as Dyson's Brownian motion of type $D$.

In this chapter we will construct a Markov process on two dimensional arrays that can be considered as a positive temperature analogue of the Brownian motion on the symplectic Gelfand-Tsetlin cone.

We fix $n \in \mathbb{N}$ and assume that $N=2n$ or $2n-1$. Let $\lambda = (\lambda_{1},...,\lambda_{n})$ be a real-valued vector such that $\lambda_{1}>...>\lambda_{n}>0$. Let $B^k_{i},1\leq i \leq \big[\frac{k+1}{2}\big],1\leq k \leq N$ be a collection of independent one dimensional standard Brownian motions where we assume that, for $1\leq i \leq k$, the Brownian motions $B^{2k}_{i}$, $B^{2k-1}_{i}$ have drifts $-\lambda_{k}$ and $\lambda_{k}$ respectively. Note that for convenience we drop the dependence of the drift from the symbol of the Brownian motion. Let us finally consider the Markov process $\mathbf{X}=(\mathbf{X}(t),t \geq 0)$, with state space $\Gamma_{N} := \{\mathbf{x}=(x^1,...,x^N): x^{2k-1},x^{2k}\in \mathbb{R}^k\}$ that evolves according to the following rules.

$X^1_{1}$ is an autonomous Markov process satisfying $dX^1_{1}=dB^1_{1} + e^{-X^1_{1}}dt$. The particles in even-indexed levels, i.e. if $k=2l$, evolve according to 
\begin{equation}
\left\{ \begin{array}{ll}
 dX^k_{1}=dB^k_{1}+e^{X^{k-1}_{1}-X^k_{1}}dt & \\
 dX^k_{m}=dB^k_{m} + (e^{X^{k-1}_{m}-X^k_{m}} - e^{X^k_{m}-X^{k-1}_{m-1}})dt, & 2\leq m \leq l
\end{array} \right.
\label{eq:SDEs_even}
\end{equation}
whereas the particles in odd-indexed level, i.e. if $k=2l-1$, evolve according to the SDEs
\begin{equation}
\left\{ \begin{array}{ll}
 dX^k_{1}=dB^k_{1}+e^{X^{k-1}_{1}-X^k_{1}}dt & \\
 dX^k_{m}=dB^k_{m} + (e^{X^{k-1}_{m}-X^k_{m}} - e^{X^k_{m}-X^{k-1}_{m-1}})dt, & 2\leq m \leq l-1\\
dX^k_{l}=dB^k_{l} + (e^{-X^k_{l}} - e^{X^k_{l}-X^{k-1}_{l-1}})dt &
\end{array} \right. .
\label{eq:SDEs_odd}
\end{equation}
So, a particle with index $(k,m)$ evolves as a standard Brownian motion with a drift that depends on its distance from the particles indexed by $(k-1,m)$ and $(k-1,m-1)$ or by the distance from $0$ if $k$ is odd and $m=\frac{k+1}{2}$. We remark that the particle-particle and the particle-wall interactions we propose for our model are similar to the interactions among the particles for the symmetric version of the process studied in \cite{O'Connell_2012}.

Our objective is to prove that under appropriate initial conditions, the even-indexed levels evolve as a positive temperature analogue of the Dyson's Brownian motion of type $B$ and the odd-indexed levels evolve as a positive temperature analogue of the Dyson's Brownian motion of type $D$. 

This chapter is organised as follows. In section 8.1  propose a generalisation of the $\mathfrak{so}_{2n+1}$-Whittaker functions we defined in \ref{so-whittaker-definition}. In section 8.2 we present the main result, which states that under certain initial conditions $X^N$ evolves as a diffusion on $\mathbb{R}^{[\frac{N+1}{2}]}$. In 8.3 we will attempt to connect the first coordinate of the $\mathfrak{gl}_{N}$-Whittaker process with the process $X^N$ in a similar manner as in \cite{Borodin_et_al_2009}, where it is proved that the maximum of the first coordinate of the $N$-dimensional type $A$ Dyson Brownian motion is given by the first coordinate of the type $B$ or $D$ Dyson Brownian motion if $N$ is even or odd, respectively.

\section{An extension of the $\mathfrak{so}_{2n+1}$-Whittaker function} 
For $n\geq 1$ we consider the differential operator $\mathcal{H}^B_{n}$ acting on twice differentiable function as follows
\begin{equation}
\mathcal{H}_{n}^B = \dfrac{1}{2}\sum_{i=1}^n\dfrac{\partial^2}{\partial x_{i}^2}-\sum_{i=1}^{n-1}e^{x_{i+1}-x_{i}}- e^{-x_{n}}.
\label{eq:operator_B}
\end{equation}
We remark that the operator $\mathcal{H}^B_{n}$ is related with the operator of the Toda lattice associated with the $\mathfrak{so}_{2n+1}$ Lie algebra, $\mathcal{H}^{\mathfrak{so}_{2n+1}}$ given in \eqref{eq:Hamiltonian_so} with the transformation $x_{i}\mapsto x_{i} + \log 2$.

\begin{definition}
Let $N=2n$ or $N=2n-1$ and let $\lambda = (\lambda_{1},...,\lambda_{n})$ be a $\mathbb{C}$-valued vector. We define the function $\Phi^{(N)}_{\lambda}:\mathbb{R}^n \to \mathbb{C}$ as follows
\begin{equation}
\Phi_{\lambda}^{(N)}(x):=\int_{\Gamma_{N}[x]} e^{\mathcal{F}_{\lambda}^{(N)}(\mathbf{x})}\prod_{k=1}^{N-1}\prod_{i=1}^{\big[\frac{k+1}{2}\big]}dx^k_{i}
\label{eq:sym_whittaker}
\end{equation}
where the integration is over the set  $\Gamma_{N}[x]=\{\mathbf{x}\in \Gamma_{N}: x^{N}=x\}$ and the exponent $\mathcal{F}_{\lambda}^{(N)}(\mathbf{x})$ is given by
\begin{equation*}
\begin{split}
\mathcal{F}_{\lambda}^{(N)}(\mathbf{x}) &= \sum_{k=1}^N \bar{\lambda}_{k}\big( |x^k|-|x^{k-1}|\big)\\
&-\sum_{k=1}^N \Big(\sum_{i=1}^{[\frac{k-1}{2}]}(e^{x^{k-1}_{i}-x^k_{i}}+e^{x^k_{i+1}-x^{k-1}_{i}})+(e^{x^{k-1}_{k/2}-x^{k}_{k/2}}+2e^{-x^{k-1}_{k/2}})\mathbbm{1}_{\{k \text{ even}\}}\Big)
\end{split}
\end{equation*}
where $\bar{\lambda}=(\bar{\lambda}_{1},...,\bar{\lambda}_{N})\in \mathbb{C}^N$ with $\bar{\lambda}_{2i-1}=\lambda_{i}$ and $\bar{\lambda}_{2i}=-\lambda_{i}$.
\end{definition}

\noindent The exponent $\mathcal{F}_{0}^{(N)}(\mathbf{x})$ is given schematically in figure \ref{weight_sketch} below. 
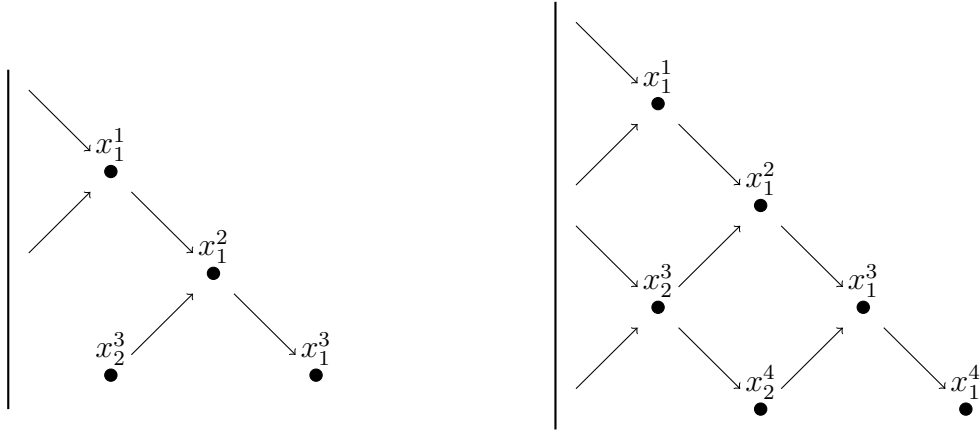
\begin{figure}[h]
\begin{center}
\begin{tikzpicture}[scale =0.9]
% right
\draw [thick] (8,1.2) -- (8,7.5) ;
\draw [fill] (9.5,6) circle [radius=0.09];
\node [above, black] at (9.5,6) {\large{$x_{1}^1$}};
\draw [fill] (11,4.5) circle [radius=0.09];
\node [above, black] at (11,4.5) {\large{$x_{1}^2$}};
\draw [fill] (12.5,3) circle [radius=0.09];
\node [above, black] at (12.5,3) {\large{$x_{1}^3$}};
\draw [fill] (9.5,3) circle [radius=0.09];
\node [above, black] at (9.5,3) {\large{$x_{2}^3$}};
\draw [fill] (14,1.5) circle [radius=0.09];
\node [above, black] at (14,1.5) {\large{$x_{1}^4$}};
\draw [fill] (11,1.5) circle [radius=0.09];
\node [above, black] at (11,1.5) {\large{$x_{2}^4$}};
\draw [->]  (8.3,1.8)--(9.2,2.7);
\draw [->]  (11.3,1.8)--(12.2,2.7);
\draw [->]  (9.8,3.3)--(10.7,4.2);
\draw [->]  (8.3,4.8)--(9.2,5.7);
\draw [->]  (9.8,2.7)--(10.7,1.8);
\draw [->]  (12.8,2.7)--(13.7,1.8);
\draw [->]  (8.3,4.2)--(9.2,3.3);
\draw [->]  (11.3,4.2)--(12.2,3.3);
\draw [->]  (9.8,5.7)--(10.7,4.8);
\draw [->]  (8.3,7.2)--(9.2,6.3);
% left
\draw [thick] (0,1.5) -- (0,6.5) ;
\draw [fill] (1.5,5) circle [radius=0.09];
\node [above, black] at (1.5,5) {\large{$x_{1}^1$}};
\draw [fill] (3,3.5) circle [radius=0.09];
\node [above, black] at (3,3.5) {\large{$x_{1}^2$}};
\draw [fill] (4.5,2) circle [radius=0.09];
\node [above, black] at (4.5,2) {\large{$x_{1}^3$}};
\draw [fill] (1.5,2) circle [radius=0.09];
\node [above, black] at (1.5,2) {\large{$x_{2}^3$}};
\draw [->]  (1.8,2.3)--(2.7,3.2);
\draw [->]  (0.3,3.8)--(1.2,4.7);
\draw [->]  (3.3,3.2)--(4.2,2.3);
\draw [->]  (1.8,4.7)--(2.7,3.8);
\draw [->]  (0.3,6.2)--(1.2,5.3);
\end{tikzpicture}
\end{center} 
\caption{The case $N=3$ (left) and $N=4$ (right). The vertical line corresponds to $0$. The exponent $\mathcal{F}^N_{0}(\mathbf{x})$ is given by $-\sum_{a \rightarrow b}e^{a-b}$.} 
\label{weight_sketch}
\end{figure}

The functions $\Phi_{\lambda}^{(N)}$ are related with the $\mathfrak{so}_{2n+1}$-Whittaker functions as follows. If $N=2n$ then $\Phi^{(2n)}_{\lambda}(x)=\Psi^{\mathfrak{so}_{2n+1}}_{\lambda}(x+1_{n}\cdot \log 2)$, where $1_{n}$ denotes the vector consisting of $n$ entries equal to $1$. This observation is immediate comparing the definition of the $\Phi^{(2n)}$ function with the integral representation for the $\mathfrak{so}_{2n+1}$-Whittaker function given in \ref{so-whittaker-definition}. Therefore, we may conclude that the function $\Phi^{(2n)}_{\lambda}$ is an eigenfunction of $\mathcal{H}^B_{n}$.

For $n \in \mathbb{N}$ and 
$\theta \in \mathbb{C}$ we define a second differential operator
\begin{equation}
\mathcal{H}_{n,\theta}^D = \dfrac{1}{2}\sum_{i=1}^n\dfrac{\partial^2}{\partial x_{i}^2}-\sum_{i=1}^{n-1}e^{x_{i+1}-x_{i}}+e^{-x_{n}}\dfrac{\partial}{\partial x_{n}}-\theta e^{-x_{n}}.
\label{eq:operator_D}
\end{equation}
Later we will show that, if $N=2n-1$, the function $\Phi_{\lambda}^{(2n-1)}$ is an eigenfunction of the operator $\mathcal{H}_{n,\lambda_{n}}^D$.

We introduce the integral operator $Q_{\theta}^{n,n}$ defined on a suitable class of functions by
\[Q_{\theta}^{n,n}f(x)=\int_{\mathbb{R}^n}Q_{\theta}^{n,n}(x;y)f(y)dy\]
with kernel given, for $x,y \in \mathbb{R}^n$, by
\begin{equation*}
Q_{\theta}^{n,n}(x;y)=\exp \Big( \theta \big(|y|-|x|\big) -\big(2e^{-y_{n}}+\sum_{i=1}^n e^{y_{i}-x_{i}}+\sum_{i=1}^{n-1} e^{x_{i+1}-y_{i}}\big)\Big ).
\end{equation*}
\begin{proposition}
\label{kernel_odd2even}
The following intertwining relation holds
\[\mathcal{H}_{n}^{B}\circ Q_{\theta}^{n,n}=Q_{\theta}^{n,n} \circ \mathcal{H}^{D}_{n,\theta}.\]
\end{proposition}
\begin{proof}
The intertwining of the operators follows directly from the following identity for the kernel $Q_{\theta}^{n,n}(x;y)$
\[\mathcal{H}^B_{n}Q_{\theta}^{n,n}(x;y)=(\mathcal{H}^D_{n,\theta})^* Q_{\theta}^{n,n}(x;y)\]
where $(\mathcal{H}^D_{n,\theta})^* $ is the adjoint of $\mathcal{H}^D_{n,\theta}$ with respect to the Lebesgue measure on $\mathbb{R}^n$ and acts on the $y$-variable and $\mathcal{H}^B_{n}$ acts on the $x$-variable.
Therefore we only need to prove the identity for the kernels. The adjoint of $\mathcal{H}^D_{n,\theta}$ with respect to the Lebesgue measure on $\mathbb{R}^n$ is
\[(\mathcal{H}_{n,\theta}^D)_{y}^* = \dfrac{1}{2}\sum_{i=1}^n\dfrac{\partial^2}{\partial y_{i}^2}-\sum_{i=1}^{n-1}e^{y_{i+1}-y_{i}}-\dfrac{\partial}{\partial y_{n}}e^{-y_{n}}-\theta e^{-y_{n}}.\]
We calculate for $1\leq i \leq n$
\[\dfrac{\partial^2}{\partial x_{i}^2}Q_{\theta}^{n,n}(x;y)=\Big((-\theta +e^{y_{i}-x_{i}}-e^{x_{i}-y_{i-1}}\mathbbm{1}_{i\geq 2})^2-e^{y_{i}-x_{i}}-e^{x_{i}-y_{i-1}}\mathbbm{1}_{i\geq 2}\Big)Q_{\theta}^{n,n}(x;y)\]
\begin{equation*}
\begin{split}
\dfrac{\partial^2}{\partial y_{i}^2}Q_{\theta}^{n,n}(x;y)=\Big((\theta -e^{y_{i}-x_{i}}&+e^{x_{i+1}-y_{i}}\mathbbm{1}_{i\leq n-1}+2e^{-y_{n}}\mathbbm{1}_{i=n})^2\\
&-e^{y_{i}-x_{i}}-e^{x_{i+1}-y_{i}}\mathbbm{1}_{i\leq n-1}-2e^{-y_{n}}\mathbbm{1}_{i=n}\Big)Q_{\theta}^{n,n}(x;y)
\end{split}
\end{equation*}
then
\begin{equation*}
\begin{split}
\Big(\sum_{i=1}^n \dfrac{\partial^2}{\partial x_{i}^2}-\sum_{i=1}^n \dfrac{\partial^2}{\partial y_{i}^2}\Big)Q_{\theta}^{n,n}(x;y)=\Big(&2\sum_{i=1}^{n-1}e^{x_{i+1}-x_{i}}+2e^{-x_{n}}-2\sum_{i=1}^{n-1}e^{y_{i+1}-y_{i}}\\
&-4e^{-2y_{n}}-4\theta e^{-y_{n}}+2e^{-x_{n}}+2e^{-y_{n}}\Big)Q_{\theta}^{n,n}(x;y).
\end{split}
\end{equation*}
We conclude the proof of the identity by observing that
\[\Big(-\dfrac{\partial}{\partial y_{n}}2e^{-y_{n}}-2\theta e^{-y_{n}}\Big)Q_{\theta}^{n,n}(x;y)=\Big(-4e^{-2y_{n}}-4\theta e^{-y_{n}}+2e^{-x_{n}}+2e^{-y_{n}}\Big)Q_{\theta}^{n,n}(x;y).\]
\end{proof}

Finally, we define the integral operator 
\[Q_{\theta}^{n,n-1}f(x)=\int_{\mathbb{R}^{n-1}}Q_{\theta}^{n,n-1}(x;y)f(y)dy\]
with kernel given for $x \in \mathbb{R}^n$, $y \in \mathbb{R}^{n-1}$ by
\[Q_{\theta}^{n,n-1}(x;y)=\exp \Big( \theta \big(|x|-|y|\big) -\sum_{i=1}^{n-1}\big (e^{x_{i+1}-y_{i}}+ e^{y_{i}-x_{i}}\big)\Big ).\]
This operator provides a relation between $\mathcal{H}^D_{n,\theta}$ and $\mathcal{H}^B_{n-1}$ as follows.
\begin{proposition}
\label{kernel_even2odd}
The following intertwining relation holds
\[(\mathcal{H}_{n,\theta}^D-\dfrac{1}{2}\theta^2)\circ Q_{\theta}^{n,n-1}=Q_{\theta}^{n,n-1}\circ \mathcal{H}^B_{n-1}.\]
\end{proposition}
\begin{proof}
The result follows from the identity
\[\big(\mathcal{H}^D_{n,\theta}-\dfrac{1}{2}\theta^2\big)Q_{\theta}^{n,n-1}(x;y)=(\mathcal{H}^B_{n-1})^{\ast}Q_{\theta}^{n,n-1}(x;y)\]
where $\mathcal{H}^D_{n,\theta}$ acts on the $x$-variable and $(\mathcal{H}^B_{n-1})^{\ast}$ denotes the adjoint of $\mathcal{H}^B_{n-1}$ with respect to the Lebesgue measure on $\mathbb{R}^{n-1}$ and acts on the $y$-variable. We note that the operator $\mathcal{H}^B_{n-1}$ is self-adjoint, i.e. $(\mathcal{H}^B_{n-1})^{\ast}=\mathcal{H}^B_{n-1}$. \\
We calculate for $1\leq i \leq n$
\begin{equation*}
\begin{split}\dfrac{\partial^2}{\partial x_{i}^2}Q_{\theta}^{n,n-1}(x;y)=\Big((\theta +e^{y_{i}-x_{i}}&\mathbbm{1}_{i\leq n-1}-e^{x_{i}-y_{i-1}}\mathbbm{1}_{i\geq 2})^2\\
&-e^{y_{i}-x_{i}}\mathbbm{1}_{i\leq n-1}-e^{x_{i}-y_{i-1}}\mathbbm{1}_{i\geq 2}\Big)Q_{\theta}^{n,n-1}(x;y)
\end{split}
\end{equation*}
and for $1\leq i \leq n-1$
\[\dfrac{\partial^2}{\partial y_{i}^2}Q_{\theta}^{n,n-1}(x;y)=\Big((-\theta +e^{x_{i+1}-y_{i}}-e^{y_{i}-x_{i}})^2-e^{x_{i+1}-y_{i}}-e^{y_{i}-x_{i}}\Big)Q_{\theta}^{n,n-1}(x;y)\]
Subtracting the two Laplacians we obtain the following form 
\begin{equation*}
\begin{split}
\Big(\sum_{i=1}^n \dfrac{\partial^2}{\partial x_{i}^2}-\sum_{i=1}^n \dfrac{\partial^2}{\partial y_{i}^2}\Big)&Q_{\theta}^{n,n-1}(x;y)\\
&=\Big(\theta^2 - 2\sum_{i=1}^{n-2}e^{y_{i+1}-y_{i}}+2\sum_{i=1}^{n-1}e^{x_{i+1}-x_{i}}\Big)Q_{\theta}^{n,n-1}(x;y).
\end{split}
\end{equation*}
We complete the proof by checking that
\[\big(2e^{-x_{n}}\frac{\partial}{\partial x_{n}}-2e^{-x_{n}}\theta \big)Q_{\theta}^{n,n-1}(x;y)=-2e^{-y_{n-1}}Q_{\theta}^{n,n-1}(x;y).\]
\end{proof}

The two kernels $Q^{n,n}$ and $Q^{n,n-1}$ already appeared in the literature in the work of Gerasimov-Lebedev-Oblezin \cite{Gerasimov_et_al_2009} where they provide intertwining relations for the Toda operators of different root systems. The two kernels we present here relate operators for the root systems of type $B_{n}$ and $B_{n-1}$ with a degeneration of the operator for the root system of type $BC_{n}$.

In Propositions \ref{kernel_odd2even} and \ref{kernel_even2odd} we proved that the operators $\mathcal{H}^B$ and $\mathcal{H}^D$ are intertwined. This implies that eigenfunctions of the one operator can be used to determine eigenfunctions for the other. We can more specifically state the following result.
\begin{proposition}
For $b \in \mathbb{C}$, let $f^r_{b}:\mathbb{R}^{n-1} \to \mathbb{C}$ be a right $b$-eigenfunction of $\mathcal{H}^B_{n-1}$, i.e. $\mathcal{H}^B_{n-1}f^r_{b}=bf^r_{b}$, then $Q_{\theta}^{n,n-1}f_{b}^r$ is a right $(b+\frac{1}{2}\theta^2)$-eigenfunction of $\mathcal{H}_{n,\theta}^D$.
\end{proposition}
\begin{proof}
The result follows directly from the intertwining of the two operators in proposition \ref{kernel_even2odd} as follows
\begin{equation*}
\begin{split}
\mathcal{H}^D_{n,\theta}(Q_{\theta}^{n,n-1}f_{b}^r)&=((\mathcal{H}^D_{n,\theta}-\dfrac{1}{2}\theta^2)\circ Q_{\theta}^{n,n-1})f_{b}^r+\dfrac{1}{2}\theta^2Q_{\theta}^{n,n-1}f_{b}^r\\
&=(Q_{\theta}^{n,n-1}\circ\mathcal{H}^B_{n-1})f^r_{b}+\dfrac{1}{2}\theta^2Q_{\theta}^{n,n-1}f_{b}^r\\
&=(b+\dfrac{1}{2}\theta^2)Q_{\theta}^{n,n-1}f^r_{b}.
\end{split}
\end{equation*}
\end{proof}
We observe that the function $\Phi_{\lambda}^{(N)}$ can be re-expressed using the two integral kernels as follows
\begin{equation}
\label{eq:recursion_odd2even}
\Phi_{\lambda_{1},...,\lambda_{n}}^{(2n)}=Q_{\lambda_{n}}^{n,n}\Phi_{\lambda_{1},...,\lambda_{n}}^{(2n-1)}
\end{equation} 
and
\begin{equation}
\label{eq:recursion_even2odd}
\Phi_{\lambda_{1},...,\lambda_{n}}^{(2n-1)}=Q_{\lambda_{n}}^{n,n-1}\Phi_{\lambda_{1},...,\lambda_{n-1}}^{(2n-2)}.
\end{equation} 

\begin{proposition}
\label{eigenrelation_phi_whittaker}
For $\lambda=(\lambda_{1},...,\lambda_{n})\in \mathbb{C}^n$, it holds that
\[\mathcal{H}^B_{n}\Phi_{\lambda}^{(2n)}(x)=\dfrac{1}{2}\sum_{i=1}^n\lambda_{i}^2 \Phi_{\lambda}^{(2n)}(x)\]
and
\[\mathcal{H}^D_{n,\lambda_{n}}\Phi_{\lambda}^{(2n-1)}(x)=\dfrac{1}{2}\sum_{i=1}^n \lambda_{i}^2\Phi_{\lambda}^{(2n-1)}(x).\]
\end{proposition}
\begin{proof}
Although the eigenrelation for $\Phi^{(2n)}_{\lambda}$ follows directly from the eigenrelation for the $\mathfrak{so}_{2n+1}$-Whittaker function stated in \ref{eq:eigenrelation_so} after performing the transform $x_{i}\mapsto x_{i}+\log 2$, for $1\leq i \leq n$, for reasons of completeness we will prove the eigenrelation for $\Phi_{\lambda}^{(N)}(x)$ for $N$ both even and odd. \\
Let us start with the simple case $N=1$. We have $\Phi^{(1)}_{\lambda_{1}}(x)=e^{\lambda_{1}x}$ and $\mathcal{H}_{1,\lambda_{1}}^D=\frac{1}{2}\frac{d^2}{dx^2}+e^{-x}\frac{d}{dx}-\lambda_{1}e^{-x}$. Then differentiating the function $\Phi_{\lambda_{1}}^{(1)}(x)$ leads directly to the relation
\[\mathcal{H}_{1,\lambda_{1}}^D\Phi^{(1)}_{\lambda_{1}}(x)=\frac{1}{2}\lambda_{1}^2\Phi^{(1)}_{\lambda_{1}}(x).\]
Assume now that the function $\Phi_{\lambda}^{(2n-1)}$, for $\lambda=(\lambda_{1},...,\lambda_{n})\in \mathbb{C}^n$, is an eigenfunction of the operator $\mathcal{H}_{n,\lambda_{n}}^D$, i.e. $\mathcal{H}_{n,\lambda_{n}}^D \Phi_{\lambda}^{(2n-1)}=\frac{1}{2}\sum_{i=1}^n \lambda_{i}^2 \Phi_{\lambda}^{(2n-1)}$, then for the function $\Phi_{\lambda}^{(2n)}$, determined via \eqref{eq:recursion_odd2even}, we have
\begin{equation*}
\begin{split}
\mathcal{H}_{n}^B\Phi_{\lambda_{1}}^{(2n)}(x)&= (\mathcal{H}_{n}^B \circ Q_{\lambda_{n}}^{n,n})\Phi_{\lambda}^{(2n-1)}(x)\\
&= (Q_{\lambda_{n}}^{n,n} \circ \mathcal{H}^{D}_{n,\lambda_{n}})\Phi_{\lambda_{1}}^{(2n-1)}(x)\\
&=\dfrac{1}{2}\sum_{i=1}^n \lambda_{i}^2Q_{\lambda_{n}}^{n,n}\Phi_{\lambda}^{(2n-1)}(x)\\
&=\dfrac{1}{2}\sum_{i=1}^n \lambda_{i}^2\Phi_{\lambda}^{(2n)}(x)
\end{split}
\end{equation*}
where for the second equality, we used the intertwining relation stated at Proposition \ref{kernel_odd2even}, and for the third comes from the induction hypothesis.\\
Moreover, for $\lambda_{n+1}\in \mathbb{C}$ the function $\Phi_{(\lambda, \lambda_{n+1})}^{(2n+1)}$ can be written, using \eqref{eq:recursion_even2odd}, as $\Phi_{(\lambda, \lambda_{n+1})}^{(2n+1)}(x)=Q_{\lambda_{n+1}}^{n+1,n}\Phi_{\lambda}^{(2n)}(x)$. Therefore
\begin{equation*}
\begin{split}
\mathcal{H}_{n+1,\lambda_{n+1}}^D\Phi_{(\lambda,\lambda_{n+1})}^{(2n+1)}(x)&= (\mathcal{H}_{n+1,\lambda_{n+1}}^D \circ Q_{\lambda_{n}}^{n+1,n})\Phi_{\lambda}^{(2n)}(x)\\
&= (Q_{\lambda_{n+1}}^{n+1,n} \circ (\mathcal{H}^{B}_{n}+\dfrac{1}{2}\lambda_{n+1}^2))\Phi_{\lambda}^{(2n)}(x)\\
& = Q_{\lambda_{n+1}}^{n+1,n}(\mathcal{H}^B_{n}\Phi_{\lambda}^{(2n)}(x)+\dfrac{1}{2}\lambda_{n+1}^2\Phi_{\lambda}^{(2n)}(x))\\
&=\dfrac{1}{2}\sum_{i=1}^{n+1} \lambda_{i}^2Q_{\lambda_{n+1}}^{n+1,n}\Phi_{\lambda}^{(2n)}(x)\\
&=\dfrac{1}{2}\sum_{i=1}^{n+1} \lambda_{i}^2\Phi_{(\lambda,\lambda_{n+1})}^{(2n+1)}(x)
\end{split}
\end{equation*}
where for the second equality we used the intertwining of the kernels in Proposition \ref{kernel_even2odd} and the fourth follows from the formula we proved for $\Phi^{(2n)}_{\lambda}$. The result then follows by induction on $N$.
\end{proof}
Similarly to the $\mathfrak{so}_{2n+1}$-Whittaker function, the function $\Phi^{(N)}$ also has controlled exponential decay. More specifically, we have the following result.
\begin{proposition}
\label{exponential_decay}
Let $n\geq 1$, $N=2n$ or $N=2n-1$ and $\lambda = (\lambda_{1},...,\lambda_{n})\in \mathbb{R}^n$ such that $\lambda_{1}>...>\lambda_{n}>0$. The function $e^{-\sum_{i=1}^n\lambda_{i}x_{i}}\Phi_{\lambda}^{(N)}(x)$ is bounded with the following limiting behaviour at infinity
\[e^{-\sum_{i=1}^n\lambda_{i}x_{i}}\Phi_{\lambda}^{(N)}(x) \to C, \text{ as }x\to \infty\]
for some $C=C(\lambda)>0$. Here $x\to \infty$ can be understood as $x_{i}-x_{i+1}\to \infty$, for $1\leq i <n$ and $x_{n}\to \infty$.
\end{proposition}

\begin{proof}
When $N=2n$ the result follows from Proposition \ref{Whittaker_characterization}, as $\Phi_{\lambda}^{(2n)}$ corresponds to $\Psi_{\lambda}^{\mathfrak{so}_{2n+1}}$ after shifting the variables by $\log 2$.

Let us then prove the result for $N=2n-1$. For $\lambda = (\lambda_{1},...,\lambda_{n})\in \mathbb{R}^n$ with $\lambda_{1}>...>\lambda_{n}>0$ and $x\in \mathbb{R}^n$ we have
\begin{equation*}
\begin{split}
\Phi^{(2n-1)}_{\lambda}(x)&e^{-\sum_{i=1}^n \lambda_{i}x_{i}}\\
&=e^{-\sum_{i=1}^n \lambda_{i}x_{i}}\int_{\mathbb{R}^{n-1}}Q^{n,n-1}_{\lambda_{n}}(x;y)\Phi^{(2n-2)}_{\tilde{\lambda}}(y)dy\\
&=e^{-\sum_{i=1}^n \lambda_{i}x_{i}}\int_{\mathbb{R}^{n-1}}Q^{n,n-1}_{\lambda_{n}}(x;y)e^{\sum_{i=1}^{n-1} \lambda_{i}y_{i}}\Big(e^{-\sum_{i=1}^{n-1} \lambda_{i}y_{i}}\Phi^{(2n-2)}_{\tilde{\lambda}}(y)\Big)dy.
\end{split}
\end{equation*}
Since the result holds in the even case, there exists constant $c\geq 0$, that depends on $\tilde{\lambda}$ such that
\[e^{-\sum_{i=1}^{n-1} \lambda_{i}y_{i}}\Phi^{(2n-2)}_{\tilde{\lambda}}(y)\leq c.\]
Therefore 
\begin{equation*}
\begin{split}
\Phi^{(2n-1)}_{\lambda}(x)&e^{-\sum_{i=1}^n \lambda_{i}x_{i}}\\
&\leq c \,e^{-\sum_{i=1}^n \lambda_{i}x_{i}}\int_{\mathbb{R}^{n-1}}Q^{n,n-1}_{\lambda_{n}}(x;y)e^{\sum_{i=1}^{n-1} \lambda_{i}y_{i}}dy\\
&= c \,e^{\sum_{i=1}^n (\lambda_{n}-\lambda_{i})x_{i}}\prod_{i=1}^{n-1}\int_{\mathbb{R}}e^{(\lambda_{i}-\lambda_{n})y_{i}}\exp \Big(-e^{x_{i+1}-y_{i}}-e^{y_{i}-x_{i}}\Big)dy_{i}.
\end{split}
\end{equation*}
Making the change of variables
\[y_{i}\mapsto y_{i}+x_{i+1}, \qquad 1\leq i \leq n-1\]
we have that
\begin{equation*}
\begin{split}
\Phi^{(2n-1)}_{\lambda}(x)&e^{-\sum_{i=1}^n\lambda_{i}x_{i}}\\
&\leq c\,e^{\sum_{i=1}^n (\lambda_{n}-\lambda_{i})x_{i}}\prod_{i=1}^{n-1}e^{(\lambda_{i}-\lambda_{n})\frac{x_{i}+x_{i+1}}{2}}\Psi_{(\lambda_{i}-\lambda_{n})/2}^{\mathfrak{so}_{2\cdot 1 + 1}}(x_{i}-x_{i+1})\\
& = c\prod_{i=1}^{n-1 }e^{-\frac{\lambda_{i}-\lambda_{n}}{2}(x_{i}-x_{i+1})}\Psi_{(\lambda_{i}-\lambda_{n})/2}^{\mathfrak{so}_{2\cdot 1 + 1}}(x_{i}-x_{i+1}).
\end{split}
\end{equation*}
We may use again Proposition \ref{Whittaker_characterization} to conclude that there exists $c_{i}\geq 0$ such that
\[e^{-\frac{\lambda_{i}-\lambda_{n}}{2}(x_{i}-x_{i+1})}\Psi_{(\lambda_{i}-\lambda_{n})/2}^{\mathfrak{so}_{2\cdot 1 + 1}}(x_{i}-x_{i+1}) \leq c_{i}\]
for $1\leq i \leq n-1$, which completes the proof for the bound.

Finally, we use Dominated Convergence theorem to conclude that $\Phi_{\lambda}^{(2n-1)}(x)$ behaves like $e^{\sum_{i=1}^n \lambda_{i}x_{i}}$ at infinity. 
\end{proof}
\section{The main result}
In section 8.1 we considered the operators $\mathcal{H}^B_{n}$ and $\mathcal{H}^D_{n,\theta}$ and calculated their eigenfunctions. We observe that 
\[\mathcal{H}^B_{n}-\frac{1}{2}\sum_{i=1}^n \lambda_{i}^2\]
can be thought as an $n$-dimensional standard Brownian motion with killing given by
\[c(x)=\sum_{i=1}^{n-1}e^{x_{i+1}-x_{i}}+e^{-x_{n}}+\frac{1}{2}\sum_{i=1}^n \lambda_{i}^2.\]
Similarly,
\[\mathcal{H}^D_{n,\lambda_{n}}-\frac{1}{2}\sum_{i=1}^n\lambda_{i}^2\]
is the generator of a diffusion with drift $b_{i}(x)=e^{-x_{n}}\mathbbm{1}_{i=n}$, for $1\leq i \leq n$ and killing given by
\[c(x) = \sum_{i=1}^{n-1}e^{x_{i+1}-x_{i}}+\lambda_{n}e^{-x_{n}}+\frac{1}{2}\sum_{i=1}^n\lambda_{i}^2.\]

In this section we will prove that under appropriate initial conditions, the bottom level of the process $\mathbf{X}$ with evolution described by the differential equations \eqref{eq:SDEs_even} and \eqref{eq:SDEs_odd}, is the $h$-transform, with $h=\Phi_{\lambda}^{(N)}$, of the process with generator $\mathcal{H}^B_{n}-\frac{1}{2}\sum_{i=1}^n \lambda_{i}^2$, if $N=2n$ and $\mathcal{H}^D_{n,\lambda_{n}}-\frac{1}{2}\sum_{i=1}^n$ if $N=2n-1$. \\

For $x, \lambda \in \mathbb{R}^n$, denote by $\sigma_{\lambda}^x$ the probability measure on the set $\Gamma_{N}$ defined by
\[\int fd\sigma_{\lambda}^x = \Phi_{\lambda}^{(N)}(x)^{-1}\int_{\Gamma_{N}[x]}f(\mathbf{x})e^{\mathcal{F}_{\lambda}^{(N)}(\mathbf{x})}\prod_{k=1}^{N-1}\prod_{i=1}^{\big[\frac{k+1}{2} \big]}dx^k_{i}.\]
\begin{theorem}
\label{first_result}
Fix $x\in \mathbb{R}^n$, $\lambda = (\lambda_{1}> ... > \lambda_{n}>0)$ and let $\mathbf{X}$ be the process evolving according to the system of stochastic differential equations described in \eqref{eq:SDEs_even}, \eqref{eq:SDEs_odd} with initial law $\sigma_{\lambda}^x$. Then $X^N$ is a Markov process started at $x$ with infinitesimal generator
\[\mathcal{L}_{n,\lambda}^{B} =(\Phi_{\lambda}^{(2n)})^{-1}\Big(\mathcal{H}^B_{n}-\dfrac{1}{2}\sum_{i=1}^n\lambda_{i}^2 \Big)\Phi_{\lambda}^{(2n)}, \text{ if  }N=2n \]
and
\[\mathcal{L}_{n,\lambda}^{D} =(\Phi_{\lambda}^{(2n-1)})^{-1}\Big(\mathcal{H}^D_{n,\lambda_{n}}-\dfrac{1}{2}\sum_{i=1}^n\lambda_{i}^2 \Big)\Phi_{\lambda}^{(2n-1)}, \text{ if }N=2n-1\]
where $\mathcal{H}^B_{n}$ and $\mathcal{H}^D_{n,\lambda_{n}}$ are as in \eqref{eq:operator_B} and \eqref{eq:operator_D}.
\end{theorem}
\begin{proof}
The proof of Theorem \ref{first_result} uses Theorem \ref{Kurtz}, therefore we need a Markov kernel $K$ that satisfies the property
\[K(f \circ \gamma) = f\]
where $\gamma(\mathbf{x})=x^N$ is the projection to the bottom level of the pattern $\mathbf{x}$.

Let us define the kernel $K^{N}_{\lambda}$ acting on bounded, measurable functions on $\Gamma_{N}$ as follows
\[(K^{N}_{\lambda}f)(x) = \Phi_{\lambda}^{(N)}(x)^{-1}\int_{\Gamma_{N}[x]}f(\mathbf{x})e^{\mathcal{F}_{\lambda}^{(N)}(\mathbf{x})}\prod_{k=1}^{N-1}\prod_{i=1}^{\big[\frac{k+1}{2} \big]}dx^k_{i}.\]
The kernel $K^{N}_{\lambda}$ satisfies the properties:
\begin{itemize}
\item if $f\equiv 1$ then, by the definition of the function $\Phi^{(N)}$, it holds
\[K^{N}_{\lambda}f = 1;\]
\item since $(f\circ \gamma)(\mathbf{x}) = f(x^N)$, it follows that
\begin{equation*}
\begin{split}
\big(K^{N}_{\lambda}(f\circ \gamma)\big)(\mathbf{x}) &= \Phi_{\lambda}^{(N)}(x)^{-1}\int_{\Gamma_{N}[x]}f(x^N)e^{\mathcal{F}_{\lambda}^{(N)}(\mathbf{x})}\prod_{k=1}^{N-1}\prod_{i=1}^{\big[\frac{k+1}{2} \big]}dx^k_{i}\\
& = f(x)\Phi_{\lambda}^{(N)}(x)^{-1}\int_{\Gamma_{N}[x]}e^{\mathcal{F}_{\lambda}^{(N)}(\mathbf{x})}\prod_{k=1}^{N-1}\prod_{i=1}^{\big[\frac{k+1}{2} \big]}dx^k_{i}\\
&= f(x).
\end{split}
\end{equation*}
\end{itemize}

Let us now denote the infinitesimal generator of the process $\mathbf{X}$ by $\mathcal{A}^{N}_{\lambda}$. The key step towards the proof of Theorem \ref{first_result} is the proof an intertwining relation among the generator $\mathcal{A}^{N}_{\lambda}$ and $\mathcal{L}_{n,\lambda}^{B}$ if $N=2n$ or $\mathcal{L}_{n,\lambda}^{D}$ if $N=2n-1$.

\begin{proposition}
\label{main_intert_proposition(continuous)}
Let $n\geq 1$. The following intertwining relations hold:
\[\mathcal{L}^B_{n,\lambda}\circ K^{N}_{\lambda} = K^{N}_{\lambda} \circ \mathcal{A}^{N}_{\lambda}, \text{ if }N=2n\]
and
\[\mathcal{L}^D_{n,\lambda}\circ K^{N}_{\lambda} = K^{N}_{\lambda} \circ \mathcal{A}^{N}_{\lambda},\text{ if }N=2n-1.\]
\end{proposition}
\begin{proof}(of Proposition \ref{main_intert_proposition(continuous)})
We will prove the result by induction on $N$.

When $N=1$, we have $K^{1}_{\lambda_{1}}f = f$
and since $\Phi^{(1)}_{\lambda_{1}}(x) = e^{\lambda_{1}x}$ we can easily check that
\[\mathcal{A}^{1}_{\lambda_{1}} = \dfrac{1}{2}\dfrac{d^2}{dx^2} + (\lambda_{1}+e^{-x})\dfrac{d}{dx} = \mathcal{L}_{1,\lambda_{1}}^D.\]
Therefore the intertwining relation holds trivially.\\
% even2odd
\paragraph*{Iterate from even to odd.} Suppose that $N=2n-1$ and the statement of Proposition \ref{main_intert_proposition(continuous)} holds for $N-1$, i.e. the relation
\[\mathcal{L}_{n-1,\tilde{\lambda}}^B \circ K^{N-1}_{\tilde{\lambda}} = K^{N-1}_{\tilde{\lambda}}\circ \mathcal{A}^{N-1}_{\tilde{\lambda}}\]
holds.

By the construction of the process, we observe that the evolution of $X^N$ depends on the rest of the pattern only via $X^{N-1}$. Therefore, it is convenient to obtain a helper intertwining relation that focuses only on the bottom two levels of the pattern and then deduce the required intertwining relation from the helper.

We consider the operator $\mathcal{G}^{n,n-1}_{\lambda}$ acting on appropriate functions on $\mathbb{R}^{n}\times \mathbb{R}^{n-1}$ as follows
\begin{equation*}
\begin{split}
(\mathcal{G}^{n,n-1}_{\lambda}f)(x,y)=\Big[\mathcal{L}^B_{n-1, \tilde{\lambda}}&+\frac{1}{2}\sum_{k=1}^n \frac{\partial ^2}{\partial x_{k}^2}+\frac{1}{2}(\lambda_{n}+e^{y_{1}-x_{1}})\frac{\partial}{\partial x_{1}}\\
&+ \frac{1}{2}\sum_{k=2}^{n-1}(\lambda_{n}+e^{y_{k}-x_{k}}-e^{x_{k}-y_{k-1}})\frac{\partial}{\partial x_{k}}\\
&+\frac{1}{2}(\lambda_{n}+e^{-x_{n}}-e^{x_{n}-y_{n-1}})\frac{\partial}{\partial x_{n}}\Big]f(x,y)
\end{split}
\end{equation*}
where $\mathcal{L}^B_{n-1, \tilde{\lambda}}$ acts on the $y$-variable. 

We also consider the operator  $\Lambda_{\lambda}^{n,n-1}$ defined for bounded, measurable functions on $\mathbb{R}^n\times \mathbb{R}^{n-1}$ by
\[(\Lambda^{n,n-1}_{\lambda} f)(x)=\Phi_{\lambda}^{(N)}(x)^{-1}\int_{\mathbb{R}^{n-1}}Q_{\lambda_{n}}^{n,n-1}(x;y)\Phi_{\tilde{\lambda}}^{(N-1)}(y)f(x,y)\prod_{i=1}^{n-1}dy_{i}.\]

\noindent \textbf{Claim.} The intertwining relation 
\begin{equation}
\mathcal{L}^D_{n,\lambda}\circ \Lambda_{\lambda}^{n.n-1} = \Lambda_{\lambda}^{n.n-1} \circ \mathcal{G}^{n,n-1}_{\lambda}
\label{eq:helper_even2odd(continuous)}
\end{equation}
holds.

Before we confirm \eqref{eq:helper_even2odd(continuous)} let us use it to complete the proof of the main intertwining relation.  

We observe that the operator $\mathcal{A}^{N}_{\lambda}$ consists of two parts as follows
\begin{equation*}
\begin{split}
\mathcal{A}^{N}_{\lambda} &= \mathcal{A}^{N-1}_{\tilde{\lambda}} + \frac{1}{2}\sum_{k=1}^n \frac{\partial ^2}{\partial (x^N_{k})^2}+(\lambda_{n}+e^{x^{N-1}_{1}-x^N_{1}})\frac{\partial}{\partial x^N_{1}}\\
&+ \sum_{k=2}^{n-1}(\lambda_{n}+e^{x^{N-1}_{k}-x^N_{k}}-e^{x^N_{k}-x^{N-1}_{k-1}})\frac{\partial}{\partial x^N_{k}}\\
&+(\lambda_{n}+e^{-x^N_{n}}-e^{x^N_{n}-x^{N-1}_{n-1}})\frac{\partial}{\partial x^N_{n}}
\end{split}
\end{equation*}
where the operator $\mathcal{A}^{N-1}_{\tilde{\lambda}}$ acts on the top $N-1$ levels of the pattern.

Then for any appropriate function $f$ on $\Gamma_{N}$ we have
\[(K^{N}_{\lambda}\circ \mathcal{A}^{N}_{\lambda})f(x) = \Phi_{\lambda}^{(N)}(x) ^{-1}\int_{\Gamma_{N}[x]}e^{\mathcal{F}^{(N)}_{\lambda}(\mathbf{x})}(\mathcal{A}^{N}_{\lambda}f)(\mathbf{x})\prod_{k=1}^{N-1}\prod_{i=1}^{\big[\frac{k+1}{2} \big]}dx^k_{i}=:\mathcal{I}_{1} + \mathcal{I}_{2}\]
where 
\[\mathcal{I}_{1} = \Phi_{\lambda}^{(N)}(x) ^{-1}\int_{\Gamma_{N}[x]}e^{\mathcal{F}^{(N)}_{\lambda}(\mathbf{x})}(\mathcal{A}^{N-1}_{\tilde{\lambda}}f)(\mathbf{x})\prod_{k=1}^{N-1}\prod_{i=1}^{\big[\frac{k+1}{2} \big]}dx^k_{i}\]
and
\begin{equation*}
\begin{split}
\mathcal{I}_{2} = \Phi_{\lambda}^{(N)}(x) ^{-1}\int_{\Gamma_{N}[x]}e^{\mathcal{F}^{(N)}_{\lambda}(\mathbf{x})}&\Big[ \frac{1}{2}\sum_{k=1}^n \frac{\partial ^2}{\partial x_{k}^2}+\frac{1}{2}(\lambda_{n}+e^{x^{N-1}_{1}-x_{1}})\frac{\partial}{\partial x_{1}}\\
&+ \frac{1}{2}\sum_{k=2}^{n-1}(\lambda_{n}+e^{x^{N-1}_{k}-x_{k}}-e^{x_{k}-x^{N-1}_{k-1}})\frac{\partial}{\partial x_{k}}\\
&+\frac{1}{2}(\lambda_{n}+e^{-x_{n}}-e^{x_{n}-x^{N-1}_{n-1}})\frac{\partial}{\partial x_{n}}\Big]f(\mathbf{x})\prod_{i=1}^{\big[\frac{k+1}{2} \big]}dx^k_{i}.
\end{split}
\end{equation*}

\noindent The operator $K^{N}_{\lambda}$ has a recursive structure as follows
\[K^{N}_{\lambda} = \Lambda_{\lambda}^{n,n-1}\circ K^{N-1}_{\tilde{\lambda}},\]
therefore $\mathcal{I}_{1}$ equals
\begin{equation*}
\begin{split}
\mathcal{I}_{1}&= \Phi^{(N)}_{\lambda}(x)^{-1}\int_{\mathbb{R}^{n-1}}Q_{\lambda_{n}}^{n,n-1}(x;y)\Phi_{\tilde{\lambda}}^{(N-1)}(y)(K^{N-1}_{\tilde{\lambda}}\circ \mathcal{A}^{N-1}_{\tilde{\lambda}})f(x,y)\prod_{i=1}^{n-1}dy_{i}\\
&= \Phi^{(N)}_{\lambda}(x)^{-1}\int_{\mathbb{R}^{n-1}}Q_{\lambda_{n}}^{n,n-1}(x;y)\Phi_{\tilde{\lambda}}^{(N-1)}(y)(\mathcal{L}_{n-1,\tilde{\lambda}}^B\circ K^{N-1}_{\tilde{\lambda}})f(x,y)\prod_{i=1}^{n-1}dy_{i}
\end{split}
\end{equation*}
where the second equality follows from the induction hypothesis, and $\mathcal{I}_{2}$ equals
\begin{equation*}
\begin{split}
\mathcal{I}_{2} = \Phi^{(N)}_{\lambda}(x)^{-1}\int_{\mathbb{R}^{n-1}}&Q_{\lambda_{n}}^{n,n-1}(x;y)\Phi_{\tilde{\lambda}}^{(N-1)}(y)\\
\times &\Big[ \frac{1}{2}\sum_{k=1}^n \frac{\partial ^2}{\partial x_{k}^2}+\frac{1}{2}(\lambda_{n}+e^{y_{1}-x_{1}})\frac{\partial}{\partial x_{1}}\\
&+ \frac{1}{2}\sum_{k=2}^{n-1}(\lambda_{n}+e^{y_{k}-x_{k}}-e^{x_{k}-y_{k-1}})\frac{\partial}{\partial x_{k}}\\
&+\frac{1}{2}(\lambda_{n}+e^{-x_{n}}-e^{x_{n}-y_{n-1}})\frac{\partial}{\partial x_{n}}\Big]f(x,y)\prod_{i=1}^{n-1}dy_{i}.
\end{split}
\end{equation*}
Combining the two integrals we then conclude that
\[(K^{N}_{\lambda}\circ \mathcal{A}^{N}_{\lambda})f(x) = (\Lambda_{\lambda}^{n,n-1}\circ \mathcal{G}^{n,n-1}_{\lambda}\circ K^{N-1}_{\tilde{\lambda}})f(x)\]
which, using the helper intertwining relation \eqref{eq:helper_even2odd(continuous)}, equals
\[(\mathcal{L}_{n,\lambda}^D  \circ \Lambda^{n,n-1}_{\lambda}\circ K^{N-1}_{\tilde{\lambda}})f(x)=(\mathcal{L}_{n,\lambda}^D \circ K^{N}_{\lambda}) f(x)\]
as required.\\

Let us conclude the proof for the even-to-odd case by proving relation \eqref{eq:helper_even2odd(continuous)}. The operator $\mathcal{G}^{n,n-1}_{\lambda}$ can then be re-expressed as follows
\[\mathcal{G}^{n,n-1}_{\lambda}=\Phi_{\tilde{\lambda}}^{(N-1)}(y)^{-1}\Big(V_{\lambda_{n}}^{n,n-1}-\frac{1}{2}\sum_{i=1}^{n-1}\lambda_{i}^2\Big)\Phi_{\tilde{\lambda}}^{(N-1)}(y)\]
with
\begin{equation*}
\begin{split}
V_{\lambda_{n}}^{n,n-1}=& (\mathcal{H}_{n-1}^B)_{y}+\frac{1}{2}\sum_{k=1}^n \frac{\partial ^2}{\partial x_{k}^2}+(\lambda_{n}+e^{y_{1}-x_{1}})\frac{\partial}{\partial x_{1}}\\
&+ \sum_{k=2}^{n-1}(\lambda_{n}+e^{y_{k}-x_{k}}-e^{x_{k}-y_{k-1}})\frac{\partial}{\partial x_{k}}\\
&+(\lambda_{n}+e^{-x_{n}}-e^{x_{n}-y_{n-1}})\frac{\partial}{\partial x_{n}}.
\end{split}
\end{equation*}
where $(\mathcal{H}_{n-1}^B)_{y}$ acts as the operator $\mathcal{H}_{n-1}^B$ on the $y$-variable.
 
\noindent Let $(\mathcal{H}^D_{n,\lambda_{n}})_{x}$ be the operator $\mathcal{H}_{n,\lambda_{n}}^D$ acting on the $x$-variable, then
\begin{equation*}
\begin{split}
(\mathcal{H}_{n,\lambda_{n}}^D)_{x}\int_{\mathbb{R}^{n-1}}&Q_{\lambda_{n}}^{n,n-1}(x;y)\Phi_{\tilde{\lambda}}^{(N-1)}(y)f(x,y)dy\\
=& \int_{\mathbb{R}^{n-1}}((\mathcal{H}_{n,\lambda_{n}}^D)_{x}Q_{\lambda_{n}}^{n,n-1}(x;y))\Phi_{\tilde{\lambda}}^{(N-1)}(y)f(x,y)dy\\
&+ \int_{\mathbb{R}^{n-1}}Q_{\lambda_{n}}^{n,n-1}(x;y)\frac{1}{2} \sum_{i=1}^n\dfrac{\partial^2}{\partial x_{i}^2}\Phi_{\tilde{\lambda}}^{(N-1)}(y)f(x,y)dy\\
&+ \int_{\mathbb{R}^{n-1}}\sum_{i=1}^n \Big(\dfrac{\partial}{\partial x_{i}}Q_{\lambda_{n}}^{n,n-1}(x;y)\Big)\Big(\dfrac{\partial}{\partial x_{i}}\Phi_{\tilde{\lambda}}^{(N-1)}(y)f(x,y)\Big)dt\\
&+ \int_{\mathbb{R}^{n-1}}Q_{\lambda_{n}}^{n,n-1}(x;y)e^{-x_{n}}\dfrac{\partial}{\partial x_{n}}\Phi_{\tilde{\lambda}}^{(N-1)}(y)f(x,y)dt.
\end{split}
\end{equation*}
Using Proposition \ref{kernel_even2odd} and the fact that $\mathcal{H}^B_{n-1}$ is self-adjoint, the first integral equals
\[ \int_{\mathbb{R}^{n-1}}Q_{\lambda_{n}}^{n,n-1}(x;y)((\mathcal{H}_{n-1}^B)_{y}+\frac{1}{2}\lambda_{n}^2)\Phi_{\tilde{\lambda}}^{(N-1)}(y)f(x,y)dy.\]
Using, for $1\leq i \leq n$, the formula
\[\dfrac{\partial}{\partial x_{i}}Q_{\lambda_{n}}^{n,n-1}(x;y)=(\lambda_{n}+e^{y_{i}-x_{i}}\mathbbm{1}_{i\leq n-1}-e^{x_{i}-y_{i-1}}\mathbbm{1}_{i>1})Q_{\lambda_{n}}^{n,n-1}(x;y)\]
we conclude that the last two integrals combined equal
\begin{equation*}
\begin{split}
\int_{\mathbb{R}^{n-1}}Q_{\lambda_{n}}^{n,n-1}\Big((\lambda_{n}+e^{y_{1}-x_{1}})&\frac{\partial}{\partial x_{1}}+\sum_{i=2}^{n-1}(\lambda_{n}+e^{y_{i}-x_{i}}-e^{x_{i}-y_{i-1}})\frac{\partial}{\partial x_{i}}\\
&(\lambda_{n}+e^{-x_{n}}-e^{x_{n}-y_{n-1}})\frac{\partial}{\partial x_{n}}\Big)\Phi_{\tilde{\lambda}}^{(N-1)}(y)f(x,y)dy
\end{split}
\end{equation*}
therefore 
\begin{equation*}
\begin{split}
(\mathcal{H}_{n,\lambda_{n}}^D)_{x}\int_{\mathbb{R}^{n-1}}Q_{\lambda_{n}}^{n,n-1}&(x;y)\Phi_{\tilde{\lambda}}^{(N-1)}(y)f(x,y)dy\\
&=\int_{\mathbb{R}^{n-1}}Q_{\lambda_{n}}^{n,n-1}(x;y)(V_{\lambda_{n}}^{n,n-1}+\frac{1}{2}\lambda_{n}^2)\Phi_{\tilde{\lambda}}^{(N-1)}(y)f(x,y)dy
\end{split}
\end{equation*}
We then conclude that for appropriate functions on $\mathbb{R}^n\times \mathbb{R}^{n-1}$
\[(\mathcal{L}_{n,\lambda}^D \circ \Lambda_{\lambda}^{n,n-1})f(x) = (\Lambda_{\lambda}^{n,n-1} \circ \mathcal{G}_{\lambda}^{n,n-1})f(x)\]
as required.

\paragraph*{Iterate from odd to even.}
Assume that $N=2n$ and the result holds for $N-1$, i.e. the relation
\[\mathcal{L}^D_{n-1,\lambda}\circ K_{N-1} = K_{N-1}\circ \mathcal{A}_{N-1}\]
holds.

Instead of proving the intertwining relation of Proposition \ref{main_intert_proposition(continuous)} we will provide a relation that focuses on the bottom two levels of the pattern. The required intertwining relation then follows using the same arguments as for the even-to-odd case.

We consider the operator $\mathcal{G}^{n,n}_{\lambda}$ that acts on appropriate functions on $\mathbb{R}^n\times \mathbb{R}^n$ as follows
\[(\mathcal{G}^{n,n}_{\lambda}f)(x,y) = \Big[\mathcal{L}^D_{n,\lambda}+\frac{1}{2}\sum_{i=1}^n\frac{\partial ^2}{\partial x_{i}^2}+\sum_{i=1}^n(-\lambda_{n}+e^{y_{i}-x_{i}}-e^{x_{i}-y_{i-1}}\mathbbm{1}_{i>1})\frac{\partial}{\partial x_{i}}\Big]f(x,y)\]
where the operator $\mathcal{L}^D_{n,\lambda}$ acts on the $y$-variable.

We also consider the operator $\Lambda_{\lambda}^{n,n}$ defined for bounded measurable functions on $\mathbb{R}^n\times \mathbb{R}^n$ by
\[(\Lambda_{\lambda}^{n,n}f)(x)=\Phi_{\lambda}^{(N)}(x)^{-1}\int_{\mathbb{R}^n} Q_{\lambda_{n}}^{n,n}(x;y)\Phi_{\lambda}^{(N-1)}(y)f(x,y)dy.\]

\noindent \textbf{Claim.} The intertwining relation
\begin{equation}
\mathcal{L}^B_{n,\lambda}\circ \Lambda^{n,n}_{\lambda} = \Lambda_{\lambda}^{n,n}\circ \mathcal{G}^{n,n}_{\lambda}
\label{eq:helper_odd2even(continuous)}
\end{equation}
holds.

For the proof of \eqref{eq:helper_odd2even(continuous)} we observe that $\mathcal{G}^{n,n}_{\lambda}$ can be re-written as follows
\[\mathcal{G}_{\lambda}^{n,n}=\Phi_{\lambda}^{(N-1)}(y)^{-1}\Big(V^{n,n}_{\lambda_{n}}-\frac{1}{2}\sum_{i=1}^n\lambda_{i}^2\Big)\Phi_{\lambda}^{(N-1)}(y)\]
with
\[V^{n,n}_{\lambda_{n}}=(\mathcal{H}_{n,\lambda_{n}}^D)_{y}+\frac{1}{2}\frac{\partial ^2}{\partial x_{i}^2}+\sum_{i=1}^n(-\lambda_{n}+e^{y_{i}-x_{i}}-e^{x_{i}-y_{i-1}}\mathbbm{1}_{i>1})\frac{\partial}{\partial x_{i}},\]
where $(\mathcal{H}^D_{n,\lambda_{n}})_{y}$ acts as the operator $\mathcal{H}^D_{n,\lambda_{n}}$ on the $y$-variable.

\noindent Let $(\mathcal{H}^B_{n})_{x}$ be the operator $\mathcal{H}_{n}^B$ acting on the $x$-variable, then
\begin{equation*}
\begin{split}
(\mathcal{H}_{n}^B)_{x}\int_{\mathbb{R}^n}Q_{\lambda_{n}}^{n,n}(x;y)&\Phi_{\lambda}^{(N-1)}(y)f(x,y)dy\\
=& \int_{\mathbb{R}^n}((\mathcal{H}_{n}^B)_{x}Q_{\lambda_{n}}^{n,n}(x;y))\Phi_{\lambda}^{(N-1)}(y)f(x,y)dy\\
&+ \int_{\mathbb{R}^n}Q_{\lambda_{n}}^{n,n}(x;y)\frac{1}{2} \sum_{i=1}^n\dfrac{\partial^2}{\partial x_{i}^2}\Phi_{\lambda}^{(N-1)}(y)f(x,y)dy\\
&+ \int_{\mathbb{R}^n}\sum_{i=1}^n \Big(\dfrac{\partial}{\partial x_{i}}Q_{\lambda_{n}}^{n,n}(x;y)\Big)\Big(\dfrac{\partial}{\partial x_{i}}\Phi_{\lambda}^{(N-1)}(y)f(x,y)\Big)dy .
\end{split}
\end{equation*}
Using Proposition \ref{kernel_odd2even}, the first integral equals
\[ \int_{\mathbb{R}^n}Q_{\lambda_{n}}^{n,n}(x;y)\mathcal{H}^D_{n,\lambda_{n}}\Phi_{\lambda}^{(N-1)}(y)f(x,y)dy.\]
The third integral is given by
\[\int_{\mathbb{R}^n}Q_{\lambda_{n}}^{n,n}(x;y)\sum_{i=1}^n(-\lambda_{n}+e^{y_{i}-x_{i}}-e^{x_{i}-y_{i-1}}\mathbbm{1}_{i>1})\frac{\partial}{\partial x_{i}}\Phi_{\lambda}^{(N-1)}(y)f(x,y)dy\]
since, for $1\leq i \leq n$,
\[\dfrac{\partial}{\partial x_{i}}Q_{\lambda_{n}}^{n,n}(x;y)=(-\lambda_{n}+e^{y_{i}-x_{i}}-e^{x_{i}-y_{i-1}}\mathbbm{1}_{i>1})Q_{\lambda_{n}}^{n,n}(x;y).\]
Combining all the above calculations we conclude that 
\begin{equation*}
\begin{split}
(\mathcal{H}_{n}^B)_{x}\int_{\mathbb{R}^n}Q_{\lambda_{n}}^{n,n}(x;y)&\Phi_{\lambda}^{(N-1)}(y)f(x,y)dy\\
=& \int_{\mathbb{R}^n}Q_{\lambda_{n}}^{n,n}(x;y)V_{\lambda_{n}}^{n,n}\Phi_{\lambda}^{(N-1)}(y)f(x,y)dy.
\end{split}
\end{equation*}
and the intertwining relation follows.
\end{proof}
Now that we have established the intertwining relation, let us check the remaining assumptions of Theorem \ref{Kurtz}. We have that the domain of the operator $\mathcal{A}^N_{\lambda}$ is $\mathcal{D}(\mathcal{A}^N_{\lambda}) = C_{c}^2(\Gamma_{N})$, the set of continuously, twice differentiable, compactly supported functions on $\Gamma_{N}$, which is closed under multiplication, separates points and is convergence determining. Moreover, the drift $\mathbf{x}\mapsto (b^k_{i}(\mathbf{x}),1\leq i \leq [\frac{k+1}{2}], 1\leq k \leq N)$ where for $k=2l$
\[b^{k}_{i}(\mathbf{x}) =-\lambda_{l} +e^{x^{k-1}_{i}-x^k_{i}}-e^{x^k_{i}-x^{k-1}_{i-1}}\mathbbm{1}_{i>1},\quad  1\leq i \leq l\]
and for $k=2l-1$
\[b^{k}_{i}(\mathbf{x}) =\lambda_{l} +e^{x^{k-1}_{i}-x^k_{i}}\mathbbm{1}_{i<l} +e^{-x^k_{l}}\mathbbm{1}_{i=l} -e^{x^k_{i}-x^{k-1}_{i-1}}\mathbbm{1}_{i>1}, \quad 1\leq i \leq l\]
is locally Lipschitz. Therefore, for every $\mathbf{x}\in \Gamma_{N}$, there exists a unique process with evolution given by the differential equations \eqref{eq:SDEs_even} and \eqref{eq:SDEs_odd} that starts from $\mathbf{x}$.

The domain of the operator $\mathcal{L}^{B/D}_{n,\lambda}$ is $\mathcal{D}(\mathcal{L}^{B/D}_{n,\lambda}) = C^2_{c}(\mathbb{R}^n)$. We need to show that a process associated with the operator $\mathcal{L}^{B/D}_{n,\lambda}$ does not explode in finite time. Baudoin and O'Connell already handled the operator $\mathcal{L}^B_{n,\lambda}$ in Theorem 3.1 of \cite{Baudoin-O'Connell_2011}, so we will only focus on the operator $\mathcal{L}^D_{n,\lambda}$. 

A process associated with the operator $\mathcal{L}^D_{n,\lambda}$ cannot explode in finite time if we can find a Lyapunov function $U$ for the operator $\mathcal{L}^D_{n,\lambda}$ that satisfies the following properties
\begin{itemize}
\item $U(x)\to \infty$ as $x\to \infty$;
\item $\mathcal{L}_{n,\lambda}^DU\leq cU$, for some $c\geq 0$.
\end{itemize}

\noindent We consider the function
\[U(x) =\dfrac{ \exp \big( \sum_{i=1}^{n-1} 2\lambda_{i}x_{i}\big)\Phi^{(1)}_{\lambda_{n}}(x_{n})}{\Phi^{(2n-1)}_{\lambda}(x)}.\]
The first property follows from the exponential decay of the function $\Phi^{(2n-1)}$ we proved in Proposition \ref{exponential_decay}. For the second property we have
\begin{equation*}
\begin{split}
(\mathcal{L}^D_{n,\lambda}U)(x) &= \dfrac{1}{\Phi_{\lambda}^{(N)}(x)}\Big(\mathcal{H}^D_{n,\lambda} - \frac{1}{2}\sum_{i=1}^n \lambda_{i}^2\Big)\exp \big( \sum_{i=1}^{n-1} 2\lambda_{i}x_{i}\big)\Phi^{(1)}_{\lambda_{n}}(x_{n})\\
&\leq  \dfrac{1}{\Phi_{\lambda}^{(N)}(x)}\Big( \frac{1}{2}\Delta + e^{-x_{n}}\frac{\partial}{\partial x_{n}}-\lambda_{n}e^{-x_{n}}- \frac{1}{2}\sum_{i=1}^n \lambda_{i}^2\Big)\exp \big( \sum_{i=1}^{n-1} 2\lambda_{i}x_{i}\big)\Phi^{(1)}_{\lambda_{n}}(x_{n})\\
&= \frac{3}{2}\sum_{i=1}^{n-1} \lambda_{i}^2 U(x)
\end{split}
\end{equation*}
where the last equality follows from the eigenrelation in Proposition \ref{eigenrelation_phi_whittaker}.

Therefore, for $x\in \mathbb{R}^n$, $\{X_{t}\}_{t\geq 0}$ with
\[X_{t} = x + \int_{0}^t \nabla \log \Phi^{(2n-1)}_{\lambda}(X_{s})ds + B_{t}\]
where $\{B_{t}\}_{t\geq 0}$ is a standard Brownian motion on $\mathbb{R}^n$, is a process with generator $\mathcal{L}^D_{n,\lambda}$ started from $x$. The uniqueness of the process follows from the uniqueness of the process with generator $\mathcal{A}^N_{\lambda}$ for any starting point $\mathbf{x}\in \Gamma_{N}$ along with the intertwining of the operators and the assumption on the initial distribution.

\end{proof}

\section{An informal connection with $\mathfrak{gl}_{N}$-Whittaker process} 
In the previous section we proved that under certain initial conditions the bottom level of the pattern with evolution described by the system of SDEs \eqref{eq:SDEs_even}, \eqref{eq:SDEs_odd} evolves as a diffusion on $\mathbb{R}^n$ with infinitesimal generator $\mathcal{L}^B_{n}$ if the pattern has $2n$ levels and $\mathcal{L}^D_{n}$ if its contains $2n-1$ levels. In this section we will discuss how these processes may relate with the partition function of a semi-discrete random polymer and via this relation we will conclude with the $\mathfrak{gl}_{N}$-Whittaker process. 

Let $N=2n$ or $N=2n-1$ for some $n\in \mathbb{N}$. We fix a vector $\lambda=(\lambda_{1},...,\lambda_{n})\in \mathbb{R}^n$ and consider the independent standard Brownian motions $\beta_{1},...,\beta_{N}$ such that $\beta_{2i-1}$ has drift $\lambda_{i}$ and $\beta_{2i}$ has drift $-\lambda_{i}$. For $t\geq 0$, define the random variables
\[\mathcal{Z}^k(t)\equiv \mathcal{Z}^{k}_{\lambda}(t)=\log \int_{0<s_{1}<...<s_{k}<s_{k+1}=t}e^{\sum_{i=1}^k \beta_{i}(s_{i},s_{i+1})}ds_{1}...ds_{k}, \qquad 1\leq k \leq N,\]
where for $s<t$, $\beta(s,t):=\beta(t)-\beta(s)$. \\
The evolution of the Markov process $(\mathcal{Z}^k(t),1\leq k \leq N)_{t\geq 0}$ is described by the system of stochastic differential equations
\begin{equation*}
\begin{array}{ll}
 d\mathcal{Z}^1=d\beta_{1}+e^{-\mathcal{Z}^{1}}dt & \\
 d\mathcal{Z}^k=d\beta_{k} + e^{\mathcal{Z}^{k-1}-\mathcal{Z}^k}dt, & 2\leq k \leq N.
\end{array} 
\end{equation*}
So, the process $(\mathcal{Z}^k(t), 1\leq k \leq N, t\geq 0)$ has the same evolution as the process on the \enquote{edge} of the pattern we described in the previous section. 

We remark that $\mathcal{Z}^N$ is the logarithm of the partition function for a model of a directed random polymer with up/right paths from $(0,0)$ to $(t,N)$. The energy of a horizontal segment $[t,t+\delta t]$ at level $i$ is given by $\beta_{i}(t,t+\delta t)=\beta_{i}(t+\delta t)-\beta_{i}(t)$, if $i\geq 1$ and equals $0$ if $i=0$.

The $\mathfrak{gl}_{N}$-Whittaker process is a diffusion on $\mathbb{R}^N$ with infinitesimal generator 
\begin{equation}
\mathcal{L}^A_{N,\nu}=[\Psi_{\nu}^{\mathfrak{gl}_{N}}]^{-1}\big(\frac{1}{2}\sum_{i=1}^N \dfrac{\partial^2}{\partial x_{i}^2}-\sum_{i=1}^{N-1}e^{x_{i+1}-x_{i}}-\frac{1}{2}\sum_{i=1}^N \nu_{i}^2\big)\Psi_{\nu}^{\mathfrak{gl}_{N}}
\label{eq:generator}
\end{equation}
where $\Psi_{\nu}^{\mathfrak{gl}_{N}}(x)$ is the $\mathfrak{gl}_{N}$-Whittaker function parametrized by the vector $\nu \in \mathbb{R}^N$.\\
In \cite{O'Connell_2012}, it was proved that the Whittaker process is related with the partition function of the semi-discrete random polymer introduced in \cite{O'Connell_Yor} as follows.\\

\noindent For $\gamma_{1},...,\gamma_{N}$ standard, independent Brownian motions with drift $\nu_{1},...,\nu_{N}$ respectively and for $t>0$, define, for $1\leq k \leq N$, the random variables
\[\mathcal{Y}^{k}(t)=\log \int_{0<s_{1}<...<s_{k-1}<t}e^{\gamma_{1}(s_{1})+\gamma_{2}(s_{1},s_{2})+...+\gamma_{k}(s_{n-1},t)}ds_{1}...ds_{k-1}.\]
\begin{theorem}[\cite{O'Connell_2012}] Let $\nu \in \mathbb{R}^N$ and $X^A$ be a Whittaker process on $\mathbb{R}^N$ with infinitesimal generator $\mathcal{L}_{N,\nu}$ such that $\lim_{t\to 0}(X_{i}^A(t)-X_{i+1}^A(t))=-\infty$, it follows that 
\begin{equation}\mathcal{Y}^{N}\overset{\mathcal{L}}{=}X_{1}^A .\label{eq:connection_2}\end{equation}
\label{comparison_2}
\end{theorem}

Let us prove a connection between the partition functions two polymers $\mathcal{Z}^N$ and $\mathcal{Y}^N$
\begin{theorem}
Let $\nu \in \mathbb{R}^n$ and let $\overleftarrow{\nu} = (\nu_{N},...,\nu_{1})$. Suppose that, if $N=2n$ then $\nu$ is chosen to satisfy $\overleftarrow{\nu} = (\lambda_{1},-\lambda_{1},...,\lambda_{n},-\lambda_{n})$ and if $N=2n-1$ then $\nu$ is chosen to satisfy $\overleftarrow{\nu} = (\lambda_{1},-\lambda_{1},...,\lambda_{n})$. Then for fixed $t \geq 0$
\begin{equation}
\mathcal{Z}^N(t)\overset{d}{=}\log \int_{0}^t e^{\mathcal{Y}^{N}(s)}ds
\label{eq:law_of_polymers}.
\end{equation}
\end{theorem}
\begin{proof}
Recall that $\mathcal{Z}^N(t)$ has the integral representation
\[\mathcal{Z}^N(t)=\log \int_{0<t_{1}<...<t_{N}<t_{N+1}= t}e^{\sum_{i=1}^N\beta_{i}(t_{i},t_{i+1})}dt_{1}...dt_{N}\]
If we set $s_{i}=t-t_{N-i+1}$, then 
\begin{equation*}
\begin{split}
\mathcal{Z}(t)&=\log \int_{0= s_{0}<s_{1}<...<s_{N}<t}e^{\sum_{i=1}^N\beta_{i}(t-s_{N-i+1},t-s_{N-i})}ds_{1}...ds_{N}\\
&=\log \int_{0= s_{0}<s_{1}<...<s_{N}< t}e^{\sum_{i=1}^N\beta_{i}(t-s_{i},t-s_{i-1})}ds_{1}...ds_{N}
\end{split}
\end{equation*}
set $\hat{\beta}_{j}(s):= \beta_{N-j+1}(t)-\beta_{N-j+1}(t-s)$, then $\hat{\beta}_{j}$ is a standard Brownian motion with drift $\bar{\lambda}_{N-j+1}$, where $\bar{\lambda_{2i}}=-\bar{\lambda}_{2i-1}=-\lambda_{i}$. So $\mathcal{Z}^N_{t}$ is rewritten as
\[\mathcal{Z}^N(t)=\log \int_{0= s_{0}<s_{1}<...<s_{N}<t}e^{\sum_{i=1}^N\hat{\beta}_{i}(s_{i-1},s_{i})}ds_{1}...ds_{N}.\]
Recall that $Y_{N}(t)$ is given by 
\[\mathcal{Y}^{N}(t)=\log \int_{0<s_{1}<...<s_{N-1}<t}e^{\gamma_{1}(s_{1})+\gamma_{2}(s_{1},s_{2})+...+\gamma_{N}(s_{n-1},t)}ds_{1}...ds_{N-1}\]
So if we choose the drift $\nu$ to satisfy $\nu_{i}=\bar{\lambda}_{N-i+1}$, we find $\hat{\beta}_{i}\overset{d}{=} \gamma_{i}$ and we conclude the result.
\end{proof} 

We observe that for $1\leq k \leq N$,
\[\lim_{t\to 0}\exp(\mathcal{Z}^k(t))=\lim_{t \to 0}\int_{0<s_{1}<...<s_{k}<s_{k+1}=t}e^{\sum_{i=1}^k \beta_{i}(s_{i},s_{i+1})}ds_{1}...ds_{k}=0, \quad a.s.\]
so the process $\mathcal{Z}$ starts from $-\infty$.
At this stage we would like to be able to say that the process on $\mathbf{X}$ can be started with all the particles at \enquote{$-\infty$} and even with this special initial condition the result in Theorem \ref{first_result} holds. Unfortunately, we have not managed to prove Theorem \ref{first_result} for general starting point. Let us assume for the sake of argument that the following holds.

\begin{conjecture}Let $X^B, X^D$ be two diffusions on $\mathbb{R}^n$ with infinitesimal generators $\mathcal{L}^B_{n,\lambda}$, $\mathcal{L}_{n,\lambda}^D$, respectively. If $\lim_{t\to 0}(X^{B/D}_{i}(t)-X^{B/D}_{i+1}(t))=-\infty$, for $1\leq i \leq n$, then 
\begin{equation}
\mathcal{Z}^{N}\overset{\mathcal{L}}{=} \left\{ \begin{array}{ll}
X^B_{1} & ,N=2n\\
X^D_{1} & ,N=2n-1.
\end{array} \right.
\label{eq:connection_1}
\end{equation} 
\label{conjecture_polymers}
\end{conjecture}
Then, using relations \eqref{eq:connection_2} and  \eqref{eq:law_of_polymers} we obtain the following result.
\begin{corollary} 
\label{connection_so_gl}
Let $X^B, X^D$ be two diffusions on $\mathbb{R}^n$ with generators $\mathcal{L}_{n,\lambda}^B$, $\mathcal{L}^D_{n,\lambda}$ satisfying the initial condition $\lim_{t\to 0}(X_{i}^{B/D}(t)-X_{i+1}^{B/D})(t)=-\infty$. Let $X^A$ be a diffusion on $\mathbb{R}^N$ with infinitesimal generator $\mathcal{L}_{N,\nu}$ and initial condition $\lim_{t\to 0}(X_{i}^A(t)-X_{i+1}^A)(t)=-\infty$. Assuming that conjecture \ref{conjecture_polymers} holds. we have the following\\
(i) If $N=2n$ and $\nu$ is chosen to satisfy $\overleftarrow{\nu} = (\lambda_{1},-\lambda_{1},...,\lambda_{n},-\lambda_{n})$ then for fixed $t \geq 0$
\[X_{1}^{(B)}(t)\overset{d}{=} \log \int_{0}^t e^{X_{1}^A(s)}ds.\]
(ii) if $N=2n-1$ and $\nu$ is chosen to satisfy $\overleftarrow{\nu} = (\lambda_{1},-\lambda_{1},...,\lambda_{n})$ then for fixed $t \geq 0$
\[X_{1}^{(D)}(t)\overset{d}{=} \log \int_{0}^t e^{X_{1}^A(s)}ds.\]
In both statement $\overleftarrow{\nu} = (\nu_{N},...,\nu_{1})$.
\end{corollary}

We remark that the result in Corollary \ref{connection_so_gl} is the positive temperature analogue of the result in \cite{Borodin_et_al_2009} where the authors proved that the running maximum of the first coordinate of a type $A$ $N$-dimensional Dyson's Brownian motion is distributed as the first coordinate of an $n$-dimensional Dyson's Brownian motion of type $B$ if $N=2n$ and of type $C$ if $N=2n-1$.